\documentclass[a4paper,11pt]{report}
\usepackage{amssymb,amsbsy,amsmath,amsfonts,amssymb,amscd,amsthm}
\usepackage{latexsym}
\usepackage{geometry}
\usepackage{graphics}
\usepackage{color}
\usepackage{comment}
\usepackage{hyperref}
\usepackage{multirow}
\usepackage{framed}
\usepackage{tikz}
\usetikzlibrary{shapes.geometric, arrows,decorations.pathreplacing,calligraphy}
\usepackage{tikz-cd}
\usepackage{booktabs}
\usepackage{longtable}
\usepackage{enumitem}
\usepackage[toc,page]{appendix}
\usepackage{listings}
\usepackage{tabularx}
\usepackage{caption}
\usepackage{cleveref}
\usepackage{titletoc}
\input xy
\xyoption{all}

\theoremstyle{plain}
\newtheorem{thm}{Theorem}[section]
\newtheorem{theorem}[thm]{Theorem}

\newtheorem{lemma}[thm]{Lemma}

\newtheorem{prop}[thm]{Proposition}
\newtheorem{cor}[thm]{Corollary}
\theoremstyle{definition}
\newtheorem{rem}[thm]{Remark}

\newtheorem{defin}[thm]{Definition}

\newtheorem{example}[thm]{Example}

\newtheorem{question}[thm]{Question}
\newtheorem{problem}[thm]{Problem}
\newtheorem{questions}[thm]{Questions}

\newtheorem{script}[thm]{GAP Code}
\newtheorem*{notation}{Notation}
\newtheorem*{notationfacts}{Notation and Facts}
\newtheorem*{facts}{Facts}

\numberwithin{equation}{section}
\newcommand\Oh{{\mathcal O}}

\newcommand{\sC}{{\mathcal C}}

\newcommand{\CC}{\ensuremath{\mathbb{C}}}
\newcommand{\RR}{\ensuremath{\mathbb{R}}}
\newcommand{\ZZ}{\ensuremath{\mathbb{Z}}}
\newcommand{\QQ}{\ensuremath{\mathbb{Q}}}

\newcommand{\NN}{\ensuremath{\mathbb{N}}}

\newcommand\om{\omega}
\newcommand\la{\lambda}

\newcommand\al{\alpha}

\newcommand{\Lam}{\Lambda}

\DeclareMathOperator{\Fix}{Fix}

\DeclareMathOperator{\Hom}{Hom}

\DeclareMathOperator{\GL}{GL}
\DeclareMathOperator{\diag}{diag}
\DeclareMathOperator{\im}{im}
\DeclareMathOperator{\Eig}{Eig}

\DeclareMathOperator{\mult}{mult}

\DeclareMathOperator{\ord}{ord}
\DeclareMathOperator{\id}{id}
\DeclareMathOperator{\Heis}{Heis}

\DeclareMathOperator{\lcm}{lcm}
\DeclareMathOperator{\Dic}{Dic}
\DeclareMathOperator{\triv}{triv}
\DeclareMathOperator{\Irr}{Irr}
\DeclareMathOperator{\Stab}{Stab}
\DeclareMathOperator{\ab}{ab}
\DeclareMathOperator{\res}{res}

\DeclareMathOperator{\tr}{tr}
\DeclareMathOperator{\Hol}{Hol}
\DeclareMathOperator{\kod}{kod}

\newcommand{\FF}{\ensuremath{\mathbb{F}}}

\def\eea{\end{eqnarray*}}
\def\bea{\begin{eqnarray*}}

\newcommand\dual{\mathrel{\raise3pt\hbox{$\underline{\mathrm{\thinspace d
\thinspace}}$}}}
\newcommand\qe{\ifhmode\unskip\nobreak\fi\quad $\Box$}       % box for QED

\def\BOX{\hfill\lower.5\baselineskip\hbox{$\Box$}}
\newtheorem{theo}{Theorem}
\theoremstyle{plain}
\newtheorem{mainthm}[theo]{Main Theorem}

\DeclareMathOperator{\Aut}{Aut}

\DeclareMathOperator{\Bihol}{Bihol}

\DeclareMathOperator{\SL}{SL}

\DeclareMathOperator{\isog}{isog.}

\lstdefinelanguage{GAP}{%
	morekeywords={%
		Assert,Info,IsBound,QUIT,%
		TryNextMethod,Unbind,and,break,%
		continue,do,elif,%
		else,end,false,fi,for,%
		function,if,in,local,%
		mod,not,od,or,%
		quit,rec,repeat,return,%
		then,true,until,while%
	},%
	sensitive,%
	morecomment=[l]\#,%
	morestring=[b]",%
	morestring=[b]',%
}[keywords,comments,strings]

\newcommand{\End}{{\rm End}}
\newcommand{\Gal}{{\rm Gal}}

\hypersetup{colorlinks=true, unicode=true, linkcolor=[rgb]{0.10,0.05,0.67}, citecolor=[rgb]{0.10,0.05,0.67}, filecolor=[rgb]{0.10,0.05,0.67}, urlcolor=[rgb]{0.10,0.05,0.67}} %hyperlinks

\title{The Classification of Hyperelliptic Groups in Dimension $4$}	
\author{Andreas Demleitner \\ \ \\ University of Freiburg \\ Ernst-Zermelo-Str. 1 \\ 79104 Freiburg im Breisgau \\ Germany \\ \ \\ andreas.demleitner@math.uni-freiburg.de \\ \ \\ \emph{2010 Mathematics Subject Classification:} \\
	 Primary: 14L30, 14J10, 32Q57\\
	 Secondary: 20H15, 20C15, 53C29 \\ \ \\
\emph{Keywords:}\\ Hyperelliptic Surfaces, Hyperelliptic Manifolds, Crystallographic Groups, Bieberbach Groups, \\ Holonomy Representation, Complex Tori}

\begin{document}
\makeatletter
\renewcommand*\l@subsection{\@dottedtocline{2}{1.8em}{3.6em}}
\makeatother

	\maketitle
	\tableofcontents

\chapter{Introduction}

Hyperelliptic surfaces appear classically in the Kodaira-Enriques classification of compact complex surfaces as the surfaces $S$, which are uniquely determined by the invariants $\kod(S) = 0$, $p_g(S) = 0$ and $q(S) = 1$. Furthermore, $12K_S$ is linearly trivial. Due to the work of Enriques-Severi \cite{EnrSev09}, \cite{EnrSev10}, and Bagnera-de Franchis \cite{BdF}, these surfaces are well understood. They can equivalently be described as quotients of an Abelian surface $A$ by the action of a non-trivial finite group $G$, which acts freely on $A$ and contains no translations. The latter characterization allows a full classification of hyperelliptic surfaces, which was carried out in the aforementioned work by Enriques-Severi and Bagnera-de Franchis. The classification shows in particular that the Abelian surface $A$ is isogenous to a product of two elliptic curves and that the group $G$ is cyclic of order $2$, $3$, $4$, or $6$. \\ Higher-dimensional analogs -- \emph{hyperelliptic manifolds} -- were first introduced by Lange \cite{Lange} and are defined as quotients of complex tori by actions of finite groups satisfying the same properties as in the $2$-dimensional case. Unlike the case of surfaces, a classification of higher-dimensional hyperelliptic manifolds is much more involved. This is mainly for two reasons:
\begin{enumerate}[ref=(\theenumi)]
	\item \label{intro-1} First of all, it is a difficult algebraic problem to determine those finite, non-trivial groups $G$ admitting a free and translation-free holomorphic action on a complex torus $T$ of dimension $n$. Such a group $G$ will be called \emph{hyperelliptic in dimension $n$}. 
	\item \label{intro-2} Secondly, the task of determining all free and translation-free actions and classifying the quotients up to isomorphism is much more involved in higher dimensions.
\end{enumerate}
Progress on step \ref{intro-1} has first been made by Uchida-Yoshihara \cite{Uchida-Yoshihara}, who showed that if $G$ is hyperelliptic in dimension $3$, then $G \cong D_4$ (the dihedral group of order $8$), or $G$ is a direct product of two cyclic groups (where they give an explicit list of $16$ possible products). Later, Lange \cite{Lange} and Catanese-Demleitner \cite{CD-Hyp3} showed that the $17$ groups determined by Uchida-Yoshihara are indeed all hyperelliptic in dimension $3$. Furthermore, they obtained classification results for hyperelliptic threefolds in the sense of \ref{intro-2}. \\
Except in special cases, there have been no results in dimension $4$ so far. This monograph aims to perform step \ref{intro-1} in the $4$-dimensional case. Our main result is as follows:

\begin{mainthm} \label{mainthm}
	Let $G$ be a hyperelliptic group in dimension $4$. Then $G$ is isomorphic to one of the $79$ groups listed in \hyperref[table:main]{Table~\ref*{table:main}} below. Conversely, every group contained in \hyperref[table:main]{Table~\ref*{table:main}} is hyperelliptic in dimension $4$.
	\begin{center}
		\resizebox{0.8\textwidth}{!}{
			{\def\arraystretch{1.4}
				\begin{tabularx}{\textwidth}{l||c||r}
					\begin{tabular}{l|l}
						ID & Group Name \\ \hline \hline 
						$[2,1]$ & $C_2$ \\
						$[3,1]$ & $C_3$ \\
						$[4,1]$ & $C_4$ \\
						$[5,1]$ & $C_5$ \\
						$[4,2]$ & $C_2 \times C_2$\\
						$[6,1]$ & $S_3$ \\
						$[6,2]$ & $C_6$  \\
						$[7,1]$ & $C_7$\\
						$[8,1]$ & $C_8$  \\
						$[8,2]$ & $C_2 \times C_4$  \\
						$[8,3]$ & $D_4$ \\
						$[8,4]$ & $Q_8$ \\
						$[8,5]$ & $C_2 \times C_2 \times C_2$  \\
						$[9,1]$ & $C_9$  \\
						$[9,2]$ & $C_3 \times C_3$ \\
						$[10,2]$ & $C_{10}$  \\
						$[12,1]$ & $G(3,4,2)$ \\
						$[12,2]$ & $C_{12}$ \\
						$[12,3]$ & $A_4$ \\
						$[12,4]$ & $S_3 \times C_2$ \\
						$[12,5]$ & $C_2 \times C_6$  \\
						$[14,2]$ & $C_{14}$ \\
						$[15,1]$ & $C_{15}$ \\
						$[16,2]$ & $C_4 \times C_4$  \\
						$[16,3]$ & $(C_4 \times C_2)\rtimes C_2$ \\
						$[16,4]$ & $G(4,4,3)$  \\
						$[16,5]$ &  $C_2 \times C_8$ 
					\end{tabular}
					&
					\begin{tabular}{l|l}
						ID & Group Name \\ \hline \hline 
						$[16,6]$ & $G(8,2,5)$\\
						$[16,8]$ & $G(8,2,3)$ \\
						$[16,10]$ & $C_2 \times C_2 \times C_4$ \\
						$[16,11]$ & $D_4 \times C_2$\\
						$[16,13]$ & $D_4 \curlyvee C_4$ \\
						$[18,2]$ & $C_{18}$  \\
						$[18,3]$ & $S_3 \times C_3$ \\
						$[18,5]$ & $C_3 \times C_6$ \\
						$[20,2]$ & $C_{20}$ \\
						$[20,5]$ & $C_2 \times C_{10}$ \\
						$[24,1]$ & $G(3,8,2)$   \\
						$[24,2]$ & $C_{24}$ \\
						$[24,5]$ & $S_3 \times C_4$ \\
						$[24,8]$ & $(C_2 \times C_6) \rtimes C_2$  \\
						$[24,9]$ & $C_2 \times C_{12}$ \\
						$[24,10]$ & $D_4 \times C_3$ \\
						$[24,11]$ & $Q_8 \times C_3$  \\
						$[24,13]$ & $A_4 \times C_2$ \\
						$[24,15]$ & $C_2 \times C_2 \times C_6$ \\
						$[27,3]$ & $\Heis(3)$ \\
						$[27,5]$ & $C_3 \times C_3 \times C_3$ \\
						$[30,4]$ & $C_{30}$ \\
						$[32,3]$ & $C_4 \times C_8$  \\
						$[32,4]$ & $G(8,4,5)$  \\
						$[32,11]$ & $(C_4 \times C_4) \rtimes C_2$  \\
						$[32,21]$ & $C_2 \times C_4 \times C_4$  \\
						$[32,24]$ & $(C_4 \times C_4) \rtimes C_2$
					\end{tabular}
					&
					\begin{tabular}{l|l}
						ID & Group Name \\ \hline \hline 
						$[32,37]$ & $G(8,4,5) \times C_2$ \\
						$[36,6]$ & $G(3,4,2) \times C_3$ \\
						$[36,8]$ & $C_3 \times C_{12}$ \\
						$[36,12]$ & $S_3 \times C_6$  \\
						$[36,14]$ & $C_6 \times C_6$ \\
						$[40,9]$  & $C_2 \times C_{20}$  \\
						$[48,20]$ & $C_4 \times C_{12}$\\
						$[48,21]$ & $((C_4 \times C_2) \rtimes C_2) \times C_3$  \\
						$[48,22]$ & $G(4,4,3) \times C_3$  \\
						$[48,23]$ & $C_2 \times C_{24}$ \\
						$[48,31]$ & $A_4 \times C_4$\\
						$[48,44]$ & $C_2 \times C_2 \times C_{12}$ \\
						$[54,12]$ & $S_3 \times C_3 \times C_3$  \\
						$[54,15]$ & $C_3 \times C_3 \times C_6$ \\
						$[60,13]$ & $C_2 \times C_{30}$ \\
						$[72,12]$ & $G(3,8,2) \times C_3$\\
						$[72,27]$ & $S_3 \times C_{12}$  \\
						$[72,30]$ & $((C_2 \times C_6) \rtimes C_2) \times C_3$ \\
						$[72,36]$ & $C_6 \times C_{12}$ \\
						$[72,50]$ & $C_2 \times C_6 \times C_6$ \\
						$[96,161]$ & $C_2 \times C_4 \times C_{12}$\\
						$[108,32]$ & $G(3,4,2) \times C_3 \times C_3$ \\
						$[108,42]$ & $S_3 \times C_6 \times C_3$ \\
						$[108,45] $ & $C_3 \times C_6 \times C_6$ \\
						$[144,178]$ & $C_2 \times C_6 \times C_{12}$ \\
						\ & \  \\
						\ & \ 
					\end{tabular}
				\end{tabularx}
		}}
		\captionof{table}{Shows the hyperelliptic groups in dimension $4$ together with their IDs in the Database of Small Groups. Here, $C_m$ is the cyclic group of order $m$, $G(m,n,r) =  \langle g,h ~ | ~ g^m = h^m = 1, ~ h^{-1}gh = g^r \rangle$, and $D_4$ resp. $Q_8$ denote the dihedral resp. quaternion group of order $8$. Finally, $D_4 \curlyvee C_4$ is the central product of $D_4$ and $C_4$.} \label{table:main}
	\end{center}
\end{mainthm}

Similarly to Uchida-Yoshihara's work, the proof of \hyperref[mainthm]{Main Theorem~\ref*{mainthm}} is mainly based on group- and representation-theoretic methods. A big part of the proof will be to obtain an adequate bound on the group order: this will then allow the use of the computer algebra system GAP \cite{GAP} to search the Database of Small Groups, which contains all groups up to order $2000$ (with the exception of groups of order $1024$). \\

Our second main result concerns the canonical divisor of a hyperelliptic manifold. As mentioned in the beginning, the canonical divisor of a hyperelliptic surface is $12$-torsion. Since $p_g(S) = h^0(S,\Oh_S(K_S)) = 0$, the canonical divisor $K_S$ of a hyperelliptic surface $S$ is non-trivial. Moreover, a closer investigation of hyperelliptic surfaces reveals that one of $4K_S$ or $6K_S$ is always trivial. Thus, the number $\lcm(4,6) =12$ is the minimal number $m$ with the property that $mK_S$ is trivial for any hyperelliptic surface $S$. We prove analogous results in dimensions $3$ and $4$:

\begin{mainthm} \label{mainthm-can-div}
	For $n \geq 2$, set
	\begin{align*}
	\tau(n) := \min\{m \geq 1 ~ | ~ m K_X \text{ is trivial for any hyperelliptic manifold } X \text{ of dimension } n\}.
	\end{align*}
	Then
	\begin{align*}
	\tau(2) = 12, \qquad \tau(3) = 60, \qquad \tau(4) = 1260.
	\end{align*}
\end{mainthm}

The structure of the text is as follows: \\

\hyperref[chapter:grouptheory]{Chapter~\ref*{chapter:grouptheory}} is used to recall basic notions from group and representation theory. In \hyperref[chapter:prerequisities]{Chapter~\ref*{chapter:prerequisities}}, we review complex tori and their associated rational and complex representations, introduce hyperelliptic manifolds and explain how a group action on a complex torus induces a decomposition up to isogeny. \hyperref[chapter:properties]{Chapter~\ref*{chapter:properties}} serves as the start of our investigation of hyperelliptic groups in dimension $4$. First structural results are shown, for instance, that the order of a hyperelliptic group in dimension $4$ is $$2^a \cdot 3^b \cdot 5^c \cdot 7^d,$$ where $c$ and $d$ are at most $1$ -- this will play a crucial role throughout the entirety of the text. Before we move on to proving \hyperref[mainthm]{Main Theorem~\ref*{mainthm}}, we explain the strategy for the classification in \hyperref[chapter:outline]{Chapter~\ref*{chapter:outline}}. Part of the classification result is shown in \hyperref[chapter:abelian]{Chapter~\ref*{chapter:abelian}}: this chapter is dedicated to the investigation of Abelian groups as hyperelliptic groups (in any dimension). It culminates in the classification of Abelian hyperelliptic groups in dimension $4$. Moving on to the non-Abelian case, we obtain information about the $2$- and $3$-Sylow subgroups of a hyperelliptic group in dimension $4$, in \hyperref[section-2-sylow]{Chapters~\ref*{section-2-sylow}} and \ref{section-3-sylow}, respectively: in particular, we prove that $a \leq 7$ (which will be improved to $a \leq 5$ later on) and $b \leq 3$. \hyperref[chapter:sylow-alternative]{Chapter~\ref*{chapter:sylow-alternative}} describes a different, more computer-algebraic approach to describing the $2$- and $3$-Sylow subgroups, which was communicated to the author by Christian Gleissner. The greatest part of \hyperref[mainthm]{Main Theorem~\ref*{mainthm}} will be shown in \hyperref[chapter:2a3b]{Chapter~\ref*{chapter:2a3b}}, where we classify the groups of order $2^a \cdot 3^b$, which are hyperelliptic in dimension $4$: here, our knowledge about the $2$- and $3$-Sylow subgroups obtained in previous chapters is heavily used. Finally, the last part of \hyperref[mainthm]{Main Theorem~\ref*{mainthm}} will be shown in \hyperref[chapter:5-or-7-divides-G]{Chapter~\ref*{chapter:5-or-7-divides-G}}, where we show that if the order of a hyperelliptic group in dimension $4$ is divisible by $5$ or $7$, then it is Abelian. Since the Abelian hyperelliptic groups in dimension $4$ were already classified, this will finish the proof of \hyperref[mainthm]{Main Theorem~\ref*{mainthm}}. The canonical divisor of a hyperelliptic manifold will be discussed in \hyperref[chapter:canonical-divisor]{Chapter~\ref*{chapter:canonical-divisor}}, which will then yield the proof of  \hyperref[mainthm-can-div]{Main Theorem~\ref*{mainthm-can-div}}. In the last chapter, \hyperref[chapter:finalremarks]{Chapter~\ref*{chapter:finalremarks}}, we elaborate on several open problems. \hyperref[appendix:gap]{Appendix~\ref*{appendix:gap}} contains the GAP codes that were used throughout the text.

\chapter{A Short Group-Theoretic Account} \label{chapter:grouptheory}

As the method of proof of our classification result \hyperref[mainthm]{Main Theorem~\ref*{mainthm}} is group-theoretic in nature, we recall basic facts from group and representation theory. 

\begin{notation}
	Let $G$ be a finite group.
	\begin{itemize}
		\item We denote the order of $G$ by $|G|$. The \emph{exponent} of $G$ is the least common multiple of all the orders of elements of $G$.
		\item By $Z(G) = \{ g \in G ~ | ~ gh = hg \text{ for all } h \in G\}$ we denote the \emph{center} of $G$. It is a normal subgroup of $G$, and $G$ is Abelian if and only if $Z(G) = G$.
		\item $N_G(H) = \{g \in G ~ | ~ g^{-1}Hg = H\}$ denotes the \emph{normalizer} of a subgroup $H$ in $G$. Clearly, $H \subset N_G(H)$, and $N_G(H)$ is the largest subgroup of $G$ in which $H$ is normal. The subgroup $H$ is normal itself if and only if $N_G(H) = G$.
		\item We denote by $[G,G]$ the \emph{commutator subgroup} (or \emph{derived subgroup}) of $G$. By definition, $[G,G]$ is the subgroup of $G$ spanned by all elements of the form $ghg^{-1}h^{-1}$. It is a normal subgroup with the property that the factor group $G^{\ab} = G/[G,G]$ is Abelian. The group $G^{\ab}$ is called the \emph{Abelianization} of $G$. Furthermore, the Abelianization satisfies the following universal property: if $f \colon G \to A$ is a group homomorphism of $G$ to an Abelian group, then there is a unique group homomorphism $\overline{f} \colon G^{\ab} \to A$ such that $f = \overline{f} \circ \pi$, where $\pi \colon G \to G^{\ab}$ is the canonical projection.
		\begin{center}
			\begin{tikzcd}
			G  \arrow[r, "f"] \arrow[rd, "\pi"] & A \\
			& G^{\ab} \arrow[u, "\overline{f}"]
			\end{tikzcd}
		\end{center}
		\item Two elements $g, g' \in G$ are \emph{conjugate} if there is $h \in G$ such that $h^{-1}gh = g'$. Being conjugate defines an equivalence relation on $G$. The equivalence classes are called the \emph{conjugacy classes} of $G$.
		
		\item Recall \emph{Sylow's Theorems}: they assert that if $p$ is prime and $n \geq 1$ is the largest number such that $|G|$ is divisible by $p^n$, then:
		\begin{itemize}
			\item $G$ contains subgroups of order $p^i$ for every $1 \leq i \leq n$; for $i = n$, these are called the \emph{$p$-Sylow subgroups} of $G$,
			\item if $P, P' \subset G$ are $p$-Sylow subgroups of $G$, then there is $g \in G$ such that $g^{-1}Pg = P'$, i.e., $P$ and $P'$ are conjugate; in particular, if $G$ contains only one $p$-Sylow subgroup, then it is normal,
			\item the number of $p$-Sylow subgroups of $G$ divides $|G|/p^n$ and is $\equiv 1 \pmod p$.
		\end{itemize}
		\item The \emph{structure theorem for finitely generated Abelian groups} states that every finite Abelian group is isomorphic to a group of the form
		\begin{align*}
		C_{d_1} \times ...\times C_{d_k}, \qquad \text{ where } d_i ~ | ~ d_{i+1}
		\end{align*}
		and $C_d$ denotes the cyclic group of order $d$.
	\end{itemize}
\end{notation}

\begin{notationfacts}[on group actions]
	Let $G$ be a finite group and $X$ a set.
	\begin{itemize}
		\item By an \emph{action} of $G$ on $X$ we mean a map $\mu \colon G \times X \to X$ such that $\mu(1,x) = x$ and $\mu(g,\mu(h,x)) = \mu(gh,x)$ for all $x \in X$ and $g,h \in G$. We often write $g(x)$ instead of $\mu(g,x)$ if the context is clear.
		\item Given a group action of $G$ on $X$, we obtain a group homomorphism of $G$ into the symmetric group $S_X$ of $X$, which maps $g$ to $g(x)$. Conversely, given a group homomorphism $f \colon G \to S_X$, we obtain a group action of $G$ on $X$ via $g(x) = f(g)(x)$. 
		\item  An action of an element $g \in G$ on $X$ is called \emph{free} (or \emph{fixed point free}) if the equation $g(x) = x$ has no solution in $x \in X$. The action of $G$ on $X$ is called \emph{free} if every non-trivial element of $G$ acts freely, i.e., if $g(x) = x$ for some $x \in X$ implies $g = 1$. The freeness of $g \in G$ only depends on the conjugacy class of $G$, since if $g(x) = x$, then $(h^{-1}gh)(h^{-1}(x)) = h^{-1}(x)$ for all $h \in G$.
		\item An action of $G$ on $X$ is called \emph{faithful} if $g(x) = x$ for all $x \in X$ implies $g = 1$. Faithfulness is equivalent to the homomorphism $G \to S_X$ being injective.
	\end{itemize}
\end{notationfacts}

Next, we recall basic facts concerning the solvability of groups.

\begin{facts}[on solvable groups] \
	\begin{itemize}
		\item A finite group $G$ is called \emph{solvable} if its derived sequence $G^{(0)} := G$, $G^{(i+1)} := [G^{(i)}, G^{(i)}]$ terminates in the trivial group, i.e., if there is $N \geq 1$ such that $G^{(N)} = \{1\}$.
		\item All subgroups and quotients of solvable groups are solvable. Abelian groups and $p$-groups are solvable. \\ 
		
		\item Burnside's famous $p^aq^b$-Theorem asserts more generally:
		
		\begin{thm} \cite[15.3 Theorem]{Huppert} \label{thm:burnside}
			If $|G| = p^a q^b$, where $p$, $q$ are primes, then $G$ is solvable.
		\end{thm}
		
		Thus the group order of a non-solvable group must be divisible by at least three distinct primes.
		
		\item There exists an important generalization of Sylow's Theorems to solvable groups, which is due to P. Hall:
		
		\begin{theorem}[abridged.] \cite[Theorem 9.3.1]{Hall} \label{thm:hall}
			Let $G$ be a finite solvable group. If $|G| = mn$, $\gcd(m,n) = 1$, then $G$ contains a subgroup of order $m$ and all such subgroups are conjugate in $G$.
		\end{theorem}
	\end{itemize}
	
\end{facts}

Furthermore, we shall need basic facts from character theory, which we will recall below. We refer to the books \cite{Huppert} and \cite{Isaacs} for detailed treatments.

\begin{notationfacts}
	Let $G$ be a finite group.
	\begin{itemize}
		\item A \emph{(complex) representation} of $G$ is a finite-dimensional $\CC$-vector space $V$ together with a group homomorphism $\rho \colon G \to \GL(V)$. In other words, $G$ acts on $V$ by linear maps. The dimension of $V$ is called the \emph{degree} of $\rho$. The representation $\rho$ is called \emph{trivial}, if $\rho(g) = \id_V$ for all $g \in G$.
		\item A representation $\rho \colon G \to \GL(V)$ is called \emph{faithful}, if it is injective, i.e., if the action of $G$ on $V$ is faithful.
		\item The \emph{character} of a representation $\rho \colon G \to \GL(V)$ is the map that maps $g$ to the trace of $\rho(g)$. Characters of representations are often denoted by the letter $\chi$. Characters are only group homomorphisms if the degree of $\rho$ is $1$, in which case $\rho$ equals its character. We will thus mean a representation of degree $1$ if we speak of a \emph{linear character}. The restriction of a linear character to the derived subgroup $[G,G]$ is trivial.
		\item Two representations $\rho \colon G \to \GL(V)$ and $\rho' \colon G  \to \GL(W)$ are called \emph{equivalent} if there is an \emph{$G$-equivariant} isomorphism $f \colon V \to W$, i.e.,
		\begin{align*}
		f^{-1} \circ \rho'(g) \circ f = \rho(g) \qquad \text{ for all } g \in G.
		\end{align*}
		\item Let $\rho \colon G \to \GL(V)$ be a representation. A linear subspace $U \subset V$ is called \emph{$G$-stable}, if $\rho(g)|_U \in \GL(U)$. A \emph{sub-representation} of $\rho$ is $\rho|_U \colon G \to \GL(U)$ for a $G$-stable subspace $U$.
		The representation $\rho$ is called \emph{irreducible}, if $\dim(V) \geq 1$ and its only $G$-stable subspaces are $\{0\}$ and $V$.
		\item	The character of a representation depends only on its equivalence class. Thus, it makes sense to speak of the degree, irreducibility, and faithfulness of a character of $G$. The set of irreducible characters of $G$ will be denoted $\Irr(G)$. The cardinality of $\Irr(G)$ equals the number of conjugacy classes of $G$, and 
		\begin{align*}
		|G| = \sum_{\chi \in \Irr(G)}  \chi(1)^2.
		\end{align*}
		\item Degrees of irreducible characters divide the group order. More generally, we will make frequent use of the following result due to N. Ito:
		
		\begin{thm}[N. Ito's Degree Theorem] \cite[19.9 Theorem]{Huppert}\label{thm:huppert-degree}
			Let $G$ be a finite group containing a normal Abelian subgroup $A$ and $\chi \in \Irr(G)$. Then the degree $\chi(1)$ divides the index of $A$ in $G$. 
		\end{thm}
		
		\item Every representation $\rho$ is equivalent to a direct sum of irreducible sub-representations. While these sub-representations are not unique, the equivalence classes of the sub-representations are. Hence the character $\chi$ of $\rho$ can be uniquely written as a sum of irreducible characters. This is called the \emph{isotypical decomposition} of $\chi$. %{\ really?}
		
		\item Schur's Lemma \cite[2.3 Theorem]{Huppert} implies that if $\rho \colon G \to \GL(V)$ is irreducible and $g \in Z(G)$, then $\rho(g)$ is a multiple  of the identity. It follows that if $Z(G)$ is non-cyclic, then no irreducible representation of $G$ is faithful (see also \cite[2.5 Proposition]{Huppert}). Albeit false in general, the converse holds for $p$-groups:
		
		\begin{thm}\cite[7.2 Theorem]{Huppert} \label{thm:huppert:p-groups}
			A $p$-group has a faithful irreducible character if and only if its center is cyclic.
		\end{thm}
	\end{itemize}
\end{notationfacts}

In addition to all of the above, we will make use following result (especially in  \hyperref[section-2-sylow]{Section~\ref*{section-2-sylow}}), which is a special case of \cite[Theorem 12.5.2]{Hall}:

\begin{thm} \label{thm:hall-unique-p}
	Let $G$ be a finite $2$-group of exponent $2$ or $4$. If $G$ contains only one subgroup of order $2$, then $G$ is isomorphic to $C_2$, $C_4$ or the quaternion group $Q_8 = \langle g,h ~ | ~ g^4 = 1, \ g^2 = h^2, \ h^{-1}gh=g^3\rangle$.
\end{thm}

\chapter{Prerequisites on Complex Tori and Hyperelliptic Manifolds} \label{chapter:prerequisities}

\section{The Complex Representation} \label{section:cmplxrep}

Let $T = V/\Lam$, $T'=V'/\Lam'$ be complex tori. It is well-known that any holomorphic map $f \colon T \to T'$ is affine, i.e., there exist a linear map $\alpha \colon V \to V'$ such that $\alpha(\Lambda) \subset \Lambda'$ and a vector $b \in V'$ such that
\begin{align*}
f([z]) = [\alpha(z) + b],
\end{align*}
see for instance \cite[Proposition 1.2.1]{Birkenhake-Lange}.
Note that the condition that $\alpha(\Lambda) \subset \Lambda'$ guarantees that $f$ is well-defined. Furthermore, $\alpha$ is uniquely determined, whereas only the class $[b]$ of $b$ in $T'$ is uniquely determined. Quite suggestively, the map $\alpha$ is called the \emph{linear part} of $f$, whereas $[b]$ is called the \emph{translation part} of $f$. \\

Denoting by $\Hol(T)$ the set of holomorphic maps $T \to T$, we obtain two homomorphisms of $\ZZ$-algebras
\begin{align*}
\rho \colon \Hol(T) \to \End(V), \qquad
\rho_\Lambda \colon \Hol(T) \to \End(\Lambda),
\end{align*}
both of which map an element of $\Hol(T)$ to its linear part.

\begin{defin}
	The maps $\rho$ and $\rho_{\Lambda}$ are called the \emph{complex} and \emph{rational representation} of $T$, respectively.
\end{defin}

\begin{rem} \
	\begin{enumerate}[ref=(\theenumi)]
		\item Abusing notation, we will drop the square brackets to denote residue classes modulo $\Lambda$ from now on.
		\item There is a one-to-one correspondence
		\begin{align*}
		\{ f \in \Hol(T) ~ | ~ f(0) = 0 \} \longleftrightarrow \{\alpha \in \End(V) ~ | ~ \alpha(\Lambda) \subset \Lambda\}.
		\end{align*}
		It is therefore unambiguous to view a map $\alpha \in \End(V)$ such that $\alpha(\Lambda) \subset \Lambda$ as an endomorphism of $T$ fixing the origin.
	\end{enumerate}
\end{rem}

The Hodge decomposition $\Lambda \otimes_\ZZ \CC = V \oplus \overline{V}$ implies that $\rho$ and $\rho_{\Lambda}$ are related via
\begin{align} \label{integralrep}
\rho_{\Lambda} \otimes \CC = \rho \oplus \overline{\rho}.
\end{align}

Equation (\ref{integralrep}) can be reformulated as ``$\rho \oplus \overline{\rho}$ is equivalent to an integral representation''. Let us mention that the proof of \hyperref[mainthm]{Main Theorem~\ref*{mainthm}} will often require investigating if a given representation $\tilde \rho \colon G \to \GL(4,\CC)$ occurs as a restriction of the complex representation of some $4$-dimensional complex torus. It is, however, extremely difficult to verify (\ref{integralrep}) for a given representation $\tilde \rho$.  The following weaker, more usable form will be sufficient for our purposes:

\begin{cor} \label{cor-characpol}
	Let $\rho$ be the complex representation of a complex torus $T$. Then the characteristic polynomial of $\rho(f) \oplus \overline{\rho(f)}$ has integer coefficients for every $f \in \Hol(T)$. 
\end{cor}

\begin{rem}
	Let $T = E_i \times E_i$, where $E_i = \CC/(\ZZ+i\ZZ)$ is Gauss' elliptic curve. Then $[z] \mapsto [iz]$ is a linear automorphism of $E_i$ of order $4$. Thus $\Hol(T)$ contains a multiplicative group isomorphic to the quaternion group
	\begin{align*}
	Q_8 = \langle g,h \ | \ g^4 = 1, ~ g^2 = h^2, ~ gh = g^{-1}h\rangle
	\end{align*}
	of order $8$, namely the one spanned by the images of the unique irreducible representation
	\begin{align} \label{q8-rep-rem}
	g = \begin{pmatrix}
	i & 0 \\
	0 & -i
	\end{pmatrix}, \qquad h = \begin{pmatrix}
	0 & -1 \\ 1 & 0
	\end{pmatrix}
	\end{align}
	of $Q_8$. 	Since $\rho$ defines an action on the complex torus¸ $E_i \times E_i$, the representation $\rho \oplus \overline{\rho}$ is integral.  Moreover, since $Q_8$ has a unique equivalence class of irreducible representations of degree $2$, the representation (\ref{q8-rep-rem}) is equivalent to its complex conjugate representation. It follows that
	\begin{align*}
		\rho \oplus \overline{\rho} = \rho \oplus \rho
	\end{align*}
	is an integral representation. In particular, the characteristic polynomial of $f \oplus \overline{f}$ is integral for every $f \in \langle g,h \rangle$. 
	However, a standard result from representation theory shows that (\ref{q8-rep-rem}) cannot be defined over $\RR$. Hence \hyperref[cor-characpol]{Corollary~\ref*{cor-characpol}} is strictly weaker than requiring $\rho \oplus \overline{\rho}$ to be an integral representation.
\end{rem}

\medskip

We will mostly be interested in the case where $f \colon T \to T$ is biholomorphic, which we will denote by ``$f \in \Bihol(T)$''. In this case, the linear part of $f$ is invertible, both as an endomorphism of $V$ and $\Lambda$. Consequently, the complex and rational representations restrict to group homomorphisms
\begin{align*}
\rho \colon \Bihol(T) \to \GL(V), \qquad
\rho_\Lambda \colon \Bihol(T) \to \Aut(\Lambda).
\end{align*}

If $\ord(f) < \infty$, then all eigenvalues of $\rho(f)$ are roots of unity. We use an example to illustrate how \hyperref[cor-characpol]{Corollary~\ref*{cor-characpol}} gives information about the set of eigenvalues of $\rho(f)$.

\begin{example}
	If $\dim(T) = 2$ and $f \in \End(T)$, then the set of eigenvalues of $\rho(f)$ cannot be $\{\zeta_5, \ \zeta_5^4\}$, since the polynomial $(X - \zeta_5)(X - \zeta_5^4)(X - \overline{
		\zeta_5})(X-\overline{\zeta_5^4}) = (X - \zeta_5)^2(X - \zeta_5^4)^2$ does not have integer coefficients. Indeed, the minimal polynomial of $\zeta_5$ over $\QQ$ has degree $4$, and its roots are $\zeta_5,~ \zeta_5^2,~ \zeta_5^3,~ \zeta_5^4$.
\end{example}

The above example shows that the statement of \hyperref[cor-characpol]{Corollary~\ref*{cor-characpol}} is related to \emph{Euler's totient function} $\varphi$ which maps an integer $d \geq 1$ to the number of integers from $1$ to $d-1$ that are coprime to $d$. For convenience, we repeat its basic properties and its relevance in algebra:
\begin{enumerate}[ref=(\theenumi)]
	\item It is a multiplicative function, i.e., if $d_1,d_2 \in \ZZ_{>0}$ are coprime, then $\varphi(d_1d_2) = \varphi(d_1)\varphi(d_2)$. Moreover, if $p$ is prime, then $\varphi(p^k) = p^{k-1}(p-1)$ for every $k \geq 1$.
	\item The minimal polynomial $\Phi_d$ of a primitive $d$th root of unity over $\QQ$ has degree $\varphi(d)$: it is given by
	\begin{align*}
	\Phi_d = \prod_{\zeta \in \mu_d^*} (X-\zeta),
	\end{align*}
	where $\mu_d^*$ is the set of all primitive $d$th roots of unity.
\end{enumerate}
These remarks have the following elementary but significant consequences.

\begin{lemma} \label{lemma:integrality}
	Suppose $\al \in \Aut(T)$ is a linear automorphism of finite order $d$. For each positive divisor $k$ of $d$, we define
	\begin{align*}
	\nu_k = \max\{\nu \geq 0 ~ | ~ \Phi_k^\nu \text { divides the characteristic polynomial of } \alpha \oplus \overline{\alpha}\}.
	\end{align*}
	Furthermore, we denote by $\mult \colon \mu_d \to \ZZ$ the function which assigns to a $d$th root of unity its multiplicity as an eigenvalue of $\al$. Then:
	\begin{enumerate}[ref=(\theenumi)]
		\item \label{integrality-1} the function
		\begin{align*}
		\mult_k \colon \mu_k^* \to \ZZ, \quad	
		\zeta \mapsto \mult(\zeta) + \mult(\overline{\zeta})
		\end{align*}
		is constant and equal to $\nu_k$ for every positive divisor $k$ of $d$,
		\item \label{integrality-2} if $\nu_k = 1$, then $k \geq 3$ and $\alpha$ has exactly $\varphi(k)/2$ pairwise non-complex conjugate eigenvalues of order $k$,
		%	{\ \item \label{integrality-2} for a fixed divisor $k$ of $d$, either $\nu_k = 0$ or
		%	\begin{align*} 
		%		(*)_k\colon \qquad \nu_k \geq \begin{cases}
		%			2 & \text{ if } k \in \{1,2\} \\
		%			\varphi(k) & \text{ if } k \geq 3.
		%		\end{cases}
		%	\end{align*}
		%	\item \label{integrality-3} equality in $(*)_k$ holds if and only if $\alpha$ has exactly
		%	\begin{align*}
		%		\begin{cases}
		%		 \text{ one eigenvalue of order } k, & \text{ if } k \in \{1,2\} \\
		%		 \varphi(k)/2 \text{ pairwise non-complex conjugate eigenvalues of order } k, & \text{ if } k \geq 3.
		%		\end{cases}
		%	\end{align*}}
		
		\item \label{integrality-3} $2 \dim(T) = \sum_{k|d} \varphi(k)\nu_k$.
	\end{enumerate}
	
\end{lemma}

\begin{proof}
	%Observe preliminarily that $\mult_k(\zeta)$  eigenvalues of the characteristic polynomial $P$ of $\alpha \oplus \overline{\alpha}$.	\\
	
	\ref{integrality-1}	Denote by $\zeta$ a primitive $k$th root of unity. Observe that $\mult(\overline \zeta)$ is not only the multiplicity of $\overline{\zeta}$ as an eigenvalue of $\alpha$, but also the multiplicity of $\zeta$ as an eigenvalue of $\overline{\alpha}$. Hence $\mult_k(\zeta)$ equals the multiplicity of $\zeta$ as an eigenvalue of $\alpha \oplus \overline{\alpha}$. \\
	On the other hand, the characteristic polynomial $P$ of $\alpha \oplus \overline{\alpha}$ is integral by \hyperref[cor-characpol]{Corollary~\ref*{cor-characpol}}, so that
	\begin{align} \label{eq:integrality}
	P = \prod_{k|d} \Phi_k^{\nu_k}.
	\end{align}
	It follows that $\nu_k = \mult_k(\zeta)$. \\
	
	%Considering that the roots of $\Phi_k$ are exactly the primitive $k$th roots of unity and all of them are simple roots, we conclude that $\mult_k$ is constant and equal to $\nu_k$. \\
	
	%\ref{integrality-2} Denote by $\zeta$ a primitive $k$th root of unity. Observe that $\mult(\overline \zeta)$ is not only the multiplicity of $\overline{\zeta}$ as an eigenvalue of $\alpha$, but also the multiplicity of $\zeta$ as an eigenvalue of $\overline{\alpha}$. Hence $2\nu_k = \mult_k(\zeta)$ is the multiplicity of $\zeta$ as an eigenvalue of $\alpha \oplus \overline{\alpha}$.
	%
	%Denote by $\zeta$ a primitive $k$th root of unity and suppose that $\mult(\zeta) \geq 1$. If $k \in \{1,2\}$, then $\zeta$ is real (i.e., $\zeta = \overline{\zeta}$) and thus $\nu_k \geq 2$. On the other hand, if $k \geq 3$, then $\alpha \oplus \overline{\alpha}$ has at least $\varphi(k)$ eigenvalues of order $k$. 
	
	% and $\zeta$ denotes a primitive $k$th root, then $\mult(\zeta) + \mult(\overline{\zeta}) = \nu_k$ implies that $\mult(\zeta) \geq 1$ or $\mult(\overline{\zeta}) \geq 1$. We may assume that $\mult(\zeta) \geq 1$. 
	
	\ref{integrality-2} Denote by $\zeta$ a primitive $k$th root of unity. Observe that if $k \in \{1,2\}$, then any primitive $k$th root $\zeta$ is real, i.e., $\zeta = \overline{\zeta}$. It follows that $\nu_k = 2\mult(\zeta)$ is even. Hence $\nu_k = 1$ implies that $k \geq 3$. In this case,
	\begin{align*}
	1 = \nu_k = \mult(\zeta) + \mult(\overline{\zeta}) 
	\end{align*}
	implies that \emph{either} $\mult(\zeta) = 1$ or $\mult(\overline{\zeta}) = 1$. Moreover, \ref{integrality-1} implies that this holds independently from the choice of $\zeta$. Hence, if
	\begin{align*}
	\mu_k^* = O_1 \sqcup ... \sqcup O_{\varphi(k)/2}
	\end{align*}
	denotes the orbit decomposition of $\mu_k^*$ with respect to complex conjugation, there is a unique $\zeta^{(j)} \in O_j$ such that $\mult(\zeta^{(j)}) = 1$. \\
	
	\ref{integrality-3} follows from (\ref{eq:integrality}), since $\alpha \oplus \overline{\alpha}$ is a square matrix of size $2\dim(T)$ and $\deg(\Phi_k) = \varphi(k)$.
\end{proof}

\begin{rem}
	Let $T$ be a complex torus of dimension $n$, and $m \geq 1$ an integer. Then $z \mapsto mz$ defines a surjective endomorphism of $T$. We denote its kernel, the \emph{subgroup of $m$-torsion points}, by $T[m]$. As a group, $T$ is isomorphic to the product of $2n$ copies of the circle group $S_1$, and thus $T[m] \cong C_m^{2n}$.
\end{rem}

\begin{example}
	If $E = \CC/(\ZZ+\tau \ZZ)$ is an elliptic curve (where we assume as usual that $\tau$ has positive imaginary part), then $E[m]$ is spanned by $1/m$ and $\tau/m$. 
\end{example}

The upcoming result concerns fixed points of linear automorphisms of a complex torus.

\begin{lemma} \label{fixed-point-prime-power}
	Let $T$ be a complex torus of dimension $n$ and $\alpha \in \Aut(T)$ a linear automorphism of order $d$ with the property that all of its eigenvalues are primitive $d$th roots of unity. Denote by $\Fix(\alpha)$ the subgroup $T$ consisting of the fixed points of $\alpha$. Then:
	\begin{enumerate}[ref=(\theenumi)]
		\item \label{fixed-pts-1} if $d$ is the power of a prime $p$, then $\Fix(\alpha) \cong C_p^{2n/\varphi(d)}$.
		\item \label{fixed-pts-2}  if $d$ is not a prime power, then $\Fix(\alpha) = \{0\}$.
	\end{enumerate}
\end{lemma}

\begin{proof}
	We refer to \cite[Proposition 1.3]{DemleitnerToulouse}. Let us just remark that \ref{fixed-pts-1} $\implies$ \ref{fixed-pts-2}, since if $p_1 \neq p_2$ are primes dividing $d$, we choose $m_1, m_2 \geq 1$ such that $\ord(\alpha^{m_i}) = p_i$. By \ref{fixed-pts-1}, $\Fix(\alpha^{m_i})$ is an elementary Abelian $p_i$-group, and thus $\Fix(\alpha^{m_1}) \cap \Fix(\alpha^{m_2}) = \{0\}$. The statement follows because
	\begin{align*}
	\Fix(\alpha) \subset \Fix(\alpha^{m_1}) \cap \Fix(\alpha^{m_2}) = \{0\}.
	\end{align*}
\end{proof}

\begin{example} \
	\begin{enumerate}[ref=(\theenumi)]
		\item Let $T$ be a complex torus of dimension $n$. Then $\Fix(-\id_T) = T[2]$, and the statement of \hyperref[fixed-point-prime-power]{Lemma~\ref*{fixed-point-prime-power}} is just a reformulation of the fact that $T[2] \cong C_2^{2n}$.  
		\item Denote by $F = \CC/(\ZZ+ \zeta_3\ZZ)$ the equianharmonic elliptic curve. It has a linear automorphism of order $3$, namely multiplication by $\zeta_3$. \hyperref[fixed-point-prime-power]{Lemma~\ref*{fixed-point-prime-power}} implies that $\Fix(\zeta_3)$ is cyclic of order $3$. Indeed, using that $\zeta_3^2 + \zeta_3 + 1 = 0$, we obtain
		\begin{align*}
		\zeta_3 \cdot \frac{1-\zeta_3}{3} = \frac{\zeta_3-\zeta_3^2}{3} = \frac{1+2\zeta_3}{3} = \frac{1-\zeta_3}{3} \qquad \text{ in } \quad F.
		\end{align*}
		Hence $\Fix(\zeta_3) = \langle (1-\zeta_3)/3\rangle$. \\
		\item Similarly, the harmonic elliptic curve $E_i = \CC/(\ZZ+i\ZZ)$ has a linear automorphism of order $4$, namely multiplication by $i$. As above, we calculate that $\Fix(i) = \langle (1-i)/2\rangle$.
	\end{enumerate}
\end{example}

%{\ how to present applications?¸}
%
%\begin{cor}
%In the situation of \hyperref[lemma:integrality]{Integrality Lemma~\ref*{lemma:integrality}}, suppose that $\alpha$ has an eigenvalue of order $\ell$ for some $\ell ~ | ~ d$. Then $\varphi(k) \leq 2\dim(T)$.
%\end{cor}
%
%\begin{proof}
%This follows immediately from
%\begin{align*}
%\varphi(\ell) \leq \varphi(\ell) \nu_\ell \leq  \sum_{k | d} \varphi(k) \nu_k =	2\dim(T).
%\end{align*}
%\end{proof}

%\hyperref[lemma:integrality]{Lemma~\ref*{lemma:integrality}} is used many times throughout the text, mostly in the following fashion:
%
%\begin{example}
%Suppose that $\dim(T) = 4$ and that $\alpha \in \Aut(T)$ is a linear automorphism of order $5$. Then $\nu_5 \geq 1$ and
%\begin{align*}
%	8 = \nu_1 + 4\nu_5,
%\end{align*} 
%so that $\nu_5 \in \{1,2\}$. Furthermore, if $\nu_5 = 1$, then $\alpha$ has exactly two eigenvalues of order $5$, which are non-complex conjugate, and the two remaining eigenvalues are $1$.
%\end{example}

\section{Hyperelliptic Manifolds} \label{section:hyperellmfds}

\emph{Hyperelliptic surfaces} occur in the Kodaira-Enriques classification of compact complex surfaces and are characterized by the invariants $p_g = 0$, $q = 1$ and $\kappa = 0$. The classification of hyperelliptic surfaces was achieved by Bagnera-de Franchis \cite{BdF} as well as \cite{EnrSev09}, \cite{EnrSev10} and shows in particular that hyperelliptic surfaces can be described as quotients of a product of elliptic curves by a free action of a finite, non-trivial group, which is not a subgroup of the group of translations. In particular, all hyperelliptic surfaces are algebraic. The main objects of our study will be higher-dimensional analogs of hyperelliptic surfaces, which we define as follows.

\begin{defin} \label{def:hypmfd}
	A \emph{hyperelliptic manifold} is the quotient $X = T/G$ of a complex torus $T = \CC^n/\Lam$ by the action of a finite, non-trivial subgroup $G \subset \Bihol(T)$ that acts freely on $T$ and contains no translations. A projective hyperelliptic manifold is called a \emph{hyperelliptic variety}. The restriction of the complex representation $\rho \colon \Bihol(T) \to \GL(n,\CC)$ to $G$ is called the \emph{complex representation} of $X$, which we will again denote by $\rho$.
\end{defin}

\begin{rem}
	Let $T = \CC^n/\Lambda$ be a complex torus and $X = T/G$ be a hyperelliptic manifold, where $G$ satisfies the properties of \hyperref[def:hypmfd]{Lemma~\ref*{def:hypmfd}}. Let $\rho \colon G \to \GL(n,\CC)$ be its complex representation.
	\begin{enumerate}[ref=(\theenumi)]
		\item The assumption that $G$ does not contain any translations precisely means that the complex representation $\rho \colon G \to \GL(n,\CC)$ is faithful.
		\item The underlying $\mathcal C^\infty$-manifold of $X$ is a compact flat Riemannian manifold, whose holonomy representation is simply the complex representation $\rho \colon G \to \GL(n,\CC)$. The group $G$ is therefore called the \emph{(complex) holonomy group} of $X$. We refer to \cite[p. 51 f.]{Charlap} for more details.
		\item The conditions in the definition of a hyperelliptic manifold can be explained as follows: 
		\begin{itemize}
			\item The assumptions that $G$ is finite and acts freely on $T$ imply that hyperelliptic manifolds are compact Kähler manifolds.
			\item The assumption that $G$ contains no translations can always be achieved. Indeed, let $N \subset G$ denote the normal subgroup of translations. Then $T/N$ is again a complex torus, the group $N$ being finite. Furthermore, $G/N$ acts on $T/N$ without translations, and the canonical map $(T/N)/(G/N) \to T/G$ is biholomorphic.
			\item The non-triviality of $G$ combined with the assumption that $G$ does not contain any translations imply that hyperelliptic manifolds are not biholomorphic to complex tori. Indeed, the fundamental group of $X$ consists of all lifts of elements of $G$ to $\CC^n$. In particular, $\pi_1(X)$ is necessarily non-Abelian if $G$ contains a non-translation.
		\end{itemize}
		
	\end{enumerate}
\end{rem}

\begin{rem} \label{rem:ev1}
	Let $T = \CC^n/\Lam$ be a complex torus and $g \colon T \to T$, $g(z) = \rho(g)z+\tau(g)$ holomorphic. Then $g$ has no fixed point on $T$ if and only if the equation
	\begin{align*}
	(\rho(g)-\id_T)z = - \tau(g)
	\end{align*}
	has no solution $z \in T$. In particular, $\rho(g)$ has the eigenvalue $1$ if $g$ has no fixed point on $T$.
\end{rem}

\begin{example}\
	\begin{enumerate}[ref=(\theenumi)]
		\item In view of \hyperref[rem:ev1]{Remark~\ref*{rem:ev1}}, there is no hyperelliptic manifold of dimension $1$.
		\item In dimension $2$, the classification of Bagnera-de Franchis and Enriques-Severi shows that there are exactly the following seven families:
		\begin{center}
			\begin{tabular}{l|c|c}
				$T$ & The Holonomy Group $G$ & Dimension of the family \\ \hline \hline
				$E_\tau \times E_{\tau'}$ & $C_2 = \langle (z_1,z_2) \mapsto (z_1+1/2, -z_2) \rangle$ & $2$ \\ \hline
				$(E_\tau \times E_{\tau'})/\langle (\tau/2, 1/2)\rangle$ & $C_2 = \langle (z_1,z_2) \mapsto (z_1+1/2, -z_2) \rangle$ & $2$ \\ \hline 
				$E_\tau \times E_{\zeta_3}$ & $C_3 = \langle (z_1,z_2) \mapsto (z_1+1/3, \zeta_3z_2) \rangle$ & $1$ \\ \hline 
				$(E_\tau \times E_{\zeta_3})/\langle (\tau/3,(1-\zeta_3)/3)\rangle$ & $C_3 = \langle (z_1,z_2) \mapsto (z_1+1/3, \zeta_3z_2) \rangle$ & $1$ \\ \hline 
				$E_\tau \times E_{i}$ & $C_4 = \langle (z_1,z_2) \mapsto (z_1+1/4, iz_2) \rangle$ & $1$ \\ \hline 
				$(E_\tau \times E_{i})/\langle (\tau/4, (1-i)/2)\rangle$ & $C_4 = \langle (z_1,z_2) \mapsto (z_1+1/4, iz_2) \rangle$ & $1$ \\ \hline 
				$E_\tau \times E_{\zeta_3}$ & $C_6 = \langle (z_1,z_2) \mapsto (z_1+1/6, -\zeta_3z_2) \rangle$ & $1$ \\ \hline 
			\end{tabular}
		\end{center}
		Here, $E_\tau = \CC/(\ZZ+\tau \ZZ)$ for some $\tau$ with positive imaginary part.
		In particular, the holonomy group is cyclic.
		\item In dimension $3$, Uchida-Yoshihara \cite{Uchida-Yoshihara} showed that the holonomy group is either the dihedral group $D_4$ of order $8$, or is Abelian and isomorphic to $C_{d_1} \times C_{d_2}$, where $(d_1,d_2)$ belongs to the following list:
		\begin{align*}
		&\text{cyclic: }(1,2), ~ (1,3), ~ (1,4), ~ (1,5), ~ (1,6), ~ (1,8), ~ (1,10), ~ (1,12),\\ 
		&\text{non-cyclic: } (2,2), ~ (2,4), ~ (2,6), ~ (2,12), ~ (3,3), ~ (3,6), ~ (4,4), ~ (6,6).
		\end{align*}
		Hyperelliptic threefolds were later investigated extensively by Lange \cite{Lange} and Catanese-Demleitner \cite{CD-Hyp3}. It is, in particular, shown that all of the above groups occur as the holonomy group of a hyperelliptic threefold.
		\item Not much is known in dimension $\geq 4$. In the article \cite{DemleitnerToulouse}, hyperelliptic manifolds with cyclic holonomy group are investigated and classified.
	\end{enumerate}
\end{example}

\begin{defin}
	A finite, non-trivial abstract group $G$ is called \emph{hyperelliptic in dimension $n$} if it occurs as the holonomy group of some hyperelliptic manifold of dimension $n$. 
\end{defin}

\begin{rem}
	\
	\begin{enumerate}[ref=(\theenumi)]
		\item It is clear that the set of hyperelliptic groups in a given dimension is closed under taking non-trivial subgroups.
		\item If $G$ is hyperelliptic in dimension $n$, then $G$ is hyperelliptic in every dimension $\geq n$.
		\item In \cite[Theorem 17]{Ekedahl}, the authors show that any hyperelliptic manifold admits arbitrarily small projective deformations. Since any deformation is locally diffeomorphic to a trivial family and the holonomy group is a smooth invariant, it follows that $G$ is hyperelliptic in dimension $n$ if it occurs as the holonomy group of a hyperelliptic \emph{variety} (i.e., a \emph{projective} hyperelliptic manifold).
	\end{enumerate}
\end{rem}

Throughout this text, we will mainly focus on hyperelliptic fourfolds. Since a complete classification would be too involved, our goal is to classify the hyperelliptic groups in dimension $4$. \\
Albeit we will not fully classify hyperelliptic manifolds in dimension $4$, it is still worthwhile to briefly explain when two hyperelliptic manifolds (or, more generally, quotients of complex tori by finite group actions) are biholomorphic. \\ \label{torusquot-bihol}
For $i = 1,2$, let $T_i = \CC^n/\Lambda_i$ be a complex torus and $G_i \subset \Bihol(T_i)$ be a finite group containing no translations. (Note that we do not require $G_i$ to act freely on $T_i$ here.) Denote by $\phi_i \colon G_i \to \Bihol(T_i)$ the corresponding actions and let $X_i = T_i/G_i$ be the quotient of $T_i$ by the action $\phi_i$. \\
We recall that the fundamental groups $\pi_1(X_i)$ are \emph{crystallographic groups}, which implies that $\Lambda_i$ is its unique maximal normal Abelian subgroup and that $\pi_1(X_i)/\Lambda_i \cong G_i$ (we refer to \cite{Charlap}, \cite{Szczepanski} and \cite[Section 5.2]{topmethods} for more details). Any biholomorphic map $f \colon X_1 \to X_2$ induces an isomorphism $f_* \colon \pi_1(X_1) \to \pi_1(X_2)$, which then necessarily has to satisfy $f_*(\Lambda_1) = \Lambda_2$. In particular, $f$ induces an isomorphism $G_1 \cong G_2$. \\
Furthermore, by topology, there is a biholomorphic lift $\hat f \colon T_1 \to T_2$ of $f$, that is, the following diagram commutes:
\begin{center}
	\begin{tikzcd}
	T_1 \arrow[r, "\hat f"] \arrow[d] & T_2  \arrow[d] \\
	X_1 \arrow[r, "f"] & X_2
	\end{tikzcd}
\end{center}
Spelling out the commutativity of the diagram, we obtain an isomorphism $\varphi \colon G_1 \to G_2$ such that for every $g \in G_1$ and $z \in T_1$:
\begin{align} \label{eq:bihol}
\hat f\left(\phi_1(g)(z)\right) = \phi_2(\varphi(g))(\hat f(z))).
\end{align}
The map $\hat f$ being holomorphic, we may write $\hat f (z) = Az +b $, where $A$ is linear and $b \in T_2$. Furthermore, we denote the linear part (resp. the translation part) of $\phi_i$ by $\rho_i$ (resp. $\tau_i$). 
Comparing the linear and translations parts of (\ref{eq:bihol}) then yields that the following two conditions are satisfied for every $g \in G_1$:
\begin{enumerate}[ref=(\theenumi)]
	\item \label{bihol-cond-1} $A \rho_1(g) A^{-1} = \rho_2(\varphi(g))$, and
	\item \label{bihol-cond-2} $(\rho_2(\varphi(g)) - \id_{T_2})b = A\tau_1(g) - \tau_2(\varphi(g))$.
\end{enumerate}
Conversely, given a biholomorphic map $g \colon T_1 \to T_2$,  $g(z) = Az +b$ and an isomorphism $\varphi \colon G_1 \to G_2$ such that \ref{bihol-cond-1} and \ref{bihol-cond-2} are satisfied, the map $g$ induces a biholomorphic map $X_1 \to X_2$. \\

\textbf{Interpretation of \ref{bihol-cond-1} and \ref{bihol-cond-2}:}

Condition \ref{bihol-cond-1} simply means that $\rho_1$ and $\rho_2 \circ \varphi$ are equivalent representations. The second condition, \ref{bihol-cond-2}, has the following interpretation in group cohomology:

\begin{rem} \label{rem:cocycle}
	Let $T$ be a complex torus and let $G \subset \Bihol(T)$ be a finite group. We consider two actions of $G$ on $T$:
	\begin{enumerate}[label=(\roman*),ref=(\roman*)]
		\item \label{usual-action}the ``usual'' action $\phi \colon G \to \Bihol(T)$ as above; we denote the linear part and the translation part of $\phi$ by $\rho$ and $\tau$, respectively.
		\item the linear action $\rho \colon G \to \Aut(T)$. 
	\end{enumerate}
	Now, since
	\begin{align*}
	\tau(g_1g_2) = \rho(g_1)\tau(g_2) + \tau(g_1),
	\end{align*} 
	the translation part defines an element in $Z^1(G,T)$ with respect to the linear action (ii). Observe that the $1$-coboundaries $B^1(G,T) \subset Z^1(G,T)$ are the elements of the form
	\begin{align*}
	\rho(g)b-b
	\end{align*} 
	for some $b \in T$.
\end{rem}
\bigskip 

Rewriting condition \ref{bihol-cond-2} as
\begin{align*}
\forall u \in G_2 \colon \qquad \rho_2(u)b - b = A \tau_1(\varphi^{-1}(u)) - \tau_2(u),
\end{align*}
\hyperref[rem:cocycle]{Remark~\ref*{rem:cocycle}} shows that \ref{bihol-cond-2}  simply means that $A\tau_1 \varphi^{-1}$ and $\tau_2$ define the same cohomology class in $H^1(G_2,T_2)$. \\

\begin{rem} \label{rem:translation-g-torsion}
	We use the notation of \hyperref[rem:cocycle]{Remark~\ref*{rem:cocycle}}. The cohomological interpretation above has the following important consequences:
	\begin{enumerate}[ref=(\theenumi)]
		\item freeness of the action $\phi$ depends only on the cohomology class $[\tau] \in H^1(G,T)$.
		\item since modifying the translation part of $\phi$ by a coboundary does not change the isomorphism type of the quotient $T/G$, and $|G|$ kills $H^1(G,T)$ (see \cite[Ch. III, (10.2) Corollary]{Brown}), we are allowed to assume  $\tau(g)$ to be $|G|$-torsion for every $g \in G$.
		\item let $U \subset G$ be a subgroup. The restriction of $\rho$ to $U$ induces a restriction map $$\res^G_U\colon H^1(G,T) \to H^1(U,T),$$ which maps $[\tau]$ to $[\tau|_U]$. Since the diagram
		
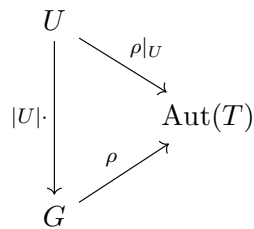
\begin{center}
			\begin{tikzcd}
			U \arrow[rd, "\rho|_U"] \arrow[dd, "|U|\cdot" left] & \\
			&\Aut(T) \\
			G \arrow[ur, "\rho"] &
			\end{tikzcd}
		\end{center}
		commutes, $\res^G_U$ commutes with $|U|$ as well. By the previous bullet point, we obtain that
		\begin{align*}
		\res^G_U(|U| \cdot [\tau]) = |U| \cdot [\tau|_U] = 0 \in H^1(U,T).
		\end{align*}
		We may therefore not only assume that $\tau(g)$ is $|G|$-torsion for every $g \in G$, but also in addition that $\tau(u)$ is $|U|$-torsion for every $u \in U$. This observation will be beneficial in the forthcoming discussion.
	\end{enumerate}
\end{rem}

\section{Equivariant Decomposition up to Isogeny} \label{isogeny}
Suppose that we are given a linear action $\rho \colon G \to \Aut(T)$ of the finite group $G$ on the complex torus $T = V/\Lambda$. We aim to explain how $\rho$ induces a $G$-equivariant isogeny\footnote{Recall that an isogeny is a surjective holomorphic group homomorphism, whose kernel is finite.} $T_1 \times ... \times T_k \to T$, where $T_1, ..., T_k \subset T$ are sub-tori, which are stable under the $G$-action. We follow the exposition in \cite[Ch. 13.6]{Birkenhake-Lange}.\\ 

The first step is to extend  the action $\rho \colon G \to \Aut(T)$ to a map
\begin{align*}
\tilde\rho \colon \QQ[G] \to \End_\QQ(T)
\end{align*}
and consider the decomposition
\begin{align*}
\QQ[G] = Q_1 \oplus ... \oplus Q_r
\end{align*}
into blocks. Write
\begin{align*}
1 = e_1 + ... + e_r,
\end{align*}
where the $e_j$ are central, primitive, and orthogonal idempotents of $\QQ[G]$, i.e., they satisfy the following properties:
\begin{enumerate}[ref=(\theenumi)]
	\item $e_j \in Q_j$,
	\item $e_j \in Z(\QQ[G])$,
	\item $e_j^2 = e_j$,
	\item $e_i e_j = 0$ for $i \neq j$.
\end{enumerate}

Choose positive integers $m_j \in \ZZ_{> 0}$ such that $T_j := \im(\tilde \rho(m_je_j)) \subset T$. Observe that $T_j$ does not depend on the particular choice of $m_j$ (since $T$ is a divisible group). We can therefore choose an $m$ such that $T_j = \im(\tilde \rho(me_j))$ for all $j$ (we will see later in  \hyperref[cor-torsion]{Corollary~\ref*{cor-torsion}} that $m = |G|$ works).

\begin{lemma} The following statements hold:
	\begin{enumerate}[ref=(\theenumi)]
		\item \label{lemma-isog-1} the tori $T_1, ..., T_k$ defined above are $G$-stable.
		\item \label{lemma-isog-2} if $i \neq j$, then there are no homomorphisms $T_i \to T_j$ respecting the $G$-action.
		\item \label{lemma-isog-3} the addition map
		\begin{align*}
		\mu \colon T_1 \times ... \times T_k \to T
		\end{align*}
		is a $G$-equivariant isogeny.
	\end{enumerate}
\end{lemma}

\begin{proof}
	\ref{lemma-isog-1} Let $g \in G$. Then 
	\begin{align*}
	\rho(g)T_j = \tilde\rho(g m e_j) T = \tilde\rho(m e_j)\rho(g)T = \tilde\rho(m e_j)T = T_j, 
	\end{align*}
	where we used that the $e_j$ are central. \\
	
	\ref{lemma-isog-2} Every element of $\Hom_G(T_i,T_j)$ gives rise to a $G$-equivariant map between the corresponding $\QQ$-representations, which are irreducible. By Schur's lemma, there can only be a non-trivial map if $i \neq j$. \\
	
	\ref{lemma-isog-3} It follows from $e_1 + ... + e_k = 1$ that
	\begin{align*}
	\dim(T_1) + ... + \dim(T_k) = \dim(T)
	\end{align*}
	and that the addition map $\mu$ is surjective. Thus $\mu$ is an isogeny. It is clearly equivariant.
\end{proof}

As mentioned above, the projections $\QQ[G] \to Q_j$ are given by multiplication by $e_j$. We describe them explicitly.

\begin{prop} \label{projector}
	Let $V$ be a simple $\CC[G]$-module with character $\chi$.
	\begin{enumerate}[ref=(\theenumi)]
		\item \label{projector-1} Denote by $K$ the field generated by the values of $\chi$ over $\QQ$. Then $K/\QQ$ is a Galois extension.
		\item \label{projector-2} Let $W_j$ be a simple $\QQ[G]$-module that corresponds to the ideal $Q_j$ and -- viewed as a $\CC[G]$-module -- contains $V$ as a direct summand. Then the corresponding central idempotent $e_j$ is given as follows:
		\begin{align*}
		e_j = \frac{\chi(1)}{|G|} \sum_{g \in G} \tr_{K/\QQ}(\chi(g)) \cdot g.
		\end{align*}
	\end{enumerate} 
\end{prop}

\begin{proof}
	\ref{projector-1} The character values are sums of roots of unity, and thus $K$ is contained in a cyclotomic field. Since cyclotomic fields are Abelian extensions, the extension $K_j/\QQ$ is Galois. \\
	
	\ref{projector-2} We start with a general observation. Let $\Irr(G) = \{\chi_1, ..., \chi_c\}$ be the set of irreducible $\CC$-characters of $G$. It is well-known that the central idempotent $p_i$ of $\CC[G]$ corresponding to $\chi_i$ is
	\begin{align*}
	p_i = \frac{\chi_i(1)}{|G|} \sum_{g \in G} \overline{\chi_i(g)} \cdot g \in \CC[G].
	\end{align*}
	These are orthogonal, i.e., $p_i p_k = 0$ for $i \neq k$. \\
	Fix $i$ and let $K_i$ be the field spanned by the values of $\chi_i$ as in part \ref{projector-1}. The Galois group $\Gal(K_i/\QQ)$ acts on $\Irr(G)$ via
	\begin{align*}
	(\sigma \chi)(g) := \sigma(\chi(g)), \qquad \text{where } \sigma \in \Gal(K_i/\QQ), \quad \chi \in \Irr(G), \quad g \in G.
	\end{align*}
	Using this action, we obtain an action of $\Gal(K_i/\QQ)$ on the set $\{p_1, ..., p_c\}$ of complex central idempotents:
	\begin{align*}
	\sigma p_k = \frac{(\sigma\chi_k)(1)}{|G|} \sum_{g \in G} \overline{(\sigma\chi_k)(g)} \cdot g.
	\end{align*}
	We are now in the situation to describe the central idempotent $e_i$ explicitly. Let $p = p_1$ be the complex central idempotent corresponding to $\chi$. By construction,
	\begin{align*}
	e_i = \sum_{\sigma \in \Gal(K/\QQ)} \sigma p
	\end{align*}
	is the central idempotent in the rational group algebra $\QQ[G]$ corresponding to $Q_j$. Now we can compute $e_i$ explicitly:
	\begin{align*}
	e_i = \sum_{\sigma \in \Gal(K/\QQ)} \sigma p = \frac{\chi(1)}{|G|} \sum \tr_{K/\QQ}(\chi(g))g.
	\end{align*}
	Here, we used that $(\sigma\chi)(1) = \chi(1)$ for all $\sigma \in \Gal(K/\QQ)$ and that the complex conjugation can be omitted because the field traces are rational numbers (they are even integers).
\end{proof}

\hyperref[projector]{Proposition~\ref*{projector}} gives an effective way of computing the $T_j$. We illustrate the process by giving an example.
%The proof gives an effective way of computing the $T_j$: we just need to decompose $\rho \colon G \to \GL(V)$ into $\CC$-irreducible characters and then compute the projectors $e_j$ according to Proposition \ref{projector}. Since $[{\ K}:\QQ] \leq 2$ in our case by ..., the field trace is simply given by
%\begin{align*}
%	\tr_{K/\QQ}(\chi(g)) = \begin{cases}
%	\chi(g), & \text{ if } \chi \text{ is real} \\
%	\chi(g) + \overline{\chi(g)}, & \text{ if } \chi \text{ is not real}.
%	\end{cases}
%\end{align*}

\begin{example} \label{isog-ex-quat}
	Consider the quaternion group $$Q_8 = \langle g,h \ | \ g^4 = 1, ~ g^2 = h^2, ~ gh= h^{-1}g\rangle$$ of order $8$. Consider the representation $\rho \colon Q_8 \to \GL(4,\CC)$ defined by
	\begin{align*}
	a \mapsto \begin{pmatrix}
	0 & -1 & & \\
	1 & 0 && \\
	&& -1 & \\
	&&& 1
	\end{pmatrix}, \qquad b \mapsto \begin{pmatrix}
	i & 0 && \\ & -i && \\ && 1 &\\ &&& 1
	\end{pmatrix}.
	\end{align*}
	The matrices are already written in block diagonal form so that the character of $\rho$ decomposes as the direct sum of the following three distinct irreducible characters:
	\begin{itemize}
		\item $\chi_1(1) = 2, \qquad \chi_1(g) = 0, \qquad \ \, \,  \chi_1(h) = 0$,
		\item $\chi_2(1) = 1, \qquad \chi_2(g) = -1, \qquad \chi_2(h) = 1$,
		\item $\chi_3(1) = 1, \qquad \chi_3(g) = 1, \qquad  \ \ \, \chi_3(h) = 1$.
	\end{itemize}
	For each $j \in \{1,2,3\}$, the field spanned by the values of $\chi_j$ is simply $\QQ$. Hence the corresponding rational idempotents $e_1$, $e_2$, $e_3$ are given by
	\begin{itemize}
		\item $e_1 = \frac12(1 - g^2)$, 
		\item $e_2 = \frac18(1-g+g^2-g^3+h+h^3-gh-(gh)^3)$,
		\item $e_3 = \frac18(1+g+g^2+g^3+h+h^3+gh+(gh)^3)$.
	\end{itemize}
	Moreover,
	\begin{itemize}
		\item $\rho(e_1) = \diag(1,1,0,0)$,
		\item $\rho(e_2) = \diag(0,0,1,0)$, 
		\item $\rho(e_3) = \diag(0,0,0,1)$.
	\end{itemize}
	Thus a complex torus $T = \CC^4/\Lambda$ on which $G$ acts via $\rho$ is equivariantly isogenous to a product of a complex torus $S$ of dimension $2$ (the image of $2\rho(e_1)$ in $T$) and two elliptic curves $E$, $E'$ (the images of $8\rho(e_2)$ and $8\rho(e_3)$ in $T$, respectively). More precisely, the addition map
	\begin{align*}
	\mu \colon S \times E \times E' \to T
	\end{align*}
	is an isogeny and $\ker(\mu) \subset S[2] \times E[8] \times E'[8]$. 
\end{example}

As a corollary of our discussion, we obtain information about the torsion subgroup $\ker(\mu)$.

\begin{cor} \label{cor-torsion}
	The kernel of the addition map $\mu \colon T_1 \times ... \times T_k \to T$ is contained in $T_1[|G|] \times ... \times T_k[|G|]$.
\end{cor}

\begin{proof}
	By definition, $T_j$ is the image of $m_j\tilde\rho(e_j)$, where $m_j \in \ZZ_{> 0}$ is chosen such that $m_j\tilde\rho(e_j) \in \End(T)$. Considering that the field trace of an algebraic integer over $\QQ$ is an integer, the explicit description of $e_j$ given in  \hyperref[projector]{Proposition~\ref*{projector}} \ref{projector-2} implies that $|G|e_j \in \ZZ[G]$ and thus $m_j = |G|$ works.
\end{proof}

We end this section by giving two well-known results.

\begin{prop}\cite[Proposition (5.7)]{Catanese-Ciliberto} \label{prop:isogenous-ord-3-4} 
	Let $T$ be a complex torus of dimension $n$, Suppose that $\zeta_d \id_T \in \Aut(T)$ for some $d \in \{3,4\}$. Then $T$ is isomorphic to $E^n$, where 
	\begin{align*}
	E = \begin{cases}
	F := \CC/(\ZZ + \zeta_3\ZZ) \text{ is the equianharmonic elliptic curve}, & \text{ if } d = 3 \\
	E_i := \CC/(\ZZ + i\ZZ) \text{ is the harmonic elliptic curve}, & \text{ if } d = 4.
	\end{cases}
	\end{align*}
	(Here, we write $i$ instead of $\zeta_4$.)
\end{prop}

The analogous result for automorphisms of order $8$ is a special case of \cite[Proposition (5.9)]{Catanese-Ciliberto}:

\begin{prop} \label{prop:isogenous-ord-8}
	Let $T$ be a $2$-dimensional complex torus, and suppose that $\alpha \in \Aut(T)$ is a linear automorphism of order $8$. Then $T$ is isomorphic to one of $E_{\sqrt 2 i} \times E_{\sqrt 2 i}$ or $E_i \times E_i$, where $E_{\sqrt 2 i} = \CC/(\ZZ+\sqrt{2}i\ZZ)$ and $E_i = \CC/(\ZZ + i \ZZ)$. Which case occurs when depends on the eigenvalues of $\alpha$ and will be specified in the proof.
\end{prop}

\begin{proof}
	According to \hyperref[lemma:integrality]{Lemma~\ref*{lemma:integrality}} \ref{integrality-2}, $\alpha$ has two non-complex conjugate eigenvalues of order $8$. Up to raising $\alpha$ by an appropriate power, we may assume that $\zeta_8$ is an eigenvalue of $\alpha$. Then the set of eigenvalues of $\alpha$ is $\{\zeta_8, \zeta_8^r\}$ for some $r \in \{3,5\}$. According to \cite[Proposition (5.8)]{Catanese-Ciliberto}, the isomorphism class of $T$ is uniquely determined in both cases $r \in \{3,5\}$. Hence it suffices to show that $E_{\sqrt 2 i} \times E_{\sqrt 2 i}$ and $E_i \times E_i$ admit automorphisms of order $8$:
	\begin{itemize}
		\item $\begin{pmatrix}
		0 & 1 \\ 1 & \sqrt 2 i
		\end{pmatrix}$ defines an automorphism of order $8$ of $E_{\sqrt 2 i} = \CC/(\ZZ+\sqrt{2}i\ZZ)$: here, $r = 3$, 
		\item $\begin{pmatrix}
		0 & i \\ 1 & 0 
		\end{pmatrix}$ defines an automorphism of order $8$ of $E_i \times E_i$: here, $r = 5$.
	\end{itemize}
\end{proof}

\begin{rem}
	Suppose that $T$ is a complex torus of dimension $n$ (resp. $2$) and that $G \subset \Aut(T)$ is a finite subgroup containing $\zeta_d \id_T$ for some $d \in \{3,4\}$ (resp. an automorphism $\alpha$ of order $8$). The previous propositions imply that $T$ is isomorphic to a certain product $T'$ of elliptic curves. By conjugating the action of $G$ with the isomorphism $T \to T'$, we obtain an action of $G$ on $T'$. However, the complex representation of the action of $G$ on $T'$ can be quite messy. As a workaround, we will not make use of an isomorphism $T \cong T'$. Instead, we fix the complex representation and use an equivariant isogeny $T' \to T$.
\end{rem}

\chapter{Properties of Hyperelliptic Groups} \label{chapter:properties}

To obtain our classification result, we derive necessary properties of hyperelliptic groups in dimension $4$. More precisely, we show in \hyperref[section:group-order]{Section~\ref*{section:group-order}} that the order of a hyperelliptic group $G$ in dimension $4$ is $2^a \cdot 3^b \cdot 5^c \cdot 7^d$, where $b,c \leq 1$. A central part of our argument is the \hyperref[order-cyclic-groups]{Integrality Lemma~\ref*{order-cyclic-groups}} \ref{ocg-3}, which will be used many times throughout the text. In \hyperref[section:centralnormal]{Section~\ref*{section:centralnormal}}, we then show that if $g \in G$ is conjugate to its inverse, then $\ord(g) \in \{1,2,3,4,6\}$. Moreover, we prove structural results about the center of $G$, see \hyperref[thm:center]{Theorem~\ref*{thm:center}}. Finally, \hyperref[section-metacyclic]{Section~\ref*{section-metacyclic}} is devoted to studying metacyclic groups, which are hyperelliptic in dimension $4$. After having studied the representation theory of metacyclic groups, we show that the complex representation of a metacyclic group that is hyperelliptic in dimension $4$ contains a non-trivial linear character as a direct summand (see \hyperref[lemma-two-generators]{Proposition~\ref*{lemma-two-generators}}). This result will allow us to show that certain direct products of (non-Abelian) metacyclic groups with cyclic groups are not hyperelliptic in dimension $4$.

\section{The Group Order} \label{section:group-order}

The first step in classifying the hyperelliptic groups in dimension $4$ is to obtain information on the group order. In this section, we show that the order of a hyperelliptic group in dimension $4$ is $2^a \cdot 3^b \cdot 5^c \cdot 7^d$, where $c,d \leq 1$. \\

Recall from \hyperref[rem:ev1]{Remark~\ref*{rem:ev1}} that if $T$ is a complex torus and $g \in \Bihol(T)$ acts freely on $T$, then the linear part of $g$ has the eigenvalue $1$.
%\begin{lemma}[Integrality Lemma] \label{lemma:integrality}
%Let $X = T/G$ be a hyperelliptic manifold of dimension $n$ with associated complex representation $\rho \colon G \to \GL(n,\CC)$.
 %\hyperref[lemma-phi-d-2-non-conj-ev]{Lemma~\ref*{lemma-phi-d-2-non-conj-ev}}

\begin{lemma} \label{lemma:order-n}
\
\begin{enumerate}[ref=(\theenumi)]
	\item \label{lemma:order-n-1} Let $T$ be a $n$-dimensional complex torus and $\alpha \in \Aut(T)$ a linear automorphism of order $d$, such that its set of eigenvalues $\Eig(\alpha)$ satisfies $\Eig(\alpha) \subset \{1\} \cup \mu_d^*$ and $1 \in \Eig(\alpha)$.
	Then $\varphi(d) \leq 2n-2$. 
	\item \label{lemma:order-n-2} Conversely, if $\varphi(d) \leq 2n-2$, then there exists a complex torus $T$ of dimension $n$ and a linear automorphism $\alpha \in \Aut(T)$ of order $d$ such that $\Eig(\alpha) \subset \{1\} \cup \mu_d^*$ and $1 \in \Eig(\alpha)$.
\end{enumerate}
\end{lemma}

\begin{proof}
\ref{lemma:order-n-1} We may assume that $d \geq 2$. By the \hyperref[order-cyclic-groups]{Integrality Lemma~\ref*{order-cyclic-groups}} \ref{ocg-3}, 
\begin{align*}
	2n = \nu_1 + \varphi(d) \nu_d,
\end{align*}
where $\nu_k$ is the maximal $\nu$ such that $\Phi_k^\nu$ divides the characteristic polynomial of $\rho(g) \oplus \overline{\rho(g)}$. By assumption, $\nu_2 \geq 2$ and $\nu_d \geq 1$, which shows the statement. \\

\ref{lemma:order-n-2} The statement is clear if $d \leq 2$. If $d \geq 3$, we may take the direct product of a CM-torus of dimension $\varphi(d)/2$ (on which we let $\alpha$ act by choosing $\varphi(d)/2$ pairwise non-complex conjugate primitive $d$th roots of unity) and any complex torus of dimension $n - \varphi(d)/2$ (on which we let $\alpha$ act trivially).
\end{proof}

Next, we investigate the special case $n = 4$.

\begin{lemma} \label{order-cyclic-groups}
Let $T$ be a $4$-dimensional complex torus and $\alpha \in \Aut(T)$ a (linear) automorphism of finite order $d$. Suppose that $\alpha$ has the eigenvalue $1$. Then:
\begin{enumerate}[ref=(\theenumi)]
	\item \label{ocg-1} if $\Eig(\alpha) \subset \{1\} \cup \mu_d^*$, then
		$d \in \{1,~ 2,~ 3,~ 4,~ 5,~ 6,~ 7,~ 8,~ 9,~ 10,~ 12,~ 14,~ 18\}$,
	\item \label{ocg-2} $d \in \{1,~ 2,~ 3,~ 4,~ 5,~ 6,~ 7,~ 8,~ 9,~ 10,~ 12,~ 14,~ 15,~ 18,~ 20,~ 24,~ 30\}$,
	\item \label{ocg-3} \emph{(Integrality Lemma)} if $\varphi(d) \geq 4$ and $\Eig(\alpha)$ contains a primitive $d$th root of unity, then $\alpha$ contains exactly $\varphi(d)/2$ eigenvalues of order $d$, and these are pairwise non-complex conjugate,
	\item \label{ocg-4} if $d \in \{5,7,9,14,18\}$, then $\Eig(\alpha) \subset \{1\} \cup \mu_d^*$.
\end{enumerate}
\end{lemma}

\begin{proof}
\ref{ocg-1} By \hyperref[lemma:order-n]{Lemma~\ref*{lemma:order-n}} \ref{lemma:order-n-1}, we have $\varphi(d) \leq 6$, which shows that $d$ is contained in the stated set. \\

\ref{ocg-2} We use the notation from \hyperref[lemma:integrality]{ Lemma~\ref*{lemma:integrality}}. Let $S$ be the set of divisors $k \geq 2$ of $d$ such that $\nu_k \geq 1$. Then the cited lemma implies that
\begin{align*}
	\nu_1 + \sum_{k \in S} \nu_k \varphi(k) = 8.
\end{align*}
Since $\nu_1 \geq 2$, we obtain
\begin{align*}
	\sum_{k \in S} \nu_k \varphi(k) \leq 6.
\end{align*}
Since $\varphi(k) \geq 2$ for $k \geq 3$ and $\nu_2$ is even, we obtain that $|S| \leq 3$ and that
\begin{itemize}
	\item $|S| = 3 \iff S \subset \{2,3,4,6\}$: in particular, the least common multiple of the elements in $S$ divides $12$.
	\item $|S| = 2 \iff |S \cap \{5,8,12\}| = 1$ and $S \setminus \{5,8,12\} \subset \{2,3,4,6\}$: here, the least common multiple of the elements in $S$ divides one of the numbers $12$, $15$, $20$, $30$, or $24$.
	\item $|S| = 1 \iff S \subset \{7,14,18\}$: there is nothing to show in this case.
\end{itemize}

\ref{ocg-3} Here, $\nu_1 \geq 2$ and $\nu_d \geq 1$, so that
\begin{align*}
	8 = \sum_{k|d} \nu_k \varphi(k) \geq \nu_1 + \nu_d \varphi(d) \geq 2 + 4 \nu_d.
\end{align*}
We obtain that $\nu_d = 1$. The statement thus follows from \hyperref[lemma:integrality]{Lemma~\ref*{lemma:integrality}} \ref{integrality-2}. \\

\ref{ocg-4} The statement is clear for prime $d$, i.e., for $d \in \{5,7\}$. Furthermore, for $d = 9$, we have $\nu_9 \geq 1$. Hence we obtain that
\begin{align*}
	8 = \nu_1 + 2\nu_3 + 6\nu_9 \geq 8 + 2 \nu_3. 
\end{align*}
Hence $\nu_3=0$. \\
The remaining values are $d \in \{14,18\}$, and we will only deal with the case $d = 14$, the other case being similar.  
Here, there are two possibilities:
\begin{itemize}
	\item  both $\nu_2$ and $\nu_7$ are $\geq 1$, or 
	\item $\nu_{14} \geq 1$. 
\end{itemize}
We need to show that the second possibility occurs: this follows directly from 
	$$8 = \nu_1 + \nu_2 + 6\nu_7 + 6 \nu_{14} \geq 2 + \nu_2 + 6\nu_7 + 6 \nu_{14},$$
which shows that $\nu_7 \geq 1$ implies that $\nu_2 = 0$.
%
%	
%	Assertion \ref{order-cyc-1} follows from the observation that the characteristic polynomial of $\rho_\Lam$ is an integral polynomial of degree $2n$, which has the root $1$ with multiplicity at least $2$ by the previous remarks.\\ To show part (b), we observe that $\varphi(\ord(g)) \leq 6$ implies that $$\ord(g) \in \{1,2,3,4,5,6,7,8,9,10,12,14,18\}.$$ 
%	If $\varphi(\ord(g)) > 6$, then by (a) and Remark \ref{rem-below-lemma}, $g$ is a product of elements with coprime orders $d$, $d'$ satisfying $\varphi(d) = 2, \varphi(d') = 4$. This gives the remaining possible orders on the list. 

\end{proof}

\begin{cor} \label{cor:group-order}
Let $G$ be a hyperelliptic group in dimension $4$. Then $|G| = 2^a \cdot 3^b \cdot 5^c \cdot 7^d$ for some $a,b,c,d \geq 0$.
\end{cor}

\begin{proof}
\hyperref[order-cyclic-groups]{Lemma~\ref*{order-cyclic-groups}} \ref{ocg-2} shows that no prime other than $2$, $3$, $5$ and $7$ divides $|G|$.
\end{proof}

\begin{lemma}\cite{Catanese-Ciliberto}, \cite[Sections 3.1 \& 3.2]{DemleitnerToulouse}\\ \label{lemma-table}
	Let $d \in \NN$ such that $\varphi(d) \in \{4,6\}$ and that $\dim(T) = \varphi(d)/2$. Suppose $\al \in \Aut(T)$ is a linear automorphism of order $d$, the set of whose eigenvalues is denoted $\Eig(\alpha)$.
	Then, up to the action of $\Gal(\QQ(\zeta_d)/\QQ)$, the pair $(d,\Eig(\al))$ is contained in the following table:
	\begin{center}
		\begin{tabular}{|c||c|c|c|c|c|c|c|c|c|} \hline
			$d$ & $5$  & $7$ & $8$ & $9$ \\ \hline
			Possibilities & \multirow{2}{*}{$\{1,2\}$}  &$\,\{1,2,3\},$ &\,$\{1,3\},$ & $\,\{1,2,4\},$       \\
			for $\Eig(\al)$ & &$\{1,2,4\}$ &$\{1,5\}$ & $\{1,4,7\}$   \\ \hline \hline 
			$d$& $10$ & $12$ & $14$ & $18$ \\ \hline 
			Possibilities& \multirow{2}{*}{$\{1,3\}$} & \,$\{1,5\},$ & $\{1,3,5\},$  & $\{1,5,7\},$ \\
			for $\Eig(\al)$ & &$\{1,7\}$ & \,$\{1,5,11\}$ & $\, \{1,7,13\}$ \\ \hline
		\end{tabular}
	\end{center}
\end{lemma}

\begin{proof}[Sketch of Proof.]
	The hypothesis that $\varphi(d) \in \{4,6\}$ implies that $d \in \{5,7,8,9,10,12,14,18\}$.
	
	The cases $d = 5, 8, 10, 12$ are dealt with in \cite{Catanese-Ciliberto}, and the cases $d = 7,9$ are contained in \cite{DemleitnerToulouse}. Thus, the cases $d = 14,18$ are missing. We prove the assertion for $d = 18$, the other case being treated similarly. The units in $C_{18}$ form a cyclic group of order $6$. We divide our investigation into several cases, distinguished by the number of generators of $C_{18}^* = \{1,5,7,11,13,17\}$ in the triple, and also by the property whether there is the neutral element $1$ or not. The generators of $C_{18}^*$ are $5$ and $11$.\\
	Hence, there are just the following cases:
	\begin{align*}
	&\{1, 5,7\} \colon \text{one generator and }1. \\
	&\{5,7,17\}\colon \text{one  generator and not }1. \text{ This is } 5 \text{ times } \{1, 5,7\}.\\	
	&\{7,13,17\}\colon  \text{no generators and not } 1. \text{ This is } 7 \text{ times } \{1,5,17\}. \\
	&\{1,5,11\}\colon \text{two generators and }1.  \text{ This is } 11 \text{ times } \{1, 5,7\}. \\
	&\{1,11,13\} \colon \text{one generator and } 1. \text{ This is } 13 \text{ times } \{1,5,7\}. \\
	&\{11,13,17\} \colon \text{one generator and not } 1. \text{ This is } 17 \text{ times } \{1,5,7\}. \\
	&\{1,7,13\} \colon  \text{no generators and } 1. \text{ This is never a multiple of } \{1,5,7\}. \\ 
	&\{5,11,17\}\colon \text{two generators and not }1. \text{ This is } 5 \text{ times } \{1,7,13\}.
	\end{align*}
	Hence, up to Galois automorphisms, the only possibilities are $\{1,5,7\}$ and $\{1,7,13\}$.
\end{proof}

\begin{cor} \label{cor:not-divisible-25}
Let $G$ be a hyperelliptic group in dimension $4$ with associated complex representation $\rho$. Then $|G|$ is not divisible by $25$.
\end{cor}

\begin{proof}
Consider the determinant exact sequence
\begin{center}
\begin{tikzcd}
	0 \arrow[r] & N \arrow[r] & G \arrow[r, "\det(\rho)"] & \mu_m \arrow[r]  & 1. 
\end{tikzcd}
\end{center}
Now, if $g \in G$ is an element of order $5$, then $\det(\rho(g))$ is a primitive fifth root of unity (see \hyperref[lemma-table]{Lemma~\ref*{lemma-table}}) and hence $g \notin N$. It follows that $5 \nmid |N|$. We have already seen that $G$ does not contain elements of order $25$ (\hyperref[order-cyclic-groups]{Lemma~\ref*{order-cyclic-groups}}), hence $25 \nmid m$. It follows that $25 \nmid |G|$.
\end{proof}

\begin{lemma} \label{lemma:not-divisible-49}
Let $G$ be a hyperelliptic group in dimension $4$ with associated complex representation $\rho$. Then $|G|$ is not divisible by $49$.
\end{lemma}

\begin{proof}
Since hyperelliptic groups in dimension $4$ do not contain elements of order $49$ (see \hyperref[order-cyclic-groups]{Lemma~\ref*{order-cyclic-groups}}), it suffices to prove that $G$ does not contain a subgroup, which is isomorphic to $C_7 \times C_7$. Suppose for a contradiction that $U \cong C_7 \times C_7$ is a subgroup of $G$ and denote by $\{u_1,u_2\}$ a generating set of $U$. According to \hyperref[lemma-table]{Lemma~\ref*{lemma-table}} and the \hyperref[order-cyclic-groups]{Integrality Lemma~\ref*{order-cyclic-groups}} \ref{ocg-3}, after choosing a suitable basis, we may write
\begin{align*}
\rho(u_1) = \diag(\zeta_7, ~ \zeta_7^2, ~ \zeta_7^a, ~ 1), \qquad \text{ where } a \in \{3,4\}.
\end{align*}
After multiplying $u_2$ by an appropriate power of $u_1$ and raising $u_2$ to a suitable power, we may assume that the first two diagonal entries of $\rho(u_2)$ are $1$ and $\zeta_7$, respectively, i.e.,
\begin{align*}
\rho(u_2) = \diag(1, ~ \zeta_7, ~ \zeta_7^b, ~ \zeta_7^c), \qquad \text{ for some } b,c \in \{1,...,6\}.
\end{align*}
However then $\rho(u_1u_2^{-1}) = \diag(\zeta_7, ~ \zeta_7, ~ \zeta_7^{a-b}, ~ \zeta_7^{-c})$ has the eigenvalue $\zeta_7$ with multiplicity $2$, contradicting the \hyperref[order-cyclic-groups]{Integrality Lemma~\ref*{order-cyclic-groups}} \ref{ocg-3}.
\end{proof}

\begin{rem} \
\begin{enumerate}[ref=(\theenumi)]
\item The statements that $C_5 \times C_5$ and $C_7 \times C_7$ are not hyperelliptic in dimension $4$ will be generalized later in   \hyperref[abelian-cor-cd^n-2]{Corollary~\ref*{abelian-cor-cd^n-2}}. 
\item The proof of \hyperref[lemma:not-divisible-49]{Lemma~\ref*{lemma:not-divisible-49}} shows more generally that there is no $4$-dimensional complex torus $T$ such that $C_7 \times C_7 \subset \Aut(T)$. (As usual, $\Aut(T)$ is the group of automorphisms of $T$ fixing the origin.)
\end{enumerate} 
\end{rem}

\begin{cor} \label{cor:group-order-with-bounds}
If $G$ is hyperelliptic in dimension $4$, then $|G| = 2^a \cdot 3^b \cdot 5^c \cdot 7^d$, where $c,d \leq 1$. 
\end{cor}

\begin{proof}
Follows from  \hyperref[cor:group-order]{Corollary~\ref*{cor:group-order}}, \hyperref[cor:not-divisible-25]{Corollary~\ref*{cor:not-divisible-25}} and \hyperref[lemma:not-divisible-49]{Lemma~\ref*{lemma:not-divisible-49}}.
\end{proof}

Note that there is no a priori bound on $a$ and $b$. Later, in \hyperref[section-2-sylow]{Chapter~\ref*{section-2-sylow}} and \hyperref[section-3-sylow]{Chapter~\ref*{section-3-sylow}}, we will show that $a \leq 7$ (which will then be improved to $a \leq 5$ in \hyperref[chapter:2a3b]{Chapter~\ref*{chapter:2a3b}}) and $b \leq 3$.

\section{Central and Normal Subgroups} \label{section:centralnormal}

This section deals with the following question: given a hyperelliptic group $G$ in a fixed dimension, what can be said about the center, or, more generally, normal subgroups of $G$? Before specializing to dimension $4$, we prove a statement in arbitrary dimensions. As usual, we let $\rho \colon G \to \GL(n,\CC)$ be a representation such that
\begin{enumerate}[label=(\Roman*),ref=(\Roman*)]
	\item $\rho$ is faithful,
	\item every matrix in $\rho(G)$ has the eigenvalue $1$, and
	\item $\rho$ satisfies the ``integrality criterion'', i.e., the characteristic polynomial of $\rho(g) \oplus \overline{\rho(g)}$ is integral for every $g \in G$.
\end{enumerate}

A simple starting point is

%\begin{center} \label{meta-assumptions}
%	\noindent\fbox{%
%		\parbox{0.975\textwidth}{
%			The letter $G$ will always denote a finite subgroup of $\Bihol(A)$, where $A$ is an Abelian variety of dimension $n$, such that the following properties hold:
%			\begin{center}
%				\begin{itemize}
%					\item[(1)] $G$ is embedded into $\GL(n,\CC)$ via some faithful representation $\rho \colon G \hookrightarrow \GL(n,\CC)$ (this is equivalent to requiring that $G$ does not contain any translations).
%					\item[(2)] The matrix $\rho(g)$ has the eigenvalue $1$ for any $g \in G$.
%					\item[(3)] The associated complex representation of the embedding $G \subset \Bihol(A)$ is $\rho$.
%				\end{itemize}
%			\end{center}
%		}
%	}
%\end{center} \vspace{0.5cm}

\begin{lemma} \label{lemma:conjugate}
	Suppose that $G$ contains an element $g$ of order $d$ with the following properties:
	\begin{enumerate}
		\item[(a)] $\varphi(d) \geq n$,
		\item[(b)] the elements $g$ and $g^{-1}$ are conjugate in $G$,
		\item[(c)] $\rho(g)$ has the eigenvalue $1$.
	\end{enumerate}
	Assuming, furthermore, that $\rho(g)$ has an eigenvalue of order $d$, then the characteristic polynomial of $\rho(g) \oplus \overline{\rho(g)}$ is not integral. In particular, $G$ is not hyperelliptic in dimension $n$.
\end{lemma}

\begin{proof}
	By hypothesis, the matrices $\rho(g)$ and $\rho(g)^{-1}$ are conjugate and thus have the same eigenvalues. Hence the set $\Eig(\rho(g))$ of eigenvalues of $\rho(g)$ is invariant under complex conjugation. However, then the characteristic polynomial of $\rho(g) \oplus \overline{\rho(g)}$ can only be integral if $\Eig(\rho(g))$ contains all primitive $d$th roots of unity, of which there are $\varphi(d)$ many. This contradicts hypothesis (c).
\end{proof}

Applied to dimension $4$, we obtain a stronger statement:

\begin{prop} \label{prop:conjugate}
	Suppose that $g \in G$ is an element of order $d \geq 3$ and that $G$ is hyperelliptic in dimension $4$ with associated complex representation $\rho$. If the elements $g$ and $g^{-1}$ are conjugate in $G$, then $d \in \{3,4,6\}$.
\end{prop}

\begin{proof}
We assume that $\varphi(d) \geq 4$ and argue towards a contradiction. According to  \hyperref[order-cyclic-groups]{Lemma~\ref*{order-cyclic-groups}},
	\begin{align*}
		d \in \{5,~ 7,~8,~9,~10,~12,~14,~15,~18,~24,~30\}.
	\end{align*}
	If $\rho(g)$ has an eigenvalue of order $d$ (which is necessarily the case for $d \in \{5,7,8,9\}$, see \hyperref[order-cyclic-groups]{Lemma~\ref*{order-cyclic-groups}} \ref{ocg-4}), then  \hyperref[lemma:conjugate]{Lemma~\ref*{lemma:conjugate}} together with the assumption that $G$ is hyperelliptic in dimension $4$ yield the desired contradiction. We may therefore assume that $\rho(g)$ has no eigenvalues of order $d$ (and hence in particular $d \in \{10,12,14,15,18,24,30\}$). Now, if $d \neq 12$, then some power of $g$ (which is then also conjugate to its inverse) has eigenvalues of order $5$, $7$, $8$ or $9$ and we may reduce to the previous case. It remains to deal with the case $d = 12$: here, if $\rho(g)$ has no eigenvalue of order $12$, it must necessarily have an eigenvalue of order $4$ and an eigenvalue of order $3$ or $6$.  However, $\rho(g)$ and $\rho(g)^{-1} = \overline{\rho(g)}$ have the same eigenvalues -- it follows that $\rho(g)$ has two eigenvalues of order $4$ as well as two eigenvalues of order $3$ or $6$, and hence $\rho(g)$ does not have the eigenvalue $1$, a contradiction.
\end{proof}

\begin{cor} \label{cor:conjugate}
Let $G$ be hyperelliptic in dimension $4$ and $g \in G$ an element of order $d$ such that $\varphi(d) \geq 4$. Suppose that $\Aut(\langle g\rangle)$ is spanned by $g \mapsto g^j$. Then $g$ and $g^j$ are not conjugate in $G$.
\end{cor}

\begin{proof}
Some power of $g \mapsto g^j$ is the map $g \mapsto g^{-1}$, hence if $g$ and $g^j$ are conjugate, then $g$ and $g^{-1}$ are conjugate. \hyperref[prop:conjugate]{Proposition~\ref*{prop:conjugate}} then forces $\varphi(d) \leq 2$.
\end{proof}

\begin{rem} \label{rem:conjugate}
It is implicit in the assumptions of \hyperref[cor:conjugate]{Corollary~\ref*{cor:conjugate}} that $\Aut(\langle g \rangle)$ is cyclic. This holds in particular for $d = \ord(g) \in \{5,7,9\}$.  The following table lists the generators for $\Aut(\langle g \rangle)$ in these cases:
\begin{center}
	\begin{tabular}{c|c}
	$d$ & all $j$ such that $g \mapsto g^j$ generates $\Aut(\langle g\rangle)$ \\ \hline \hline 
	$5$ & $2$, $3$ \\ \hline
	$7$ & $3$, $4$ \\ \hline 
	$9$ & $2$, $7$                                                                                     
	\end{tabular}
\end{center}
Observe that the symmetry is not surprising, since $\varphi(d) = 4$ for $d \in \{5,7,9\}$ implies that $\Aut(\langle g \rangle) \cong C_d^*$ has $\varphi(4) = 2$ generators, which then must $g \mapsto g^j$ and $g \mapsto g^{-j}$ for some $j$.
\end{rem}

Next, we study central subgroups of non-Abelian hyperelliptic groups. As above, we first show a general lemma (\hyperref[lemma:central-non-cyc]{Lemma~\ref*{lemma:central-non-cyc}}), which holds in arbitrary dimension, before we specialize to dimension $4$. Here, much more can be said -- in particular, we will be investigating (and partly answering) the following

\begin{questions} \label{question:center} Let $G$ be a hyperelliptic group in dimension $4$, which is non-Abelian.
	\begin{enumerate}[ref=(\theenumi)]
		\item \label{question:center-2} If $g \in Z(G)$, what can be said about the order of $g$?
		\item \label{question:center-1} How large can $Z(G)$ be?
	\end{enumerate}
\end{questions}

We start with the promised lemma.

\begin{lemma} \label{lemma:central-non-cyc}
	Suppose that $G$ is a non-Abelian hyperelliptic group in dimension $n$, and that $Z(G)$ contains a subgroup isomorphic to $C_d^k$, $d \geq 2$, $k \geq 0$. Then the complex representation $\rho$ splits as the direct sum of at least $k+1$ irreducible representations. In particular:
	\begin{enumerate}[ref=(\theenumi)]
		\item \label{lemma:central-non-cyc-1} $k \leq n-2$,
		\item if $k = n-2$, then $\rho$ splits as the direct sum of $n-1$ irreducible representations of respective dimensions $2$, $1$, ..., $1$, exactly one of whose restriction to $Z(G)$ is trivial,
		\item If $\rho$ is irreducible, then $Z(G)$ is trivial.
	\end{enumerate}
\end{lemma}

\begin{proof}
	The assertions follow immediately since $\rho$ is faithful, every matrix in $\im(\rho)$ has the eigenvalue $1$ and irreducible representations of $G$ map every element in $Z(G)$ to a multiple of the identity.
\end{proof}

The above lemma is complemented by

\begin{lemma} \label{lemma:number-subrep}
	Let $G$ be a finite group containing a central element $g$ of order $d \geq 3$. Suppose, furthermore, that $\rho \colon G \to \GL(n,\CC)$ is a faithful representation with the property that $\rho(g)$ has the eigenvalue $1$ and at least one eigenvalue of order $d$. If the characteristic polynomial of $\rho(g) \oplus \overline{\rho(g)}$ is integral, then the number of irreducible sub-representations of $\rho$ is at least $\frac{\varphi(d)}{2} + 1$.
\end{lemma}

\begin{proof}
	Again, a central element is mapped to a multiple of the identity by every irreducible representation of $G$. The statement thus follows from the \hyperref[order-cyclic-groups]{Integrality Lemma~\ref*{order-cyclic-groups}} \ref{ocg-3}.
\end{proof}

We now specialize to the case $n = 4$. Applying \hyperref[lemma:number-subrep]{Lemma~\ref*{lemma:number-subrep}}, we obtain a first result in the direction of \hyperref[question:center]{Question~\ref*{question:center}} \ref{question:center-2}:

\begin{cor} \label{cor:central-7-9}
	Let $G$ be a hyperelliptic group in dimension $4$. If $G$ contains a central element of order $7$ or $9$, then $G$ is Abelian.
\end{cor}

\begin{proof}
	If $\rho \colon G \to \GL(4,\CC)$ is a faithful representation of some hyperelliptic fourfold with holonomy group $G$ and $d \in \{7,9\}$, then  \hyperref[lemma:number-subrep]{Lemma~\ref*{lemma:number-subrep}} asserts that $\rho$ decomposes into a direct sum of $\frac{\varphi(d)}{2}+1 = 4$ sub-representations. Finally, the faithfulness of $\rho$ implies that $G$ is Abelian.
\end{proof}

%\begin{cor}
%Suppose that $g \in G$ is an element of order $d \geq 3$ and that $G$ is hyperelliptic in dimension $4$ with associated complex representation $\rho$. If $\varphi(d) \geq 4$ and $\Aut(C_d)^*$ is cyclic, then $g$ is not conjugate to any 
%\end{cor}
%
%\begin{proof}
%As in the proof of  \hyperref[prop:conjugate]{Proposition~\ref*{prop:conjugate}}, the hypotheses imply that $d \in \{5,7,9,10,12,14,15,18,30\}$. Furthermore, 
%\end{proof}

Using other methods, we show:

\begin{lemma} \label{lemma:central-order-5-8}
Let $G$ be a finite non-Abelian group containing an element $g$ of order $d \in \{5,8\}$. If $g$ is central, then $G$ is not hyperelliptic in dimension $4$.
\end{lemma}

\begin{proof}
We assume for a contradiction that $\rho$ is the complex representation of some hyperelliptic fourfold.  \hyperref[lemma:number-subrep]{Lemma~\ref*{lemma:number-subrep}} implies that $\rho$ decomposes as the direct sum of an irreducible degree $2$ representation $\rho_2$ and two linear characters $\chi$, $\chi'$. According to the \hyperref[order-cyclic-groups]{Integrality Lemma~\ref*{order-cyclic-groups}} \ref{ocg-3}, $\rho(g)$ has exactly two non-conjugate eigenvalues of order $d$: this can only be possible if $g \in \ker(\rho)$ and $\chi(g)$, $\chi'(g)$ are primitive $d$th roots of unity. Consider now the determinant exact sequence
\begin{align*}
1 \to N \to G \to \mu_m \to 1,
\end{align*}
where the map $G \to \mu_m$ is given by $\det(\rho_2)$. Since $G/N$ is Abelian, we conclude that $N$ contains the derived subgroup $[G,G]$ of $G$, which is non-trivial since $G$ is non-Abelian. Take $1 \neq h \in [G,G]$. Then $h \notin \ker(\rho_2)$ and in view of $h \in N$ we conclude that the eigenvalues of $\rho_2(h)$ are $\zeta$ and $\overline{\zeta}$ for some root of unity $\zeta \neq 1$. It follows that $\rho(gh)$ does not have the eigenvalue $1$.
\end{proof}

As a consequence, we obtain 

\begin{cor} \label{cor:normal_ord_5_implies_abelian}
Suppose that the non-Abelian group $G$ contains a normal subgroup $\langle g \rangle$ of order $5$. Then $G$ is not hyperelliptic in dimension $4$.
\end{cor}

\begin{proof}
If $\langle g \rangle$ is central, then
\hyperref[lemma:central-order-5-8]{Lemma~\ref*{lemma:central-order-5-8}} shows that $G$ is not hyperelliptic in dimension $4$. If $\langle g \rangle$ is not central, then $g$ is conjugated to some $g^j$, $j \in \{2,3,4\}$. The case $j = 4$ was dealt with in  \hyperref[prop:conjugate]{Proposition~\ref*{prop:conjugate}}, whereas we recalled in \hyperref[rem:conjugate]{Remark~\ref*{rem:conjugate}} that $g \mapsto g^j$ generates $\Aut(\langle g \rangle)$ for $j \in \{2,3\}$, so that \hyperref[cor:conjugate]{Corollary~\ref*{cor:conjugate}} deals with the remaining cases.
\end{proof}

Furthermore, combining \hyperref[cor:central-7-9]{Corollary~\ref*{cor:central-7-9}} and \hyperref[lemma:central-order-5-8]{Lemma~\ref*{lemma:central-order-5-8}}, we obtain a complete answer to \hyperref[question:center]{Question~\ref*{question:center}} \ref{question:center-2}:

\begin{cor}  \label{cor:center-dim4}
Let $G$ be a non-Abelian hyperelliptic group in dimension $4$ and $g \in Z(G)$. Then $\ord(g) \in \{1,2,3,4,6,12\}$.
\end{cor}

Next, we work towards \hyperref[question:center]{Question~\ref*{question:center}} \ref{question:center-1}. The following corollary will soon be overshadowed by  \hyperref[thm:center]{Theorem~\ref*{thm:center}}.

\begin{cor}
If $G$ is non-Abelian and hyperelliptic in dimension $4$, then $Z(G)$ is -- up to isomorphism -- a subgroup of $C_{12} \times C_{12}$.
\end{cor}

\begin{proof}
According to \hyperref[lemma:central-non-cyc]{Lemma~\ref*{lemma:central-non-cyc}} \ref{lemma:central-non-cyc-1}, $Z(G)$ is a subgroup of $C_e \times C_e$ for some $e$. \hyperref[cor:center-dim4]{Corollary~\ref*{cor:center-dim4}} shows that $e = 12$ works.
\end{proof}

To provide a complete answer to  \hyperref[question:center]{Question~\ref*{question:center}} \ref{question:center-1}, it remains to investigate which subgroups of $C_{12} \times C_{12}$ actually arise as centers of non-Abelian hyperelliptic groups in dimension $4$. The proof of the following lemma is analogous to the one of  \hyperref[lemma:central-order-5-8]{Lemma~\ref*{lemma:central-order-5-8}} and is left to the reader.

\begin{lemma} \label{lemma:central12}
Let $G$ be a finite non-Abelian group containing a central element $g$ of order $12$ such that $\rho(g)$ has an eigenvalue of order $12$. Then $G$ is not hyperelliptic in dimension $4$.
\end{lemma}

We are now in the situation to prove 

\begin{thm} \label{thm:center}
Suppose that $G$ is a non-Abelian hyperelliptic group in dimension $4$. If $Z(G)$ is non-cyclic, then $Z(G)$ is, as an abstract group, contained in one of the following groups:
\begin{align*}
C_4 \times C_4, \qquad C_6 \times C_6 \qquad \text{ or } 
	\qquad 	C_2 \times C_{12}.
\end{align*}
\end{thm}

%We will see later in {\ referenz} that  \hyperref[thm:center]{Theorem~\ref*{thm:center}} is sharp. 

\begin{proof}
Given the previous results, it remains to exclude that $C_3 \times C_{12}$ and $C_4 \times C_{12}$ arise as centers of non-Abelian hyperelliptic groups in dimension $4$. For this, assume the contrary, i.e., that $Z(G)$ is isomorphic to $C_d \times C_{12}$, where $d \in \{3,4\}$. According to \hyperref[lemma:central-non-cyc]{Lemma~\ref*{lemma:central-non-cyc}}, $\rho$ splits as a direct sum of exactly three irreducible representations. \\
Denote generators of $Z(G)$ by $g$ and $h$, where $\ord(g) = d$ and $\ord(h) = 12$. In view of \hyperref[lemma:central12]{Lemma~\ref*{lemma:central12}}, we may assume that $\rho(h)$ and $\rho(gh)$ have no eigenvalues of order $12$. In the following, we only consider the case $d = 4$. The case $d = 3$ is similar and left to the reader.\\
By faithfulness of $\rho$, the elements $g$ and $h^3$ of order $4$ are mapped to primitive fourth roots of unity by \emph{different} irreducible summands of $\rho$. Furthermore, $\rho(g)$ and $\rho(h^3)$ have the eigenvalue $1$, i.e., there is exactly one direct summand $\rho'$ of $\rho$ whose kernel contains both $g$ and $h^3$. Furthermore, according to the previous remarks, neither $\rho(h)$ nor $\rho(gh)$ have eigenvalues of order $12$, and hence the element $h^4$ of order $3$ is mapped to a primitive third root of unity by $\rho'$. It follows that $\rho(gh)$ does not have the eigenvalue $1$, a contradiction.
\end{proof}

\begin{rem} \label{rem:center}
In the end, our classification of the hyperelliptic groups (\hyperref[mainthm]{Main Theorem~\ref*{mainthm}}) in dimension $4$ will show that if $G$ is a non-Abelian hyperelliptic group in dimension $4$, then $Z(G)$ is not isomorphic to $C_2 \times C_{12}$, $C_4 \times C_4$ or $C_6 \times C_6$. However, we were unable to find a proof, which does not rely on the full classification.
\end{rem}

\section{Hyperelliptic Fourfolds with Metacyclic Holonomy Group} \label{section-metacyclic}

As it will turn out, most non-Abelian groups which are hyperelliptic in dimension $4$ are (quotients of) metacyclic groups or direct products of metacyclic groups with cyclic groups. This justifies investigating metacyclic hyperelliptic groups in greater generality. 

\begin{defin} \label{def:metacyclic}
	A (finite) group $G$ is called \textit{metacyclic}, if it contains a cyclic normal subgroup $N$ such that $G/N$ is cyclic. The semidirect products
	\begin{align*}
	G(m,n,r) := \langle g,h \ | \ g^m = h^n = 1, \ h^{-1}gh = g^r \rangle,
	\end{align*}
	where $m,n,r \in \NN_{> 0}$ and $\gcd(r,m) = 1$ will be called \textit{standard} metacyclic groups. 
\end{defin}

\begin{example}
	Standard metacyclic groups include
	\begin{align*}
	&G(3,2,2) = S_3 = \langle \sigma, \tau \ | \ \sigma^3 = \tau^2 = 1, \ \tau^{-1}\sigma\tau = \sigma^2\rangle, \\
	&G(4,2,3) = D_4 = \langle r,s \ | \ r^4 = s^2 = 1, \ s^{-1}rs = r^{-1} \rangle, \\
	&G(m,n,1) =C_n \times C_m = \langle g,h \ | \ g^m = h^n = 1, \ h^{-1}gh = g\rangle.
	\end{align*}
	A non-standard metacyclic group is 
	\begin{align*}
	Q_8 = \langle g,h \ | \ g^4 = 1, \ g^2 = h^2, \ h^{-1}gh = g^3\rangle.
	\end{align*}
	It is the quotient of $G(4,4,3)$ by the relation $g^2 = h^2$.
\end{example}

\begin{rem} \label{metacyclic-rem}
	Let $G(m,n,r)$ be a standard metacyclic group as in  \hyperref[def:metacyclic]{Definition~\ref*{def:metacyclic}}.
	\begin{enumerate}[ref=(\theenumi)]
		\item The condition $\gcd(r,m) = 1$ ensures that $\ord(g) = m$ and thus $|G(m,n,r)| = mn$.
		\item \label{metacyclic-rem-2} By definition, $G(m,n,r)$ contains the cyclic normal subgroup $\langle g \rangle$ of order $m$ such that $$G(m,n,r)/\langle g \rangle \cong C_n.$$ In particular, the derived subgroup of $G(m,n,r)$ is contained in $\langle g \rangle$. \\
		Moreover, in combination with  \hyperref[thm:huppert-degree]{N. Ito's Degree Theorem~\ref*{thm:huppert-degree}}, this proves that the dimensions of irreducible cha\-racters of $G(m,n,r)$ divide $n$. 
	\end{enumerate}
\end{rem}

\begin{center}
	\noindent\fbox{%
		\parbox{0.975\textwidth}{
			In the following, we will always assume that $G(m,n,r)$ is presented as in \hyperref[def:metacyclic]{Definition~\ref*{def:metacyclic}} and that $r \not\equiv 1\pmod m$. This implies that $m \geq 3$ and ensures that $G(m,n,r)$ is non-Abelian.
		}
	}
\end{center}
\bigskip 
%
%{special quotients, standard.... rewrite everything}
%
We will often consider \textit{quotients} of $G(m,n,r)$: this amounts to imposing additional relations between $g$ and $h$. However, we would like that the images of $g$ and $h$ in the quotient still have orders $m$ and $n$, respectively.

\begin{defin}
	A quotient $G$ of $G(m,n,r)$ is called \textit{special}, if $G$ is non-Abelian and if the images of $g$ and $h$ in $G$ have orders $m$ and $n$, respectively. By slight abuse of notation, we will denote the generators of special quotients of $G(m,n,r)$ again by $g$ and $h$.
\end{defin}

\begin{example}
	$Q_8$ is a special quotient of $G(4,4,3)$.
\end{example}

%\begin{rem}
%	By what we remarked above, quotients of $G(m,n,r)$ can only be special if $r \not\equiv 1 \pmod m$.
%\end{rem}

\begin{rem}
	The same group $G$ can be a special quotient of two different metacyclic groups. Consider for instance the dicyclic group of order $12$,
	\begin{align*}
	\Dic_{12} = \langle a,b \ | \ a^6 = 1, \ a^3 = b^2, \ b^{-1}ab = a^5\rangle.
	\end{align*}
	Written in this form, $\Dic_{12}$ is a special quotient of $G(6,4,5)$. However, $\Dic_{12} = G(3,4,2)$: an isomorphism is given by $g \mapsto a^2$, $h \mapsto b$.
\end{rem}

\begin{lemma} \label{lemma:metacyc-irredrep}
Let $G$ be a special quotient of $G(m,n,r)$, which is hyperelliptic in dimension $4$. Then:
\begin{enumerate}[ref=(\theenumi)]
	\item \label{lemma:metacyc-irredrep1} the order of $g \mapsto g^r$ in $\Aut(\langle g \rangle)$ is $2$ or $3$,
	\item  \label{lemma:metacyc-irredrep2} $m$ is not divisible by $5$,
	\item \label{lemma:metacyc-irredrep3} the degree of an irreducible representation of $G$ is $\leq 3$.
\end{enumerate}
In particular, the complex representation of a hyperelliptic fourfold with holonomy group $G$ contains an irreducible sub-representation of degree $2$ or $3$.
\end{lemma}

\begin{proof}
\ref{lemma:metacyc-irredrep1} Denote by $s$ the order of the automorphism $g \mapsto g^r$. It follows from \hyperref[order-cyclic-groups]{Lemma~\ref*{order-cyclic-groups}} that $s \in \{2,3,4,6\}$, and that
\begin{align} \label{eq:lemma:metacyc-irredrep}
	s \in \{4,6\} \iff m \in \{5,~ 7,~ 9,~ 10,~ 14,~ 15,~ 18,~ 20,~ 30\}.
\end{align}
We assume that $s \in \{4,6\}$ and work towards a contradiction.
For the values  of $m$ listed in (\ref{eq:lemma:metacyc-irredrep}), we may write $m = q \cdot k$, where $q \in \{5,7,9\}$ and $k \in \{1,2,3,4\}$, $\gcd(q,k) = 1$. Then
\begin{align*}
\Aut(\langle g \rangle) \cong \underbrace{\Aut(\langle g^k\rangle)}_{\text{cyclic of order }s} \times \underbrace{\Aut(\langle g^q\rangle)}_{\text{cyclic of order } \leq 2}.
\end{align*}
Hence, $\Aut(\langle g^k\rangle)$ is generated by $g^k \mapsto g^{rk}$. Moreover, the elements $g^k$ and $g^{rk}$ of order $s$ are conjugated by $h$, and thus \hyperref[cor:conjugate]{Corollary~\ref*{cor:conjugate}} yields the desired contradiction.\\
	
\ref{lemma:metacyc-irredrep2}  \hyperref[cor:normal_ord_5_implies_abelian]{Corollary~\ref*{cor:normal_ord_5_implies_abelian}} states that a non-Abelian hyperelliptic group in dimension $4$ does not have a normal subgroup of order $5$. \\

\ref{lemma:metacyc-irredrep3} According to \hyperref[thm:huppert-degree]{N. Ito's Degree Theorem~\ref*{thm:huppert-degree}}, $G$ must have an irreducible representation of degree $d$ that
\begin{align*}
	d \text{ divides }|G/\langle g, h^s \rangle| =\gcd(n,s),
\end{align*}
where $s$ denotes again the order of the automorphism $g \mapsto g^r$. We have shown in \ref{lemma:metacyc-irredrep1} that $s \in \{2,3\}$, thus $d \leq 3$. 
%Furthermore, \hyperref[order-cyclic-groups]{Lemma~\ref*{order-cyclic-groups}} asserts that $s \in \{2,3,4,6\}$, and
%\begin{align*}
%	s \in \{4,6\} \iff m \text{ is divisible by } 5, 7 \text{ or } 9.
%\end{align*}
%If $k \geq 1$ is such that $\ord(g^k) \in \{5,7,9\}$, then the restriction of $g \mapsto g^r$ to $\langle g^k\rangle$ induces an automorphism of $\langle g^k\rangle$ of order $s$.
% Since $\Aut(\langle g^k \rangle)$ is cyclic of order $s$, it is generated by $g^k \mapsto g^{kr}$. Moreover, $g^k$ and $g^{kr}$ are conjugated by $h$, and hence the assertion follows from \hyperref[cor:conjugate]{Corollary~\ref*{cor:conjugate}}.
\end{proof}

\begin{lemma} \label{lemma:repmetacyc}
Suppose that $\rho' \colon G \to \GL(d,\CC)$ is an irreducible representation of degree $d \in \{2,3\}$ of a special quotient of $G(m,n,r)$. Then:
\begin{enumerate}[ref=(\theenumi)]
	\item \label{lemma:repmetacyc1} the matrices $\rho'(g)$ and $\rho'(h)$ do not commute,
	\item \label{lemma:repmetacyc2} $\rho'(g)$ and $\rho'(g)^r$ are conjugate, but different matrices, 
	\item \label{lemma:repmetacyc3} $\rho'(g)$ has $d$ distinct eigenvalues, all of which have the same order,
	\item \label{lemma:repmetacyc4} the matrix $\rho'(g)^{r-1}$ does not have the eigenvalue $1$.
\end{enumerate}
\end{lemma}

\begin{proof}
\ref{lemma:repmetacyc1} Since $\rho'$ is irreducible of degree $\geq 2$, the matrices $\rho'(g)$ and $\rho'(h)$ do not share a common eigenvector. In particular, $\rho'(g)$ and $\rho'(h)$ do not commute. \\
	
\ref{lemma:repmetacyc2} The relation $h^{-1}gh = g^r$ implies that $\rho'(h)^{-1}\rho'(g)\rho'(h) = \rho'(g)^r$. The second assertion follows from \ref{lemma:metacyc-irredrep1}.  \\

\ref{lemma:repmetacyc3}, \ref{lemma:repmetacyc4} We first show that $\rho'(g)$ has $d$ distinct eigenvalues. By \ref{lemma:repmetacyc1}, $\rho'(g)$ has at least two different eigenvalues, which already proves the statement for $d = 2$. For $d = 3$, suppose that $\rho'(g)$ has only two distinct eigenvalues, say $\la_1$ and $\la_2$, where the multiplicity of $\la_2$ is $2$. Then, since $\rho'(g)$ and $\rho'(g)^r$ are conjugate, it follows that $\la_2 = \la_2^r$ and thus $\la_1 = \la_1^r$, contradicting \ref{lemma:repmetacyc2}. \\
It remains to show that all eigenvalues of $\rho'(g)$ have the same order. 
The proof of \hyperref[lemma:metacyc-irredrep]{Lemma~\ref*{lemma:metacyc-irredrep}} \ref{lemma:metacyc-irredrep3} shows that $d = \deg(\rho')$ is equal to the order of the automorphism $g \mapsto g^r$. Hence conjugation by $\rho'(h)$ defines an automorphism of order $d$ on the set $\Eig(\rho'(g))$ of eigenvalues of $\rho'(g)$, which has cardinality $d$. Since $d \in \{2,3\}$, the corresponding element in $S_d$ is a $d$-cycle. \\
Now, if $d = 2$ and $\Eig(\rho'(g)) = \{ \la_1, \la_2\}$, then $\la_2^r = \la_1$ and $\la_1^r = \la_2$, and the hypothesis $r \equiv 1 \pmod m$ implies that $\ord(\la_1) = \ord(\la_2)$. \\
Similarly, in the case $d = 3$, we write $\Eig(\rho'(g)) = \{\la_1, \la_2, \la_3\}$. After renumbering, we may assume that the action of $\rho'(h)$ on $\Eig(\rho'(g))$ is given by the $3$-cycle $(1,2,3)$. Then $\la_2^r = \la_1$, $\la_3^r = \la_1$ and $\la_1^r = \la_3$. We reach the same conclusion.
\end{proof}

So far, only the representation theory of special quotients of metacyclic groups was discussed. For the rest of the section, we will apply the representation theory to hyperelliptic groups in dimension $4$, which are special quotients of metacyclic groups.

\begin{cor} \label{cor:metacyclic-rep}
Suppose that $G$ is a special quotient of $G(m,n,r)$, which is hyperelliptic in dimension $4$. Let $\rho \colon G \to \GL(4,\CC)$ be the complex representation of a hyperelliptic fourfold with holonomy group $G$. Then:
\begin{enumerate}[ref=(\theenumi)]
\item \label{cor:metacyclic-rep-1} the representation $\rho$ is equivalent to the direct sum of an irreducible representation $\rho'$ of degree $d \in \{2,3\}$ and $4-d$ linear characters. 
\item \label{cor:metacyclic-rep-2} if $m$ is a prime power, then all eigenvalues of $\rho'(g)$ have order $m$.
\end{enumerate}
\end{cor}

\begin{proof}
\ref{cor:metacyclic-rep-1} It suffices to show that $\rho$ is not the direct sum of two irreducible representations of degree $2$. This follows, since $\rho(g)$ has the eigenvalue $1$, but \hyperref[lemma:repmetacyc]{Lemma~\ref*{lemma:repmetacyc}} asserts that any irreducible degree $2$ representation of $G$ maps $g$ to a matrix without the eigenvalue $1$. \\

\ref{cor:metacyclic-rep-2} Let $\chi$ be a linear character of $G$. Since $g^{r-1}$ is a commutator, $\ord(\chi(g)) < m$. Hence, for $\rho$ to be faithful, the matrix $\rho'(g)$ must have order $m$ (here we use that $m$ is a prime power). Finally, \hyperref[lemma:repmetacyc]{Lemma~\ref*{lemma:repmetacyc}} \ref{lemma:repmetacyc3} implies that all eigenvalues of $\rho'(g)$ have order $m$. 
\end{proof}

The following Proposition is an analog of \cite[Lemma 3]{Uchida-Yoshihara} in complex dimension $4$.

\begin{prop} \label{lemma-two-generators}
Suppose that $G$ is a special quotient of $G(m,n,r)$, which is hyperelliptic in dimension $4$. Let $\rho \colon G \to \GL(4,\CC)$ be the complex representation of a hyperelliptic fourfold with holonomy group $G$. Then $\rho$ contains a non-trivial linear character.
\end{prop}

\begin{proof}
%	\ref{ltg-1} {\ delete?} Since $9\nmid m$ and $|G| = 2^a \cdot 3^b$, we obtain that $3 \nmid \varphi(m)$. Consequently, an element of order $3$ contained in the subgroup $\langle h\rangle$ acts trivially on $C_m$, hence spans a central subgroup. According to \hyperref[thm:huppert-degree]{N. Ito's Degree Theorem~\ref*{thm:huppert-degree}}, $G$ does not have an irreducible representation of dimension $3$. The second assertion follows since  \hyperref[cor:metacyclic-rep-1]{Corollary~\ref*{cor:metacyclic-rep}} shows that $\rho$ contains a summand of degree $1$. \\
By  \hyperref[cor:metacyclic-rep]{Corollary~\ref*{cor:metacyclic-rep}} \ref{cor:metacyclic-rep-1}, $G$ contains a linear character. We only prove the assertion in the case where $\rho$ is the direct sum of an irreducible degree $2$ and two degree $1$ representations, $\rho = \rho_2 \oplus \chi \oplus \chi'$ (the other case is proved in the same way). Let $T$ be a $4$-dimensional complex torus $T$ endowed with a free $G$-action such that the associated complex representation is $\rho$.\\ 
We prove the statement by contradiction, i.e., we assume that $\chi$ and $\chi'$ are both trivial. According to  \hyperref[isogeny]{Section~\ref*{isogeny}}, the complex torus $T$ is equivariantly isogenous to $T_1 \times T_2$, where the linear action of $G$ on $T_1$ is via $\rho_2$ and trivial on $T_2$.
	%By  \hyperref[metacyclic-eigenvalues]{Lemma~\ref*{metacyclic-eigenvalues}}, both eigenvalues of $\rho_2(g)$ have order $m$. {\ rewrite accor} Therefore
	%\begin{align*}
	%	A_1 := \im(\rho_2(g) - I) \text{ and } A_2 := \ker(\rho(g) - I)^0
	%\end{align*}
	%are Abelian surfaces, and $A$ is equivariantly isogenous to $A_1 \times A_2$. Thus there is a finite group of translations $H$ such that
	We write the action $G$ on $T = (T_1 \times T_2)/H$ as follows:
	\begin{align*}
	&g(z_1,z_2) = (\rho_2(g)z_1 + a_1, \ z_2 + a_2), \\
	&h(z_1,z_2) = (\rho_2(h)z_1 + b_1, \ z_2 + b_2).
	\end{align*}
	Spelling out the relation $h^{-1}gh = g^r$ shows that
	\begin{align*}
	(t_1, \ (r-1)b_2) \in H, \qquad \text{ where } \quad t_1 := \sum_{i=0}^{r-1} \rho_2(g)^i a_1 - \left(\rho_2(h^{-1})(\rho_2(g)-\id_{T_1})b_1 + \rho_2(h^{-1})a_1\right).
	\end{align*}
	Thus the action of $g^{r-1}$ is equal to
	\begin{align*}
	g^{r-1}(z_1,z_2) = \left(\rho_2(g)^{r-1}z_1 + \sum_{i=0}^{r-2} \rho_2(g)^i a_1 - t_1 , \ z_2\right).
	\end{align*}
	Since $\rho_2(g)^{r-1}$ does not have the eigenvalue $1$ (\hyperref[lemma:repmetacyc]{Lemma~\ref*{lemma:repmetacyc}} \ref{lemma:repmetacyc4}), $g^{r-1}$ has a fixed point on $T = (T_1 \times T_2)/H$.
\end{proof}

The above Proposition has several interesting and useful consequences.

\begin{cor} \label{cor-metacyclic-derived}
	Let $G$ be a finite group. If the derived subgroup $[G,G]$ of $G$ contains a special quotient of a metacyclic group, then $G$ is not hyperelliptic in dimension $4$.
\end{cor}

\begin{proof}
	Any linear character of $G$ maps $[G,G]$ to $1$. By \hyperref[cor:metacyclic-rep]{Corollary~\ref*{cor:metacyclic-rep}} \ref{cor:metacyclic-rep-1} and \hyperref[lemma-two-generators]{Lemma~\ref*{lemma-two-generators}}, the special quotient of a metacyclic group contained in $[G,G]$ cannot act freely on any hyperelliptic fourfold.
\end{proof}

\begin{example}
	The above corollary shows for instance that $\SL(2,3)$ is not hyperelliptic in dimension $4$: its derived subgroup is $Q_8$, the quaternion group of order $8$.
\end{example}

\begin{prop} \label{prop:metacyclic-times-cd}
Consider a special quotient $G$ of $G(m,n,r)$ and a cyclic group $C_d = \langle k \rangle$, whose order $d \geq 2$ is divisible by the exponent of $G/[G,G]$. Suppose that $\rho_2$ is an irreducible degree $2$ representation of $G \times C_d$ satisfying the following two properties:
\begin{enumerate}
	\item[(i)] $k \in \ker(\rho_2)$, and
	\item[(ii)] $\rho_2(h)$ does not have the eigenvalue $1$.
\end{enumerate}
Furthermore, let $\rho \colon G \times C_d \to \GL(4,\CC)$
be a faithful representation that contains $\rho_2$ as an irreducible sub-representation. Then $\rho$ does not occur as the complex representation of any hyperelliptic fourfold. %{\ reformulate and introduce, $\chi$, $\chi'$ also in the statement. Drop the assumption that the exponent of $G^{\ab}$ divides $d$ and require $\chi^d = \chi'^d = 1$}
\end{prop}

\begin{proof}
Assume that $\rho$ is the complex representation of some hyperelliptic fourfold. Then, according to  \hyperref[cor:metacyclic-rep]{Corollary~\ref*{cor:metacyclic-rep}} \ref{cor:metacyclic-rep-1},
\begin{align*}
	\rho = \rho_2 \oplus \chi \oplus \chi',
\end{align*}
where $\chi$ and $\chi'$ are linear characters and $\chi'|_{G}$ is non-trivial. Since $\rho_2(g)^{r-1}$ does not have the eigenvalue $1$ (\hyperref[lemma:repmetacyc]{Lemma~\ref*{lemma:repmetacyc}} \ref{lemma:repmetacyc4}), the matrix
\begin{align*}
	\rho(g^{r-1}k) = \diag(\rho_2(g^{r-1}), ~ \chi(k), ~ \chi'(k))
\end{align*}
can only have the eigenvalue $1$ if at least one of $\chi(k)$, $\chi'(k)$ is $1$. We show that both possibilities lead to a contradiction: \\ 

\emph{The case $\chi'(k) = 1$:} in this case $\chi(k) \neq 1$ by faithfulness of $\rho$. As already stated above, $\chi'|_G$ is non-trivial and thus we are in one of the following two situations: 
\begin{itemize}
	\item \underline{$\chi'(g) \neq 1$:} since $\rho_2(g)$ does not have the eigenvalue $1$, the matrix $\rho(g)$ does not have the eigenvalue $1$, unless $\chi(g) = 1$. However then
	\begin{align*}
	\rho(gk) = \diag(\rho_2(g), \ \chi(k), \ \chi'(g))
	\end{align*}
	does not have the eigenvalue $1$.
	\item \underline{$\chi'(h) \neq 1$:} this is excluded in the same way, since $\rho_2(h)$ does not have the eigenvalue $1$ according to our hypothesis (ii). 
\end{itemize}

\emph{The case $\chi(k) = 1$:} here, $\chi'(k)$ is a primitive $d$th root of unity. Let $i_0,j_0$ such that $\chi'(gk^{i_0}), \chi'(hk^{j_0}) \neq 1$. It is necessary for the matrices
\begin{align*}
	&\rho(gk^{i_0}) = \diag(\rho_2(g), \ \chi(g), \ \chi'(gk^{i_0})), \\
	&\rho(hk^{j_0}) = \diag(\rho_2(h), \ \chi(h), \ \chi'(hk^{j_0}))
\end{align*}
to have the eigenvalue $1$ that $\chi(g) = \chi(h) = 1$, i.e., $\chi$ is trivial. However, since the exponent of $G/[G,G]$ divides $d = \ord(k)$ by hypothesis, there are indices $i_1, j_1$ such that
\begin{align*}
	\chi'(gk^{i_1}) = \chi'(hk^{j_1}) = 1.
\end{align*}
By assumption, $\langle gk^{i_1}, hk^{j_1} \rangle \cong G$ is hyperelliptic in dimension $4$, and the associated complex representation contains two copies of the trivial representation. This contradicts  \hyperref[lemma-two-generators]{Proposition~\ref*{lemma-two-generators}}.
\end{proof}

\begin{rem} \label{rem:meta-weakened-hypotheses}
The hypotheses in \hyperref[prop:metacyclic-times-cd]{Proposition~\ref*{prop:metacyclic-times-cd}} can be weakened: 
\begin{enumerate}[label=(\roman*), ref=(\roman*)]
	\item \label{rem:meta-weakened-hypotheses-1} First of all, the assumption ``$k \in \ker(\rho_2)$'' can be replaced by the weaker assumption ``there is $u \in G$ such that the first two diagonal entries of $\rho_2(u) = \diag(\zeta_d, ~ \zeta_d)$'': then $u$ is central and $ku^j \in \ker(\rho_2)$ for a suitable $j$, and hence $k$ can be replaced by $ku^j$.
	\item \label{rem:meta-weakened-hypotheses-2} Moreover, the assumption ``$d$ is divisible by the exponent of $G/[G,G]$'' can be weakened to the assumption ``the $d$th powers of $\chi$ and $\chi'$ are trivial''.
\end{enumerate}
\end{rem}

\begin{prop} \label{metacyclic-c2-excluded}
The following groups are not hyperelliptic in dimension $4$:
\begin{enumerate}[ref=(\theenumi)]
	\item \label{cor:meta-excl-1} $Q_2 \times C_2$ (ID $[16,12]$),
	\item \label{cor:meta-excl-2} $G(3,4,2) \times C_2$ (ID $[24,7]$),
	%{\ \item \label{cor:meta-excl-3} $S_3 \times C_2 \times C_2$ (ID $[24,14]$),
	\item \label{cor:meta-excl-3} $G(3,8,2) \times C_2$ (ID $[48,9]$), 
%	\item \label{cor:meta-excl-4} $G(8,2,5) \times C_4$ (ID $[64,85]$).
\end{enumerate}
%\begin{align*}
%%	&G(8,2,5) \times C_3 \text{ (ID [48,24])}, \\
%%	&G(8,2,3) \times C_3 \text{ (ID [48,26])}, \\
%%	&(D_4 \curlyvee C_4) \times C_3 \text{ (ID [48,47])}, \\
%&G(8,2,5) \times C_4 \text{ (ID [64,85])}.
%%	&D_4 \times C_3 \times C_3 \text{ (ID [72,37])}, \\
%%	&Q_8 \times C_3 \times C_3 \text{ (ID [72,38])}, \\
%	\end{align*}
\end{prop}

\begin{proof}
Each of the three groups \ref{cor:meta-excl-1}  -- \ref{cor:meta-excl-3} is of the form $G \times C_2$, where $G$ is a special quotient quotient of a metacyclic group. As usual, we denote generators of $G$ by $g$ and $h$, and by $k$ a central element of order $2$ such that $\langle g,h,k \rangle = G \times C_2$. Furthermore, we denote by $\rho \colon G \times C_2 \to \GL(4,\CC)$ a representation satisfying the usual three properties
\begin{enumerate}[label=(\Roman*), ref=(\Roman*)]
	\item \label{rho-I} $\rho$ is faithful, 
	\item \label{rho-II} every matrix in the image of $\rho$ has the eigenvalue $1$, and
	\item \label{rho-III} if $\ord(u) = 8$, then $\rho(u)$ has exactly two eigenvalues of order $8$, and these are non-complex conjugate (see the \hyperref[order-cyclic-groups]{Integrality Lemma~\ref*{order-cyclic-groups}} \ref{ocg-3}).
\end{enumerate}

Then $\rho$ is equivalent to a direct sum $\rho_2 \oplus \chi \oplus \chi'$, where $\rho_2$ is irreducible of degree $2$ and $\chi$, $\chi'$ are linear characters. \\
	
%	
%Denote by $\tilde G$ one of the groups \ref{cor:meta-excl-1}  -- \ref{cor:meta-excl-4} and by $T$ a $4$-dimensional complex torus, such that $\tilde G \subset \Bihol(T)$ with associated faithful complex representation $\rho \colon \tilde G \to \GL(4,\CC)$. We show that $\tilde G$ cannot act freely in each case. Observe that $\tilde G$ contains a special quotient $G$ of a metacyclic group. Hence, according to \hyperref[cor:metacyclic-rep]{Corollary~\ref*{cor:metacyclic-rep}} \ref{cor:metacyclic-rep-1} and \hyperref[lemma-two-generators]{Proposition~\ref*{lemma-two-generators}} \ref{ltg-2}, it suffices to exclude the case in which $\rho$ is the direct sum of an irreducible degree $2$ representation $\rho_2$ and two linear characters $\chi$ and $\chi'$, where $\chi'|_G$ is non-trivial. We now treat the five groups separately; as usual, we denote generators of $G$ by $g$ and $h$, and we let $k$ be such that $\tilde G = G \times \langle k \rangle$.  \\
	
\ref{cor:meta-excl-1} The quaternion group $Q_8$ of order $8$ has a unique irreducible representation of degree $2$, namely
\begin{align*}
	g \mapsto \begin{pmatrix}
	0 & -1 \\ 1 & 0
	\end{pmatrix}, \qquad h \mapsto \begin{pmatrix}
	i  & \\ & -i
	\end{pmatrix}.
\end{align*}
Hence $\rho_2|_{Q_8}$ is the above representation -- in particular, $\rho_2(g^2) = \diag(-1, ~ -1)$, thus \ref{rem:meta-weakened-hypotheses-1} of \hyperref[rem:meta-weakened-hypotheses]{Remark~\ref*{rem:meta-weakened-hypotheses}} is satisfied. Moreover, $Q_8/[Q_8,Q_8] \cong C_2 \times C_2$ and $\rho_2(h)$ does not have the eigenvalue $1$. The non-hyperellipticity of $Q_8 \times C_2$ in dimension $4$ thus follows from \hyperref[prop:metacyclic-times-cd]{Proposition~\ref*{prop:metacyclic-times-cd}}. \\

\ref{cor:meta-excl-2} The group $G(3,4,2)$ has two irreducible degree $2$ representations:
\begin{align*}
	&g \mapsto \begin{pmatrix}
		\zeta_3 & \\ & \zeta_3^2
	\end{pmatrix}, \qquad h \mapsto \begin{pmatrix}
		0 & -1 \\ 1 & 0
	\end{pmatrix}, \quad \text{ and } \\
	&g \mapsto \begin{pmatrix}
		\zeta_3 & \\ & \zeta_3^2
	\end{pmatrix}, \qquad h \mapsto \begin{pmatrix}
		0 & 1 \\ 1 & 0
	\end{pmatrix}.
\end{align*}
The first one of these is faithful, while the second is not. \\

\underline{Claim 1:} There is no hyperelliptic fourfold with holonomy group $G(3,4,2)$, whose complex representation is -- up to automorphisms and equivalence -- given by:
\begin{align*}
\tilde \rho(g) = \begin{pmatrix}
	\zeta_3 &&& \\ & \zeta_3^2 && \\ && 1 & \\ &&& 1
\end{pmatrix}, \qquad \tilde \rho(h) = \begin{pmatrix}
0 & \pm 1 && \\ 1 & 0 && \\ && 1 & \\ &&& i
\end{pmatrix}.
\end{align*}

\underline{Proof of Claim 1:}
Let $T$ be a $4$-dimensional complex torus admitting an action of $G(3,4,2)$, whose associated complex representation is $\tilde \rho$. We write the action of $G(3,4,2)$ on $T$ as follows:
\begin{align*}
	&g(z) = (\zeta_3 z_1 + a_1, ~ \zeta_3^2 z_2 + a_2, ~ z_3 + a_3, ~ z_4 + a_4), \\
	&h(z) = (\pm z_2 + b_1, ~ z_1 + b_2, ~ z_3 + b_3, ~ iz_4 + b_4).
\end{align*}
Calculating the last two coordinates of $g^2 - h^{-1}gh$ (which is the zero endomorphism of $T$) yields that
\begin{align*}
	(w_1, ~ w_2, ~ a_3, ~ (i+2)a_4), \qquad \text{ where } \quad (w_1, ~ w_2) = (-\zeta_3^2 a_1 - a_1 + (1-\zeta_3^2)b_2, ~ -a_1 - \zeta_3 a_2 \pm (1-\zeta_3)b_1)
\end{align*}
is zero in $T$. Similarly, $gh^2 - h^2g$ is zero as well, and thus
\begin{align*}
	(u_1, ~ u_2, ~ 0, ~ 2a_4), \qquad \text{ where }\quad (u_1, ~ u_2) = ((1 \mp 1)a_1 + (\zeta_3-1)(b_1 \pm b_2), ~ 
	(1 \mp 1)a_2) + (\zeta_3^2-1)(b_1+b_2))
\end{align*}
is zero in $T$ as well. Taking differences, we obtain that 
\begin{align} \label{G(3,4,2)-element}
	(w_1 - u_1, ~ w_2 - u_2, ~ a_3,~ ia_4)
\end{align}
is zero, too. Finally, we apply $\tilde\rho(h^3)$ and obtain that
\begin{align*}
	(w_2-u_2, ~ \pm (w_1-u_1), ~ a_3,~ a_4) = 0 \text{ in } T.
\end{align*}
Hence the action of $g$ on $T$ simplifies to
\begin{align*}
	g(z) = (\zeta_3z_1 + a_1 - (w_2-u_2), ~ \zeta_3^2 + a_2 \mp (w_1-u_1), ~ z_3, ~ z_4),
\end{align*}
which shows that $g$ does not act freely on $T$. $\hfill \Box_{\text{Claim 1}}$ \\

\underline{Claim 2:} If $\rho_2|_{G(3,4,2)}$ is non-faithful, then there is no hyperelliptic fourfold with holonomy group $G(3,4,2) \times C_2$ whose complex representation is $\rho$. \\

\underline{Proof of Claim 2:} Since $\rho$ is faithful, but $\rho_2|_{G(3,4,2)}$ is non-faithful, we may assume that $\chi'(h)$ is a primitive fourth root. Thus, by replacing $k$ by $kh^2$ if necessary, we may assume that the last diagonal entry of $\rho(k)$ is $1$. 
Since
\begin{align*}
\rho(gh^2) = \diag(\zeta_3, ~ \zeta_3^2, ~ \chi(h^2), ~ -1)
\end{align*}
has the eigenvalue $1$, we obtain that
\begin{align} \label{G(3,4,2)-eq}
	\chi(h) =1 \qquad \text{ or } \qquad \chi(h) = -1.
\end{align}
Now we use that $\rho(gh^2k)$ has the eigenvalue $1$, so that the third diagonal entry of $\rho(k)$ is $1$, as well. Thus it suffices to exclude the case $\rho(k) = \diag(-1, ~ -1, ~ 1, ~ 1)$. \\

The case $\chi(h) = 1$ is now excluded in view of Claim 1. According to (\ref{G(3,4,2)-eq}), it remains to exclude the case $\chi(h) = -1$. Here, we let $T$ and $g \in \Bihol(T)$ be as in Claim 1 -- we show that
\begin{align*}
	h(z) = (z_2 + b_1, ~ z_1 + b_2, ~ -z_3 + b_3, ~ iz_4 + b_4)
\end{align*}
has a fixed point on $T$. Indeed, since $h^4 = \id_T$, we obtain that the element
\begin{align*}
(b_1 + b_2,  ~ b_1 + b_2, ~ 0, ~ 0)
\end{align*}
is zero in $T$. It follows that
\begin{align} \label{G(3,4,2)-element-2}
	(\rho(g^2) - \id_T)(b_1 + b_2,  ~ b_1 + b_2, ~ 0, ~ 0) = (\zeta_3^2(b_1+b_2), ~ \zeta_3(b_1+b_2), ~ 0, ~ 0) = 0 \text{ in } T.
\end{align}
Using $\zeta_3^2 + \zeta_3 + 1 = 0$, we then calculate
\begin{align*}
	h&\left(\zeta_3^2 b_2 - \zeta_3 b_1¸, ~ 0 , ~ \frac{b_3}{2}, ~ \frac{(1+i)b_4}{2} \right ) \\ = &\left(b_1, ~ -\zeta_3 (b_1 + b_2), ~ \frac{b_3}{2}, ~ \frac{(1+i)b_4}{2}\right) \stackrel{(\ref{G(3,4,2)-element-2})}{=} \left(\zeta_3^2 b_2 - \zeta_3 b_1, ~ 0, ~ \frac{b_3}{2}, ~ \frac{(1+i)b_4}{2}\right).
\end{align*}
This shows that $h$ has a fixed point, which excludes the case $\chi(h) = -1$ (observe that we did not need the central element $k$ to exclude the case $\chi(h) = -1$, $\chi'(h) = i$, see also \cite[Lemma 5]{Uchida-Yoshihara} and \hyperref[s3-fixed-point]{Lemma~\ref*{s3-fixed-point}}). $\hfill \Box_{\text{Claim 2}}$ \\

Finally, we deal with the last remaining case, namely the case in which $\rho_2|_{G(3,4,2)}$ is faithful. Here, $\rho_2(h^2) = \diag(-1, ~ -1)$, so that
\ref{rem:meta-weakened-hypotheses-1} of \hyperref[rem:meta-weakened-hypotheses]{Remark~\ref*{rem:meta-weakened-hypotheses}}  is satisfied. Furthermore, since $\rho(h)$ has the eigenvalue $1$ but $\rho_2(h)$ does not, we may assume that $\chi(h) = 1$. In view of Claim 1, the only remaining cases left to exclude are $\chi'(h) \in \{1,-1\}$: here, however, \ref{rem:meta-weakened-hypotheses-2} of \hyperref[rem:meta-weakened-hypotheses]{Remark~\ref*{rem:meta-weakened-hypotheses}} is satisfied, and thus these cases are excluded by \hyperref[prop:metacyclic-times-cd]{Proposition~\ref*{prop:metacyclic-times-cd}}. \\

%First of all, the relation $h^{-1}gh = g^2$ implies that $h^2$ is central and thus $\rho_2(h^2) = \pm \diag(1, ~ 1)$. It follows that either
%\begin{enumerate}[label=(\alph*), ref=(\alph*)]
%	\item \label{G(3,4,2)-a} the weakened assumption \ref{rem:meta-weakened-hypotheses-1} of \hyperref[rem:meta-weakened-hypotheses]{Lemma~\ref*{rem:meta-weakened-hypotheses}} is satisfied (if $\rho_2(h)$ has order $4$), or
%	\item \label{G(3,4,2)-b} $\rho_2|_{G(3,4,2)}$ is not faithful.
%\end{enumerate}

\ref{cor:meta-excl-3} Observe that $h^2$ is central, and thus $\rho_2(h^2)$ is a multiple of the identity. It follows that either $\rho_2(h)$ has two eigenvalues of order $8$, or $\rho_2|_{G(3,8,2)}$ is not faithful. In the latter case, since $\rho$ satisfies the properties \ref{rho-I} and \ref{rho-III}, both $\chi(h)$ and $\chi'(h)$ must be primitive eighth roots: however then $\rho(gh^4)$ does not have the eigenvalue $1$, violating property \ref{rho-II}. Hence $\rho_2(h)$ has two eigenvalues of order $8$. Hence, since $\rho(h)$ must have the eigenvalue $1$, we may hence assume that $\chi(h) = 1$. We now exclude the case that $\chi'(h)$ is a primitive fourth root exactly as in Claim 1 above. The remaining cases, $\chi'(h) \in \{1,-1\}$, are then also excluded in view of \ref{rem:meta-weakened-hypotheses-2} of \hyperref[rem:meta-weakened-hypotheses]{Remark~\ref*{rem:meta-weakened-hypotheses}} and \hyperref[prop:metacyclic-times-cd]{Proposition~\ref*{prop:metacyclic-times-cd}}.
\end{proof}

Another result in a similar direction is

\begin{prop} \label{prop:metacyclic-c3}
Consider a special quotient $G$ of $G(m,n,r)$, where $m,n \in \{2,4,8\}$ and a cyclic group $C_3 = \langle k \rangle$ of order $3$. Suppose that $G \times C_3$ is hyperelliptic in dimension $4$, and denote by $\rho \colon G \times C_3 \to \GL(4,\CC)$
the associated complex representation of some hyperelliptic fourfold $T/(G \times C_3)$ with holonomy group $G \times C_3$. Then:
\begin{enumerate}[ref=(\theenumi)]
	\item  $\rho$ is equivalent to a direct sum $\rho_2 \oplus \chi \oplus \chi'$, where $\rho_2$ is irreducible of degree $2$ and $\chi$, $\chi'$ are linear characters, and
	\item $k \in \ker(\chi) \cap \ker(\chi')$, and thus $k \notin \ker(\rho_2)$.
\end{enumerate}
\end{prop}

\begin{proof}
The first statement follows from \hyperref[cor:metacyclic-rep]{Corollary~\ref*{cor:metacyclic-rep}} \ref{cor:metacyclic-rep-1}, hence only the second statement requires a proof. 

 Here, we assume that $k \notin \ker(\chi)$ or $k \notin \ker(\chi')$ and work towards a contradiction.

\underline{Step 1:} One of the following situations occurs:
\begin{enumerate}[label= (\roman*), ref=(\roman*)]
	\item \label{enum:meta-i} $\chi(k) \neq 1$ and $\chi'(g) = \chi'(k) = 1$, or
	\item \label{enum:meta-ii} $\chi'(k) \neq 1$ and $\chi(g) = \chi(k) = 1$.
\end{enumerate}

To verify the claim, we first observe that since
\begin{align*}
	\rho(gk) = \diag(\rho_2(gk), \ \chi(gk), ~ \chi'(gk))
\end{align*}
has the eigenvalue $1$, we must have $\chi(gk) = 1$ or $\chi'(gk) = 1$. Since the orders of $m$ and $k$ are coprime, we obtain that ``$\chi(gk) = 1$'' is equivalent to ``$\chi(g) = 1$ and $\chi(k) = 1$'', and similarly for $\chi'$. This completes Step 1. \\

\underline{Step 2:} Step 1 implies that $\chi \neq \chi'$. Consequently, the discussion in \hyperref[isogeny]{Section~\ref*{isogeny}}, the complex torus $T$ is equivariantly isogenous to $S \times E \times E'$, where
\begin{itemize}
	\item $S \subset T$ is a sub-torus of dimension $2$, on which the linear part of the action of $G \times C_3$ is given by $\rho_2$, and
	\item $E,E' \subset T$ are elliptic curves, on which the linear part of the action of $G \times C_3$ is given by $\chi$ and $\chi'$, respectively.
\end{itemize}

\underline{Step 3:} We show here may assume that $G$ acts linearly on $E$ (resp. on $E'$) in Case \ref{enum:meta-i} (resp. Case \ref{enum:meta-ii}). \\

We treat Case \ref{enum:meta-i} first. Let $u \in G$ and write its action on $T \sim_{(G \times C_3)-\isog} S \times E \times E'$ as follows:
\begin{align} \label{eq:meta:u}
	u(w,z,z') = (\rho_2(u)w + \tau_S(u), ~ \chi(u)z + \tau_E(u), ~ \chi'(u)z' + \tau_{E'}(u)).
\end{align}
We may assume that 
\begin{align} \label{eq:meta:torsion}
	|G| \cdot (\tau_S(u),~ \tau_E(u),~ \tau_{E'}(u)) = 0 \text{ in } T,
\end{align}
see \hyperref[rem:translation-g-torsion]{Remark~\ref*{rem:translation-g-torsion}}. \\
After replacing $k$ by $k^2$ if necessary, we may assume that $\chi(k) = \zeta_3$. Moreover, choosing the origin in $E$ suitably allows us to assume $\tau_E(k) = 0$. The action of $k$ is, therefore, given by
\begin{align} \label{eq:meta:k}
	k(w,z,z') = (\rho_2(k)w + \tau_S(k), ~ \zeta_3 z, ~ z' + \tau_{E'}(k)).
\end{align}

We now exploit that $ku - uk$ is the zero endomorphism of $T$. Computing it using the descriptions (\ref{eq:meta:u}), (\ref{eq:meta:k}) yields that $ku-uk$ is of the following form:
\begin{align*} 
	ku-uk= (v_1, ~ (\zeta_3-1)\tau_E(u), ~ v_3) \text{ for some } v_1 \in S, ~ v_3 \in E'.
\end{align*}
It follows that 
\begin{align} \label{eq:meta:3}
	(\rho(k^2)-\id_T)(ku-uk) = ((\rho_2(k^2) - \id_S)v_1, ~ 3\tau_E(u), ~ 0 ) = 0 \text{ in } T.
\end{align}
Since $\gcd(3,|G|) = 1$, an appropriate integral linear combination of (\ref{eq:meta:torsion}) and (\ref{eq:meta:3}) of the form
\begin{align} \label{eq:meta:lincomb}
	(v_1', ~ \tau_E(u), ~ |G|\cdot \tau_{E'}(u)).
\end{align}
Subtracting the element (\ref{eq:meta:lincomb}) (which is zero in $T$ as well) from (\ref{eq:meta:u}) shows that we may assume that $\tau_E(u) = 0$, as desired. \\
Observe that we have only exploited the two properties $\chi(k) \neq 1$ and $\chi'(k) = 1$, and thus the same proof works in Case \ref{enum:meta-ii}. \\

\underline{Step 4:} Here, we reach the desired contradiction by showing that $G$ does not act freely on $T$. \\

The claim follows from  \hyperref[lemma-two-generators]{Proposition~\ref*{lemma-two-generators}}, if both $\chi|_G$ and $\chi'|_G$ is trivial, so that we will assume that $\chi'|_G$ is non-trivial. \\

We first investigate Case \ref{enum:meta-i}. 
The previous steps show that we may write the action of $g$ and $h$ on $T \sim_{(G \times C_3)-\isog} S \times E \times E'$ as follows:
\begin{align*}
	&g(w,z,z') = (\rho_2(g)w + \tau_S(g), ~ \chi(g)z, ~ z' + \tau_{E'}(g)), \\
	&h(w,z,z') = (\rho_2(h)w + \tau_S(h), ~ \chi(h)z, ~ \chi'(h)z' + \tau_{E'}(h)).
\end{align*}
Since $\chi'|_G$ is non-trivial, $\chi'(h) \neq 1$. Hence, by changing the origin in $E'$ appropriately, we may assume that $\tau_{E'}(h) = 0$. \\
We now use that $g^r$ and $h^{-1}gh$ define the same automorphisms of $T$. Their difference, which is zero, is of the form
\begin{align*}
	v := (v_1'', ~ 0, ~ (r-\overline{\chi'(h)})\tau_{E'}(g)).
\end{align*}
It follows that
\begin{align} \label{eq:meta4}
	(\rho(h)+r\id_T)v = ((\rho_2(h) - r\id_S)v'', ~ 0, ~ (r^2+1)\tau_{E'}(g))
\end{align}
is equal to zero, too. However, $m \in \{4,8\}$ and $\gcd(r,m) = 1$ implies that $r^2+1 \equiv 2 \pmod m$. Finally, as in Step 2, we exploit that (\ref{eq:meta4}) is zero to obtain that $g^{r^2+1} = g^2$ acts not only linearly on $E$, but also on $E'$. Since $\rho_2(g^2)$ does not have the eigenvalue $1$ (see \hyperref[cor:metacyclic-rep]{Corollary~\ref*{cor:metacyclic-rep}} \ref{cor:metacyclic-rep-2}), the action of $g^2$ on $T$ is not free. \\

In Case \ref{enum:meta-ii}, we write the action of $g$ and $h$ on $T$ as
\begin{align*}
&g(w,z,z') = (\rho_2(g)w + \tau_S(g), ~ z + \tau_E(g), ~ \chi'(g) z'), \\
&h(w,z,z') = (\rho_2(h)w + \tau_S(h), ~ \chi(h)z + \tau_{E}(h), ~ \chi'(h)z').
\end{align*}
We reach the same contradiction as in Case \ref{enum:meta-i} if $\chi(h) \neq 1$, so we are left with the case $\chi(h) = 1$. Here, the difference of $g^r - h^{-1}gh$ is of the form
\begin{align*}
	(v_1'', ~ (r-1)\tau_E(g), ~ 0),
\end{align*}
which shows that $g^{r-1}$ does not act freely on $T$.

%This follows immediately, because
%\begin{align*}
%	\rho(gk) = \diag(\rho_2(gk), ~ \chi(gk), ~ \chi'(gk))
%\end{align*}

\end{proof}

\begin{cor} \label{metacyclic-2-group-times-c3-excluded}
The following statements hold:
\begin{enumerate}[ref=(\theenumi)]
\item \label{cor:metacyc-c3-1} if $G$ is a special quotient of $G(4,n,r)$, $n \in \{2,4,8\}$, then $G \times C_3^2$ is not hyperelliptic in dimension $4$.
\item \label{cor:metacyc-c3-2} if $G$ is a special quotient of $G(8,n,r)$, $n \in \{2,4,8\}$, then $G \times C_3$ is not hyperelliptic in dimension $4$.
\item \label{cor:metacyc-c3-3} the group $(D_4 \curlyvee C_4) \times C_3$ (ID $[72,37]$) is not hyperelliptic in dimension $4$. Here, 
\begin{align*}
	D_4 \curlyvee C_4 = \langle r,s,\ell ~ | ~ r^4 = s^2 = (rs)^2 = \ell^4 = [r,\ell] = [s, \ell] = 1, ~ \ell^2 = r^2\rangle
\end{align*}
is the central product of $D_4$ and $C_4$.
\end{enumerate}
\end{cor}

\begin{proof}
Denote by $\rho$ degree $4$ representation of \ref{cor:metacyc-c3-1} $G \times C_3^2$, \ref{cor:metacyc-c3-2} $G \times C_3$ or \ref{cor:metacyc-c3-3} $(D_4 \curlyvee C_4)	\times C_3$ satisfying the following three properties, which are also satisfied by the complex representation of a hyperelliptic fourfold:
\begin{enumerate}[label=(\Roman*), ref=(\Roman*)]
	\item $\rho$ is faithful,
	\item every matrix in the image of $\rho$ has the eigenvalue $1$, and
	\item \label{rho-III-2} no matrix in the image of $\rho$ has eigenvalues of order $24$, or an eigenvalue of order $12$ with multiplicity $2$ (see the \hyperref[order-cyclic-groups]{Integrality Lemma~\ref*{order-cyclic-groups} \ref{ocg-3}})
\end{enumerate}

We have seen that these properties imply that $\rho$ decomposes as the direct sum of an irreducible representation $\rho_2$ of degree $2$ and two linear characters. We claim that $\rho$ does not occur as the complex representation of a hyperelliptic fourfold in each of the three cases. \\

\ref{cor:metacyc-c3-1} By faithfulness of $\rho$, there is an element $k \in C_3^2$ such that $\chi(k) \neq 1$ or $\chi'(k) \neq 1$. \hyperref[prop:metacyclic-c3]{Proposition~\ref*{prop:metacyclic-c3}} then shows that $\rho|_{G \times \langle k \rangle}$ does not occur as the complex representation of some hyperelliptic fourfold. \\ 

\ref{cor:metacyc-c3-2} According to \hyperref[cor:metacyclic-rep]{Corollary~\ref*{cor:metacyclic-rep}} \ref{cor:metacyclic-rep-2}, $\rho_2(g)$ has two eigenvalues of order $8$. Hence, by property \ref{rho-III-2}, we obtain that $k \in \ker(\rho_2)$ and thus \hyperref[prop:metacyclic-c3]{Proposition~\ref*{prop:metacyclic-c3}} shows the claim. \\

\ref{cor:metacyc-c3-3} Since $\ell^2 = r^2$ and $\rho_2|_{D_4}$ is faithful, $\rho_2(\ell) = \pm i I_2$. In particular, property \ref{rho-III-2} implies that $k \in \ker(\rho_2)$. Again, the statement follows from \hyperref[prop:metacyclic-c3]{Proposition~\ref*{prop:metacyclic-c3}} applied to the group $D_4 \times C_3$. 
\end{proof}

\chapter{Outline of the Classification} \label{chapter:outline}

Our goal is to classify the hyperelliptic groups in dimension $4$. Since the proof of our classification result (\hyperref[mainthm]{Main Theorem~\ref*{mainthm}}) is quite involved, it seems worthwhile to outline the strategy of proof.\\

Let $G$ be hyperelliptic in dimension $4$ and let $\rho \colon G \to \GL(4,\CC)$ be the complex representation of some hyperelliptic fourfold. We have seen in \hyperref[cor:group-order-with-bounds]{Corollary~\ref*{cor:group-order-with-bounds}} that $|G| = 2^a \cdot 3^b \cdot 5^c \cdot 7^d$, where $c,d \leq 1$. We find all $G$'s in the following cases:
\begin{enumerate}[ref=(\theenumi)]
	\item \label{strategy-1} \underline{$G$ is Abelian:} this step will be completed in \hyperref[section:ab-class]{Section~\ref*{section:ab-class}}. \\
	To achieve the classification, we first investigate how $G$ can embed into $\GL(4,\CC)$ while still satisfying the property that every matrix in the image has the eigenvalue $1$. It turns out that $G$ is a direct product of at most three cyclic groups when written in its elementary divisor normal form (also called Smith normal form), see \hyperref[lemma:many-factors]{Lemma~\ref*{lemma:many-factors}}. Using our knowledge about element orders of hyperelliptic groups in dimension $4$ and about integrality (\hyperref[order-cyclic-groups]{Lemma~\ref*{order-cyclic-groups}}), the actual classification will then be carried out in \hyperref[section:ab-class]{Section~\ref*{section:ab-class}} by showing the (non-)hyperellipticity of certain Abelian groups in dimension $4$.
	\item \label{strategy-2} \underline{$|G| = 2^a$:}  this case is discussed in \hyperref[section-2-sylow]{Section~\ref*{section-2-sylow}}. \\
	We already investigated the case in which $G$ is Abelian in the previous point \ref{strategy-1} (see also \hyperref[abelian-2-grp]{Corollary~\ref*{abelian-2-grp}}), hence we may focus on the case in which $G$ is non-Abelian. We show that the complex representation $\rho$ contains an irreducible degree $2$ summand $\rho_2$ and that there are two possibilities:
	\begin{itemize}
		\item[(i)] either $\rho_2$ is faithful, or
		\item[(ii)] $Z(G)$ is non-cyclic, and hence $G$ does not admit a faithful irreducible representation.
	\end{itemize}
	These two possibilities are then investigated further, mainly by analyzing the determinant exact sequence
	\begin{align*}
	0 \to N \to G \stackrel{\det(\rho_2)}{\to} C_m \to 0.
	\end{align*}
	To give a flavor of the arguments: \hyperref[lemma-table]{Lemma~\ref*{lemma-table}} shows that $m \in \{1,2,4\}$. Now, if $\rho_2$ is faithful, then $N$ contains a unique element of order $2$, so we can use the structure theorem for $2$-groups containing a unique element of order $2$ (\hyperref[thm:hall-unique-p]{Theorem~\ref*{thm:hall-unique-p}}) to conclude that $N \cong C_2, C_4$ or $Q_8$, which then gives $|G| \leq 32$. We use similar methods to prove $|G| \leq 128$ in case (ii). Since $2$-groups contain subgroups of every possible order, we obtain $|G| \leq 32$ in the second case as well if we show that there is no group of order $64$ is hyperelliptic in dimension $4$. This last reduction together with the classification of hyperelliptic $2$-groups in dimension $4$ will then be carried out later in \hyperref[chapter:2a3b]{Chapter~\ref*{chapter:2a3b}}.
	\item \label{strategy-3} \underline{$|G| = 3^b$:} in this case, the structure theory will be developed in \hyperref[section-3-sylow]{Section~\ref*{section-3-sylow}}. 
	Again, we only need to treat the non-Abelian case (see also \hyperref[abelian-case-3-sylow]{Lemma~\ref*{abelian-case-3-sylow}} for a summary of the Abelian case). If $G$ is non-Abelian, we show in \hyperref[3-group-irr-faithful-rep-of-dim-3]{Proposition~\ref*{3-group-irr-faithful-rep-of-dim-3}} that $\rho$ contains a faithful irreducible summand $\rho_3$ of degree $3$. As in \ref{strategy-2}, we then use the determinant exact sequence
	\begin{align*}
	0 \to K \to G \stackrel{\det(\rho_3)}{\to} C_m \to 0.
	\end{align*}
	Since $\rho$ embeds $K$ into $\SL(3,\CC)$, we may use the classification of subgroups of $\SL(3,\CC)$ by Miller, Blichfeldt, and Dickson \cite[Chapter XII]{MBD} as well as Yau and Yu \cite[p.2 f.]{YY} we find that $K$ is (up to isomorphism) to show that $K$ is isomorphic to one of $\{1\}$, $C_3$, $C_3 \times C_3$ or the Heisenberg group of order $27$ (see p. \pageref{3-group-inside-sl3c}). Furthermore, \hyperref[lemma-table]{Lemma~\ref*{lemma-table}} shows that $m \in \{1,3,9\}$, so that we obtain $|G| \leq 243$. A more careful analysis then excludes $|G| \in \{81, 243\}$, so that $|G|\leq 27$.\\ 
	The actual classification (i.e., showing that our candidates all occur) is the performed later in \hyperref[chapter:2a3b]{Chapter~\ref*{chapter:2a3b}}.
\end{enumerate}

\begin{center}
	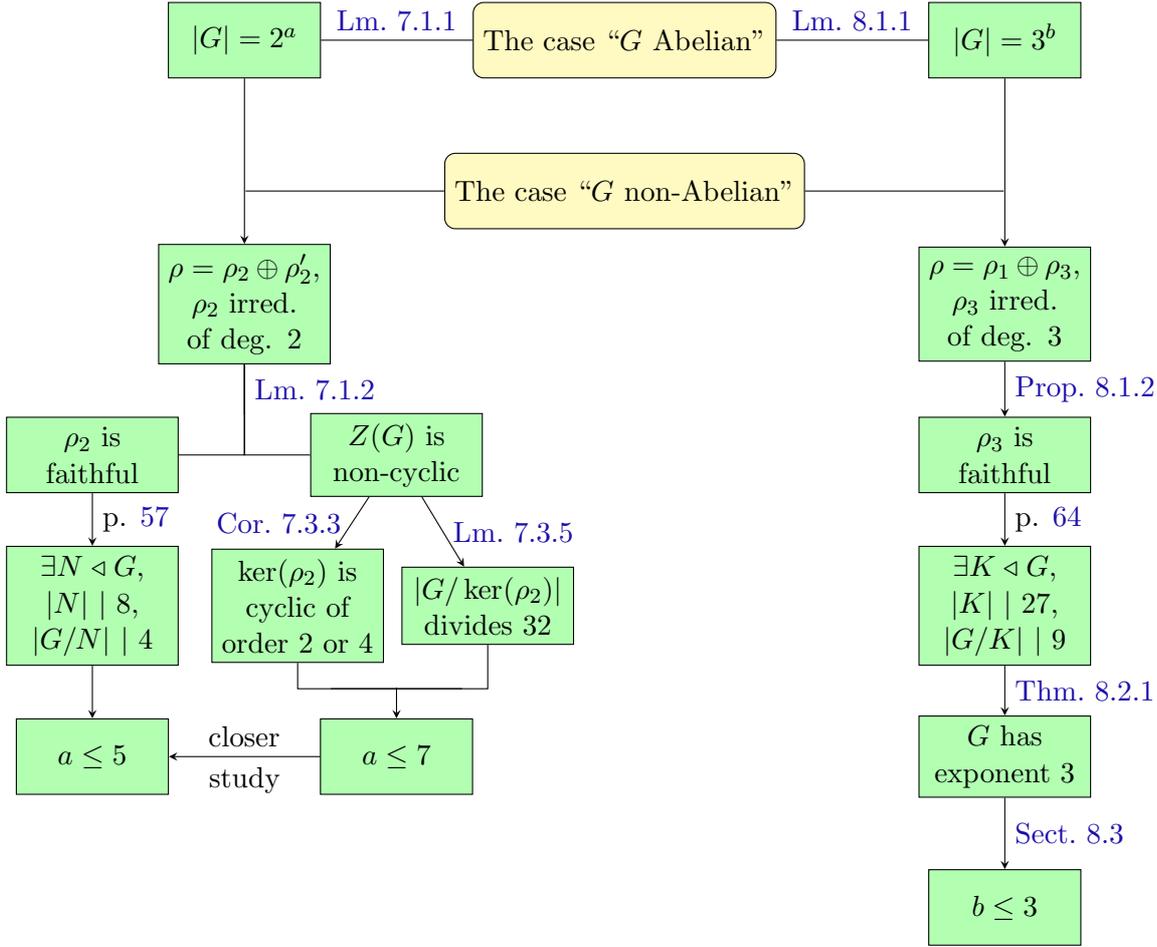
\begin{figure}
		\begin{tikzpicture}[node distance=2cm]
		\tikzstyle{arrow} = [thin,->,>=stealth]
		\tikzstyle{process} = [rectangle, minimum width=2cm, minimum height=1cm, text centered, draw=black, fill=green!30]
		\tikzstyle{process2} = [rectangle, minimum width=2cm, minimum height=1cm, text centered, text width=2cm, draw=black, fill=green!30]
		\tikzstyle{contradiction} = [diamond, minimum width=2cm, minimum height=1cm, text centered, text width = 2cm, draw=black, fill=green!30]
		\tikzstyle{ab} = [rectangle, rounded corners, minimum width=3cm, minimum height=1cm,text centered, draw=black, fill=yellow!30]
		
		%2-Sylow
		\node (2-sylow) [process] {$|G| = 2^a$};
		\node (empty-2sylow) [below of = 2-sylow, xshift=-0.12cm] {};
		\node (2-sylow-rho) [process2, below of = 2-sylow, yshift = -1.5cm] {$\rho = \rho_2 \oplus \rho_2'$, $\rho_2$ irred. of deg. $2$};
		\node (2-faithful) [process2, below of = 2-sylow-rho, xshift = -2cm] {$\rho_2$ is faithful};
		\node (2-non-faithful) [process2, below of = 2-sylow-rho, xshift = 2cm] {$Z(G)$ is non-cyclic};
		\node (2-normal) [process2, below of = 2-faithful] {$\exists N \triangleleft G$, $|N|\ |\ 8$, $|G/N|\ |\ 4$};
		\node (2-32) [process, below of = 2-normal] {$a \leq 5$};
		\node (empty-2sylow2) [below of = 2-sylow-rho, yshift = 0.4cm] {};
		\node (ker-rho2) [process2, below of = 2-non-faithful, xshift=-1.3cm] {$\ker(\rho_2)$ is cyclic of order $2$ or $4$};
		\node (quotient32) [process2, right of = ker-rho2, xshift=0.5cm] {$|G/\ker(\rho_2)|$ divides $32$};
		\node (2-128) [process, below of = 2-non-faithful, yshift = -2cm] {$a \leq 7$};
		%so that arrows are right
		\node (empty-2sylow3) [above of = 2-128, yshift = -1.1cm] {};
		\node (empty-2sylow4) [right of = empty-2sylow3, xshift = -1cm] {};
		\node (empty-2sylow5) [left of = empty-2sylow3, xshift = 1cm] {};
		\node (empty-2sylow6) [above of = empty-2sylow3, yshift = -1.85cm] {};
		
		%the (non) abelian case
		\node (ab) [ab, right of = 2-sylow, xshift = 3cm] {The case ``$G$ Abelian''};
		\node (non-ab) [ab, below of = ab] {The case ``$G$ non-Abelian''};
		
		%3-Sylow
		\node (3-sylow) [process, right of = ab, xshift = 3cm] {$|G| = 3^b$};
		\node (3-sylow-rho) [process2, below of =3-sylow, yshift = -1.5cm] {$\rho = \rho_1 \oplus \rho_3$, $\rho_3$ irred. of deg. $3$};
		\node (3-faithful) [process2, below of = 3-sylow-rho] {$\rho_3$ is faithful};
		\node (empty-3sylow) [below of = 3-sylow, xshift=0.12cm] {};
		\node (3-normal) [process2, below of = 3-faithful] {$\exists K \triangleleft G$, $|K| \ | \ 27$, $|G/K| \ | \ 9$};
		\node (exponent3) [process2, below of = 3-normal] {$G$ has exponent $3$};
		\node (3-27) [process, below of = exponent3] {$b \leq 3$};
		
		%sylows
		\draw (2-sylow) -- node[anchor=south] {\hyperref[abelian-2-grp]{Lm.~\ref*{abelian-2-grp}}} (ab);
		\draw (non-ab) -- (empty-2sylow);
		\draw (2-sylow-rho) |- (2-faithful);
		\draw (2-sylow-rho) |- (2-non-faithful);
		\draw [arrow] (2-sylow) -- (2-sylow-rho); 
		\draw [arrow] (2-non-faithful) -- node[anchor=east] { \hyperref[non-faithful-ker-rho2-cyclic]{Cor.~\ref*{non-faithful-ker-rho2-cyclic}}} (ker-rho2);
		\draw [arrow] (2-non-faithful) -- node[anchor=west] { \hyperref[2-group-bound-quotient]{Lm.~\ref*{2-group-bound-quotient}}} (quotient32);
		\draw [arrow] (2-faithful) -- node[anchor=west] {p. \pageref{page-order-32}} (2-normal);
		\draw (ker-rho2) |- (empty-2sylow4);
		\draw (quotient32) |- (empty-2sylow5);
		\draw [arrow] (empty-2sylow6) -- (2-128);
		\draw [arrow] (2-128) -- node[anchor=south] {closer} node[anchor=north]{study} (2-32);
		
		\draw [arrow] (2-normal) -- (2-32);
		\draw (2-sylow-rho) -- node[anchor=west] { \hyperref[2-sylow-2possibilities]{Lm.~\ref*{2-sylow-2possibilities}}} (empty-2sylow2);
		
		\draw (3-sylow) -- node[anchor=south] { \hyperref[abelian-case-3-sylow]{Lm.~\ref*{abelian-case-3-sylow}}} (ab);
		\draw (non-ab) -- (empty-3sylow);
		\draw [arrow] (3-sylow) -- (3-sylow-rho);
		\draw [arrow] (3-sylow-rho) -- node[anchor=west] { \hyperref[3-group-irr-faithful-rep-of-dim-3]{Prop.~\ref*{3-group-irr-faithful-rep-of-dim-3}}} (3-faithful);
		\draw [arrow] (3-faithful) -- node[anchor=west] {p. \pageref{ex-seq-3-groups}} (3-normal);
		\draw [arrow] (3-normal) -- node[anchor=west] { \hyperref[3-group-of-exp-9-is-cyclic]{Thm.~\ref*{3-group-of-exp-9-is-cyclic}}} (exponent3);
		\draw [arrow] (exponent3) -- node[anchor=west] { \hyperref[3b-b-at-least-4]{Sect.~\ref*{3b-b-at-least-4}}} (3-27);
		
		%\draw [decorate, decoration = {calligraphic brace,mirror}] (0,-15) --  (3,-15);

		\end{tikzpicture}
		\caption{A schematic illustration of the classification of hyperelliptic $2$- or $3$-groups in dimension $4$.}
	\end{figure}	
\end{center}

\begin{enumerate}[ref=(\theenumi)]
	\item[(4)] \underline{$|G| = 2^a \cdot 3^b$:} this is the most involved step of the classification and will be carried out in \hyperref[chapter:2a3b]{Chapter~\ref*{chapter:2a3b}}.  \\
	Very roughly speaking, the full classification in the case $|G| = 2^a \cdot 3^b$ runs through all such groups for $a \leq 5$ and $b \leq 3$, checks whether their Sylow subgroups are hyperelliptic in dimension $4$ (the hyperelliptic $2$- and $3$-groups in dimension $4$ were determined in steps \ref{strategy-2} and \ref{strategy-3} above) and satisfy several necessary conditions, e.g. conditions on the center of $G$. Our output will still contain many groups that are not hyperelliptic in dimension $4$. Consequently, we prove by hand that certain groups in the output are not hyperelliptic in dimension $4$ (e.g., the symmetric group $S_4$). We then rerun our computer program, but this time, we check in addition whether a given group contains $S_4$ as a subgroup. If it does, the group will not be hyperelliptic in dimension $4$ since $S_4$ is not. By iterating this process and excluding sufficiently many groups, we eventually end up with the list of groups of order $2^a \cdot 3^b$ that are all hyperelliptic in dimension $4$. We refer to  \hyperref[section:running-algo]{Section~\ref*{section:running-algo}} for a more complete description of our classification algorithm.
\end{enumerate}

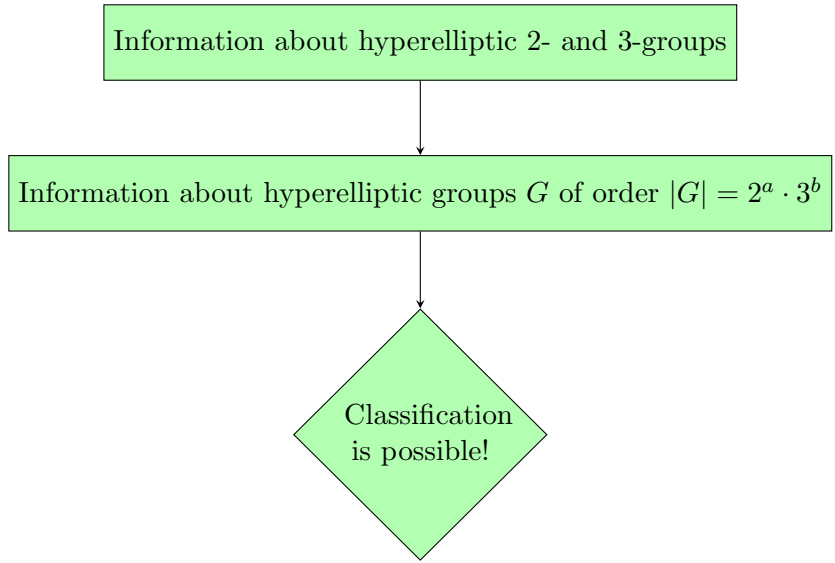
\begin{center}
	\begin{tikzpicture}[node distance=2cm]
	\tikzstyle{arrow} = [thin,->,>=stealth]
	\tikzstyle{process} = [rectangle, minimum width=2cm, minimum height=1cm, text centered, draw=black, fill=green!30]
	\tikzstyle{process2} = [rectangle, minimum width=2cm, minimum height=1cm, text centered, text width=2cm, draw=black, fill=green!30]
	\tikzstyle{contradiction} = [diamond, minimum width=2cm, minimum height=1cm, text centered, text width = 2cm, draw=black, fill=green!30]
	\tikzstyle{ab} = [rectangle, rounded corners, minimum width=3cm, minimum height=1cm,text centered, draw=black, fill=yellow!30]
	%2^a3^b
	\node (information) [process, below of = non-ab, yshift = -10cm] {Information about hyperelliptic $2$- and $3$-groups};
	\node (2a3b) [process, below of = information] {Information about hyperelliptic groups $G$ of order $|G| = 2^a \cdot 3^b$};
	\node (classification) [contradiction, below of = 2a3b, yshift = -1.2cm] {Classification is possible!};
	
	%2^a3^b
	\draw [arrow] (information) -- (2a3b);
	\draw [arrow] (2a3b) -- (classification);
	\end{tikzpicture}
\end{center}

To fully classify hyperelliptic groups in dimension $4$, we still need to investigate the case in which $|G|$ is divisible by $5$ or $7$. We show here that $G$ is necessarily Abelian:
\begin{enumerate}[ref=(\theenumi)]
	\item[(5)] \underline{$|G| = 2^a \cdot 3^b \cdot 5$:} this is investigated in  \hyperref[section:2a3b5]{Section~\ref*{section:2a3b5}}.\\ 
	We first show in  \hyperref[lemma:2a3b5]{Lemma~\ref*{lemma:2a3b5}} that if $|G| = 2^a \cdot 5$ or $|G| = 3^b \cdot 5$ that $G$ is Abelian. Finally, we treat the general case $|G| = 2^a \cdot 3^b \cdot 5$ in \hyperref[prop:2a3b5]{Proposition~\ref*{prop:2a3b5}}, making use of the fact that $G$ is solvable in this case, so that $G$ contains subgroups of order $2^a \cdot 5$ and $3^b \cdot 5$ by \hyperref[thm:hall]{Hall's Theorem~\ref*{thm:hall}}.
	\item[(6)] \underline{$|G| = 2^a \cdot 7$:} we treat this case in in  \hyperref[section:2a7]{Section~\ref*{section:2a7}}. The strategy here is to show that the $2$-Sylow subgroup of $G$ is normal if $G$ is hyperelliptic in dimension $4$ and to use GAP to search all groups satisfying the developed necessary properties. We will then exclude the single non-Abelian group that serves our criteria. 
	\item[(7)] \underline{$|G| = 3^b \cdot 7$:} we show in \hyperref[section:3b7]{Section~\ref*{section:3b7}} show that $|G| = 3^b \cdot 7$ implies $b = 0$. Here, the hard part is to show that the unique non-Abelian group of order $21$ is not hyperelliptic in dimension $4$. The cases $b \in \{2,3\}$ are then easy consequences.
\end{enumerate}

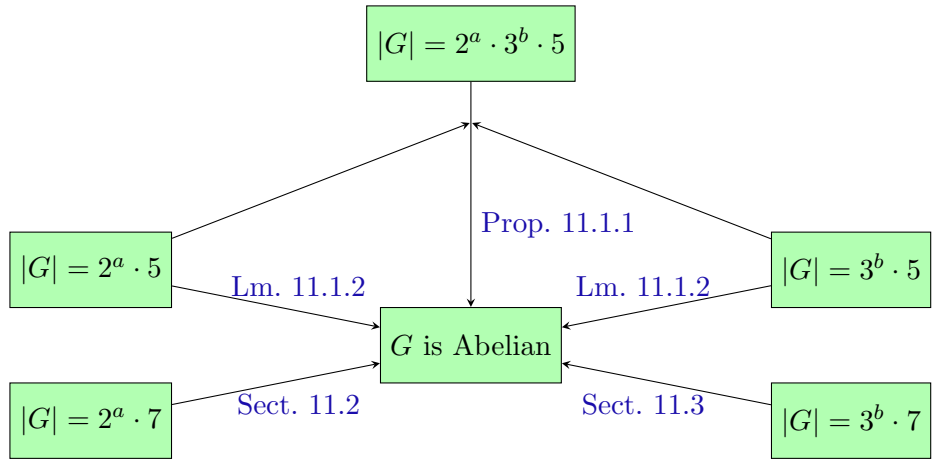
\begin{center}
	\begin{tikzpicture}[node distance=2cm]
	\tikzstyle{arrow} = [thin,->,>=stealth]
	\tikzstyle{process} = [rectangle, minimum width=2cm, minimum height=1cm, text centered, draw=black, fill=green!30]
	\tikzstyle{contradiction} = [diamond, minimum width=3cm, minimum height=1cm, text centered, draw=black, fill=red!30]
	\tikzstyle{ab} = [rectangle, rounded corners, minimum width=3cm, minimum height=1cm,text centered, draw=black, fill=yellow!30]
	\tikzstyle{process2} = [rectangle, minimum width=7.5cm, minimum height=1cm, text centered, text width=7.5cm, draw=black, fill=green!30]
	
	\node (2a3b5) [process]  {$|G| = 2^a\cdot 3^b \cdot 5$};
	\node (G-abelian) [process, below of = 2a3b5, yshift=-2cm]  {$G$ is Abelian};
	\node (2a5) [process, left of = G-abelian, xshift = -3cm, yshift = 1cm]  {$|G| = 2^a \cdot 5$};
	\node (2a7) [process, left of = G-abelian, xshift = -3cm, yshift = -1cm]  {$|G| = 2^a \cdot 7$};
	\node (3b5) [process, right of = G-abelian, xshift = 3cm, yshift = 1cm]  {$|G| = 3^b \cdot 5$};
	\node (3b7) [process, right of = G-abelian, xshift = 3cm, yshift = -1cm]  {$|G| = 3^b \cdot 7$};

	%	
	%	\node (2a3b57) [process, below of = G-abelian, yshift = -14cm]  {$|G| = 2^a \cdot 3^b \cdot 5^c \cdot 7$ ($a,b \geq 1$)};
	%	\node (non-solv-non-ab) [process, above of = 2a3b57]  {$G$ contains a non-solvable group $N$ of order $2^{\tilde a} \cdot 3^{\tilde b} \cdot 7$ ($\tilde a, \tilde b \geq 1)$};
	%	\node (2a3b7) [process2, above of = non-solv-non-ab]  {Suffices to show that non-solvable groups $N$ of order $2^{\tilde a} \cdot 3^{\tilde b} \cdot 7$ ($\tilde a,{\tilde b} \geq 1$) are not hyperell. in dim. $4$};
	%	\node (possibilities-a-b) [process, above of = 2a3b7]  {Reduce to $\tilde a \in \{3,4\}$, $\tilde b \in \{1,2\}$ and $C_2^3 \nleq N$};
	%	%\node (contains-non-ab) [process, above of = possibilities-a-b] {$C_7 \rtimes C_3 \leq N$ or $D_7 \leq K$};
	%	\node (notexists) [process, above of = possibilities-a-b] {There is no such group $N$};
	%	\node (contra) [contradiction, above of = notexists, yshift = 1cm] {Contradiction!};
	%	\node (empty1) [above of = contains-non-ab, yshift = -1cm, xshift = 0.1cm] {};
	%	\node (empty2) [above of = contains-non-ab, yshift = -1cm, xshift = -0.1cm] {};
	%\node (empty-ab1) [below of = 2a7, yshift = -1cm, xshift = -0.12cm] {};
	%\node (empty-ab2) [below of = 3b7, yshift = -1cm, xshift = 0.12cm] {};
	\node (empty-2a3b5-1) [below of = 2a3b5, yshift = 1cm, xshift = 0.12cm] {};
	\node (empty-2a3b5-2) [below of = 2a3b5, yshift = 1cm, xshift = -0.12cm] {};

	\draw [arrow] (2a3b5) -- node[anchor=west, yshift = -0.4cm] { \hyperref[prop:2a3b5]{Prop.~\ref*{prop:2a3b5}}} (G-abelian);
	\draw [arrow] (2a5) -- node[anchor=south, xshift=0.3cm] {\hyperref[lemma:2a3b5]{Lm.~\ref*{lemma:2a3b5}}} (G-abelian);
	\draw [arrow] (2a7) -- node[anchor=north, xshift=0.3cm] {\hyperref[section:2a7]{Sect.~\ref*{section:2a7}}} (G-abelian);
	\draw [arrow] (3b5) -- node[anchor=south, xshift=-0.3cm] {\hyperref[lemma:2a3b5]{Lm.~\ref*{lemma:2a3b5}}} (G-abelian);
	\draw [arrow] (3b7) -- node[anchor=north, xshift=-0.3cm] {\hyperref[section:3b7]{Sect.~\ref*{section:3b7}}} (G-abelian);
	\draw [arrow] (2a5) -- (empty-2a3b5-1);
	\draw [arrow] (3b5) -- (empty-2a3b5-2);
	\end{tikzpicture}
\end{center}	

\begin{enumerate}[ref=(\theenumi)]
	\item[(8)] \underline{$|G| = 2^a \cdot 3^b \cdot 5^c \cdot 7$ with $a, b \geq 1$:} this is the last remaining case; we will show in \hyperref[section:2a3b5c7]{Section~\ref*{section:2a3b5c7}} that there is no hyperelliptic group in dimension $4$ in this case. \\
	The idea is to reduce to the case in which $c = 0$ and the group is non-solvable. The illustration below outlines the proof and gives precise references.
\end{enumerate}

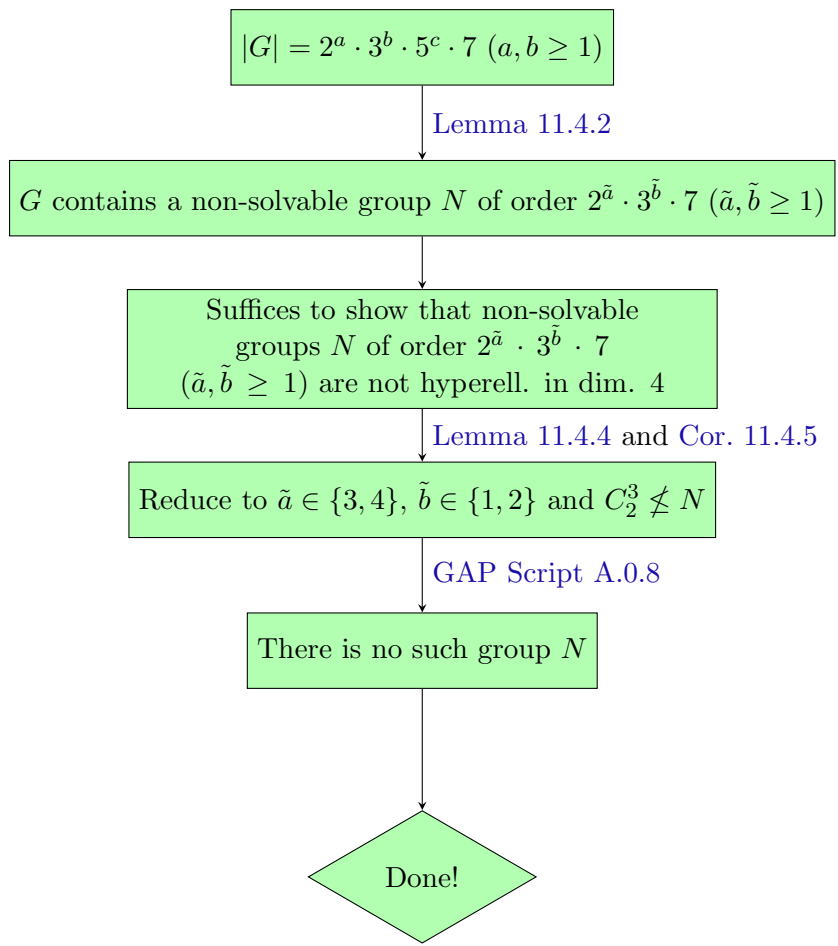
\begin{center}
	\begin{tikzpicture}[node distance=2cm]
	\tikzstyle{arrow} = [thin,->,>=stealth]
	\tikzstyle{process} = [rectangle, minimum width=2cm, minimum height=1cm, text centered, draw=black, fill=green!30]
	\tikzstyle{contradiction} = [diamond, minimum width=3cm, minimum height=1cm, text centered, draw=black, fill=green!30]
	\tikzstyle{ab} = [rectangle, rounded corners, minimum width=3cm, minimum height=1cm,text centered, draw=black, fill=yellow!30]
	\tikzstyle{process2} = [rectangle, minimum width=7.5cm, minimum height=1cm, text centered, text width=7.5cm, draw=black, fill=green!30]

	\node (2a3b57) [process]  {$|G| = 2^a \cdot 3^b \cdot 5^c \cdot 7$ ($a,b \geq 1$)};
	
	\node (non-solv-non-ab) [process, below of = 2a3b57]  {$G$ contains a non-solvable group $N$ of order $2^{\tilde a} \cdot 3^{\tilde b} \cdot 7$ ($\tilde a, \tilde b \geq 1)$};
	
	\node (2a3b7) [process2, below of = non-solv-non-ab]  {Suffices to show that non-solvable groups $N$ of order $2^{\tilde a} \cdot 3^{\tilde b} \cdot 7$ ($\tilde a,{\tilde b} \geq 1$) are not hyperell. in dim. $4$};
	
	\node (possibilities-a-b) [process, below of = 2a3b7]  {Reduce to $\tilde a \in \{3,4\}$, $\tilde b \in \{1,2\}$ and $C_2^3 \nleq N$};
	%\node (contains-non-ab) [process, above of = possibilities-a-b] {$C_7 \rtimes C_3 \leq N$ or $D_7 \leq K$};
	
	\node (notexists) [process, below of = possibilities-a-b] {There is no such group $N$};
	
	\node (contra) [contradiction, below of = notexists, yshift=-1cm] {Done!};

	\draw [arrow] (2a3b57) -- node[anchor=west] { \hyperref[lemma:non-solv]{Lemma~\ref*{lemma:non-solv}}} (non-solv-non-ab);
	\draw [arrow] (non-solv-non-ab) -- node[anchor=west] {} (2a3b7);
	\draw [arrow] (2a3b7) -- node[anchor=west] { \hyperref[lemma:non-solv-N]{Lemma~\ref*{lemma:non-solv-N}} and \hyperref[cor:168dividesN]{Cor.~\ref*{cor:168dividesN}}} (possibilities-a-b);
	\draw [arrow] (possibilities-a-b) -- node[anchor=west] { \hyperref[gap-non-solv]{GAP Script~\ref*{gap-non-solv}}} (notexists);
	\draw[arrow] (notexists) -- (contra);
	%\draw [arrow] (contains-non-ab) -- (contra);
	%\draw [arrow] (2a7) |- (empty1);
	%\draw [arrow] (3b7) |- (empty2);
	%\draw [arrow] (G-abelian) -- (empty-ab1);
	%\draw [arrow] (G-abelian) -- (empty-ab2);
	\end{tikzpicture}
\end{center}

\chapter{Abelian Hyperelliptic Groups} \label{chapter:abelian}

In this chapter, we study Abelian groups, which are hyperelliptic in a fixed dimension $n$. We first investigate in \hyperref[section:structuretheory]{Section~\ref*{section:structuretheory}} in which ways Abelian groups can embed into $\GL(n,\CC)$ such that every matrix in the image has the eigenvalue $1$ (a property, which the complex representation of a hyperelliptic manifold necessarily satisfies, see \hyperref[rem:ev1]{Remark~\ref*{rem:ev1}}). Our results will then be applied to dimension $4$ in \hyperref[cor:c6^3-c4^3-excluded]{Corollary~\ref*{cor:c6^3-c4^3-excluded}}. In the upcoming \hyperref[section:constructionmethod]{Section~\ref*{section:constructionmethod}}, we explain a general method to construct hyperelliptic manifolds with Abelian holonomy group. The classification of Abelian hyperelliptic groups in dimension $4$ will then be carried out in \hyperref[section:ab-class]{Section~\ref*{section:ab-class}}, which proves part of \hyperref[mainthm]{Main Theorem~\ref*{mainthm}}.

\section{Structure Theory} \label{section:structuretheory}

Let $G$ be a finite Abelian group, written in its elementary divisor normal form
\begin{align} \label{eq:AbelianG}
G \cong C_{d_1} \times ... \times C_{d_k}, \qquad \text{ where } d_i ~ | ~ d_{i+1}, \quad d_1 \geq 2.
\end{align}
We would like to derive (necessary or sufficient) criteria on $k$ and the $d_i$ for $G$ is hyperelliptic in a given dimension $n \geq 2$. A simple starting point is

\begin{lemma} \label{lemma:many-factors}
	Let $G$ be as in (\ref{eq:AbelianG}) and $\rho \colon G \to \GL(n,\CC)$ ($n \geq 2$) a faithful representation. Then:
	\begin{enumerate}[ref=(\theenumi)]
		\item \label{lemma:many-factors1} It holds $k \leq n$.
		\item \label{lemma:many-factors2} If $k = n$, then $\im(\rho)$ contains a matrix that does not have the eigenvalue $1$.
		\item \label{lemma:many-factors3} Suppose that $k = n-1$ and $d_1 = ... = d_{n-1} =: d$. If every matrix in $\im(\rho)$ has the eigenvalue $1$, then the image of $\rho$ is not contained in $\SL(n,\CC)$, unless $d = 2$ and $n$ is odd.
	\end{enumerate}
\end{lemma} 

\begin{proof}
	We may assume that $\rho$ embeds $G$ as a group of diagonal matrices.\\
	
	\ref{lemma:many-factors1}, \ref{lemma:many-factors2}  It suffices to prove the statements for $G = C_d^k$, $d = d_1 \geq 2$. The group $G$ is mapped into the group of diagonal matrices with entries in $\mu_{d}$ by $\rho$. Hence $C_{d}^k$ is -- up to isomorphism -- a subgroup of $\mu_{d}^{n}$ and thus $k \leq n$ by cardinality reasons. Furthermore, $\rho$ defines an isomorphism from $G$ to the group of diagonal matrices with entries in $\mu_{d}$ if $k = n$. \\
	
	\ref{lemma:many-factors3} Let $U$ be a subgroup of $G$ that is isomorphic to $C_{q}^{n-1}$, where $q$ is a prime power dividing $d$. Suppose that $\rho(G) \subset \SL(n,\CC)$. Then $\rho(U)$ is conjugate to the group of diagonal matrices taking the following form:
	\begin{align*}
	\diag(\zeta_{q}^{a_1}, ~ ..., ~ \zeta_{q}^{a_{n-1}}, ~ \zeta_{q}^{-a_1-...-a_{n-1}}).
	\end{align*}
	Every matrix in this group has the eigenvalue $1$ if and only if for every $i \in \{1,...,n-1\}$ and for every $a_i \in \{1,...,q-1\}$, the number $a_1 + ... + a_{n-1}$ is divisible by $q$. Taking $a_i = 1$ for every $i$, we conclude that $n-1$ is divisible by $q$. Now, if $q \geq 3$, then also 
	\begin{align*}
	2  + \underbrace{1+...+1}_{(n-2) \text{ times}} = n
	\end{align*}
	is divisible by $q$, a contradiction. Hence $q  = 2$ and $n$ is odd.
\end{proof}

\begin{notation}
	In the following, we let $G$ be as in (\ref{eq:AbelianG}). Furthermore, we let $\rho \colon G \to \GL(n,\CC)$ be an embedding of $G$ as a group of diagonal matrices. For simplicity, we will often write $g$ instead of $\rho(g)$. 
\end{notation}

We would like to describe interesting necessary criteria for $\rho$ to be the complex representation of some hyperelliptic $n$-fold. Of course, one cannot expect interesting structural results if the number of factors $k$ is small. Instead, we will mostly concentrate on the cases $k \in \{n-1, n-2\}$. The upcoming lemma serves as a purely combinatorial starting point.

\begin{lemma} \label{common-ev-1}
	Suppose that $k = n-1$ and that $d_1 = ... = d_{n-1} =: d \geq 3$. If every $g \in G$ has the eigenvalue $1$, then all $g \in G$ share an eigenvector for the eigenvalue $1$, i.e., $\bigcap_{g \in G} \ker(g - \id) \neq \{0\}$.
\end{lemma}

\begin{proof}
	Let $\{g_1, ..., g_{n-1}\}$ be a generating set of $G$. We first prove the assertion in the case where $d$ is a prime power and then in the general case. \\
	
	\underline{Step 1:} Assume first that $d$ is a prime power. Then there is a basis in which $g_1, ..., g_{n-1}$ are of the following form:
	\begin{align*}
	g_j = \diag(1,...,1,\underbrace{\zeta_{d},}_{j\text{-th entry}} 1,...,1,\zeta_{d}^{a_j}) \qquad \text{ for some } \quad a_j \in \{0,...,d-1\}.
	\end{align*}
	In the following, we assume that $(a_1,...,a_{n-1}) \neq 0$ and construct an element of $G$ without the eigenvalue $1$. Suppose without loss of generality that $a_1 \neq 0$. Then, if $a_1 + ... + a_{n-1} \not \equiv 0 \pmod d$, the element
	\begin{align*}
	\prod_{j=1}^{n-1} g_j = \diag\left(\zeta_{d}, ~ ..., ~ \zeta_{d}, \zeta_{d}^{a_1+ ...+ a_{n-1}}\right)
	\end{align*}
	does not have the eigenvalue $1$. Otherwise, if $d$ divides $a_1 +...+a_{n-1}$, then $2a_1 + a_2 ... + a_{n-1}$ is non-zero modulo $d$ and the element
	\begin{align} \label{eq:lemma:common-ev-1}
	g_1^2 \cdot \prod_{j=2}^{n-1}g_j = \diag\left(\zeta_{d}^2, ~ \zeta_d, ~ ..., ~ \zeta_{d}, \zeta_{d}^{2a_1+a_2+ ...+ a_{n-1}}\right)
	\end{align}
	does not have the eigenvalue $1$. Here, the hypothesis $d \geq 3$ is used to ensure that the first diagonal entry $\zeta_d^2$ of (\ref{eq:lemma:common-ev-1}) is different from $1$. \\
	
	\underline{Step 2:} We now treat the general case. Write $d = q \cdot m$, where $q \geq 3$ is a prime power coprime to the integer $m \geq 2$. According to Step 1, we may assume that the $g_j^m$ are given by
	\begin{align*}
	g_j^m = \diag(1,...,1,\underbrace{\zeta_{q},}_{j\text{th entry}} 1,...,1,1).
	\end{align*}
	Since the $g_j^m$ and the $g_i^{q}$ generate $G$, we have to prove that the last diagonal entry of every $g_i^{q}$ is $1$. If this is not the case for some $i$, then the element
	\begin{align*}
	g_i^q \cdot \prod_{j=1}^{n-1} g_j^m
	\end{align*}
	does not have the eigenvalue $1$.
\end{proof}

\begin{rem}
	The conclusion of \hyperref[common-ev-1]{Lemma~\ref*{common-ev-1}} does not hold for $d = 2$: for every $n \geq 3$, the linear subgroup of $\GL(n,\CC)$ spanned by 
	\begin{align*}
	&h_1 = \diag(-1, \ 1, \ 1, \ 1, \ ..., \ 1), \\
	&h_2 = \diag(1, \ -1, \ -1, \ 1, \ ..., \ 1), \\
	&g_j = \diag(1, \ ..., \ 1, \ \underbrace{-1}_{j\text{th entry} }, \ 1, \ ..., \ 1) \text{ for } j \in \{3,...,n\}
	\end{align*} 
	has the property that every element in it has the eigenvalue $1$. However, the elements $h_1, h_2, g_3, ..., g_n$ do not share a common eigenvector for the eigenvalue $1$.
\end{rem}

\hyperref[common-ev-1]{Lemma~\ref*{common-ev-1}} only gives information about the boundary case $k = n-1$.  \hyperref[common-ev-1-better]{Proposition~\ref*{common-ev-1-better}} below shows that the conclusion of \hyperref[common-ev-1]{Lemma~\ref*{common-ev-1}} holds true even if $k = n-2$ under the stronger assumption $d \geq 4$. We start with a preliminary lemma:

\begin{lemma} \label{lemma-eqs-z/d}
	Let $r \geq 1$ and $d = p^{k}\geq 4$ a prime power. Let $(a_1, ..., a_r)$, $(b_1, ..., b_r)$ be non-zero tuples in $(\ZZ/d\ZZ)^r$. Then there are elements $\lambda_1, ..., \lambda_r \in (\ZZ/d\ZZ) \setminus \{0\}$ such that
	\begin{align*}
	\lambda_1 a_1 + ... + \lambda_r a_{r} \neq 0 \quad \text{ and } \quad
	\lambda_1 b_1 + ... + \lambda_r b_{r} \neq 0 \qquad \text{ in } \ZZ/d\ZZ.
	\end{align*} 
\end{lemma}

\begin{proof}
	First, observe that if there is no index $i$ such that both $a_i$ and $b_i$ are non-zero, then it is clearly possible to choose $\lambda_1, ..., \lambda_r$ accordingly. We may hence assume that $a_1, b_1 \neq 0$. Choose $\lambda_1, ..., \lambda_r$ according to the following constructive procedure:
	\begin{enumerate}[ref=(\theenumi)]
		\item \label{lemma-eqs-z/d-1} If $a_1$ and $b_1$ are both units, we may choose $\lambda_2, ..., \lambda_r$ arbitrarily and different from zero: then there are at most two values $x \in \ZZ/d\ZZ$ such that
		\begin{align*}
		&a_1 x + \lambda_2 a_2 +... + \lambda_r a_r = 0, \text{ or} \\
		&b_1 x + \lambda_2 b_2 +... + \lambda_r b_r = 0.
		\end{align*}
		Since $d \geq 4$, we can certainly find $\lambda_1 \neq 0$ such that none of the two equations above are satisfied.
		\item \label{lemma-eqs-z/d-2} If $a_1$ is a unit and $b_1$ is a zero-divisor, we first claim that we may choose $\lambda_2, ..., \lambda_r \neq 0$ such that
		\begin{align*}
		\lambda_2 b_2 + ... + \lambda_r b_r \neq \pm b_1.
		\end{align*}
		Indeed, if every $b_2, ..., b_r$ are zero or zero-divisors, it suffices to take $\lambda_2 = ... = \lambda_r = p^{k-1}$ so that $\lambda_2 b_2 + ... + \lambda_r b_r = 0$. If (say) $b_2$ is a unit, then we may choose $\lambda_3, ..., \lambda_r \neq 0$ arbitrarily and define $\lambda_2$ accordingly. This settles the claim.\\
		Now, let $\lambda_2, ..., \lambda_r \neq 0$ be as in the claim. Then, as in \ref{lemma-eqs-z/d-1}, the equation $a_1x + \lambda_2 a_2 + ... + \lambda_r a_r = 0$ has a unique solution in $x$. Therefore at least one of the (different!) numbers $1$ or $-1$ does not solve this equation, and it suffices to define $\lambda_1$ as this non-solution.
		\item \label{lemma-eqs-z/d-3} If $a_1, b_1$ are both zero-divisors, we choose $\lambda_1, ..., \lambda_r$ in the following way:
		\begin{itemize}
			\item If $r = 1$ or if $a_2 = ... = a_r = b_2 = ... = b_r = 0$, then we may define $\lambda_1 = ... = \lambda_r = 1$.  
			\item If $r \geq 2$ and (say) $a_2 = ... = a_r = 0$, but $b_2 \neq 0$, we may choose $\lambda_3, ..., \lambda_r \neq 0$ such that $\lambda_3 b_3 + ... + \lambda_r b_r = 0$ (this is clearly possible). Furthermore, define
			\begin{align*}
			\lambda_1 = 1 \qquad \text{ and } \qquad \lambda_2 = \begin{cases}
			1, & \text{ if } b_2 \text{ is a unit} \neq 0 \\
			p^{k-1}, & \text{ if } b_1 \text{ is a zero divisor} \neq 0.
			\end{cases}
			\end{align*}
			Then, since $b_1 \neq 0$ is a zero divisor, our construction implies that
			\begin{align*}
			\lambda_1b_1 + ... + \lambda_r b_r = b_1 + \lambda_2 b_2 \neq 0.
			\end{align*}
			\item Finally, if $r \geq 2$ and not all $a_2, ..., a_r$ as well as not all $b_2, ..., b_r$ are zero, set $\lambda_1 = p^{k-1}$ and continue inductively according to the procedure described in \ref{lemma-eqs-z/d-1}, \ref{lemma-eqs-z/d-2} and \ref{lemma-eqs-z/d-3}.
		\end{itemize}
	\end{enumerate}
	%{\ I think the proof is now correct. Add picture of the proof}
\end{proof}

Finally, we prove the announced improvement of \hyperref[common-ev-1]{Lemma~\ref*{common-ev-1}} to $n-2$ factors.

\begin{prop} \label{common-ev-1-better}
	Suppose that $k = n-2$ and that $d_1 = ... = d_{n-2} =: d \geq 4$. If every $g \in G$ has the eigenvalue $1$, then all $g \in G$ have a common eigenvector for $1$, i.e., $\bigcap_{g \in G} \ker(g - \id) \neq \{0\}$.
\end{prop}

\begin{proof} Write $d = q \cdot m$, where $q \geq 3$ is a prime power and $\gcd(m,q) = 1$. Let $\{g_1, ..., g_{n-2}\}$ be a generating set of $G$. The powers $g_j^m$ have order $q$, and thus, there exists a basis in which they can be written as follows:
	\begin{align*}
	g_j^m = \diag(1, \ ..., \ 1, \underbrace{\zeta_{q}}_{j\text{th entry}}, 1, ..., 1, \zeta_{q}^{a_j}, \ \zeta_{q}^{b_j}), \qquad \text{where } \quad a_j,b_j \in \{0,...,q-1\}.
	\end{align*}
	We will make frequent use of the easy observation that the first $n-2$ diagonal entries of the products
	\begin{align*}
	\prod_{j=1}^{n-2} (g_j^m)^{k_j}, \qquad \text{ where } \quad k_j \in \{1,...,q-1\}
	\end{align*}
	are different from $1$. \\
	
	Observe that \hyperref[lemma-eqs-z/d]{Lemma~\ref*{lemma-eqs-z/d}} is only applicable if $d$ is divisible by a prime power $\geq 4$, which is only the case if $d \neq 6$. We henceforth show the statement separately in the two cases
	\begin{enumerate}[ref=(\theenumi)]
		\item \label{common-ev-better-1} $d \neq 6$ -- here, we will assume that $q \geq 4$,
		\item \label{common-ev-better-2} $d = 6$, i.e., $q = 3$ and $m = 2$.
	\end{enumerate}
	
	\ref{common-ev-better-1} As already mentioned above, \hyperref[lemma-eqs-z/d]{Lemma~\ref*{lemma-eqs-z/d}} asserts in this case that if not all $a_i, b_j$ are zero, then there are $\lambda_1, ..., \lambda_{n-1} \in \{1,...,q-1\}$ such that $\prod_{j=1}^{n-2} (g_j^m)^{\lambda_j} $
	does not have the eigenvalue $1$. We may therefore assume that $b_1 = ... = b_{n-2} = 0$.\\ 
	The statement is then vacuous if the $n$th diagonal entry of every $g_i^q$ is $1$, so let us assume that the $n$th diagonal of $g_{1}^q$ is different from $1$. Given the hypothesis that every element of $g$ has the eigenvalue $1$, we conclude that the $(n-1)$st diagonal entry of all elements of the form
	\begin{align}
	g_{1}^q \cdot \prod_{j=1}^{n-2} (g_j^m)^{k_j}, \qquad \text{ where } \quad k_j \in \{1,...,q-1\}
	\end{align}
	is $1$. Since $\gcd(m,q) = 1$, this is equivalent to the following two conditions:
	\begin{enumerate}
		\item[(i)] the $(n-1)$st diagonal entry of $g_1^q$ is $1$, and
		\item[(ii)] the kernel of the homomorphism
		\begin{align*}
		\langle g_1^m, ..., g_{n-2}^m \rangle \to C_q, \quad
		g_j^m \mapsto a_j
		\end{align*}
		contains the elements of the form $\prod_{j=1}^{n-2} (g_j^m)^{k_j}$, $k_j \in \{1,...,q-1\}$, of which there are $(q-1)^{n-2}$ many.
	\end{enumerate}
	In particular, (ii) implies that $a_j = 0$ for all $j$. Finally, the elements
	\begin{align*}
	g_1^q \cdot g_i^q \cdot \prod_{j=1}^{n-2} g_j^m \qquad \text{ and } \qquad g_i^q \cdot \prod_{j=1}^{n-2} g_j^m
	\end{align*}
	have the eigenvalue $1$ for $i \geq 2$, so that the $(n-1)$st diagonal entry of each $g_i^q$ must be $1$. \\
	
	\ref{common-ev-better-2} The proof of case \ref{common-ev-better-1} can be copied verbatim if $a_1 = ... = a_{n-2} = 0$ or $b_1 = ... = b_{n-2} = 0$. It is, therefore, enough to derive a contradiction if there are indices $j_1, j_2$ such that $a_{j_1} \neq 0$ and $b_{j_2} \neq 0$. \\
	
	We claim that there are $h_0, h_1 \langle g_1^2, ..., g_{n-2}^2\rangle$ such that
	\begin{itemize}
		\item all but the $n$th diagonal entries of $h_0$ are different from $1$,
		\item all but the $(n-1)$st diagonal entries of $h_1$ are different from $1$.
	\end{itemize}
	
	This is shown as follows. By hypothesis, one of the last two diagonal entries of $\prod_{j=1}^{n-2} g_j^2$ is $1$, say the $n$th one. Then
	\begin{align*}
	h_0 := \begin{cases}
	\prod_{j=1}^{n-2} g_j^2, & \text{ if the } (n-1)\text{st diagonal entry of } g \text{ is } \neq 1, \\
	g_{j_1}^2 \cdot \prod_{j=1}^{n-2} g_j^2, & \text{ if the } (n-1)\text{st diagonal entry of } g \text{ is } 1
	\end{cases} \qquad \text{and}\qquad h_1 := g_{j_2}^{2} \cdot \prod_{j=1}^{n-2} g_j^2
	\end{align*}
	have the desired properties. \\
	
	Now, since the matrices $h_0 \cdot g_i^3$ and $h_1 \cdot g_i^3$ have the eigenvalue $1$, we conclude that the $(n-1)$st and $n$th diagonal entries of all $g_i^3$ are $1$. The group $C_2^{n-2} \cong \langle g_1^3, ..., g_{n-2}^{n-2}\rangle$ hence embeds into $\GL(n-2,\CC)$ by forgetting the last two diagonal entries. In particular, our group $G$ contains the element
	\begin{align*}
	h := \diag(-1, \ ..., \ -1, \ 1, \ 1).
	\end{align*}
	Moreover, there exists $g \in \langle g_{j_1}^2, g_{j_2}^2\rangle$ with the property that both of its last two diagonal entries are different from $1$. Then $g\cdot h$ does not have the eigenvalue $1$, a contradiction. 
\end{proof}

\begin{rem}
	The statement of  \hyperref[common-ev-1-better]{Proposition~\ref*{common-ev-1-better}} is wrong for $d = 3$ and $n \geq 4$. In fact, consider the $(n \times n)$-matrices
	\begin{align*}
	&h_1 = \diag(\zeta_3, 1,..., 1, \zeta_3^2, \zeta_3), \\
	&h_2 = \diag(1, \zeta_3, 1,..., 1, \zeta_3, \zeta_3), \\
	&g_j = \diag(1,...,1,\underbrace{\zeta_3,	}_{j\text{th entry}} 1,...,1), \; \text{ for } 3 \leq j \leq n-2.
	\end{align*}
	Then $h_1, h_2, g_3, ..., g_{n-2}$ span a subgroup $G$ of $\GL(n,\CC)$ that is isomorphic to $C_3^{n-2}$ and has the following properties:
	\begin{enumerate}
		\item[(i)] every $g \in G$ has the eigenvalue $1$, 
		\item[(ii)] $h_1, h_2, g_3, ..., g_{n-2}$ do not share a common eigenvector for the eigenvalue $1$.
	\end{enumerate}
\end{rem}
\vspace{0.5cm}

Interpreting \hyperref[common-ev-1-better]{Proposition~\ref*{common-ev-1-better}} in the language of hyperelliptic manifolds, we obtain the following consequence.

\begin{cor} \label{abelian-cor-cd^n-2}
	Let $d \geq 2$ and suppose that $C_d^{n-2}$ is hyperelliptic in dimension $n \geq 4$. Then $d \in \{2,3,4,6\}$.
\end{cor}

\begin{proof}
	Since $C_d^{n-2}$ is hyperelliptic in dimension $n$, there is an $n$-dimensional complex torus $T$ and an embedding $\iota \colon C_d^{n-2} \to \Bihol(T)$ such that $\iota(C_d^{n-2})$ does not contain any translations and acts freely on $T$. Let $\rho \colon C_d^{n-2} \hookrightarrow \GL(n,\CC)$ be the associated complex representation -- recall that the freeness of the action on $T$ implies that every matrix in the image of $\rho$ has the eigenvalue $1$. Assume that $d  \notin \{2,3,4,6\}$. It then suffices to exclude the following two cases:
	\begin{enumerate}[ref=(\theenumi)]
		\item \label{hyper-ab-case1} $d$ is divisible by a prime power $q \geq 5$, or
		\item \label{hyper-ab-case2} $d = 12$.
	\end{enumerate}
	Case \ref{hyper-ab-case1} is easily excluded, since we find a matrix in the image of $\rho$ that has the eigenvalue $\zeta_q$ with multiplicity at least $n-2 \geq 2$, which contradicts the \hyperref[order-cyclic-groups]{Integrality Lemma~\ref*{order-cyclic-groups}} \ref{ocg-3}. \\
	We henceforth focus on the subtler Case \ref{hyper-ab-case2}. Here, \hyperref[common-ev-1-better]{Proposition~\ref*{common-ev-1-better}} allows us to assume that the generators $\{g_1^3, ..., g_{n-2}^3\}$ of $C_4^{n-2} \subset C_{12}^{n-2}$ are embedded by
	\begin{align*}
	\rho(g_j^3) =  \diag(1,\ ..., 1,\ \underbrace{\zeta_4}_{j\text{th entry}}, \ 1, \ ..., \ 1, \ \zeta_4^{a_j}, 1) \qquad \text{for some } \quad a_j \in \{0,...,3\}.
	\end{align*}
	Consider now an element $g \in C_3^{n-2} \subset C_{12}^{n-2}$. Then $\rho(g \cdot \prod_{j=1}^{n-2} g_j^3)$ only has the eigenvalue $1$, if one of the last two diagonal entries of $\rho(g)$ is $1$. \\
	It hence remains to exclude the case in which every $g \in C_3^{n-2}$ has the property that at least one of the last two diagonal entries of $\rho(g)$ is $1$. Here, since $n - 2 \geq 2$, there is an element $h \in C_3^{n-2}$ such that:
	\begin{enumerate}
		\item[(i)] the $(n-1)$st diagonal entry of $\rho(h)$ is $1$, and
		\item[(ii)] there exists $j \in \{1,..., n-2\}$ such that the $j$th diagonal entry of $\rho(h)$ is a primitive third root of unity.
	\end{enumerate}
	However then the matrix $\rho(g_j^3h)$ has only one eigenvalue of order $12$, contradicting the \hyperref[order-cyclic-groups]{Integrality Lemma~\ref*{order-cyclic-groups}} \ref{ocg-3}.
\end{proof}

%Note that we are not claiming that the groups $C_d^{n-2}$ are all hyperelliptic in dimension $n$ for all $d \in \{2,3,4,6\}$. In fact:
%
%\begin{thm}
%Let $d \in \{2,3,4,6\}$ and $r \geq 1$. Then the minimal dimension in which $C_d^r$ is hyperelliptic is $r+1$ if $d \in \{2,3\}$ and $2r$ if $d \in \{4,6\}$.
%\end{thm}
%
%\begin{proof}
%{\ maybe only give an example?}
%\end{proof}

\hyperref[common-ev-1-better]{Proposition~\ref*{common-ev-1-better}} is the most important step in the proof of the following Theorem, which generalizes \cite[Lemma 6.5]{Lange} to arbitrary dimension (Lange proved it in dimension $n = 3$).
\medskip

\begin{theorem} \label{abelian-two-cases} 
	Suppose that $k = n-1$ and that that every $g \in G$ has the eigenvalue $1$. Then at least one of the two following possibilities holds:
	\begin{enumerate}[label=(\roman*), ref=(\roman*)]
		\item \label{abelian-first-poss} all $g \in G$ share a common eigenvector for the eigenvalue $1$, i.e., $$\bigcap_{g \in G} \ker(g-\id) \neq \{0\}.$$
		\item up to isomorphism, $G$ can be realized as a subgroup of $C_2^2 \times C_d^{n-3}$ for some even $d \geq 4$.
	\end{enumerate}
\end{theorem}

\begin{proof}
	According to \hyperref[common-ev-1]{Lemma~\ref*{common-ev-1}}, the groups $C_m^{n-1}$ satisfy property \ref{abelian-first-poss} for any $m \geq 3$. It therefore remains to prove that the groups $C_2 \times C_d^{n-2}$ for an even integer $d \geq 4$ satisfy property \ref{abelian-first-poss} as well. \\
	Denote by $\{g_1, ..., g_{n-2}, h\}$ a generating set of $C_2 \times C_d^{n-2}$, where $\ord(g_j) = d$ and $\ord(h) = 2$. From \hyperref[common-ev-1-better]{Proposition~\ref*{common-ev-1-better}}, we may assume that  
	\begin{align} \label{ab-2-cases-eq}
	{}^t (0, ..., 0,1) \in	\bigcap_{j=1}^{n-2} \ker(g_j - \id) \neq \{0\}.
	\end{align}
	Moreover, we write $d = q \cdot m$, where $q \geq 3$ is a prime power and $\gcd(m,q) = 1$. Then $G$ is generated by $\{g_1^m,g_1^q, ..., g_{n-2}^m, g_{n-2}^q, h\}$, and a reformulation of (\ref{ab-2-cases-eq}) is that the last diagonal entries of all $g_j^q$ and $g_j^m$ is $1$. Choosing an appropriate basis, we may thus write
	\begin{align*}
	g_j^m = \diag(1,\ ..., 1,\ \underbrace{\zeta_q}_{j\text{th entry}}, \ 1, \ ..., \ 1, \ \zeta_q^{a_j}, 1), \qquad \text{ where } \quad a_j \in \{0,...,q-1\},
	\end{align*}
	and
	\begin{align*}
	h = \diag((-1)^{b_1}, ~ ..., ~ (-1)^{b_n}), \qquad \text{ where } \quad b_1, ..., b_n \in \{0,1\}.
	\end{align*}
	There is nothing to show if $b_n = 0$, so let us assume that $b_n = 1$. Here, we first consider the elements
	\begin{align*}
	h \cdot \prod_{j=1}^{n-2} g_j^{mk_j} = \diag((-1)^{b_1}\zeta_q, ~ ..., ~ (-1)^{b_{n-2}}\zeta_q, ~ (-1)^{b_{n-1}}\zeta_q^{k_1a_1 +... + k_{n-2}a_{n-2}}, ~ -1) 
	\end{align*}
	for $k_1, ..., k_{n-2} \in \{1,...,q-1\}$. By hypothesis, it has the eigenvalue $1$, which implies that 
	\begin{align*}
	\forall k_1, ..., k_{n-2} \in \{1, ..., d-1\}\colon \qquad (-1)^{b_{n-1}}\zeta_q^{k_1a_1 +... + k_{n-2}a_{n-2}} = 1.
	\end{align*}
	This can only be satisfied if $a_1 = ... = a_{n-2} = b_{n-1} = 0$.\\
	We are left with proving that the $(n-1)$st diagonal entry of each $g_{j_0}^{q}$ is $1$. This follows immediately since the element
	\begin{align*}
	hg_{j_0}^q \cdot \prod_{j=1}^{n-2} g_j^m
	\end{align*}
	must have the eigenvalue $1$.
\end{proof}

It seems interesting to investigate the optimality of  \hyperref[abelian-two-cases]{Theorem~\ref*{abelian-two-cases}} when applied to hyperelliptic mani\-folds. We will, however, not pursue this problem further here. \\

Next, we interpret \hyperref[abelian-two-cases]{Theorem~\ref*{abelian-two-cases}} and \hyperref[common-ev-1]{Lemma~\ref*{common-ev-1}} in the setting of hyperelliptic manifolds:

\begin{prop} \label{prop:excluding-common-ev-1}
	Let $n \geq 4$ and $G = C_{d'}^{n-4} \times C_d^3$, where $(d',d) \in \{(2,4), (2,6), (3,6)\}$. Suppose that $\rho \colon G \to \GL(n,\CC)$ is a faithful representation such that $\bigcap_{g \in G} \ker(\rho(g) - \id) \neq \{0\}$. Then there is no hyperelliptic manifold with holonomy group $G$, whose associated complex representation is $\rho$. 
\end{prop}

An immediate consequence is:

\begin{cor} \label{cor:c6^3-c4^3-excluded}
	\
	\begin{enumerate}[ref=(\theenumi)]
		\item The group $C_3^{n-4} \times C_6^3$ is not hyperelliptic in dimension $n \geq 4$. In particular, $C_6^3$ is not hyperelliptic in dimension $4$.
		\item The group $C_4^3$ is not hyperelliptic in dimension $4$.
	\end{enumerate}
\end{cor}

\begin{proof}[Proof of Proposition \ref{prop:excluding-common-ev-1}.]
	Denote by $U = \langle g_1, ..., g_{n-1}\rangle$ a subgroup of $G = C_{d'}^{n-4} \times C_d^3$, which is isomorphic to $C_{d'}^{n-1}$. 
	
	By hypothesis, $\bigcap_{j=1}^n \ker(\rho(g_j) - \id) \neq \{0\}$, and hence there is a basis such that
	\begin{align*}
	\rho(g_j) = \diag(1, ~ ..., ~ 1, ~ \underbrace{\zeta_{d'}}_{j\text{th entry}}, ~ 1, ~ ..., ~ 1).
	\end{align*}
	Consider now a complex torus $T$ of dimension $n$, such that $G \subset \Bihol(T)$ with associated complex representation $\rho$. The complex representation $\rho$ is the direct sum of $n$ distinct linear characters, and thus $T$ is equivariantly isogenous to a product $E_1 \times ... \times E_n$ of elliptic curves, where $E_j \subset T$ (see \hyperref[isogeny]{Section~\ref*{isogeny}}) -- we write
	\begin{align*}
	T = (E_1 \times ... \times E_n)/H,
	\end{align*}
	where $H \subset E_1 \times ... \times E_n$ is a finite subgroup of translations. The action of the generators $g_j$ of $U$ on $T$ then takes the following form:
	\begin{align*}
	g_{j}(z) =\left(z_1+a_{1j},\ ...,\ z_{j-1} + a_{j-1,j},\ \zeta_{d'}z_j + a_{jj},\ z_{j+1} +a_{j+1,j},\ ...,\ z_n + a_{nj}\right).
	\end{align*}
	Denote by $h_1, h_2, h_3$ elements of order $d$ such that $\{g_1, ..., g_{n-1}, h_1, h_2, h_3\}$ is a generating set of $G$, i.e., $\langle h_1, h_2, h_3 \rangle \cong C_d^3$.\\ 
	We now make use of the fact that $E_n[d] \cong C_d \times C_d$, i.e., there exists $h \in \langle h_1,h_2,h_3\rangle$ of order $d$ that acts trivially on $E_n$\footnote{The kernel of every homomorphism $C_d^3 \to E_n[d]$ contains an element of order $d$.}. We claim that $h^{d'}$ has a fixed point on $T$. Let $J \subset \{1,...,n-1\}$ be the set of indices $j$ such that the $j$th diagonal entry of $\rho(h)^{d/d'}$ is $1$ (or, equivalently, \emph{not} a primitive $d'$th root of unity). Clearly,
	\begin{align*}
	h^{d'} = \widetilde{h}^{d'}, \qquad \text{where} \quad \widetilde{h} :=  h \cdot \prod_{j \in J} g_j.
	\end{align*}
	By construction, the $j$th diagonal entry of $\rho(\widetilde{h})$ is
	\begin{enumerate}[label=(\roman*)]
		\item a primitive $d'$th root of unity, if $j \in J$,
		\item a primitive $d$th root of unity, if $j \in \{1, ..., n-1\}$, $j \notin J$, and
		\item equal to $1$, if $j = n$.
	\end{enumerate}
	It follows that $h^{d'} = \widetilde{h}^{d'}$ acts on $E_j$ in the following way:
	\begin{enumerate}[label=(\roman*)]
		\item trivially, if $j \in J$,
		\item by a non-translation, if $j \in \{1, ..., n-1\}$, $j \notin J$, and
		\item trivially, if $j = n$.
	\end{enumerate}
	This shows that $h^{d'}$ has a fixed point on $T$.
\end{proof}

\section{A Construction Method} \label{section:constructionmethod}

We sketch a method to construct hyperelliptic manifolds with a given Abelian holonomy group. Let
\begin{align*}
G = C_{e_1} \times ... \times C_{e_r}, \qquad \text{where } \quad e_i \geq 2.
\end{align*}
Note that we do \emph{not} require $e_i ~ | ~ e_{i+1}$ here. Furthermore, we define $\delta_i \geq 1$ as follows:

\begin{align*}
\delta_i := \begin{cases}
1, & \text{ if } e_i = 2, \\
\varphi(e_i)/2, & \text{ if } e_i \geq 3.
\end{cases}
\end{align*}

We construct a hyperelliptic manifold $X = T/G$ of dimension $n$, where $n$ is such that
\begin{align} \label{eq:constructionmethod}
\sum_{i=1}^r \delta_i \leq n - \frac{r}{2}.
\end{align}

For each $1 \leq i \leq r$, we define a complex torus $T_i$ of dimension $\delta_i$ as follows. If $e_i = 2$, then we let $T_i$ be any elliptic curve and if $e_i \geq 3$, we let $T_i$ be a CM-torus corresponding to the cyclotomic field $\QQ(\zeta_{e_i})$. Furthermore, let $T'$ be any complex torus of dimension $d := n - \sum_{i=1}^r \delta_i$ and define
\begin{align*}
T := T_1 \times ... \times T_r \times T'.
\end{align*}
The $G$-action on $T$ is constructed as follows. Let $g_j \in C_{e_j}$ be generators, so that $g_1, ..., g_r$ generate $G$. Then $g_j$ acts on $T$ as follows:
\begin{itemize}
	\item $g_j$ acts by multiplication by $-1$ on $T_j$ if $e_j = 2$ and linearly by the corresponding CM-type of $T_j$ if $e_j \geq 3$,
	\item $g_j$ acts trivially on $T_i$ if $i \neq j$,
	\item $g_j$ acts on $T'$ by a translation $z' \mapsto z' + t_j$ of order $e_j$. 
\end{itemize}
Then, by construction, the associated complex representation is faithful. \\
Note that in order for the $G$-action to be free, the translation parts $t_j$ cannot be chosen arbitrarily. We choose them as follows. Since $T'[|G|]$ is isomorphic to the direct product of $2d$ copies of $C_{|G|}$ and $r \leq 2d$ by (\ref{eq:constructionmethod}), we may choose the translation parts $t_1, ..., t_r$ such that
\begin{align*}
\bigcap_{i=1}^r \langle t_i\rangle = \{0\}.
\end{align*}
This ensures that the sums
\begin{align*}
\sum_{i=1}^r k_i t_i, \qquad 0 \leq k_i < e_i, \quad (k_1,\ ...,\ k_r) \neq 0
\end{align*}
are non-zero, which then proves that the elements $\prod_{i=1}^r g_i^{k_i}$ all act freely on $T$. It follows that $X = T/G$ is indeed a hyperelliptic manifold of dimension $n$. \\

\begin{rem}
	It is, of course, neither true that all hyperelliptic manifolds of dimension $n$ with holonomy group $G$ occur in this way nor that $n$ is the minimal dimension of a hyperelliptic manifold with holonomy group $G$.
\end{rem}

\section{Classification of Abelian Hyperelliptic Groups in Dimension $4$} \label{section:ab-class}

We apply the results of the previous sections to dimension $4$. More precisely, our goal is to prove

\begin{thm} \label{thm:abelian-classification}
	The following list consists precisely of the Abelian hyperelliptic groups in dimension $4$:
	\begin{center}
		\begin{longtable}{l|r}
			\begin{tabular}{l|l}
				ID & Group \\ \hline \hline 
				$[2,1]$ & $C_2$\\
				$[3,1]$ & $C_3$\\
				$[4,1]$ & $C_4$ \\
				$[4,2]$ & $C_2 \times C_2$ \\
				$[5,1]$ & $C_5$ \\
				$[6,2]$ & $C_6$  \\
				$[7,1]$ & $C_7$ \\
				$[8,1]$ & $C_8$  \\
				$[8,2]$ & $C_2 \times C_4$ \\
				$[8,5]$ & $C_2 \times C_2 \times C_2$ \\
				$[9,1]$ & $C_9$ \\
				$[9,2]$ & $C_3 \times C_3$\\
				$[10,2]$ & $C_{10}$ \\
				$[12,2]$ & $C_{12}$ \\
				$[12,5]$ & $C_2 \times C_6$ \\
				$[14,2]$ & $C_{14}$ \\
				$[15,1]$ & $C_{15}$ \\
				$[16,2]$ & $C_4 \times C_4$ \\
				$[16,5]$ &  $C_2 \times C_8$  \\
				$[16,10]$ & $C_2 \times C_2 \times C_4$ \\
				$[18,2]$ & $C_{18}$ \\
				$[18,5]$ & $C_3 \times C_6$
			\end{tabular}
			&
			\begin{tabular}{l|l}
				ID & Group \\ \hline \hline 
				$[20,2]$ & $C_{20}$ \\
				$[20,5]$ & $C_2 \times C_{10}$ \\
				$[24,2]$ & $C_{24}$  \\
				$[24,9]$ & $C_2 \times C_{12}$ \\
				$[24,15]$ & $C_2 \times C_2 \times C_6$ \\
				$[27,5]$ & $C_3 \times C_3 \times C_3$ \\
				$[30,4]$ & $C_{30}$ \\
				$[32,3]$ & $C_4 \times C_8$ \\
				$[32,21]$ & $C_2 \times C_4 \times C_4$ \\
				$[36,8]$ & $C_3 \times C_{12}$ \\
				$[36,14]$ & $C_6 \times C_6$ \\
				$[40,9]$ & $C_2 \times C_{20}$ \\
				$[48,20]$ & $C_4 \times C_{12}$ \\
				$[48,23]$ & $C_2 \times C_{24}$ \\
				$[48,44]$ & $C_2 \times C_2 \times C_{12}$\\
				$[54,15]$ & $C_3 \times C_3 \times C_6$  \\
				$[60,13]$ & $C_2 \times C_{30}$ \\
				$[72,36]$ & $C_6 \times C_{12}$\\
				$[72,50]$ & $C_2 \times C_6 \times C_6$ \\
				$[96,161]$ & $C_2 \times C_4 \times C_{12}$ \\
				$[108,45] $ & $C_3 \times C_6 \times C_6$ \\
				$[144,178]$ & $C_2 \times C_6 \times C_{12}$ \\
			\end{tabular}
		\end{longtable}
	\end{center}
\end{thm}

A closer inspection of the table above reveals that every group in the table is -- up to isomorphism -- a subgroup of one of the following nine groups:
\begin{center}
	$C_{14}, \qquad C_{18}, \qquad C_4 \times C_8, \qquad C_2 \times C_{20}, \qquad C_2 \times C_{24}, \qquad C_2 \times C_{30},$ \\ $C_2 \times C_4 \times C_{12}, \qquad C_3 \times C_6 \times C_6, \qquad C_2 \times C_6 \times C_{12}$.
\end{center}

The stategy for proving \hyperref[thm:abelian-classification]{Theorem~\ref*{thm:abelian-classification}} is therefore:
\begin{enumerate}[ref=(\roman*), label=(\roman*)]
	\item \label{thm:ab-class1} show that the nine ``maximal'' groups above are indeed hyperelliptic in dimension $4$,
	\item \label{thm:ab-class2} show that the following groups are \emph{not} hyperelliptic in dimension $4$:
	\begin{center}
		$C_2 \times C_{14}, \qquad C_2 \times C_2 \times C_8, \qquad C_2 \times C_{18}, \qquad C_3 \times C_{15}, \qquad C_4 \times C_4 \times C_4, \qquad C_3 \times C_{24},$ \\ $C_8 \times C_8, \qquad C_4 \times C_{20}, \qquad C_4 \times C_{24}, \qquad C_3 \times C_3 \times C_{12}, \qquad C_{12} \times C_{12}, \qquad C_6 \times C_6 \times C_6$.
	\end{center}
\end{enumerate}
Recall that a hyperelliptic group in dimension $4$ has order $2^a \cdot 3^b \cdot 5^c \cdot 7^c$ with $c,d \leq 1$ (see \hyperref[cor:group-order-with-bounds]{Corollary~\ref*{cor:group-order-with-bounds}}) and that the order of every group element is one of
\begin{align*}
1,~ 2,~ 3,~ 4,~ 5,~ 6,~ 7,~ 8,~ 9,~ 10,~ 12,~ 14,~ 15,~ 18,~ 20,~ 24,~ 30,
\end{align*}
see \hyperref[order-cyclic-groups]{Lemma~\ref*{order-cyclic-groups}} \ref{ocg-2}. Hence, taking into account that the class of hyperelliptic groups in a given dimension is closed under taking non-trivial subgroups,  it suffices to complete steps \ref{thm:ab-class1} and \ref{thm:ab-class2} in order to prove \hyperref[thm:abelian-classification]{Theorem~\ref*{thm:abelian-classification}}.

We start with \ref{thm:ab-class1}. First of all, the construction method described in \hyperref[section:constructionmethod]{Section~\ref*{section:constructionmethod}} proves that 

\begin{cor} \label{abelian:cor}
	The following groups are hyperelliptic in dimension $4$:
	\begin{align*}
	C_{14}, \quad C_{18}, \quad  C_4 \times C_8, \quad C_6 \times C_8 \cong C_2 \times C_{24}, \quad C_4 \times C_{10} \cong C_2 \times C_{20}, \quad C_6 \times C_{10} \cong C_2 \times C_{30}.
	\end{align*}
\end{cor} 

\begin{example} \label{abelian:ex}
	We give concrete examples to show that
	\begin{enumerate}[ref=(\theenumi)]
		\item \label{abelian:ex-1} $C_2 \times C_4 \times C_{12}$, 
		\item \label{abelian:ex-2} $C_2 \times C_6 \times C_{12}$,
		\item \label{abelian:ex-3} $C_3 \times C_6 \times C_6$
	\end{enumerate}
	are hyperelliptic in dimension $4$. \\
	
	\ref{abelian:ex-1} Let $T =  F \times E_i \times E_i \times E$, where
	\begin{itemize}
		\item $F = \CC/(\ZZ + \zeta_3 \ZZ)$ is Fermat's elliptic curve,
		\item $E_i= \CC/(\ZZ+i\ZZ)$ is Gauss' elliptic curve, and
		\item $E = \CC/(\ZZ + \tau\ZZ)$ is an arbitrary elliptic curve in normal form.
	\end{itemize}
	Furthermore, we define biholomorphisms of $T$ as follows: 
	\begin{align*}
	&g_1(z) = \left(-z_1, ~ z_2 + \frac{1-i}2, ~ z_3, ~ z_4\right), \\
	&g_2(z) = \left(z_1, ~ z_2, ~ iz_3, ~ z_4 + \frac14\right), \\
	&g_3(z) = \left(\zeta_3 z_1, ~ i z_2, ~ z_3, ~ z_4 + \frac{\tau}{12}\right).
	\end{align*} 
	The following relations are clearly satisfied:
	\begin{align*}
	\ord(g_1) = 2, \qquad \ord(g_2) = 4,\qquad \ord(g_3) = 12, \qquad g_1 g_2 = g_2 g_1, \qquad \text{ and }\qquad g_2 g_3 = g_1 g_2. 
	\end{align*}
	Furthermore, $g_1$ and $g_3$ commute, since $i \frac{1-i}{2} = \frac{1-i}{2}$ in $E_i$. Hence $G := \langle g_1, g_2, g_3 \rangle \subset \Bihol(T)$ is isomorphic to $C_2 \times C_4 \times C_{12}$. By construction, the associated complex representation is faithful; hence $G$ does not contain any translations. We claim that $G$ acts freely on $T$. This is seen by spelling out the action of $g_1^{a_1} g_2^{a_2} g_3^{a_3}$ ($0 \leq a_j <\ord(g_j)$) on the second and the fourth factor on $T$: it is given by
	\begin{align*}
	g_1^{a_1} g_2^{a_2} g_3^{a_3}(z_2,z_4) = \left(  i^c z_2 + a_1 \cdot \frac{1-i}{2}, ~ z_4 + \frac{a_2}{4} + \frac{a_3 \tau}{12}\right).
	\end{align*}
	Consequently, $g_1^{a_1} g_2^{a_2} g_3^{a_3}$ has a fixed point on $T$ only if $a_1 = a_2 = a_3 = 0$, which shows that $G$ acts freely on $T$. Hence $T/G$ is a hyperelliptic fourfold with holonomy group $C_2 \times C_4 \times C_{12}$. \\
	
	\ref{abelian:ex-2} Here, we define $T = F \times E_i \times F \times E$, where $F$, $E_i$ and $E$ are as above, and 
	\begin{align*}
	&g_1(z) = \left(-z_1, ~ z_2 + \frac{1-i}{2}, ~ z_3, ~ z_4\right), \\
	&g_2(z) = \left(z_1, ~ z_2, ~ -\zeta_3 z_3, ~ z_4 + \frac16\right), \\
	&g_3(z) = \left(\zeta_3 z_1, ~ i z_2, ~ z_3, ~ z_4 + \frac{\tau}{12}\right).
	\end{align*} 
	The definitions immediately imply that
	\begin{align*}
	\ord(g_1) = 2, \qquad \ord(g_2) = 6,\qquad \ord(g_3) = 12, \qquad g_1 g_2 = g_2 g_1, \qquad \text{ and }\qquad g_2 g_3 = g_1 g_2. 
	\end{align*}
	We show exactly as in \ref{abelian:ex-1} that $\langle g_1, g_2, g_3 \rangle \Bihol(T)$ is isomorphic to $C_2 \times C_6 \times C_{12}$, does not contain any translations and acts freely on $T$. Hence $C_2 \times C_6 \times C_{12}$ is hyperelliptic in dimension $4$.\\
	
	\ref{abelian:ex-3} Let $T = F \times F \times F \times E$. Define $g_j \in \Bihol(T)$ as follows:
	\begin{align*}
	&g_1(z) = \left(\zeta_3 z_1, ~ z_2, ~ z_3 + \frac{1-\zeta_3}{3}, ~ z_4 \right), \\
	&g_2(z) = \left(z_1, ~ -\zeta_3z_2, ~ z_3, ~ z_4 + \frac16 \right), \\
	&g_3(z) = \left(-z_1, ~ z_2, ~ \zeta_3z_3, ~ z_4 + \frac{\tau}{6}\right).
	\end{align*}
	The definitions immediately imply that 
	\begin{align*}
	\ord(g_1) = 3, \qquad \ord(g_2) = 6, \qquad \ord(g_3) = 6, \qquad g_1 g_2 = g_2 g_1, \qquad \text{ and } \qquad g_2 g_3 = g_3 g_2. 
	\end{align*}
	Furthermore, $\zeta_3 \frac{1-\zeta_3}{3} = \frac{1-\zeta_3}{3}$ in $F$, and hence $g_1$ and $g_3$ commute as well. It follows that $G := \langle g_1, g_2, g_3 \rangle \Bihol(T)$ is isomorphic to $C_3 \times C_6 \times C_6$. Since the associated complex representation is faithful, $G$ does not contain any translations. \\
	In order to show that $G$ acts freely on $T$, we observe that the action $g_1^{a_1} g_2^{a_2} g_3^{a_3}$ ($0 \leq a_j <\ord(g_j)$) on the last two factors of $T$ is given as follows:
	\begin{align*}
	g_1^{a_1} g_2^{a_2} g_3^{a_3}(z_3,z_4) = \left(\zeta_3^{a_3} + a_1 \cdot \frac{1-\zeta_3}{3}, ~ z_4 + \frac{a_2}{6} + \frac{a_3\tau}{12} \right).
	\end{align*}
	It follows that $g_1^{a_1} g_2^{a_2} g_3^{a_3}$ can only have a fixed point on $T$, if $a_1 = a_2 = a_3 = 0$, which shows that $G$ acts freely on $T$. We obtain that $T/G$ is a hyperelliptic fourfold with holonomy group $C_3 \times C_6 \times C_6$.
\end{example}

This finishes step \ref{thm:ab-class1}, hence it remains to complete step \ref{thm:ab-class2}. An immediate application of \hyperref[abelian-cor-cd^n-2]{Corollary~\ref*{abelian-cor-cd^n-2}} is:

\begin{cor} \label{cor:2-factors-excluded}
	The groups $C_8 \times C_8$ and $C_{12} \times C_{12}$ are not hyperelliptic in dimension $4$.
\end{cor}

Next, we exclude the remaining groups listed in \ref{thm:ab-class2} that are a product of two cyclic groups. We use the following preliminary lemma.

\begin{lemma} \label{lemma:cyclic-cplx-rep}
	Let $X = T/C_d$ be a hyperelliptic fourfold with cyclic holonomy group $C_d = \langle g \rangle$, $d \in \{14,15,18,20,24\}$. Denote by $\rho \colon C_d \to \GL(4,\CC)$ the associated complex representation. Then, $\rho$ is -- up to equivalence and automorphisms -- given as follows:
	\begin{enumerate}[ref=(\theenumi)]
		\item \label{lemma:cyclic-cplx-rep1} if $d = 14$, then $\rho(g) = \diag(\zeta_{14}, ~ \zeta_{14}^a, ~ \zeta_{14}^b, ~ 1)$, where $(a,b) \in \{(3,5), (5,11)\}$; in particular $\rho(g)$ has three eigenvalues of order $14$,
		\item \label{lemma:cyclic-cplx-rep2} if $d = 15$, then $\rho(g) = \diag(\zeta_5, ~ \zeta_5^2, ~ \zeta_3, ~ 1)$,
		\item\label{lemma:cyclic-cplx-rep3} if $d = 18$, then $\rho(g) = \diag(\zeta_{18}, ~ \zeta_{18}^a,~  \zeta_{18}^b, ~ 1)$, where $(a,b) \in \{(5,7), (7,13)\}$; in particular $\rho(g)$ has three eigenvalues of order $18$,
		\item \label{lemma:cyclic-cplx-rep4} if $d = 20$, then $\rho(g) = \diag(\varepsilon \zeta_5, ~ \varepsilon \zeta_5^2, ~ i, ~ 1)$, where $\varepsilon \in \{1,-1\}$,
		\item \label{lemma:cyclic-cplx-rep5} if $d = 24$, then $\rho(g) = \diag(\zeta_8, ~ \zeta_8^r, ~ \varepsilon \zeta_3, ~ 1)$, where $r \in \{3,5\}$, $\varepsilon \in \{1,-1\}$.
	\end{enumerate}
\end{lemma}

\begin{proof}
	Recall that $\rho(g)$ has the eigenvalue $1$, since $g$ acts freely on $T$. \\
	
	\ref{lemma:cyclic-cplx-rep1} According to the \hyperref[order-cyclic-groups]{Lemma~\ref*{order-cyclic-groups}} \ref{ocg-4}, the number of eigenvalues of order $14$ of $\rho(g)$ is three. The statement hence follows from \hyperref[lemma-table]{Lemma~\ref*{lemma-table}}. \\
	
	\ref{lemma:cyclic-cplx-rep2} \hyperref[order-cyclic-groups]{Lemma~\ref*{order-cyclic-groups}} shows that $\rho(g)$ has no eigenvalues of order $15$. Hence $\rho(g)$ has the eigenvalue $1$ as well as two eigenvalues of order $5$ and one eigenvalue of order $3$. The statement thus follows from \hyperref[lemma-table]{Lemma~\ref*{lemma-table}} by choosing an appropriate basis and replacing $g$ by a suitable power, if necessary. \\
	
	\ref{lemma:cyclic-cplx-rep3} This follows exactly as the case $d = 14$. \\
	
	\ref{lemma:cyclic-cplx-rep4}, \ref{lemma:cyclic-cplx-rep5} can be dealt with as the case $d = 15$.
\end{proof}

\begin{cor} \label{cor:cd1xcd2-excluded}
	The groups
	\begin{align*}
	C_2 \times C_{14}, \qquad C_2 \times C_{18}, \qquad C_3 \times C_{15}, \qquad C_3 \times C_{24}, \qquad C_4 \times C_{20}, \qquad C_4 \times C_{24}
	\end{align*}
	are not hyperelliptic in dimension $4$.
\end{cor}

\begin{proof}
	All of the groups in the statement of the lemma are of the form $G = C_{d_1} \times C_{d_2}$, where $d_1 ~ | ~ d_2$ and $d_2 \in \{14,15,18,20,24\}$. Consider a representation $\rho \colon G \to \GL(4,\CC)$ satisfying the usual properties
	\begin{enumerate}[label=(\roman*), ref=(\roman*)]
		\item \label{abelian:rho-prop1} $\rho$ is faithful,
		\item \label{abelian:rho-prop2} every matrix in $\rho(G)$ has the eigenvalue $1$, and
		\item \label{abelian:rho-prop3} if $g \in G$, then the characteristic polynomial of $\rho(g) \oplus \overline{\rho(g)}$ has integer coefficients.
	\end{enumerate}
	Denote by $\{g_1,g_2\}$ a generating set of $G$, where $\ord(g_j) = d_j$. We may then assume that $\rho(g_2)$ is as in \hyperref[lemma:cyclic-cplx-rep]{Lemma~\ref*{lemma:cyclic-cplx-rep}}. We now treat the five groups separately: \\
	
	\underline{$(d_1,d_2) = (2,14)$:} The last diagonal entry of $\rho(g_1)$ must be $1$, since else $\rho(g_1g_2)$ does not have the eigenvalue $1$. Furthermore, $\rho(g_1) \neq \rho(g_2^7) = \diag(-1, ~ -1, ~ -1, ~ 1)$, since $\rho$ is faithful. It follows that $\ord(g_1g_2) = 14$, but $\rho(g_1g_2)$ has only one or two eigenvalues of order $14$. Hence, according to \hyperref[lemma:cyclic-cplx-rep]{Lemma~\ref*{lemma:cyclic-cplx-rep}} \ref{lemma:cyclic-cplx-rep1} (or property \ref{abelian:rho-prop3}) $\rho$ cannot be the complex representation of a hyperelliptic fourfold. \\
	
	\underline{$(d_1,d_2) = (2,18)$:} This is shown exactly as the case $(d_1,d_2) = (2,14)$. \\
	
	\underline{$(d_1,d_2) = (3,15)$:} Here, property \ref{abelian:rho-prop2} is violated for $g_1g_2$, unless the last diagonal entry of $\rho(g_1)$ is $1$. Furthermore, the proof of \hyperref[lemma:cyclic-cplx-rep]{Lemma~\ref*{lemma:cyclic-cplx-rep}} \ref{lemma:cyclic-cplx-rep1} (or property \ref{abelian:rho-prop2} (or property \ref{abelian:rho-prop3}) shows that the first two diagonal entries of $\rho(g_1)$ are $1$. It follows that (up to replacing $g_1$ by $g_1^2$): \\
	\begin{align*}
	\rho(g_1) = \diag(1, ~ 1, ~ \zeta_3^2, ~ 1) = \rho(g_2^{5}).
	\end{align*}
	Hence, property \ref{abelian:rho-prop1} is violated if \ref{abelian:rho-prop2} and  \ref{abelian:rho-prop3} are satisfied. \\
	
	\underline{$(d_1,d_2) = (3,24)$:} This case can be excluded by the same arguments as the case $(d_1,d_2) = (3,15)$. \\
	
	\underline{$(d_1,d_2) = (4,20)$:} Again, this can be excluded as $(d_1,d_2) = (3,15)$. \\
	
	\underline{$(d_2,d_2) = (4,24)$:} First of all, $\rho(g_1g_2)$ only has the eigenvalue $1$, if the last diagonal entry of $\rho(g_1)$ is $1$. Furthermore it is necessary for property \ref{abelian:rho-prop3} to be satisfied, $\rho(g_1g_2)$ cannot have eigenvalues of order $12$, which implies that the third diagonal entry of $\rho(g_1)$ is $\delta \in \{1,-1\}$. By faithfulness of $\rho$, we may assume that 
	\begin{align*}
	\rho(g_1) = \diag(i, ~ i^a, ~ \delta, ~ 1), \qquad \text{ where } \quad a \in \{0,...,3\}. 
	\end{align*}
	Observe now that
	\begin{align*}
	\rho(g_2^{6}) = \diag(i, ~ i^{-r}, ~ 1, ~ 1).
	\end{align*}
	Again by faithfulness,
	\begin{align*}
	\rho(g_1g_2^6) = \diag(1, ~ i^{a-r}, ~ \delta, ~ 1) 
	\end{align*}
	is a matrix of order $4$. It follows that  $a \equiv r\pm 1 \pmod 4$ and hence -- since $r \in \{3,5\}$ -- we obtain $a  \in \{0,2\}$. Hence
	\begin{align*}
	\rho(g_1^2 g_2) = \diag(\zeta_8^3, ~ \zeta_8^r, ~ \varepsilon \zeta_3, ~ 1)
	\end{align*}
	either has an eigenvalue of order $8$ with multiplicity $2$ (if $r = 3$) or two complex conjugate eigenvalues of order $5$ (if $r = 5$). Thus property \ref{abelian:rho-prop3} is violated in any case.
\end{proof}

Finally, we exclude the groups listed in \ref{thm:ab-class2}  that are a product of three cyclic groups. The groups $C_4 \times C_4 \times C_4$ and $C_6 \times C_6 \times C_6$ were already excluded in \hyperref[cor:c6^3-c4^3-excluded]{Corollary~\ref*{cor:c6^3-c4^3-excluded}}, that we are left with excluding $C_2 \times C_2 \times C_8$ and $C_3 \times C_3 \times C_{12}$. We exclude $C_2 \times C_2 \times C_8$ first.                                                                                                                                                                                                                                                                           

\begin{lemma} \label{lemma:cd1xcd2-cplx-rep}
	Let $X = T/(C_{2} \times C_{8})$ be a hyperelliptic fourfold with holonomy group $C_{2} \times C_{8} = \langle g_1, g_2\rangle$, where $\ord(g_1) = 2$, and $\ord(g_2) = 8$. Denote $\rho \colon C_2 \times C_8 \to \GL(4,\CC)$ the associated complex representation. Then, $\rho$ is -- up to equivalence and conjugation -- given as follows:
	\begin{align*}
	\rho(g_1) = \diag(1, ~ 1, ~ -1, ~ 1) \quad \text{ and } \quad \rho(g_2) = \diag(\zeta_8, ~ \zeta_8^r, ~ i^s, ~ 1), \qquad \text{ where }\quad r \in \{3,5\}, ~ s \in \{0,...,3\}
	\end{align*}
	
\end{lemma}

\begin{proof}
	According to \hyperref[lemma-table]{Lemma~\ref*{lemma-table}}, we may assume that $\rho(g_2) = \diag(\zeta_8, ~ \zeta_8^r, ~ i^s, ~ 1)$ for some $r \in \{3,5\}$ and $s \in \{0,...,3 \}$. Then $\rho(g_2^4) = \diag(-1, ~ -1, ~ 1, ~ 1)$, and by replacing $g_1$ by $g_1 g_2^4$ if necessary, we may assume that at most one of the first two diagonal entries of $\rho(g_1)$ is $-1$.  If (say) the first diagonal entry of $\rho(g_1)$ is $-1$, then $-\zeta_8 = \zeta_8^5$, and hence $\rho(g_1g_2)$ has an eigenvalue of order $8$ with multiplicity $2$ or two complex conjugate eigenvalues of order $8$. In both cases, we obtain a contradiction to the \hyperref[order-cyclic-groups]{Integrality Lemma~\ref*{order-cyclic-groups}} \ref{ocg-3}. We may therefore assume that $\rho(g_1) = \diag(1, ~ 1, ~ \delta, ~ \varepsilon)$, where $\delta, \varepsilon \in \{1,-1\}$, $(\delta, \varepsilon) \neq (1,1)$. Since
	\begin{align*}
	\rho(g_1g_2^4) = \diag(-1, ~ -1, ~ \delta, ~ \varepsilon)
	\end{align*} 
	must have the eigenvalue $1$, we conclude that $(\delta, \varepsilon) \neq (-1,-1)$. Hence, if $\delta= -1$, then $\varepsilon = 1$ and $\rho(g_1)$ is of the desired form. Similarly, if $\varepsilon = -1$, then $\delta = 1$, and since
	\begin{align*}
	\rho(g_1g_2) = \diag(\zeta_8, ~ \zeta_8^r, ~ i^s, ~ -1)
	\end{align*}
	must have the eigenvalue $1$, we infer that $s = 0$. Here, after permuting the third and fourth coordinate, $\rho(g_1)$ and $\rho(g_2)$ are as in the statement of the lemma.
\end{proof}

Similarly, as above, we obtain the following consequence:

\begin{cor}
	The group $C_2 \times C_2 \times C_8$ is not hyperelliptic in dimension $4$. 
\end{cor}

\begin{proof}
	This follows from \hyperref[lemma:cd1xcd2-cplx-rep]{Lemma~\ref*{lemma:cd1xcd2-cplx-rep}} in the same way as \hyperref[cor:cd1xcd2-excluded]{Corollary~\ref*{cor:cd1xcd2-excluded}} was deduced from \hyperref[lemma:cyclic-cplx-rep]{Lemma~\ref*{lemma:cyclic-cplx-rep}}.
\end{proof}

Finally, we show

\begin{lemma} \label{c3xc3xc12-excluded}
	The group $C_3 \times C_3 \times C_{12}$ is not hyperelliptic in dimension $4$.
\end{lemma}

\begin{proof}
	If $\rho \colon C_3 \times C_3 \times C_{12} \to \GL(4,\CC)$ is a faithful representation with the property that every matrix in $\rho(C_3 \times C_3 \times C_{12})$ has the eigenvalue $1$, then, according to \hyperref[common-ev-1]{Lemma~\ref*{common-ev-1}}, we may assume that $C_3 \times C_3 \times C_3 = \langle g_1, g_2, g_3 \rangle$ is embedded as follows:
	\begin{align*}
	\rho(g_1^a g_2^b g_3^c) = \diag(\zeta_3^a, ~ \zeta_3^b, ~ \zeta_3^c, ~ 1).
	\end{align*}
	Denoting now by $h \in C_3 \times C_3 \times C_{12}$ an element of order $4$ (so that $G$ is spanned by $h$ and the $g_j$), then $\rho(g_1g_2g_3h)$ does not have the eigenvalue $1$, unless the last diagonal entry of $\rho(h)$ is $1$. Moreover, if (say) the first diagonal entry of $\rho(h)$ is equal to $i$, then $\rho(g_1 h)$ has exactly one eigenvalue of order $12$: in this case, the \hyperref[order-cyclic-groups]{Integrality Lemma~\ref*{order-cyclic-groups}} \ref{ocg-3} shows that $\rho$ cannot be the complex representation of any hyperelliptic fourfold.
\end{proof}

This completes the proof of  \hyperref[thm:abelian-classification]{Theorem~\ref*{thm:abelian-classification}}.

\chapter{The $2$-Sylow Subgroups of Hyperelliptic Groups in Dimension $4$} \label{section-2-sylow}

This chapter aims to give an upper bound for the order of hyperelliptic $2$-groups in dimension $4$. We prove that the order of such a $2$-group $G$ is bounded by $128$: in the proof, we distinguish between the three cases in which
\begin{itemize}
	\item $G$ is Abelian: here, we have $|G| \leq 32$, see \hyperref[abelian-2-grp]{Corollary~\ref*{abelian-2-grp}},
	\item $G$ is non-Abelian and has a faithful irreducible representation of dimension $2$: again, $|G| \leq 32$ in this case, see the discussion on p. \pageref{page-order-32},
	\item $G$ is non-Abelian and has no faithful irreducible representation: \hyperref[at-most-128]{Corollary~\ref*{at-most-128}} shows that $|G| \leq 128$.
\end{itemize}
Lemma \ref{2-sylow-2possibilities} guarantees that one of these three cases indeed occurs.\\
Later, in \hyperref[section:running-algo]{Section~\ref*{section:running-algo}} we will see that the order of $G$ is actually bounded by $32$. \\

Throughout the chapter, we will assume that

\noindent\fbox{%
	\parbox{0.975\textwidth}{\begin{center}
			The group $G$ is a hyperelliptic $2$-group in dimension $4$. By $\rho \colon G \to \GL(4,\CC)$, we denote the complex representation of a hyperelliptic fourfold with holonomy group $G$.
		\end{center}
	}
}

\section{General results}

Let $G$ be a hyperelliptic $2$-group in dimension $4$. The case where $G$ is Abelian was already dealt with in \hyperref[thm:abelian-classification]{Theorem~\ref*{thm:abelian-classification}}:

\begin{cor} \label{abelian-2-grp}
	If $G$ is an Abelian hyperelliptic $2$-group in dimension $4$, then $G$ is isomorphic to one of the following groups:
	\begin{center}
		$C_2, \qquad C_2 \times C_2, \qquad C_4, \qquad C_2 \times C_2 \times C_2, \qquad C_2 \times C_4, \qquad C_8,$ \\
		$C_2 \times C_2 \times C_4, \qquad C_4 \times C_4, \qquad C_2 \times C_8, \qquad C_2 \times C_4 \times C_4, \qquad C_4 \times C_8$.
	\end{center}
	In particular, $|G| \leq 32$.
\end{cor}

We shall therefore concentrate on the case where $G$ is non-Abelian. In this case, the representation $\rho$ splits as a direct sum of two representations $\rho_2$, $\rho_2'$ of degree $2$, where we can assume $\rho_2$ to be irreducible.

\begin{lemma} \label{2-sylow-2possibilities}
	One of the following possibilities holds:
	\begin{enumerate}[ref=(\theenumi)]
		\item \label{2-sylow-poss1} $\rho$ contains a faithful, irreducible representation of degree $2$.
		\item \label{2-sylow-poss2} No irreducible representation of $G$ is faithful.
	\end{enumerate}
\end{lemma}

\begin{proof}
	Assume that (2) does not hold. By  \hyperref[thm:huppert:p-groups]{Theorem~\ref*{thm:huppert:p-groups}}, this means that $Z(G)$ is cyclic. Since $\rho$ is faithful, we obtain
	\begin{align} \label{rep-faithful}
	\ker(\rho_2) \cap \ker(\rho_2') = \{1\}.
	\end{align}
	Now, $\ker(\rho_2)$ and $\ker(\rho_2')$ are normal subgroups of $G$. A non-trivial normal subgroup of the $2$-group $G$ intersects $Z(G)$ non-trivially\label{p-grp-cent}\footnote{This follows from a 'Sylow-type' argument and holds more generally for any finite $p$-group $G$: any normal subgroup $N$ of $G$ is the union of conjugacy classes, which have length $p^j$ for some $j$. Since $N$ contains the conjugacy class $\{1\}$ of length $1$, and because $N$ is a $p$-group, it follows that $N$ contains another conjugacy class of length $1$, or, equivalently, a non-trivial central element of $G$.}: since $Z(G)$ is cyclic we infer that (\ref{rep-faithful}) can only hold if $\ker(\rho_2) = \{1\}$ or $\ker(\rho_2') = \{1\}$. In the latter case, we observe that $\rho_2'$ must be irreducible because $G$ is non-Abelian.
\end{proof}

In the upcoming two section, we will treat the cases \ref{2-sylow-poss1} and \ref{2-sylow-poss2} separately.

\section{The case in which $G$ is non-Abelian and $Z(G)$ is cyclic} \label{section:2-sylow-case1}

We will first treat \ref{2-sylow-poss1}, i.e., we will assume for the time being that
\vspace{0.5cm}
\begin{center}
	\noindent\fbox{%
		\parbox{0.975\textwidth}{\begin{center}
				The representation $\rho$ contains a faithful irreducible representation $\rho_2$ of degree $2$.\end{center}
		}
	}
\end{center} \vspace{0.5cm}

We claim that $|G| \leq 32$. Consider the exact sequence
\begin{align*} 
1 \to N \to G \to C_m \to 1,
\end{align*}
where $G \to C_m$ is given by $g \mapsto \det \rho_2(g)$. Note that  \hyperref[lemma-table]{Lemma~\ref*{lemma-table}} shows that $m$ divides $4$.\\
Now we analyze the kernel $N$. Our main observation is that if $g \in N$, then the eigenvalues of $\rho_2(g)$ are $\zeta$ and $\zeta^{-1}$ for an $\ord(g)$th root of unity $\zeta$. This has the following crucial consequences:
\begin{enumerate}[label=(\roman*)]
	\item \label{2syl-obs-1} together with the \hyperref[order-cyclic-groups]{Integrality Lemma~\ref*{order-cyclic-groups}} \ref{ocg-3}, it implies that $N$ has exponent $\leq 4$
	\item \label{2syl-obs-2} if $g \in N$ is an element of order $2$, then $\rho_2(g) = -I_2$. Now, since $\rho_2$ is faithful, we infer that $N$ contains a unique element of order $2$.
\end{enumerate}
The structure theorem \cite[Theorem 12.5.2]{Hall} for $2$-groups with a unique subgroup of order $2$ together with \ref{2syl-obs-1} imply that $N$ is either cyclic of order dividing $4$ or a generalized quaternion group, which then necessarily has order $8$.
We have proven that $|G| \leq 32$, achieved only if $N \cong Q_8$ and $m = 4$. \label{page-order-32} \\

We use  \hyperref[script:2-groups-case1]{GAP Script~\ref*{script:2-groups-case1}} to find all non-Abelian $2$-groups $G$ of order dividing $32$ such that
\begin{enumerate} \label{2-sylows-gap}
	\item $Z(G)$ is cyclic of order $2$ or $4$ (recall from \hyperref[cor:center-dim4]{Corollary~\ref*{cor:center-dim4}} that $Z(G)$ does not contain elements of order $8$),
	\item the exponent of $G$ divides $8$ (see  \hyperref[order-cyclic-groups]{Lemma~\ref*{order-cyclic-groups}}),
	\item no element of order $8$ is conjugate to its inverse (this is necessary by \hyperref[prop:conjugate]{Proposition~\ref*{prop:conjugate}}),
	\item if $|G| = 16$, then $G$ contains a normal subgroup $N$ such that $G/N$ is cyclic of order $\leq 4$ and $N \cong C_4$ or $N \cong Q_8$,
	\item if $|G| = 32$, then $G$ contains a normal subgroup $N \cong Q_8$ such that $G/N$ is cyclic (of order $4$).
\end{enumerate}

The GAP output consists of the six groups
\begin{align*}
&D_4 = \langle r,s \ | \ r^4 = s^2 = 1, \ (rs)^2 = 1 \rangle \quad \text{(ID } [8,3] \text{)}, \\
&Q_8 = \langle a,b \ | \ a^4 = 1, \ a^2 = b^2, \ ab = b^{-1}a\rangle \quad \text{(ID } [8,4]\text{)}, \\
&G(8,2,5) = \langle g,h \ |\ g^8 = h^2 = 1, \ h^{-1}gh = g^5 \rangle \quad \text{(ID } [16,6]\text{)}, \\
&G(8,2,3) = \langle g,h \ |\ g^8 = h^2 = 1, \ h^{-1}gh = g^3 \rangle \quad \text{(ID } [16,8]\text{)}, \\
&D_4 \curlyvee C_4 = \langle r,s,k \ | \ r^4 = s^2 = (rs)^2 = 1, \ k \text{ central}, \ k^2 = r^2 \rangle \quad \text{(ID } [16,13]\text{)}, \\
&(C_4 \times C_4) \rtimes C_2 = \langle a,b,c \ | \ a^4 = b^4 = c^2 = 1, \ ab = ba, \ c^{-1}ac = a^{-1}, \ c^{-1}bc = ab \rangle  \quad \text{(ID } [32,11]\text{)}.
\end{align*}

\begin{rem} \label{rem:2-sylow-case-1-references}
	Every of the above group is a subgroup of one of the groups we will show to be hyperelliptic in dimension $4$ in \hyperref[section:examples]{Section~\ref*{section:examples}}, so that all of the above groups are hyperelliptic in dimension $4$ as well. See also the beginning of \hyperref[section:running-algo]{Section~\ref*{section:running-algo}} for precise references.
\end{rem}

\section{The case in which $G$ is non-Abelian and $Z(G)$ is non-cyclic} \label{section:2-sylow-case2}

We are left with investigating \ref{2-sylow-poss2}, i.e., we assume for the rest of this section that \vspace{0.5cm}
\begin{center}
	\noindent\fbox{%
		\parbox{0.975\textwidth}{%
			\begin{center}
				The non-Abelian $2$-group $G$ does \textbf{not} have a faithful irreducible representation.
			\end{center}
		}
	} 
\end{center} \vspace{0.5cm}

Equivalently, we may require that $Z(G)$ is non-cyclic (cf.  \hyperref[thm:huppert:p-groups]{Theorem~\ref*{thm:huppert:p-groups}}). Recall from above that the complex representation $\rho$ is the direct sum of two representations $\rho_2$ and $\rho_2'$ of degree $2$, where we assumed $\rho_2$ to be irreducible.

\begin{lemma} \label{rho2'-reducible}
	The representation $\rho_2'$ splits as a direct sum of two linear characters $\chi$ and $\chi'$.
\end{lemma}

\begin{proof}
	This is the content of  \hyperref[lemma:central-non-cyc]{Lemma~\ref*{lemma:central-non-cyc}}.
\end{proof}

\medskip
In the same manner, we prove
\begin{lemma} \label{lemma-properties-kernels}
	Define $K := \ker(\rho_2') = \ker(\chi \oplus \chi')$. Then the following statements hold. 
	\begin{enumerate}[ref=(\theenumi)]
		\item \label{lemma-properties-kernels-1} $\ker(\rho_2) \subset Z(G)$, 
		\item \label{lemma-properties-kernels-2}$K$ intersects $Z(G)$ non-trivially,
		\item \label{lemma-properties-kernels-3} $K \cap \ker(\rho_2) = \{1\}$,
	\end{enumerate}
\end{lemma}

\begin{proof}
	\ref{lemma-properties-kernels-1} follows because $\rho_2'$ is a direct sum of $1$-dimensional representations and $\rho$ is faithful. For part \ref{lemma-properties-kernels-2}, it suffices to note that is $K$ is non-trivial (because $G$ is non-Abelian); since it is normal, it intersects the center of the $2$-group $G$ non-trivially. Part \ref{lemma-properties-kernels-3} is just a reformulation of the assumption that $\rho$ is faithful.
\end{proof}

In the following, we will write the matrices in the image of $\rho$ in a block diagonal form according to the direct sum decomposition $\rho = \rho_2 \oplus \chi \oplus \chi'$.

\begin{cor} \label{non-faithful-ker-rho2-cyclic}
	The kernel of $\rho_2$ is a cyclic group of order $2$ or $4$.
\end{cor}

\begin{proof}
	\hyperref[cor:center-dim4]{Corollary~\ref*{cor:center-dim4}} shows that the exponent of $Z(G)$ divides $4$. Thus the exponent of $\ker(\rho_2)$ divides $4$, too. It remains to prove that $\ker(\rho_2)$ is cyclic.
	Assume that $\ker(\rho_2)$ is non-cyclic. Then, since $\rho$ is faithful, we can find $g,h \in G$ such that
	\begin{align*}
	\rho(g) = \diag(1,\ 1,\ \alpha,\ 1),\qquad \rho(h) = \diag(1,\ 1,\ 1,\ \beta), \qquad \text{where }\quad \alpha,\ \beta \neq 1.
	\end{align*}
	According to \hyperref[lemma-properties-kernels]{Lemma~\ref*{lemma-properties-kernels}} \ref{lemma-properties-kernels-2}, there exists $1\neq k \in Z(G) \cap K$: then $\rho(ghk)$ does not have the eigenvalue $1$, a contradiction.
\end{proof}
\medskip

%\begin{cor} \label{2-group-center-kernel-cyclic}
%	The group $Z(G)/\ker(\rho_2)$ is cyclic.
%\end{cor}
%
%\begin{proof}
%	Since $\rho$ is a direct sum of three irreducible representations,  \hyperref[non-abelian-center]{Lemma~\ref*{non-abelian-center}} implies $Z(G)$ is a subgroup of $C_4 \times C_4$. 
%	Suppose now that $Z(G)/\ker(\rho_2)$ is non-cyclic: we construct an element that is mapped to a matrix without the eigenvalue $1$ by $\rho$.\\
%	Since $Z(G)/\ker(\rho_2)$ is non-cyclic by assumption,  \hyperref[non-faithful-ker-rho2-cyclic]{Corollary~\ref*{non-faithful-ker-rho2-cyclic}} guarantees that $Z(G)$ contains an element $g$ of order $4$ such that $g^2$ generates $\ker(\rho_2)$. In other words, $\rho_2(g) = -I_2$. By faithfulness of $\rho$, the matrix
%	\begin{align*}
%	\rho(g) = \diag(-1, \ -1, \ \chi(g), \ \chi'(g))
%	\end{align*}
%	must have order $4$. Furthermore, it must have the eigenvalue $1$. We conclude that up to symmetry and replacing $g$ by $g^3$, 
%	\begin{align*}
%	\chi(g) = i, \qquad \chi'(g) = 1, \qquad \text{i.e.,} \quad \rho(g) = \diag(-1, \ -1, \ i, \ 1).
%	\end{align*}
%	We now choose a (non-trivial) central element $h \in Z(G)$ such that $\langle g,h \rangle = Z(G)$ and $\langle g \rangle \cap \langle h \rangle = \{1\}$. The kernel $\ker(\rho_2)$ is the cyclic group generated by $g^2$, and so $\langle h \rangle$ intersects $\ker(\rho_2)$ only in the identity.  {\ write proof well/complete proof, it isnt finished yet}
%\end{proof}

The following Lemma will be useful to show that certain non-solvable groups are not hyperelliptic groups in dimension $4$, cf.  \hyperref[section:2a3b5c7]{Section~\ref*{section:2a3b5c7}}.
\begin{lemma} \label{non-faithful-not-in-sl}
	$\rho(G)$ is not contained in $\SL(4,\CC)$.
\end{lemma}

\begin{proof}
	Let $1 \neq g \in \ker(\rho_2)$ and  $1 \neq h \in K \cap Z(G)$. If $\rho(G) \subset \SL(4,\CC)$, then
	\begin{align*}
	\rho(g) = \diag\left(1, \ 1, \ \chi(g), \ \chi(g)^{-1}\right) \qquad \text{ and } \qquad \rho(h) = \diag(-1, \ -1, \ 1, \ 1).
	\end{align*}
	The faithfulness of $\rho$ implies that $\chi(g) \neq 1$ and hence $\rho(gh)$ does not have the eigenvalue $1$, a contradiction.
\end{proof}

% {\ lemma below is out of context and probably not needed.
%\begin{lemma}
%The derived subgroup $[G,G]$ does not intersect the kernel of $\rho_2$. In particular $G/\ker(\rho_2)$ is non-Abelian.
%\end{lemma} 
%
%\begin{proof}
%The kernels of $\chi$ and $\chi'$ contain $[G,G]$, and hence the assertion follows from the faithfulness of $\rho$.
%\end{proof}}

We have prepared the proof of the desired bound $|G| \leq 128$ in Case \ref{2-sylow-poss2}.

\begin{prop} \label{2-group-bound-quotient}
	The order of $G/\ker(\rho_2)$ is at most $32$.
\end{prop}

\begin{proof}
	We first claim that the exponent of $N/\ker(\rho_2)$ is at most $4$. The proof goes by contradiction. Assume that the class of an element $g \in N$ in $N/\ker(\rho_2)$ has order $8$.  \hyperref[order-cyclic-groups]{Lemma~\ref*{order-cyclic-groups}} shows that the order of $g$ is $8$, too. It follows that the matrix $\rho_2(g)$ has order $8$ as well.
	Observe that $g \in N$ implies that the two eigenvalues of $\rho_2(g)$ are complex conjugate $8$th roots of unity. Using \hyperref[lemma:integrality]{Lemma~\ref*{lemma:integrality}} again, we see that in order for $\rho(g)$ to act on a $4$-dimensional complex torus, we must have four eigenvalues of order $8$: this contradicts the necessary property that $\rho(g)$ has the eigenvalue $1$. We have established the claim that the exponent of $N/\ker(\rho_2)$ is at most four. 	Consider now the commutative diagram 
	\begin{center}
		\begin{tikzpicture}
		\matrix (m) [matrix of math nodes,row sep=3em,column sep=4em,minimum width=2em]
		{
			1 & N & G & C_m & 1 \\
			1 & N/\ker(\rho_2) & G/\ker(\rho_2) & C_m & 1\\};
		\path[-stealth]
		(m-1-1) edge node [left] {} (m-1-2)
		(m-1-2) edge node [left] {} (m-1-3)
		(m-1-3) edge node [above] {$g \mapsto \det \rho_2(g)$} (m-1-4)
		(m-1-4) edge node [left] {} (m-1-5)
		(m-2-1) edge node [left] {} (m-2-2)
		(m-2-2) edge node [left] {} (m-2-3)
		(m-2-3) edge node [left] {} (m-2-4)
		(m-2-4) edge node [left] {} (m-2-5)
		
		(m-1-2) edge node [left] {} (m-2-2)
		(m-1-3) edge node [left] {} (m-2-3)
		(m-1-4) edge node [left] {} (m-2-4);
		\end{tikzpicture}
	\end{center}
	Obviously, $\rho_2$ induces a faithful representation $\overline{\rho}_2 \colon G/\ker(\rho_2) \to \GL(2,\CC)$. We are therefore able to use similar arguments as in Case \ref{2-sylow-poss1}: indeed, by faithfulness of $\overline{\rho}_2$ we obtain that $N/\ker(\rho_2)$ contains a unique element of order $2$; it is mapped to $-I_2$ by $\overline{\rho}_2$. Using again the structure theorem \cite[Theorem 12.5.2]{Hall} for $2$-groups containing a unique subgroup of order $2$, we conclude that $N/\ker(\rho_2)$ either cyclic or a generalized quaternion group. Since it is a group of exponent dividing $4$, it must be cyclic of order $2$ or $4$ or the group $Q_8$. Moreover, $m \leq 4$ by  \hyperref[lemma-table]{Lemma~\ref*{lemma-table}}.
\end{proof}
\medskip

Finally, we obtain the desired bound for $|G|$:

\begin{cor} \label{at-most-128} If $G$ falls into Case \ref{2-sylow-poss2}, i.e., $G$ is a non-Abelian $2$-group with non-cyclic center that is hyperelliptic in dimension $4$, then $|G| \leq 128$.
\end{cor}
\begin{proof}
	The group $G$ sits in the exact sequence
	\begin{align*}
	1 \to \ker(\rho_2) \to G \to G/\ker(\rho_2) \to 1.
	\end{align*} 
	According to  \hyperref[non-faithful-ker-rho2-cyclic]{Lemma~\ref*{non-faithful-ker-rho2-cyclic}} and  \hyperref[2-group-bound-quotient]{Proposition~\ref*{2-group-bound-quotient}}, we have 
	\begin{align*}
	|G| = |\ker(\rho_2)| \cdot |G/\ker(\rho_2)| \leq 4 \cdot 32 = 128.
	\end{align*}
\end{proof}

The following lemma allows for a more efficient search.

\begin{lemma} \label{lemma:2-sylow-derived}
	If $G$ falls into Case \ref{2-sylow-poss2} and the intersection $[G,G] \cap Z(G)$ is non-cyclic, then $G$ is not hyperelliptic in dimension $4$.
\end{lemma}

\begin{proof}
	If $G$ is hyperelliptic in dimension $4$, then we have seen that the associated complex representation decomposes as the direct sum of an irreducible degree $2$ summand $\rho_2$ and two linear characters $\chi$, $\chi'$. Since the restriction of a linear character to $[G,G]$ is trivial, $\rho_2|_{[G,G]}$ must be faithful. This is however impossible if $[G,G] \cap Z(G)$ is non-cyclic.
\end{proof}

We use  \hyperref[script:2-sylow-non-cyc-center]{GAP Script~\ref*{script:2-sylow-non-cyc-center}} to search all non-Abelian $2$-groups $G$ of order dividing $128$ satisfying the following properties:
\begin{enumerate}
	\item $G$ does not contain $C_2^4$ (ID $[16,14]$) (\hyperref[lemma:many-factors]{Lemma~\ref*{lemma:many-factors}} \ref{lemma:many-factors1}),
	\item no element of order $8$ is conjugate to its inverse (\hyperref[prop:conjugate]{Proposition~\ref*{prop:conjugate}}),
	\item its center $Z(G)$ satisfies the following properties:
	\begin{enumerate}
		\item it is non-cyclic;
		\item it does not contain a subgroup isomorphic to $C_2^3$ (\hyperref[lemma:central-non-cyc]{Lemma~\ref*{lemma:central-non-cyc}});
		\item it  does not contain an element of order $8$ (\hyperref[lemma:central-order-5-8]{Lemma~\ref*{lemma:central-order-5-8}}).
	\end{enumerate}
	\item the group $G$ contains a normal subgroup $N$ satisfying the following properties, which are extracted out of the proof of  \hyperref[2-group-bound-quotient]{Proposition~\ref*{2-group-bound-quotient}}:
	\begin{enumerate}
		\item the exponent of $N$ divides $4$,
		\item the quotient $G/N$ is cyclic of order $\leq 4$,
		\item there exists a normal subgroup $K$ of $N$ ``playing the role of $\ker(\rho_2)$'', i.e.,
		\begin{itemize}
			\item $K$ is cyclic of order $2$ or $4$ (\hyperref[non-faithful-ker-rho2-cyclic]{Corollary~\ref*{non-faithful-ker-rho2-cyclic}}),
			\item $K$ is also normal in $G$,
			\item $N/K$ is cyclic of order $\leq 4$ or isomorphic to $Q_8$.
		\end{itemize}
	\end{enumerate}
	\item the intersection $[G,G] \cap Z(G)$ is non-cyclic (\hyperref[lemma:2-sylow-derived]{Lemma~\ref*{lemma:2-sylow-derived}}).
\end{enumerate}

The output of the script is too long to list, but if we require in addition that $G$ does not contain any of the groups with IDs $[16,12]$, $[32,5]$, $[32,9]$, $[32,12]$, $[32,13]$, $[32,25]$, $[64,20]$, $[64,85]$ (all of which will be excluded later, see the  \hyperref[section:running-algo]{Section~\ref*{section:running-algo}} for precise references), then the output will consist only of the following six groups:
\begin{align*}
&(C_2 \times C_4) \rtimes C_2 = \langle a,b,c ~ | ~ a^2 = b^4 = c^2 = 1, ~ ab = ba, ~ ac = ca, ~ c^{-1}bc = ab \rangle  \quad (\text{ID } [16,3]\text{)}, \\
&G(4,4,3) = \langle g,h ~ | ~ g^4 = h^4 = 1, ~ h^{-1}gh = g^3\rangle \quad (\text{ID } [16,4]\text{)}, \\
&D_4 \times C_2 = \langle r,s,k ~ | ~ r^4 = s^2 = (rs)^2 = k^2 = 1, ~ k \text{ central}\rangle \quad (\text{ID } [16,11]\text{)}, \\
&G(8,4,5) = \langle g,h ~ | ~ g^8 = h^4 = 1, ~ h^{-1}gh = g^5\rangle  \quad (\text{ID } [32,4]\text{)}    \\
&C_4^2 \rtimes C_2 = \langle a,b,c ~ | ~ a^4 = b^4 = c^2 = 1, ~ ab =ba, ~ ac = ca, ~  c^{-1}bc = a^2b \rangle \quad (\text{ID } [32,24]\text{)}   \\
&G(8,2,5) \times C_2 = \langle g,h,k ~ | ~ g^8 = h^2 = k^2 = 1, ~ h^{-1}gh = g^5, ~ k \text{ central}\rangle  \quad (\text{ID } [32,37]\text{)}
\end{align*}
We will prove in the course of the book that the six groups above are hyperelliptic in dimension $4$. Again, we refer to the  \hyperref[section:running-algo]{Section~\ref*{section:running-algo}} for precise references.

\section{Summary} \label{section:2groups-summary}

Let $G$ be a hyperelliptic $2$-group in dimension $4$. If $G$ is Abelian, then  \hyperref[abelian-2-grp]{Corollary~\ref*{abelian-2-grp}} asserts that $G$ is one of the groups
\begin{center}
	$C_2, \qquad C_2 \times C_2, \qquad C_4, \qquad C_2 \times C_2 \times C_2, \qquad C_2 \times C_4, \qquad C_8,$ \\
	$C_2 \times C_2 \times C_4, \qquad C_4 \times C_4, \qquad C_2 \times C_8, \qquad C_2 \times C_4 \times C_4, \qquad C_4 \times C_8$,
\end{center}
of which we already know that they are hyperelliptic in dimension $4$, see \hyperref[thm:abelian-classification]{Theorem~\ref*{thm:abelian-classification}}.\\

We have seen in  \hyperref[section:2-sylow-case1]{Section~\ref*{section:2-sylow-case1}} that if $G$ is non-Abelian and has a faithful irreducible representation, then $G$ belongs to the following list:
\begin{align*}
&D_4 = \langle r,s \ | \ r^4 = s^2 = 1, \ (rs)^2 = 1 \rangle \quad \text{(ID } [8,3] \text{)}, \\
&Q_8 = \langle a,b \ | \ a^4 = 1, \ a^2 = b^2, \ ab = b^{-1}a\rangle \quad \text{(ID } [8,4]\text{)}, \\
&G(8,2,5) = \langle g,h \ |\ g^8 = h^2 = 1, \ h^{-1}gh = g^5 \rangle \quad \text{(ID } [16,6]\text{)}, \\
&G(8,2,3) = \langle g,h \ |\ g^8 = h^2 = 1, \ h^{-1}gh = g^3 \rangle \quad \text{(ID } [16,8]\text{)}, \\
&D_4 \curlyvee C_4 = \langle r,s,k \ | \ r^4 = s^2 = (rs)^2 = 1, \ k \text{ central}, \ k^2 = r^2 \rangle \quad \text{(ID } [16,13]\text{)}, \\
&(C_4 \times C_4) \rtimes C_2 = \langle a,b,c \ | \ a^4 = b^4 = c^2 = 1, \ ab = ba, \ c^{-1}ac = a^{-1}, \ c^{-1}bc = ab \rangle  \quad \text{(ID } [32,11]\text{)}.
\end{align*}
As already mentioned in  \hyperref[rem:2-sylow-case-1-references]{Remark~\ref*{rem:2-sylow-case-1-references}}, the above groups are subgroups of the groups we deal with in  \hyperref[section:examples]{Section~\ref*{section:examples}}, all of which are hyperelliptic in dimension $4$. \\

Finally, we will see later that if $G$ is a  non-Abelian $2$-group with non-cyclic center that is hyperelliptic in dimension $4$, then $G$ is one of the following groups:
\begin{align*}
&(C_2 \times C_4) \rtimes C_2 = \langle a,b,c ~ | ~ a^2 = b^4 = c^2 = 1, ~ ab = ba, ~ ac = ca, ~ c^{-1}bc = ab \rangle  \quad (\text{ID } [16,3]\text{)}, \\
&G(4,4,3) = \langle g,h ~ | ~ g^4 = h^4 = 1, ~ h^{-1}gh = g^3\rangle \quad (\text{ID } [16,4]\text{)}, \\
&D_4 \times C_2 = \langle r,s,k ~ | ~ r^4 = s^2 = (rs)^2 = k^2 = 1, ~ k \text{ central}\rangle \quad (\text{ID } [16,11]\text{)}, \\
&G(8,4,5) = \langle g,h ~ | ~ g^8 = h^4 = 1, ~ h^{-1}gh = g^5\rangle  \quad (\text{ID } [32,4]\text{)}    \\
&C_4^2 \rtimes C_2 = \langle a,b,c ~ | ~ a^4 = b^4 = c^2 = 1, ~ ab =ba, ~ ac = ca, ~ \qquad c^{-1}bc = a^2b \rangle \quad (\text{ID } [32,24]\text{)}   \\
&G(8,2,5) \times C_2 = \langle g,h,k ~ | ~ g^8 = h^2 = k^2 = 1, ~ h^{-1}gh = g^5, ~ k \text{ central}\rangle  \quad (\text{ID } [32,37]\text{)}
\end{align*}
The discussion at the end of the previous section implies that in order to achieve the classification in Case \ref{2-sylow-poss2}, we are left with showing that the above groups occur and that the groups with IDs $[16,12]$, $[32,5]$, $[32,9]$, $[32,12]$, $[32,13]$, $[32,25]$, $[64,20]$, $[64,85]$ are not hyperelliptic in dimension $4$.

%\subsection{The candidates for hyperelliptic $2$-groups in dimension $4$}
%
%{\text}
%
%
%
%We will see later in Theorem \ref{2-sylow} that the order of the $2$-Sylow subgroups of $G$ is indeed bounded from above by $32$.

\chapter{The $3$-Sylow Subgroups of Hyperelliptic Groups in Dimension $4$} \label{section-3-sylow}

We will now turn our attention to $3$-groups that are hyperelliptic in dimension $4$. We will show that the order of such a group is bounded by $27$ and that equality is achieved precisely by the groups $C_3^3$ and the Heisenberg group of order $27$, see \hyperref[3-grps]{Proposition~\ref*{3-grps}}. \\

Throughout the chapter, we will assume that

\noindent\fbox{%
	\parbox{0.975\textwidth}{\begin{center}
			The group $G$ is a hyperelliptic $3$-group in dimension $4$. By $\rho \colon G \to \GL(4,\CC)$, we denote the complex representation of a hyperelliptic fourfold with holonomy group $G$.
		\end{center}
	}
}

\section{General results}

Throughout the chapter, let $G$ be a hyperelliptic $3$-group in dimension $4$. We first deal with the case in which $G$ is Abelian:

\begin{lemma} \label{abelian-case-3-sylow}
	If $G$ is an Abelian hyperelliptic $3$-group in dimension $4$, it is isomorphic to one of $$ C_3,\qquad  C_3 \times C_3,\qquad C_9,\qquad C_3 \times C_3 \times C_3.$$
\end{lemma}

\begin{proof}
	By  \hyperref[lemma:many-factors]{Lemma~\ref*{lemma:many-factors}} \ref{lemma:many-factors1}, $G$ does not have a subgroup isomorphic to $C_3^4$. Moreover,  \hyperref[order-cyclic-groups]{Lemma~\ref*{order-cyclic-groups}} shows that the exponent of $G$ is $3$ or $9$. The lemma follows if we prove that $C_9 \times C_3$ is not hyperelliptic in dimension $4$. Assume that $\rho \colon C_3 \times C_9 \to \GL(4,\CC)$ is a diagonal embedding of $S$. We prove that there is an element of $g \in C_3 \times C_9$ such that $\rho(g)$ has conjugate eigenvalues of order $9$: this will violate the conclusion of the \hyperref[order-cyclic-groups]{Integrality Lemma~\ref*{order-cyclic-groups}} \ref{ocg-3} and hence show that $C_3 \times C_9$ is not hyperelliptic in dimension $4$.
	
	According to  \hyperref[lemma-table]{Lemma~\ref*{lemma-table}}, we can assume that $C_9$ is generated by
	\begin{align*}
	\text{(I) } g_1 = \diag(1,~\zeta_9,~ \zeta_9^2,~ \zeta_9^4)\quad \text{ or \quad (II) } g_1 = \diag(1,~\zeta_9,~ \zeta_9^4,~ \zeta_9^7).
	\end{align*}
	Furthermore, suppose that $g_1$ and
	\begin{align*}
	g_2 = \diag(\zeta_9^a,~\zeta_9^b,~\zeta_9^c,~\zeta_9^d), \qquad a,b,c,d \in \{0,3,6\}
	\end{align*}
	span a subgroup isomorphic to $S = C_9 \times C_3$.  In both cases, (I) and (II), the condition that $g_1g_2$ must have the eigenvalue $1$ implies that $a = 0$. From now on, we only consider case (I), case (II) is dealt with similarly. \\
	
	We prove the following \\
	
	\underline{Claim:} None of $b$, $c$ and $d$ are zero. 
	
	\underline{Proof of the Claim:} Consider the element
	\begin{align*}
	g_1 g_2 = \diag(1, ~ \zeta_9^{b+1}, ~ \zeta_9^{c+2}, ~ \zeta_9^{d+4}).
	\end{align*}
	The conditions $\zeta_9^{b+1} \neq \overline{\zeta_9^{c+2}}$ and $\zeta_9^{b+1} \neq \zeta_9^{b+4}$ are necessary for $g_1g_2$ to have pairwise different and non-conjugate eigenvalues of order $9$. They translate into
	\begin{align*}
	&b+1 \not\equiv -(c+2) \pmod 9, \qquad  &\text{i.e., } b+c \not \equiv 6 \pmod 9, \\
	&b+1 \not\equiv d+4 \pmod 9, \qquad &\text{i.e., } b-d \not\equiv 3 \pmod 9.
	\end{align*}
	Similarly, if the element
	\begin{align*}
	g_1^2 g_2 = \diag(1, ~ \zeta_9^{b+2}, ~ \zeta_9^{c+4}, ~ \zeta_9^{d+8})
	\end{align*}
	has three different, pairwise non-complex conjugate eigenvalues of order $9$, then the following conditions are in particular satisfied:
	\begin{align*}
	&b+2 \not\equiv -(c+4) \pmod 9, \qquad  &\text{i.e., } b+c \not \equiv 3 \pmod 9, \\
	&b+2 \not\equiv d+8 \pmod 9, \qquad &\text{i.e., } b-d \not\equiv 6 \pmod 9.
	\end{align*}
	In total, the four conditions above imply -- together with the assumption $b,c,d \in \{0,3,6\}$ -- that
	\begin{align*}
	b \equiv -c \pmod 9, \qquad b \equiv d \pmod 9.
	\end{align*}
	This shows the claim. \\
	
	%	If $c \neq 0$, after possibly replacing $g_2$ by $g_2^2$, we can assume that $c = 3$. Since the element
	%	\begin{align*}
	%	g_1 g_2 = \diag(1, \zeta_9^{1+b}, \zeta_9^{5}, \zeta_9^4)
	%	\end{align*}
	%	has conjugate eigenvalues, which is not possible according to  \hyperref[lemma-phi-d-2-non-conj-ev]{Lemma~\ref*{lemma-phi-d-2-non-conj-ev}} and Remark \ref{rem-below-lemma}. Thus $d \neq 0$. The other direction is proved similarly. \\
	%	(b) Again, if $b \neq 0$, we can assume that $b = 3$. If furthermore $d = 0$, considering
	%	\begin{align*}
	%	g_1 g_2 = \diag(1, \zeta_9^{4}, \zeta_9^{2+c}, \zeta_9^4),
	%	\end{align*}
	%	we arrive at the contradiction that $g_1g_2$ has an eigenvalue with multiplicity at least $2$. Again, the converse direction is proved similarly. This proves the Claim. \\
	
	Now, by assumption, $g_2$ is not contained in the group $C_9 = \langle g_1 \rangle$ and $g_1^3 = \diag(1,~\zeta_9^3,~\zeta_9^6,~\zeta_9^3)$, we obtain that $(b,c,d)$ is not a multiple of $(3,6,3)$. The remaining possibilities are (up to taking inverses):
	\begin{align*}
	&(b,c,d) = (3,3,3) \implies 	g_1g_2 = \diag(1, ~ \zeta_9^{4},~ \zeta_9^{5}, \zeta_9^{7}), \\
	&(b,c,d) = (3,3,6) \implies 	g_1g_2 = \diag(1,~ \zeta_9^{4},~ \zeta_9^{5}, ~ \zeta_9), \\
	&(b,c,d) = (3,6,6) \implies 	g_1g_2 = \diag(1,~ \zeta_9^{4},~ \zeta_9^{8},~ \zeta_9).
	\end{align*}
	Thus the element $g_1g_2$ has conjugate eigenvalues of order $9$ in all three cases, which shows that case (I) does not occur.
\end{proof}

Having dealt with the Abelian case, we shall assume in the following that $G$ is a non-Abelian $3$-group that is hyperelliptic in dimension $4$. Since $\rho$ is faithful, $\rho$ splits into a direct sum of a $3$-dimensional irreducible representation $\rho_3$ and a $1$-dimensional representation $\chi$.	
\begin{prop} \label{3-group-irr-faithful-rep-of-dim-3}
	The representation $\rho_3$ is faithful.
\end{prop}

\begin{proof}
	It follows from \hyperref[lemma:central-non-cyc]{Lemma~\ref*{lemma:central-non-cyc}} that $Z(G)$ is cyclic. Since $G$ is non-Abelian, $\ker(\rho_1)$ is non-trivial and therefore must intersect the cyclic group $Z(G)$ non-trivially (see the footnote on p. \pageref{p-grp-cent}). Since $\rho$ is faithful, $\ker(\rho_3) \cap \ker(\rho_1) = \{1\}$, and hence we may conclude that $\ker(\rho_3)$ is trivial. 
\end{proof}

Consider again the determinant exact sequence

\begin{align} \label{ex-seq-3-groups}
1 \to K \to G \to C_m \to 1, \qquad  m \in \{1,3,9\},
\end{align}
where the map $G \to C_m$ is given by $g \mapsto \det \rho_3(g)$. Since $\rho_3$ is faithful, the kernel $K$ is a subgroup of $\SL(3,\CC)$. We analyze it further -- 
the upcoming lemma is a direct consequence of \hyperref[lemma-table]{Lemma~\ref*{lemma-table}}.

%\begin{lemma}
%The group $C_3 \times C_3 \times C_3$ does not embed into $\SL(3,\CC)$.
%\end{lemma}
%
%\begin{proof}
%There is only one embedding of $C_3^3$ into $\GL(3,\CC)$ up to conjugation: it is given by
%\begin{align*}
%	(a,b,c) \mapsto \diag(\zeta_3^a, \ \zeta_3^b, \zeta_3^c).
%\end{align*}
%Since $(1,0,0)$ maps to $\diag(\zeta_3, \ 1, \ 1)$, this is not an embedding into $\SL(3,\CC)$. (This is essentially the same argument as in \hyperref[no-subgroup-of-high-rank]{Lemma~\ref*{no-subgroup-of-high-rank}}.)
%\end{proof}

\begin{lemma} \label{kernel-det-3-group}
	The kernel $K$ does not contain a subgroup that is isomorphic to $C_9$.
\end{lemma}

Using the classification of finite subgroups of $\SL(3,\CC)$, first achieved by Miller, Blichfeldt, and Dickson \cite[Chapter XII]{MBD} (which was later found to be incomplete by Stephen S.-T. Yau and Y. Yu \cite[p.2 f.]{YY} and completed in the cited book), we find that $K$ is conjugate to one of the following groups:

\begin{itemize} \label{3-group-inside-sl3c}
	\item[(a)] Diagonal Abelian groups: $\{1\}$, $C_3$ or $C_3 \times C_3$.
	\item[(b)] A group generated by a non-trivial group of type (a) and $\begin{pmatrix}0 & 1 & 0 \\ 0 & 0 & 1 \\ 1 & 0 & 0\end{pmatrix}$. 
\end{itemize}

In particular, if $K$ is of type (b), it is isomorphic to the Heisenberg group of order $27$,
\begin{align*}
K \cong \Heis(3) := \langle g, h, k \, | \, g^3 = h^3 = k^3 = 1, \; ghg^{-1}h^{-1} = k,\; gk = kg,\; hk = kh\rangle.
\end{align*}

This proves in particular that $|G| \leq 81$. In the following sections, we show that $|G| \leq 27$ and that if $G$ is non-Abelian, then $G \cong \Heis(3)$.

\section{The case where $G$ contains elements of order $9$}
Our next goal is to prove

\begin{theorem} \label{3-group-of-exp-9-is-cyclic}
	If $G$ is a hyperelliptic $3$-group in dimension $4$ containing an element of order $9$, then $G \cong C_9$.
\end{theorem}

\hyperref[abelian-case-3-sylow]{Lemma~\ref*{abelian-case-3-sylow}} shows the statement of  \hyperref[3-group-of-exp-9-is-cyclic]{Theorem~\ref*{3-group-of-exp-9-is-cyclic}} is true if $G$ is Abelian. We will henceforth assume in the following that $G$ is a non-Abelian $3$-group of exponent $9$. \\

Consider the unique group of order $27$ and exponent $9$, the metacyclic group $G(9,3,4)$:
$$G(9,3,4) := \langle g,h \, | \, g^9 = h^3 = 1,\ h^{-1}gh = g^4 \rangle.$$

\begin{rem} \label{rem-g934}
	The multiplicative inverse of $4$ in $C_9^*$ is $7$: hence the map $g \mapsto g$, $h \mapsto h^{-1}$ defines an isomorphism $G(9,3,4) \to G(9,3,7)$.
\end{rem}

The following result shows that it suffices to exclude $G(9,3,4)$ in order to prove  \hyperref[3-group-of-exp-9-is-cyclic]{Theorem~\ref*{3-group-of-exp-9-is-cyclic}}.

\begin{prop} \label{3-group-non-ab-exp-9-m27}
	If $G$ is a non-Abelian $3$-group of exponent $9$ that is hyperelliptic in dimension $4$, then $G$ contains a subgroup that is isomorphic to $G(9,3,4)$.
\end{prop}

We shall need the following well-known group theoretic lemma.

\begin{lemma} \label{group-lemma}
	Let $G$ be a finite $p$-group. Suppose that $H$ is a proper subgroup of $G$ with normalizer $N_G(H)$. Then $H$ is a proper subgroup of $N_G(H)$.
\end{lemma}

\begin{proof}
	If $H$ is normal in $G$, then $N_G(H) = G$, and there is nothing to show. We will henceforth assume that $H$ is not normal in $G$. Let $|G| = p^n$. We argue by induction on $n$. For $n = 0$, there is nothing to show. Let therefore $n > 0$ and assume for a contradiction that $H = N_G(H)$. Since the center $Z$ of $G$ is contained in $N_G(H)$, it is contained in $H$ as well. The assumption that $G$ is a $p$-group implies that $Z$ is non-trivial and hence $|G/Z| < |G|$. By induction, $H/Z$ is a proper subgroup of $N_{G/Z}(H/Z)$, and hence there is a class $uZ \in N_{G/Z}(H/Z)$ such that $u \notin H$. By definition, if $h \in H$, there is $h' \in H$ such that $u^{-1}hg \in h'Z \subset H$. We thus obtain that $u \in N_G(H)$, a contradiction to our assumption that $H = N_G(H)$.
\end{proof}

\begin{proof}[Proof of Proposition \ref{3-group-non-ab-exp-9-m27}] Let $g \in G$ be an element of order $9$. We prove the statement in two steps. \\
	
	\underline{Step 1:} We assume first that $\langle g \rangle$ is normal in $G$. \\
	Since $G$ is non-Abelian, the faithful representation $\rho$ contains an irreducible degree $3$ summand $\rho_3$. By possibly replacing $g$ by some appropriate power and choosing an appropriate basis,
	\begin{align*}
	\rho_3(g) = \begin{pmatrix}
	\zeta_9 && \\ & \zeta_9^n & \\ && \zeta_9^m
	\end{pmatrix},
	\end{align*}
	where $(n,m) \in \{(2,4), (4,7)\}$, cf.  \hyperref[lemma-table]{Lemma~\ref*{lemma-table}}. Note that this implies that $g \notin Z(G)$ (this would also follow from \hyperref[cor:center-dim4]{Corollary~\ref*{cor:center-dim4}}).
	Thus, by \hyperref[abelian-case-3-sylow]{Lemma~\ref*{abelian-case-3-sylow}}, we can find $h \in G$ of order $3$ such that $h$ and $g$ do not commute. Since $\langle g \rangle$ is a normal subgroup of $G$, there is $1 < k < 9$ such that $h^{-1}gh = g^k$. Therefore, $\det(\rho_3(g)) = \det(\rho_3(g))^k$, and because $1 < k < 9$, we obtain that $\det(\rho_3(g))$ is a third root of unity and thus $k \in \{4,7\}$. Thus $\langle g,h \rangle \cong G(9,3,k) \cong G(9,3,4)$ by  \hyperref[rem-g934]{Remark~\ref*{rem-g934}}. \\
	
	\underline{Step 2:} Assume now that $\langle g \rangle$ is not normal in $G$. \\
	Let $g \in G$ be an element of order $9$. Since $\langle g \rangle$ is not a normal subgroup of $G$, the normalizer $N := N_G(\langle g \rangle)$ is a proper subgroup of $G$ containing $g$. By  \hyperref[group-lemma]{Lemma~\ref*{group-lemma}}, $\langle g \rangle$ is a proper subgroup of $N$. The normalizer $N$ is then non-Abelian by  \hyperref[abelian-case-3-sylow]{Lemma~\ref*{abelian-case-3-sylow}} and contains the cyclic normal subgroup $\langle  g \rangle$ of order $9$. By Step 1, $N$ contains a subgroup that is isomorphic to $G(9,3,4)$.
\end{proof}

The following Proposition is last step in the proof of  \hyperref[3-group-of-exp-9-is-cyclic]{Theorem~\ref*{3-group-of-exp-9-is-cyclic}}.
\begin{prop} \label{no-m27}
	The group $G(9,3,4) = \langle g,h \ | \ g^9 = h^3 = 1, ~ h^{-1}gh = g^4\rangle$ is not hyperelliptic in dimension $4$.
\end{prop}

\begin{proof}
	Assume that $G(9,3,4)$ is hyperelliptic in dimension $4$. Then an associated complex representation $\rho$ contains an irreducible degree $3$ summand $\rho_3$ and a degree $1$ summand $\chi$. According to  \hyperref[3-group-irr-faithful-rep-of-dim-3]{Lemma~\ref*{3-group-irr-faithful-rep-of-dim-3}}, $\rho_3$ is faithful, hence the matrix $\rho_3(g)$ has three eigenvalues of order $9$. Since $\ord(gh) = 9$, the matrix $\rho_3(gh)$ has three eigenvalues of order $9$ as well. However, both of the matrices
	\begin{align*}
	\rho(g) = \diag(\rho_3(g), \ \chi(g)), \qquad \text{and} \qquad \rho(gh) = \diag(\rho_3(gh), \ \chi(gh))
	\end{align*}
	must have the eigenvalue $1$: thus $\chi(g) = \chi(h) = 1$, a contradiction to \hyperref[lemma-two-generators]{Proposition~\ref*{lemma-two-generators}}, which asserts that our assumption that $G(9,3,4)$ is hyperelliptic in dimension $4$ implies that $\chi$ is non-trivial.
\end{proof}

\section{Non-existence of hyperelliptic groups of order $3^b$, $b \geq 4$} \label{section:3^b-b>=4} \label{3b-b-at-least-4}
Finally, we prove that if $G$ is a hyperelliptic $3$-group in dimension $4$, then $|G| \leq 27$ and that if $G$ is non-Abelian, then $G \cong \Heis(3)$. \\

Assume that there exists a hyperelliptic $3$-group $G$ in dimension $4$ such that $|G| \geq 81$. First of all, according to  \hyperref[abelian-case-3-sylow]{Lemma~\ref*{abelian-case-3-sylow}}, $G$ is non-Abelian. In view of exact sequence (\ref{ex-seq-3-groups}),  \hyperref[kernel-det-3-group]{Lemma~\ref*{kernel-det-3-group}} and  \hyperref[3-group-of-exp-9-is-cyclic]{Theorem~\ref*{3-group-of-exp-9-is-cyclic}} only the possibility $K = \Heis(3)$, $m = 3$ is left to exclude.

We use  \hyperref[gap-3groups]{GAP Script~\ref*{gap-3groups}} to find all $3$-groups of exponent $3$ with cyclic center (this guarantees that $G$ has a faithful irreducible representation, which is necessary according to \hyperref[3-group-irr-faithful-rep-of-dim-3]{Proposition~\ref*{3-group-irr-faithful-rep-of-dim-3}}) and contain $\Heis(3)$ as a normal subgroup. The script tells us that there are no such groups.

\section{Summary} \label{sect-3grp-summary}

In the previous sections, we established the following
\bigskip
\begin{prop}\label{3-grps}
	Let $G$ be a hyperelliptic $3$-group in dimension $4$. Then $G$ is isomorphic to one of the following groups:
	\begin{align*}
	C_3, \qquad C_3 \times C_3, \qquad C_9, \qquad  C_3 \times C_3 \times C_3, \qquad \Heis(3).
	\end{align*}
\end{prop}

The existence of hyperelliptic fourfolds with holonomy group $C_3 \times C_3 \times C_3$ was shown in  \hyperref[abelian:ex]{Example~\ref*{abelian:ex}}, while the existence of a hyperelliptic fourfold with holonomy group $C_9$ was given in  \hyperref[examples-cyclic]{Lemma~\ref*{examples-cyclic}}. We will later show the existence of a hyperelliptic fourfold with holonomy group $\Heis(3)$ in  \hyperref[27-3-section]{Section~\ref*{27-3-section}}.

\chapter{An Alternative Way of Determining the $2$- and $3$-Sylow Subgroups} \label{chapter:sylow-alternative}

In this short chapter, we will sketch a different method (communicated to the author by Christian Gleissner) to classify the hyperelliptic $2$- and $3$-groups in dimension $4$. This method has the advantage of being less theoretically involved. \\

Consider the following three properties of a finite $2$- or $3$-group $G$:
\begin{enumerate}[label=(\theenumi), ref=(\theenumi)]
	\item \label{chartable-1} $G$ admits a faithful representation $\rho \colon G \to \GL(4,\CC)$,
	\item  \label{chartable-2}the matrix $\rho(g)$ has the eigenvalue $1$ for any $g \in G$,
	\item  \label{chartable-3}the characteristic polynomial of $\rho(g) \oplus \overline{\rho(g)}$ is integral for every $g \in G$.
\end{enumerate}
Observe that these three properties can be verified by a computer algebra system using only the character table of $G$ as input. Indeed, \ref{chartable-1} can be verified using the character table of $G$, since if $\chi$ is the character of a representation $\rho$, then
\begin{align*}
\ker(\rho) = \{g \in G \ | \ \chi(g) = \chi(1)\}.
\end{align*}
Furthermore, Newton's identities show that the characteristic polynomial of $\rho(g)$ can be determined from the character values $\chi(g^j)$ so that \ref{chartable-2} and \ref{chartable-3} can be verified using the character table of $G$. \\

The algorithm to determine the possibilities for $G$ is now as follows: \\

\underline{Initial Step:} Initialize with the groups $C_2$ and $C_3$: write them into an output file. \\

\underline{$(n+1)$th Step ($n \geq 1$):}  We use a computer algebra system to search the groups $G$ of order $2^{n+1}$ and $3^{n+1}$, respectively. For each such group $G$, we run through all proper, non-trivial subgroups $U$ of $G$, which we view as abstract groups. If some subgroup $U$ is not contained in our output file, discard $G$. Otherwise, we check if $G$ satisfies properties \ref{chartable-1} -- \ref{chartable-3}. If this is the case, we write this group into our output file. Otherwise, discard it. \\

Note that if for some $n$ there is no new group, the algorithm terminates: indeed, Sylow's Theorem states that a group of order $2^{n+1}$ (resp. $3^{n+1}$) has a subgroup of order $2^n$ (resp. $3^n$). Running the algorithm then shows that it terminates: it does not find new groups of order $2^8$ (resp. $3^5$). In other words:
\begin{itemize}
	\item If $G$ is a hyperelliptic $2$-group in dimension $4$, then $|G| \leq 128$. This is the same bound obtained in \hyperref[section-2-sylow]{Chapter~\ref*{section-2-sylow}} by more theoretical arguments.
	\item If $G$ is a hyperelliptic $3$-group in dimension $4$, then $|G| \leq 81$. Here, we obtain a worse result than in \hyperref[section-3-sylow]{Chapter~\ref*{section-3-sylow}}, where we showed that $|G| \leq 27$. The main difference why the possibility $|G| = 81$ was not excluded with the current computer-algebraic method is because it only takes the representation theory of $G$ into account.
\end{itemize}

\chapter{Hyperelliptic Groups in Dimension $4$ whose Order is $2^a \cdot 3^b$} \label{chapter:2a3b}

Here, we classify the groups of order $2^a \cdot 3^b$, which are hyperelliptic in dimension $4$. We first give plenty of examples and non-examples in \hyperref[section:examples]{Section~\ref*{section:examples}} and \hyperref[non-examples]{Section~\ref*{non-examples}}. The actual classification is then carried out in \hyperref[section:running-algo]{Section~\ref*{section:running-algo}}.

\section{Examples} \label{section:examples}
This chapter will be used to prove that the following groups are hyperelliptic in dimension $4$. We list the group together with its presentation, ID in the Database of Small Groups, and the section in which it is discussed.
\begin{center}
	
	\begin{table}
		\resizebox{\textwidth}{!}{
			{\def\arraystretch{1.8}
				\begin{tabular}{l|l|l}
					Group Name and Presentation & ID & Section \\ \hline \hline 
					$G(8,2,3) = \langle a,b \ | \ a^8 = b^2 = 1, \ b^{-1}ab = a^3\rangle$ & [16,8] & \hyperref[16-8-section]{Section~\ref*{16-8-section}} \\ \hline
					$D_4 \times C_2 = \langle r,s,k \ | \ r^4=s^2 = k^2=  1, ~ s^{-1}rs = r^3, ~ k \text{ central}\rangle$ & $[16,11]$ & \hyperref[16-11-section]{Section~\ref*{16-11-section}} \\ \hline 
					$Q_8 \times C_3 = \langle a,b,k \ | \ a^4 = k^3 = 1, \ a^2 = b^2, \ ab = b^{-1}a, \ k \text{ central} \rangle$ & $[24,11]$ & \hyperref[24-11-section]{Section~\ref*{24-11-section}} \\ \hline 
					$\Heis(3) = \langle g,h,k \ | \ g^3 = h^3 = k^3 = 1, \ [g,k] = [h,k] = 1, \ [g,h] = k \rangle$ & $[27,3]$ & \hyperref[27-3-section]{Section~\ref*{27-3-section}} \\ \hline 
					$G(8,4,5) = \langle g, h \ | \ g^8 = h^4 = 1, \ h^{-1}gh = g^5\rangle$ & $[32,4]$ & \hyperref[32-4-section]{Section~\ref*{32-4-section}}\\ \hline 
					$(C_4 \times C_4) \rtimes C_2 = \langle g_1,g_2,h \ | \ g_1^4 = g_2^4 = h^2 = [g_1,g_2] = 1, \ h^{-1}g_1h = g_2\rangle$ & $[32,11]$ & \hyperref[32-11-section]{Section~\ref*{32-11-section}} \\ \hline 
					$(C_4 \times C_4) \rtimes C_2 = \langle  a,b,c \ | \ a^4 = b^4 = c^2 = [a,b] = [a,c] = 1, \ c^{-1}bc = a^2b  \rangle$ & $[32,24]$ & \hyperref[32-24-section]{Section~\ref*{32-24-section}} \\ \hline 
					$G(8,2,5) \times C_2 = \langle g,h,k \ | \ g^8 = h^2 = k^2 = 1, \ h^{-1}gh = g^5, ~ k \text{ central} \rangle$ & $[32,37]$ & \hyperref[32-37-section]{Section~\ref*{32-37-section}} \\ \hline 
					$((C_2 \times C_4) \rtimes C_2) \times C_3 = 
					\Biggr\langle
					\begin{array}{l|cl}
					a,b, & a^4 = b^2 = c^2 = k^3 = 1, \ c^{-1}ac = ab, \\ 
					c,k & [a,b] = [b,c] = 1, ~ k \text{ central}
					\end{array} \Biggr\rangle$ & $[48,21]$ & \hyperref[48-21-section]{Section~\ref*{48-21-section}} \\ \hline 
					$G(4,4,3) \times C_3 = \langle g,h,k \ | \ g^4 = h^4 = k^3  = 1, \ h^{-1}gh = g^3, \ k \text{ central} \rangle$ & $[48,22]$ & \hyperref[48-22-section]{Section~\ref*{48-22-section}} \\ \hline 
					$A_4 \times C_4 = \langle \sigma, \xi, \kappa \ | \ \sigma^3 = \xi^2 = \kappa^4 = (\xi \sigma)^3 = 1, ~ \kappa \text{ central}\rangle$ & $[48,31]$ & \hyperref[48-31-section]{Section~\ref*{48-31-section}} \\ \hline 
					$G(3,8,2) \times C_3 = \langle g,h,k \ | \ g^3 = h^8 = k^3  = 1, \ h^{-1}gh = g^2, ~ k \text{ central}\rangle$ & $[72,12]$ & \hyperref[72-12-section]{Section~\ref*{72-12-section}} \\ \hline 
					$S_3 \times C_{12} = \langle \sigma, \tau, \kappa \ | \ \sigma^3 = \tau^2 = \kappa^{12} = 1, \ \tau^{-1}\sigma\tau = \sigma^2, \ \kappa \text{ central} \rangle$ & $[72,27]$ & \hyperref[72-27-and-108-42-section]{Section~\ref*{72-27-and-108-42-section}} \\ \hline 
					$((C_6 \times C_2) \rtimes C_2) \times C_3 = \Biggr\langle \begin{array}{l|cl}
					a,b, &a^6 = b^2 = c^2 = k^3 = 1, \ c^{-1}ac = a^{-1}b, \\
					c,k & [a,b] = [b,c] = 1, ~ k \text{ central}
					\end{array} \Biggr\rangle$ & $[72,30]$ & \hyperref[72-30-section]{Section~\ref*{72-30-section}} \\ \hline 
					$G(3,4,2) \times C_3 \times C_3 = \Biggl\langle 
					\begin{array}{l|cl}
					g,h,& g^3 = h^4 = k_1^3 = k_2^3 = 1, \\
					k_1, k_2& h^{-1}gh = g^2, ~ k_1 \text{ and } k_2 \text{ central}                     
					\end{array} 
					\Biggr\rangle$ & $[108,32]$ & \hyperref[108-32-section]{Section~\ref*{108-32-section}} \\ \hline
					$S_3 \times C_6 \times C_3 = \Biggl\langle 
					\begin{array}{l|cl}
					\sigma,\tau,& \sigma^3 = \tau^2 = (\tau \sigma)^2 = \kappa_1^6 = \kappa_2^3 = 1, \\
					\kappa_1,\kappa_2&  \kappa_1 \text{ and } \kappa_2 \text{ central}                             
					\end{array} 
					\Biggr\rangle $ & $[108,42]$ & \hyperref[72-27-and-108-42-section]{Section~\ref*{72-27-and-108-42-section}}
		\end{tabular}}}
		\caption{The groups in the table will be shown to be hyperelliptic in dimension $4$ in this chapter.} \label{table:examples}
	\end{table}	
\end{center}

A full classification of hyperelliptic fourfolds will not be given here. However, we will investigate hyperelliptic fourfolds $T/G$ for each of the groups $G$ contained in \hyperref[table:examples]{Table~\ref*{table:examples}}. More precisely, the following steps are performed for each group $G$ in the table:
\begin{enumerate}
	\item[(A)] Recall the structure and representation theory of $G$
	\item[(B)] Determine the complex representation(s) $\rho$ up to suitable equivalence and the isogeny type of $T$
	\item[(C)] Show that $G$ is indeed hyperelliptic in dimension $4$ by giving a concrete example
	\item[(D)] Calculate the Hodge diamond(s) of hyperelliptic fourfolds with holonomy group $G$
\end{enumerate}
As already mentioned at the beginning of the chapter, we dedicated the entirety of \hyperref[non-examples]{Section~\ref*{non-examples}} to proving that certain groups are \emph{not} hyperelliptic in dimension $4$. However, after having performed steps (A) -- (D) for $G$, it will sometimes be very easy to show that certain groups (e.g., groups containing $G$ as a subgroup) are not hyperelliptic in dimension $4$. To not having to reintroduce notation, it makes sense to sometimes perform a fifth step
\begin{enumerate}
	\item[(E)] Prove that certain groups are not hyperelliptic in dimension $4$.
\end{enumerate}
As a rule of thumb, if a group can be excluded in a few lines after having completed steps (A) -- (D) for $G$, we exclude it at the end of the section.
If, however, more involved arguments are needed to exclude it, we devote a separate subsection of \hyperref[non-examples]{Section~\ref*{non-examples}} it. For example, it is rather simple to show that the groups $D_4 \times C_6$ and the group $[32,9]$ (both of which contain $D_4 \times C_2$) are not hyperelliptic in dimension $4$ after having completed steps (A) -- (D) for $D_4 \times C_2$. Hence we exclude these two groups in part (E) of \hyperref[16-11-section]{Section~\ref*{16-11-section}}, in which we investigate hyperelliptic fourfolds with holonomy group $D_4 \times C_2$. On the other hand, showing that  $D_4 \times C_4$ is not hyperelliptic in dimension $4$ is quite involved, hence we give the proof in  \hyperref[D4xCd-excluded-section]{Section~\ref*{D4xCd-excluded-section}}.\\ 
In \hyperref[table:forbidden]{Table~\ref*{table:forbidden}}, we list the groups which we have seen to not be hyperelliptic in dimension $4$ throughout the text, together with references. \\

We elaborate on the steps (A) -- (D): \\

The command \textsf{IrreducibleRepresentations()} of the computer algebra system GAP allows to obtain the matrix representations of a given finite group $G$ of small order and up to equivalence. In part (A), we list a representative for each equivalence class of irreducible representations of $G$\footnote{The irreducible representations of a direct product are exactly the tensor products of irreducible representations of the factors. Hence, if $G$ is a direct product, we will only list the irreducible representations of the non-Abelian direct factors of $G$.} (defined over an appropriate cyclotomic field). Furthermore, we consider the following action of $\Aut(G)$ on the set $\Irr(G)$ of irreducible characters of $G$: 
\begin{align} \label{orbits}
\Irr(G) \times \Aut(G) \to \Irr(G), \qquad (\chi,\psi) \mapsto \chi \circ \psi.
\end{align}
Moreover, we determine the orbits of this action. The reason behind determining these orbits is that we have seen in  \hyperref[section:hyperellmfds]{Section~\ref*{section:hyperellmfds}} that biholomorphic hyperelliptic manifolds have equivalent complex representations up to automorphisms. \\

In part (B), we determine the complex representation $\rho$ of a hyperelliptic fourfold with holonomy group $G$, up to equivalence and automorphisms. Here, besides the results from part (A), we will heavily use the  following properties of $\rho$:
\begin{enumerate} [label=(\Roman*), ref=(\Roman*)]
	\item \label{nec-prop1} $\rho$ is faithful,
	\item \label{nec-prop2} every matrix in $\im(\rho)$ has the eigenvalue $1$, and
	\item \label{nec-prop3} $\rho \oplus \overline{\rho}$ is an integral representation: this property will only be a concern, if $G$ contains an element of order $8$, $9$ or $12$, see the \hyperref[order-cyclic-groups]{Integrality Lemma~\ref*{order-cyclic-groups}} \ref{ocg-3}.
\end{enumerate}

Determining $\rho$ up to equivalence and automorphisms not only serves as a first step toward a complete classification of hyperelliptic fourfolds but also dramatically simplifies showing that certain groups containing $G$ are not hyperelliptic in dimension $4$.\\
Knowing $\rho$, we then describe the equivariant decomposition of $T$ into a product of subtori up to isogeny induced by $\rho$, see \hyperref[isogeny]{Section~\ref*{isogeny}}.

%It may however happen that two hyperelliptic manifolds with non-equivalent complex representations are biholomorphic. Indeed, if $f \colon X \to X'$ is an isomorphism between two hyperelliptic manifolds $X = T/G$ and $X' = T'/G$ with associated complex representations $\rho$, $\rho'$, then Bieberbach's Theorems {\ cite} show that that there is a lift $\hat f \colon T \to T'$ of $f$. In particular, the following diagram commutes (where the vertical arrows are the quotient maps):
%\begin{center}
%	\begin{tikzcd}
%	T \arrow[r, "\hat f"] \arrow[d] & T' \arrow[d] \\
%	X \arrow[r, "f"] & X'
%	\end{tikzcd}
%\end{center}
%Denoting the linear part of $f$ by $A$, the commutativity of the diagram shows that there is some automorphism $\psi \in \Aut(G)$ such that
%\begin{align*}
%A \cdot \rho = (\rho' \circ \psi) \cdot A,
%\end{align*}
%i.e., the representations $\rho$ and $\rho'$ are equivalent up to an automorphism. This explains our interest in the orbits of the action \ref{orbits}. Our second goal in part (B) is to determine the isogeny type of the complex torus $T$ on which our group $G$ acts holomorphically: this is account for by the decomposition algorithm described in  \hyperref[isogeny]{Section~\ref*{isogeny}}. \\
Part (C) is about explicitly describing an example of a hyperelliptic fourfold $T/G$ with the given holonomy group $G$. The potential complex representation(s) $\rho$ were already determined part (B), hence giving a concrete example consists of finding a translation part $\tau \colon G \to T$ such that: 
\begin{enumerate}[label=(\roman*), ref=(\roman*)]
	\item \emph{the subgroup $\mathcal G:= \langle f_g(z) := \rho(g)z + \tau(g) ~ | ~ g \in G\rangle$ of $\Bihol(T)$ is isomorphic to $G$.} \\
	By construction, $\mathcal G/\mathcal T \cong G$, where $\mathcal T \subset \mathcal G$ is the normal subgroup consisting of translations. Finding a translation part $\tau$ such that $\mathcal G \cong G$ amounts to showing that the $f_g$ satisfy the defining relations of $G$. The necessary and sufficient  conditions on $\tau$ were found by using a computer algebra system, where the $f_g$ were implemented as a $(5 \times 5)$-matrix of the form
	\begin{align*}
	f_g = \begin{pmatrix}
	\rho(g) & \tau(g) \\
	0 & 1
	\end{pmatrix}.
	\end{align*}

	\item \emph{the group $\mathcal G$ defined above acts freely on $T$.}
\end{enumerate}

%Giving an example is about finding correct translation parts for the linear parts found in (B) such that the group generated by the given automorphisms is isomorphic to $G$, does not contain any translations, and acts freely on $T$. need to give criteria on the translation parts to obtain a well-defined action of $G$ on $T$.
%Since the complex representation determined in part (B) is faithful by construction, we obtain a translation-free action of $G$ on $T$. In order to prove the freeness of the action of $G$ on $T$, it suffices to show that a single representative of each non-trivial conjugacy class of $G$ acts freely on $T$. \\
In the last step, (D), we compute the Hodge diamond of a hyperelliptic fourfold $X = T/G$, that is the datum of the Hodge numbers $h^{p,q}(X) = \dim_\CC(H^{p,q}(X))$. This is done as follows: considering that $G$ acts freely on the complex torus $T = V/\Lam$, it follows that 
\begin{align} \label{hodge-numbers}
H^{p,q}(X) = H^{p,q}(T)^G = \left(\bigwedge^p V \otimes \bigwedge^q\overline{V}\right)^G.
\end{align}
Since the action $G \times H^{p,q}(T) \to H^{p,q}(T)$ is given by $(g,\om) \mapsto \rho(g^{-1})^* \om$, the Hodge numbers $h^{p,q}(X)$ only depend on the complex representation $\rho \colon G \to \GL(V)$ and not on the translation part of the $G$-action on $T$. Now, if $\chi_1, ..., \chi_r$ is the list of irreducible characters contained in $\rho$, then the character $\chi_{p,q}$ of the $G$-action on $\bigwedge^p V \otimes \bigwedge^q\overline{V}$ is given as follows:
\begin{align*}
\chi_{p,q} = \sum_{\substack{s_1 + ... + s_r = p \\
		t_1 + ... + t_r = q}} \prod_{i=1}^r \left(\wedge^{s_i}\chi_i \cdot \wedge^{t_i}\overline{\chi_i}\right). 
\end{align*}
Finally, equation (\ref{hodge-numbers}) implies
\begin{align*}
h^{p,q}(X) = (\chi_{p,q},~ \chi_{\triv}).
\end{align*}

\subsection{$G(8,2,3)$ (ID [16,8])} \label{16-8-section}\  \\

In this section, we describe hyperelliptic fourfolds $T/G(8,2,3)$, where
\begin{align*}
G(8,2,3) = \langle g,h \ | \ g^8 = h^2 = 1, \ h^{-1}gh = g^3 \rangle
\end{align*}
More precisely, we will prove (see  \hyperref[16-8-rho]{Proposition~\ref*{16-8-rho}} and \hyperref[16-8-example]{Example~\ref*{16-8-example}}):

\begin{prop} \label{prop:16-8}
	Let $X = T/G(8,2,3)$ be a hyperelliptic fourfold with associated complex representation $\rho$. Then:
	\begin{enumerate}[ref=(\theenumi)]
		\item \label{prop:16-8-1} Up to equivalence and automorphisms, $\rho$ is given as follows:
		\begin{align*}
		\rho(g) = \begin{pmatrix}
		0 & 1 && \\ 1 & \sqrt 2 i && \\ && 1& \\ &&&1
		\end{pmatrix}, \qquad \rho(h) = \begin{pmatrix}
		1 & \sqrt 2 i && \\ 0 & -1 && \\ &&1 & \\ &&& -1
		\end{pmatrix}.
		\end{align*}
		\item \label{prop:16-8-2} The representation $\rho$ induces an equivariant isogeny $E_{\sqrt 2i} \times E_{\sqrt 2i}  \times E \times E' \to T$, where $E_{\sqrt 2i} = \CC/(\ZZ+\sqrt 2 i \ZZ)$, and $E, E' \subset T$ are elliptic curves.
		\item \label{prop:16-8-3} Hyperelliptic fourfolds with holonomy group $G(8,2,3)$ exist.
	\end{enumerate}
	In particular, $X$ moves in a complete $2$-dimensional family of hyperelliptic fourfolds with holonomy group $G(8,2,3)$.
\end{prop}

Moreover, in part (D) below, the Hodge diamond of hyperelliptic fourfolds $T/G(8,2,3)$ is given. \\

\emph{(A) Representation Theory of $G(8,2,3)$.}

We recall the representation theory of $G(8,2,3) = \langle g,h \ | \ g^8 = h^2 = 1, \ h^{-1}gh = g^3 \rangle.$
The relation $h^{-1}gh = g^3$ implies that $g^2$ is a commutator, and thus the degree $1$ representations of $G(8,2,3)$ are $\chi_{a,b}$ for $a,b \in \{0,1\}$, defined via: 
\begin{align*}
\chi_{a,b}(g) = (-1)^a, \qquad \chi_{a,b}(h) = (-1)^b.
\end{align*}
%{red 
%\begin{tabular}{lll}
%	$\rho_{1,1}:$ & $g \mapsto 1$, & $h \mapsto 1$, \\
%	$\rho_{1,2}:$ & $g \mapsto 1$, & $h \mapsto -1$, \\
%	$\rho_{1,3}:$ & $g \mapsto -1$, & $h \mapsto 1$, \\
%	$\rho_{1,4}:$ & $g \mapsto -1$, & $h \mapsto -1$.
%\end{tabular}}

Furthermore, $G(8,2,3)$ has the following three irreducible representations of degree $2$: 
\begin{center}
	\begin{tabular}{lll}
		$\rho_{2,1}:$ &$g \mapsto \begin{pmatrix}
		0 & 1 \\ 1 & \sqrt 2 i
		\end{pmatrix}$, & $h \mapsto \begin{pmatrix}
		1 & \sqrt 2 i \\ 0 & -1
		\end{pmatrix}$, \\
		$\rho_{2,2}:$ & $g \mapsto -\begin{pmatrix}
		0 & 1 \\ 1 & \sqrt 2 i
		\end{pmatrix}$, & $h \mapsto \begin{pmatrix}
		1 & \sqrt 2 i \\ 0 & -1
		\end{pmatrix}$, \\
		$\rho_{2,3}:$ & $g \mapsto \begin{pmatrix}
		0 & 1 \\ -1 & 0
		\end{pmatrix}$, & $h \mapsto \begin{pmatrix}
		-1 & 0 \\ 0 & 1
		\end{pmatrix}$.
	\end{tabular}
\end{center}
While $\rho_{2,3}$ is non-faithful, the representations $\rho_{2,1}$ and $\rho_{2,2}$ are faithful and form an orbit under the action of $\Aut(G(8,2,3))$. Indeed, $\rho_{2,1} \circ \psi$ and $\rho_{2,2}$ are equivalent representations, where $\psi \in \Aut(G(8,2,3))$ is given by $g \mapsto g^5$, $h \mapsto h$. \\

\emph{(B) The complex representation and the isogeny type of the torus.} \\
Our goals in this part are
\begin{itemize}
	\item to determine all possible complex representations  $\rho \colon G(8,2,3) \to \GL(4,\CC)$ of hyperelliptic fourfolds with holonomy group $G(8,2,3)$ up to equivalence and automorphisms (we will see in (C) that such hyperelliptic fourfolds indeed exist), and
	\item to determine the isogeny type of a complex torus $T$ such that $T/G(8,2,3)$ is a hyperelliptic fourfold.
\end{itemize}
We first decompose $\rho$ into irreducible constituents such that properties \ref{nec-prop1} -- \ref{nec-prop3} on p. \pageref{nec-prop1} hold. According to \\hyperref[cor:metacyclic-rep]{Corollary~\ref*{cor:metacyclic-rep}} \ref{cor:metacyclic-rep-1}, $\rho$ splits as the direct sum of three irreducible representations, whose degrees are $2$, $1$, $1$, respectively. Moreover, \hyperref[cor:metacyclic-rep]{Corollary~\ref*{cor:metacyclic-rep}} \ref{cor:metacyclic-rep-2} implies that the irreducible degree $2$ sub-representation of $\rho$ is faithful and hence equivalent to one of $\rho_{2,1}$ or $\rho_{2,2}$. As remarked in part (A), $\rho_{2,1}$ and $\rho_{2,2}$ are equivalent up to an automorphism of $G(8,2,3)$.

Consequently, we may assume without loss of generality that the irreducible degree $2$ constituent of $\rho$ is $\rho_{2,1}$. \\

Denote now by $\chi$, $\chi'$ the linear characters, which are sub-representations of $\rho$. Since $G(8,2,3)$ is metacyclic, we will always assume that not both of $\chi$ and $\chi'$ are trivial (see \hyperref[lemma-two-generators]{Proposition~\ref*{lemma-two-generators}}).

\begin{lemma} \label{lemma:16-8-ev1}
	Every matrix in the image of $\rho$ has the eigenvalue $1$ if and only if
	\begin{align*}
	\{\chi, \chi'\} \in \left\{ \{\chi_{0,0}, \ \chi_{0,1}\}, ~ \{\chi_{0,0}, \ \chi_{1,0}\}, ~ \{\chi_{0,0}, \ \chi_{1,1}\}, ~ \{\chi_{0,1}, \ \chi_{1,1}\} \right\}.
	\end{align*}
\end{lemma}

\begin{proof}
	We first prove the 'if' part. There is nothing to show if the trivial character $\chi_{0,0}$ is a constituent of $\rho$. In the remaining case, $\{\chi, \chi'\} = \{\chi_{0,1}, \chi_{1,1}\}$, we observe that any element of $G(8,2,3)$ can be uniquely written in the form $g^j h^k$ for $j \in \{0,...,7\}$ and $k \in \{0,1\}$. Since
	\begin{align*}
	\chi_{0,1}(g^j h^k) = (-1)^j \qquad \text{ and } \qquad \chi_{1,1}(g^j h^k) =(-1)^{j+k},
	\end{align*}
	the matrix
	\begin{align*}
	\rho(g^j h^k) = \diag(\rho_{2,1}(g^j h^k), ~ (-1)^j, ~ (-1)^{j+k})
	\end{align*}
	definitely has the eigenvalue $1$ if $k = 0$, or if $k = 1$ and $j$ is odd. In the remaining cases (i.e., $k = 1$ and $j$ is even), the matrix $\rho_{2,1}(g^jh)$ has the eigenvalue $1$, since $\ord(g^jh) = 2$ and $g^j h$ is not a central element for even $j$. \\
	Conversely, observe that $\rho_{2,1}(g)$ does not have the eigenvalue $1$, hence $\chi(g) = 1$ or $\chi'(g) = 1$. We may assume that $\chi(g) = 1$. Since
	\begin{align*}
	\rho(gh) = \diag(\rho_{2,1}(gh), ~ \chi(h), ~ \chi'(gh))
	\end{align*}
	has the eigenvalue $1$, we conclude that either $\chi(h) = 1$ (i.e., $\chi = \chi_{0,0}$, which corresponds to the first three listed cases) or $\chi(h) = -1$ and $\chi'(g) = \chi'(h)$ (i.e., $\chi= \chi_{0,1}$ and $\chi' \in \{\chi_{0,0}, \chi_{1,1}\}$). 
\end{proof}

\begin{prop} \label{16-8-isogeny}
	The complex torus $T$ is equivariantly isogenous to a product of a $2$-dimensional complex torus $S$ and two elliptic curves $E, E'$. Furthermore, $S$ is equivariantly isogenous to $E_{\sqrt{2} i} \times E_{\sqrt 2 i}$, where $E_{\sqrt 2i} = \CC/(\ZZ+\sqrt{2}i\ZZ)$.
\end{prop}

\begin{proof}
	By the previous results, $\rho$ has exactly three isotypical components. They form three orbits under the action of $\Aut(\CC)$. The first statement thus follows from the procedure in \hyperref[isogeny]{Section~\ref*{isogeny}}. The statement about the isogeny type of $S$ is the content of \hyperref[prop:isogenous-ord-8]{Proposition~\ref*{prop:isogenous-ord-8}}.
\end{proof}

A closer analysis shows that the complex torus $S$ is isogenous to a product of elliptic curves as well. This then completes the proof of \hyperref[prop:16-8-2]{Proposition~\ref*{prop:16-8-2}}.

\begin{prop} The complex torus $S$ is equivariantly isogenous to $E_{\sqrt 2i} \times E_{\sqrt 2i}$, where $E_{\sqrt{2}i} = \CC/(\ZZ + \sqrt{2}i\ZZ)$.
\end{prop} 

\begin{proof}
	The complex torus $S$ admits the action of an automorphism of order $8$ with CM-type $\{\zeta_8, \zeta_8^3\}$. By \cite[Proposition 5.8]{Catanese-Ciliberto}, there is only one isomorphism class of complex tori with this property.
\end{proof}

We shall henceforth assume that $T$ is equivariantly isogenous to
\begin{align*}
T' := E_{\sqrt 2i} \times E_{\sqrt 2i} \times E \times E'.
\end{align*}

By a change of origin in $E_{\sqrt 2i} \times E_{\sqrt 2i}$, we may assume that $g$ acts linearly on $E_{\sqrt 2i} \times E_{\sqrt 2i}$. We will therefore write $z = (z_1, z_2, z_3, z_4)$ and
\begin{align*}
&g(z) =  (z_2,\ z_1+ \sqrt 2 i\ z_2, z_3 + a_3,\chi'(g) z_4 + a_4), \\
&h(z) =  (z_1 + \sqrt 2 i\ z_2 + b_1,\ -z_2 + b_2,\ \chi(h) z_3 + b_3,\ \chi'(h) z_4 +b_4).
\end{align*}

Of course, conditions on the $a_i$ and $b_j$ need to be imposed such that $g$ and $h$ span a group that is isomorphic to $G(8,2,3)$.

\begin{lemma} \label{SD-8-relation} \
	\begin{enumerate}[ref=(\theenumi)]
		\item \label{SD8-rel1} The relation $g^8=\id_T$ holds if and only if
		\begin{align*}
		(0\ ,0\ ,8a_3,\ 4(\chi'(g)+1)a_4)
		\end{align*}
		is zero in $T$.
		\item \label{SD8-rel2} The relation $h^2 = \id_T$ holds if and only if
		\begin{align*}
		(2b_1 + \sqrt 2 i \ b_2,\ 0,\ (\chi(h)+1)b_3, \ (\chi'(h)+1)b_4)
		\end{align*}
		is zero in $T$.
		\item \label{SD8-rel3} The relation $hg^3 = gh$ holds if and only if
		\begin{align*}
		(b_1 - b_2, \ (1-\sqrt 2 i )b_2 - b_1), \ (3\chi(h) - 1)a_3, \ (1-\chi'(g))b_4 + (\chi'(h)(2+\chi'(g)) -1)a_4) 
		\end{align*}
		is zero in $T$. 
	\end{enumerate}
	
\end{lemma}

%\begin{proof}
%{red write sth}
%\end{proof}

According to \hyperref[lemma:16-8-ev1]{Lemma~\ref*{lemma:16-8-ev1}}, there are four possible combinations of linear characters contained in $\rho$. The next proposition excludes three of them.

\begin{prop} \label{16-8-rho}
	There are no hyperelliptic fourfolds with holonomy group $G(8,2,3)$ in the cases 
	\begin{align*}
	\{\chi, \chi'\} \in \left\{\{\chi_{0,0}, \ \chi_{1,0}\}, ~ \{\chi_{0,0}, \ \chi_{1,1}\}, ~ \{\chi_{0,1}, \ \chi_{1,1}\} \right\}.
	\end{align*}
	In other words: only the combination $\chi_{0,0}$ and $\chi_{0,1}$ may potentially allow a free and translation-free action of $G(8,2,3)$.
\end{prop}

\begin{proof}
	We denote by $v$ the element defined in  \hyperref[SD-8-relation]{Lemma~\ref*{SD-8-relation}} \ref{SD8-rel3}. \\
	
	\underline{$\chi_{0,0}$ and $\chi_{1,0}$:} In this case, the element $v$ simplifies to
	\begin{align*}
	v = (b_1 - b_2, \ (1-\sqrt 2 i)b_2 - b_1, \ 2a_3, \ 2b_4-2a_4).
	\end{align*}
	Since $v = 0$ in $T$, it follows that
	\begin{align*}
	(\rho(h) - \id_T)v = (w_1, \ w_2, \ 0, \ 4b_4-4a_4)  \ \ \ \ \text{(for some }w_1, w_2\text{)}
	\end{align*}
	is zero in $T$, too (we do not need to specify the first two coordinates of that element in our arguments). Hence
	\begin{align*}
	2v - (\rho(h) - \id_T)v = (w_1', \ w_2', \ 4a_3, \ 0) \ \ \ \ \text{(for some }w_1', w_2'\text{)}
	\end{align*}
	is zero in $T$, too. This shows that
	\begin{align*}
	g^4(z) = (-z_1, \ -z_2, \ z_3 + 4a_3, \ z_4)
	\end{align*}
	does not act freely on $T$. \\
	
	\underline{$\chi_{0,0}$ and $\chi_{1,1}$:} Here, the element $v$ reads as follows:
	\begin{align*}
	v = (b_1 - b_2, \ (1-\sqrt 2 i)b_2 - b_1, \ 2a_3, \ 2b_4+2a_4)
	\end{align*}
	Similarly as in the previous case,
	\begin{align*}
	(\rho(h) - \id_T)v = (w_1, \ w_2, \ 0, \ 4b_4-4a_4)  \ \ \ \ \text{(for some }w_1, w_2\text{)}
	\end{align*}
	is zero in $T$. Thus again
	\begin{align*}
	2v - (\rho(h) - \id_T)v = (w_1', \ w_2', \ 4a_3, \ 0) \ \ \ \ \text{(for some }w_1', w_2'\text{)}
	\end{align*}
	is zero in $T$, showing that $g^4$ does not act freely on $T$. \\
	
	\underline{$\chi_{0,1}$ and $\chi_{1,1}$:} We take the element 
	\begin{align*}
	(b_1 + \sqrt 2i b_2, \ 0, \ 0, \ 0)
	\end{align*}
	given in Lemma \ref{SD-8-relation} \ref{SD8-rel2}  into account: since it is zero in $T$ and $E_{\sqrt 2i} \to T$, $z_1 \mapsto (z_1, 0, ..., 0)$ is injective, we obtain that 
	\begin{align} \label{SD8-III}
	b_1 + \sqrt 2i b_2 = 0 \text{ in } E_{\sqrt 2i}.
	\end{align}
	Using (\ref{SD8-III}), the element $v$ simplifies to
	\begin{align*}
	v = (b_1 - b_2, \ b_2, \ 4a_3, \ 2b_4 - 2a_4). 
	\end{align*}
	We calculate
	\begin{align*}
	\rho(g)v = (b_2, \ -b_2, \ 4a_3, \ 2a_4-2b_4) = 0 \text{ in } A,
	\end{align*}
	where we used (\ref{SD8-III}) again. Consequently, the following element is zero in $T$, too:
	\begin{align} \label{SD8-III-2}
	(\rho(g) + \id_T)v = (b_1, \ 0, \ 8a_3, \ 0).
	\end{align}
	According to \hyperref[SD-8-relation]{Lemma~\ref*{SD-8-relation}} \ref{SD8-rel1}, the relation $g^8 = \id_T$ implies that $(0, \ 0, \ 8a_3, \ 0)$ is zero in $T$. Again, since the map $z_3 \mapsto (0,0,z_3,0)$ is injective, we obtain that $8a_3 = 0$ in $E$. Plugging this into (\ref{SD8-III-2}), we obtain that $(b_1, \ 0, \ 0, \ 0)$ is zero in $T$, and hence $b_1 = 0$. Thus (\ref{SD8-III}) reads $\sqrt 2i b_2 = 0$, which implies that
	\begin{align*}
	(\sqrt 2i)^2 b_2 = -2b_2 = 0.
	\end{align*}
	Necessarily, $b_2 \neq 0$: else $b$ has the fixed point $$\left(0,0,-\frac{b_3}{2}, -\frac{b_4}{2}\right).$$
	We, therefore, obtain that
	\begin{align*}
	b_2 = \frac{\sqrt 2i}{2}
	\end{align*}
	as the only possibility. It remains to prove that $h$ does not act freely in this case. Indeed, the equation $h(z) = z$ is satisfied if and only if
	\begin{align*}
	\left(\sqrt 2i\ z_2, \ -2z_2 + \frac{\sqrt 2i}{2},\ -2z_3 + b_3, \ -2z_4 + b_4\right)
	\end{align*}
	is zero in $T$. For $z_2 = -\frac12$, $z_3 = \frac{b_3}{2} - 2a_3$, $z_4 = -\frac{b_4}{2} - 2b_4 + 2a_4$ and arbitrary $z_1$, this vector is equal to the vector $v$ and hence zero in $T$. Thus $h$ does not act freely in this last case as well.
\end{proof}

The proof of \hyperref[prop:16-8]{Proposition~\ref*{prop:16-8}} \ref{prop:16-8-1} is hence finished. \\

\emph{(C) An example.} \\
We give an example of a hyperelliptic fourfold $T/G(8,2,3)$. By the discussion in part (B), we may write $T = T'/H$, where $T' := E_{\sqrt 2 i} \times E_{\sqrt 2 i} \times E \times E'$, and
\begin{itemize}
	\item  $E_{\sqrt 2 i} = \CC/(\ZZ + \sqrt 2 i \ZZ)$,
	\item  $E = \CC/(\ZZ+\tau \ZZ)$ and $E' = \CC/(\ZZ + \tau'\ZZ)$ are elliptic curves in normal form.
\end{itemize}

%Furthermore, we may assume that the complex representation is given by
%\begin{align} \label{16-8-complexrep}
%g \mapsto \begin{pmatrix}
%0 & 1 & & \\ 1 & \sqrt{2}i & & \\ & & 1 & \\ &&&1
%\end{pmatrix}, \qquad h \mapsto \begin{pmatrix}
%1 & \sqrt 2 i && \\ 0 & -1 && \\ &&1& \\ &&&-1
%\end{pmatrix}.
%\end{align}
\begin{rem} \label{16-8-classes}
	The non-trivial conjugacy classes of $G(8,2,3) = \langle g,h\rangle$ are represented by 
	\begin{align*}
	g, \quad g^2, \quad g^4, \quad g^7, \quad h, \quad gh.
	\end{align*}
	Considering that $\ord(g) = 8$, we observe that the element $g$ acts freely on $T$ if and only if $g^7$ does.
\end{rem}

\begin{example} \label{16-8-example}
	As above, let $T = T'/H$, where
	\begin{align*}
	H = \left\langle \left(0,0,\frac12,\frac12\right) \right\rangle.
	\end{align*}
	Define holomorphic automorphisms $g,h \in \Bihol(T')$ by
	\begin{align} \label{16-8-g-h}
	&g(z) = \left(z_2, \ z_1 + \sqrt 2 i \ z_2, \ z_3 + \frac14, \ z_4 + \frac18\right), \\
	&h(z) = \left(z_1 + \sqrt 2 i \ z_2, \ z_3 + \frac{\tau}{2}, \ -z_4\right).
	\end{align}
	Since the linear parts of both $g$ and $h$ map $H$ to $H$, they descend to holomorphic automorphisms of $T = T'/H$ which we again denote by $g$ and $h$.  \hyperref[SD-8-relation]{Lemma~\ref*{SD-8-relation}} shows that the $\langle g,h\rangle \cong G(8,2,3)$, when viewed as a subgroup of $\Bihol(T)$. \\
	The last step is to show that the action of $\langle g,h\rangle$ on $T$ is indeed free. By  \hyperref[16-8-classes]{Remark~\ref*{16-8-classes}}, we only have to prove that $g$, $g^2$, $g^4$, $h$ and $gh$ act freely. It is immediate from the definition of $T$ that the elements $g$ and $h$ as defined in \ref{16-8-g-h} as well as
	\begin{align*}
	&g^2(z) = \left(z_1 + \sqrt 2 i \ z_2, \ \sqrt 2 i \ z_1 - z_2, \ z_3 + \frac12, \ z_4 + \frac14 \right), \\
	&g^4(z) = \left(-z_1, \ -z_2, \ z_3, \ z_4 + \frac12\right), \text{ and} \\
	&gh(z) = \left(-z_2, \ z_1, \ z_3 + \frac{1}{4} + \frac{\tau}{2}, \ -z_4 + \frac18\right)
	\end{align*} 
	all act freely on $T$. Therefore, we have described a hyperelliptic fourfold $T/G(8,2,3)$.
\end{example}

\emph{(D) The Hodge diamond.} \\
The Hodge diamond of a hyperelliptic fourfold with holonomy group $G(8,2,3)$ is
\begin{center}
	$\begin{matrix}
	&   &  &  & 1 &  &  &  &  \\
	&   &  & 1 &  & 1 &  &  &  \\
	&   & 0 &  & 3 &  & 0  &  & \\
	&  0 &  & 3 &   & 3 &   &  0 &  \\
	0&    & 1 &  & 6  &  & 1  &     &  0 \\
	\end{matrix}$
	
\end{center}

\subsection{$D_4 \times C_2$ (ID [16,11])} \label{16-11-section}\  \\ 

Our aim is to describe hyperelliptic fourfolds $T/(D_4 \times C_2)$, where $D_4 = \langle r,s \ | \ r^4 = s^2 = 1, ~ s^{-1}rs = r^{-1}\rangle$ is the dihedral group of order $8$. Denoting the generator of $C_2$ by $k$, we will prove that (see part (B) and \hyperref[D4xC2-example]{Example~\ref*{D4xC2-example}}):

\begin{prop} \label{16-11-prop}
	Let $X = T/(D_4 \times C_2)$ be a hyperelliptic fourfold with associated complex representation $\rho$. Then:
	\begin{enumerate}[ref=(\theenumi)]
		\item \label{16-11-prop1} Up to equivalence and automorphisms, $\rho$ is given as follows:
		\begin{align*}
		\rho(r) = \begin{pmatrix}
		0 & -1 && \\ 1 & 0 && \\ && 1& \\ &&&1
		\end{pmatrix}, \qquad \rho(s) = \begin{pmatrix}
		1 & && \\  & -1 && \\ &&-1 & \\ &&& 1
		\end{pmatrix}, \qquad \rho(k) = \begin{pmatrix}
		1 & && \\  & 1 && \\ &&1 & \\ &&& -1
		\end{pmatrix}.
		\end{align*}
		\item The representation $\rho$ induces an equivariant isogeny $E \times E \times E' \times E'' \to T$, where $S \subset T$ is a complex sub-torus of dimension $2$, where $E, E', E'' \subset T$ are elliptic curves.
		\item Hyperelliptic fourfolds with holonomy group $D_4 \times C_2$ exist.
	\end{enumerate}
	In particular, $X$ moves in a complete $3$-dimensional family of hyperelliptic fourfolds with holonomy group $D_4 \times C_2$.
\end{prop}

The Hodge diamond of a hyperelliptic fourfold with holonomy group $D_4 \times C_2$ is computed in part (D). Furthermore, the following negative result will be shown in part (E):

\begin{prop} \label{prop:d4xc6-excluded}
	The following groups are not hyperelliptic in dimension $4$:
	\begin{enumerate}[label=(\roman*)]
		\item the group $$(C_8 \times C_2) \rtimes C_2 = \langle a,b,c ~ | ~ a^8 = b^2 = c^2 = 1, ~ ab = ba, ~ bc = cb, ~ c^{-1}ac = a^3 b\rangle,$$
		which is labeled $[32,9]$ in the Database of Small Groups, and
		\item $D_4 \times C_6$.
	\end{enumerate}
\end{prop}

Both of the above groups contain a subgroup that is isomorphic to $D_4 \times C_2$. Indeed, $\langle r:= a^2, s := c, k := b\rangle \subset (C_8 \times C_2) \rtimes C_2$ is isomorphic to $D_4 \times C_2$. \\

\emph{(A) Representation Theory of $D_4 \times C_2$.} \\
We recall the representation theory of the dihedral group $$G(4,2,3) = D_4 = \langle r,s \ | \ r^4 = s^2 = 1, \ s^{-1}rs = r^{-1}\rangle$$
of order $8$. It has exactly one irreducible representation of degree $2$. It is faithful and given by:
\begin{align*}
r \mapsto \begin{pmatrix}
0 & -1 \\ 1 & 0
\end{pmatrix}, \qquad s \mapsto \begin{pmatrix}
1 & \\ & -1
\end{pmatrix}.
\end{align*}

The degree $1$ representations of $D_4$ are $\chi_{a,b}$, $a,b \in \{0,1\}$, defined by $\chi_{a,b}(r) = (-1)^a$ and $\chi_{a,b}(s) = (-1)^b$. \\

\emph{(B) The complex representation and the isogeny type of the torus.} \\
Our goal is to determine all possible complex representations $\rho \colon D_4 \times C_2 \to \GL(4,\CC)$ of hyperelliptic fourfolds with holonomy group $D_4 \times C_2$, up to equivalence and automorphisms. Furthermore, we determine the isogeny type of a $4$-dimensional complex torus $T$ such that $T/(D_4 \times C_2)$ is a hyperelliptic fourfold. A concrete example of such a hyperelliptic fourfold will then be given in part (C). \\

As a first step towards our goal, we decompose $\rho$ into irreducible constituents such that the necessary properties \ref{nec-prop1} -- \ref{nec-prop3} given on p. \pageref{nec-prop1} are satisfied. Let $k$ be an element of order $2$ that commutes with both $r$ and $s$ and such that $\langle r,s,k\rangle = D_4 \times C_2$. \\
Since $D_4$ is metacyclic, \hyperref[cor:metacyclic-rep]{Corollary~\ref*{cor:metacyclic-rep}} \ref{cor:metacyclic-rep-1} implies that $\rho$ is the direct sum of three irreducible representations $\rho_2$, $\chi$ and $\chi'$ whose respective degrees are $2$, $1$ and $1$. Here, $\rho_2|_{D_4}$ is the degree $2$ representation given in part (A). Since the assignment
\begin{align*}
r \mapsto r, \qquad s \mapsto s, \qquad k \mapsto r^2k
\end{align*}
defines an automorphism of $D_4 \times C_2$, we may assume that $k \in \ker(\rho_2)$. Now, $\rho$ being faithful by property \ref{nec-prop1}, we may assume that $\chi'(k) = -1$. By property \ref{nec-prop2}, the matrix
\begin{align*}
\rho(r^2k) = \diag(-1, \ -1, \ \chi(k), \ -1)
\end{align*}
must have the eigenvalue $1$: this allows us to conclude that $\chi(k) = 1$. After having applied the automorphism
\begin{align*}
r \mapsto rk^{\ord(\chi'(r)) - 1}, \qquad s \mapsto sk^{\ord(\chi'(s)) - 1}, \qquad k \mapsto k,
\end{align*}
of $D_4 \times C_2$, we may assume that $\chi'(r) = \chi'(s) = 1$. The next lemma determines the remaining values $\chi(r)$ and $\chi(s)$. 

\begin{lemma}
	It holds $\chi(r) = 1$ and $\chi(s) = -1$.
\end{lemma}

\begin{proof}
	The first statement follows since the matrix
	\begin{align*}
	\rho(rk) = \begin{pmatrix}
	0 & -1 && \\ 1 & 0 && \\ && \chi(r) & \\ &&& -1
	\end{pmatrix}
	\end{align*}
	must have the eigenvalue $1$, and the second statement follows from \hyperref[lemma-two-generators]{Proposition~\ref*{lemma-two-generators}}.
\end{proof}

This ultimately determines the complex representation $\rho$ up to equivalence and automorphisms. \\ 
Our next step is to determine the isogeny type of a hyperelliptic fourfold with holonomy group $D_4 \times C_2$. The discussion in  \hyperref[isogeny]{Section~\ref*{isogeny}} implies that $T$ is equivariantly isogenous to the product of a $2$-dimensional complex torus $S \subset T$ and two elliptic curves $E', E'' \subset T$. 

\begin{lemma}
	The $2$-dimensional complex torus $S$ is isogenous to $E \times E$ for an elliptic curve $E \subset S$.
\end{lemma}

\begin{proof}
	The group $D_4 \times C_2$ acts linearly on $S$ by the degree $2$ representation $\rho_2$. Thus 
	\begin{align*}
	E_1 := \ker(\rho_2(s) - \id_S)^0, \qquad E_2 := \im(\rho_2(s) - \id_S)
	\end{align*}
	are elliptic curves and $S$ is isogenous to $E_1 \times E_2$. Since $\rho_2(r)(E_1) = E_2$, we infer that $E:=E_1 \cong E_2$ so that $S$ is isogenous to $E \times E$.
\end{proof}

\emph{(C) An example.} \\
In this part, we will give an example of a hyperelliptic fourfold with holonomy group $D_4 \times C_2$. After a change of coordinates in the elliptic curves $E$, $E'$ and $E''$, we may write the action of $D_4 \times C_2$ on $T = (E \times E \times E' \times E'')/H$ as follows:
\begin{align*}
&r(z) = (-z_2, \ z_1, \ z_3 + a_3, \ z_4+a_4), \\
&s(z) = (z_1 + b_1, \ -z_2 + b_2, \ -z_3, \ z_4 + b_4), \\
&k(z) = (z_1 + c_1, \ z_2 + c_2, \ z_3 + c_3, \ -z_4).
\end{align*}
The following lemma will immediately show that the action given in \hyperref[D4xC2-example]{Example~\ref*{D4xC2-example}} below is well-defined. 

\begin{lemma} \label{D4xC2-relations} \
	\begin{enumerate}[ref=(\theenumi)]
		\item The relation $r^4 = \id_T$ is satisfied if and only if $(0, \ 0, \ 4a_3, \ 4a_4)$ is zero in $T$. 
		\item The relation $s^2 = \id_T$ is satisfied if and only if $(2b_1, \ 0, \ 0, \ 2b_4)$ is zero in $T$.
		\item The relation $s^{-1}rs = r^{-1}$ is satisfied if and only if $(b_1+b_2, \ b_1-b_2, \ 0, \ -2a_4)$ is zero in $T$.
		\item The relation $k^2 = \id_T$ is satisfied if and only if $(2c_1, \ 2c_2, \ 2c_3, \ 0)$ is zero in $T$.
		\item The elements $r$ and $k$ commute if and only if $(c_1+c_2, \ c_2-c_1, \ 0, \ -2a_4)$ is zero in $T$.
		\item The elements $s$ and $k$ commute if and only if $(0, \ 2c_2, \ 2c_3, \ -2b_4)$ is zero in $T$.
	\end{enumerate}
\end{lemma}

\begin{rem}
	The following elements form a system of representatives for the nine non-trivial conjugacy classes of $D_4 \times C_2$:
	\begin{align*}
	r, \qquad r^2, \qquad s, \qquad rs, \qquad k, \qquad rk, \qquad r^2k, \qquad sk, \qquad rsk.
	\end{align*}
\end{rem}

\begin{example} \label{D4xC2-example}
	Our construction is reminiscent of the one given in \cite{CD-Hyp3}. We define $T := T'/H$, where 
	\begin{center}
		$T' = E \times E \times E' \times E''$, \\
		$E = \CC/(\ZZ+\tau\ZZ)$, \quad $E' = \CC/(\ZZ+\tau'\ZZ)$, \quad $E'' = \CC/(\ZZ+\tau''\ZZ)$, \\
		$H = \left\langle \left(\frac{1+\tau}2, \frac{1+\tau}{2}, \ 0, \ 0\right) \right\rangle.$ 
	\end{center}
	Furthermore, we define the following holomorphic automorphisms of $T'$:
	\begin{align*}
	&r(z) = \left(-z_2, \ z_1, \ z_3 + \frac14, \ z_4\right), \\
	&s(z) = \left(z_1 + \frac{1}{2}, \ -z_2 + \frac{\tau}{2}, \ -z_3, \ z_4\right), \\
	&k(z) = \left(z_1, \ z_2, \ z_3 + \frac{\tau'}{2}, \ -z_4\right).
	\end{align*}
	Since the linear parts of $r$, $s$ and $k$ map $H$ to $H$, they descend to automorphisms of $T = T'/H$. Lemma \ref{D4xC2-relations} now immediately implies that $\langle r,s,k \rangle \subset \Bihol(T)$ is isomorphic to $D_4 \times C_2$ and a posteriori does not contain any translation by construction. \\
	We prove that $\langle r,s,k\rangle$ acts freely on $T$. Clearly, the elements
	\begin{align*}
	&r(z) = \left(-z_2, \ z_1, \ z_3 + \frac14, \ z_4\right), \\
	&r^2(z) = \left(-z_1, \ -z_2, \ z_3 + \frac12, \ z_4\right), \\
	&s(z) = \left(z_1 + \frac{1}{2}, \ -z_2 + \frac{\tau}{2}, \ -z_3, \ z_4\right), \\
	&k(z) = \left(z_1, \ z_2, \ z_3 + \frac{\tau'}{2}, \ -z_4\right), \\
	&rk(z) = \left(-z_2, \ z_1, \ z_3 + \frac14 + \frac{\tau'}{2}, \ -z_4 \right), \\
	&r^2k(z) = \left(-z_1, \ -z_2, \ z_3 + \frac{1+\tau'}{2}, \ -z_4 \right), \text{ and} \\
	&sk(z) = \left(z_1 + \frac12, \ -z_2 + \frac{\tau}{2}, \ -z_3 + \frac{\tau'}{2}, \ -z_4\right)
	\end{align*}
	act freely on $T$. It remains to check the freeness of $rs$ and $rsk$. These elements are given by
	\begin{align*}
	&rs(z) = \left(z_2 + \frac{\tau}{2}, \ z_1 + \frac12, \ -z_3 + \frac14, \ z_4\right), \\
	&rsk(z) = \left(z_2 + \frac{\tau}{2}, \ z_1 + \frac12, \ -z_3 + \frac14 + \frac{\tau'}{2}, \ -z_4\right).
	\end{align*}
	We prove that $rs$ acts freely. The reader will be convinced that the same proof shows that $rsk$ acts freely as well. Indeed, $rs$ has a fixed point $z = (z_1,...,z_4)$ on $T$ if and only if
	\begin{align*}
	rs(z) - z = \left(z_2 - z_1 + \frac{\tau}{2}, \ z_1-z_2 + \frac12, \ -z_3 + \frac14, \ z_4\right) 
	\end{align*}
	is zero in $T$. Reading this equation as an equation in $T'$, we obtain that
	\begin{align*}
	rs(z) - z = \left(z_2 - z_1 + \frac{\tau}{2}, \ z_1-z_2 + \frac12, \ -z_3 + \frac14, \ z_4\right) \in H
	\end{align*}
	is the sufficient and necessary condition for $rs$ to have a fixed point on $T$. By our definition of $H$, this is equivalent to the existence of $j \in \{0,1\}$ such that
	\begin{align*}
	\left(z_2 - z_1 + \frac{\tau}{2}, \ z_1-z_2 + \frac12, \ -z_3 + \frac14, \ z_4\right) = j \cdot \left(\frac{1+\tau}{2}, \ \frac{1+\tau}{2}, \ 0, \ 0\right) \text{ in } T'.
	\end{align*}
	Such a $j$ is immediately seen to not exist by comparing the first two coordinates of both sides and solving for $z_2-z_1$. This thus shows that $rs$ indeed acts freely on $T$.
	In total, we have shown that $T/\langle r,s,k\rangle$ is a hyperelliptic fourfold with holonomy group $D_4 \times C_2 = \langle r,s,k\rangle$. 
\end{example}

The proof of \hyperref[16-11-prop]{Proposition~\ref*{16-11-prop}} is now complete. \\

\emph{(D) The Hodge diamond.} \\
The Hodge diamond of a hyperelliptic fourfold with holonomy group $D_4 \times C_2$ is 
\begin{center}
	$\begin{matrix}
	&   &  &  & 1 &  &  &  &  \\
	&   &  & 0 &  & 0 &  &  &  \\
	&   & 0 &  & 3 &  & 0  &  & \\
	&  1 &  & 2 &   & 2 &   &  1 &  \\
	0&    & 0 &  & 4  &  & 0  &     &  0 \\
	\end{matrix}$
\end{center}

\emph{(E) Proof of \hyperref[prop:d4xc6-excluded]{Proposition~\ref*{prop:d4xc6-excluded}}.} \\

We first show that the group
\begin{align*}
(C_8 \times C_2) \rtimes C_2 = \langle a,b,c ~ | ~ a^8 = b^2 = c^2 = 1, ~ ab = ba, ~ bc = cb, ~ c^{-1}ac = a^3 b\rangle
\end{align*}
is not hyperelliptic in dimension $4$.
It is immediate from the presentation that the subgroup spanned by $a^2$, $b$ and $c$ is isomorphic to $D_4 \times C_2$ via the isomorphism
\begin{align*}
a^2 \mapsto r, \qquad b \mapsto k,  \qquad c \mapsto s.
\end{align*}
Assume now that $\rho \colon (C_8 \times C_2) \rtimes C_2 \to \GL(4,\CC)$ is a faithful representation. If $\rho$ is the complex representation of some hyperelliptic fourfold, then, according to  \hyperref[16-11-prop]{Proposition~\ref*{16-11-prop}} \ref{16-11-prop1}, we may assume that
\begin{align*}
\rho(a^2) = \begin{pmatrix}
0 & -1 && \\ 1 & 0 && \\ && 1& \\ &&&1
\end{pmatrix}, \qquad \rho(b) = \begin{pmatrix}
1 & && \\  & 1 && \\ &&1 & \\ &&& -1
\end{pmatrix}. \qquad \rho(c) = \begin{pmatrix}
1 & && \\  & -1 && \\ &&-1 & \\ &&& 1
\end{pmatrix}.
\end{align*}
Now, the defining relations of the group imply that $(ac)^2 = a^4b$. Since $b$ is mapped to the identity by the irreducible degree $2$ summand $\rho_2$ of $\rho$, we may conclude that $\rho_2(abc)$ and hence also $\rho(abc)$ do not have the eigenvalue $1$, a contradiction. Thus the group with ID $[32,9]$ is not hyperelliptic in dimension $4$.\\

%f := FreeGroup("a","b","c");
%G := f/[f.1^8, f.2^2, f.3^2, Comm(f.1,f.2), Comm(f.2,f.3), f.3^-1*f.1*f.3*(f.1^3*f.2)^-1];

We now turn our attention to $D_4 \times C_6$. Again, we denote by $\rho \colon D_4 \times C_6 \to \GL(4,\CC)$ a faithful representation. Let $\ell$ be a generator of the $C_6$ factor. Clearly, it is necessary for $\rho$ to be the complex representation of some hyperelliptic manifold that 
\begin{enumerate}[label=(\roman*), ref=(\roman*)]
	\item \label{enum:d4xc2} $\rho|_{D_4 \times \langle \ell^3 \rangle}$ is the complex representation of a hyperelliptic manifold with holonomy group $D_4 \times C_2$, and
	\item \label{enum:d4xc3} $\rho|_{D_4 \times \langle \ell^2 \rangle}$  is the complex representation of a hyperelliptic manifold with holonomy group $D_4 \times C_3$.
\end{enumerate}

As for \ref{enum:d4xc2}, we may assume that $\rho(r)$, $\rho(s)$ and $\rho(\ell^3)$ are as in \hyperref[16-11-prop]{Proposition~\ref*{16-11-prop}} \ref{16-11-prop1}.  
Furthermore \hyperref[prop:metacyclic-c3]{Proposition~\ref*{prop:metacyclic-c3}} shows that 
\begin{align*}
\rho(\ell^2) = \diag(\zeta_3, \ \zeta_3, \ 1, \ 1) \quad \text{ or } \quad  \rho(\ell^2) = \diag(\zeta_3^2, \ \zeta_3^2, \ 1, \ 1)
\end{align*}
is necessary for \ref{enum:d4xc3} to hold. It then follows that $\rho(s\ell)$ does not have the eigenvalue $1$, and hence $D_4 \times C_6$ is not hyperelliptic in dimension $4$ either.

\subsection{$Q_8 \times C_3$ (ID [24,11])} \label{24-11-section}\  \\

In this section, the following result about hyperelliptic fourfolds with holonomy group
\begin{align*}
Q_8 \times C_3 = \langle g,h,k ~ | ~ g^4 = k^3 = 1, ~ g^2 = h^2, ~ h^{-1}gh = g^3, ~ k \text{ central}\rangle
\end{align*}
is shown:

\begin{prop} \label{24-11-prop}
	Let $X = T/(Q_8 \times C_3)$ be a hyperelliptic fourfold with associated complex representation $\rho$. Then:
	\begin{enumerate}[ref=(\theenumi)]
		\item \label{24-11-prop1} Up to equivalence and automorphisms, $\rho$ is given as follows:
		\begin{align*}
		\rho(g) = \begin{pmatrix}
		1+2\zeta_3 & -1 && \\ -2 & -1-2\zeta_3 && \\ && -1& \\ &&&1
		\end{pmatrix}, \qquad \rho(h) = \begin{pmatrix}
		-1 & \zeta_3^2 && \\ -2\zeta_3 & 1 && \\ &&1 & \\ &&& 1
		\end{pmatrix}, \qquad \rho(k) = \begin{pmatrix}
		\zeta_3 &&& \\ & \zeta_3&& \\ && 1 & \\ &&& 1
		\end{pmatrix}.
		\end{align*}
		\item \label{24-11-prop2} The representation $\rho$ induces an equivariant isogeny $F \times F \times E \times E' \to T$, where $F = \CC/(\ZZ+\zeta_3\ZZ)$ is the equianharmonic elliptic curve, and $E, E' \subset T$ are elliptic curves.
		\item \label{24-11-prop3} Hyperelliptic fourfolds with holonomy group $Q_8 \times C_3$ exist.
	\end{enumerate}
	In particular, $X$ moves in a complete $2$-dimensional family of hyperelliptic fourfolds with holonomy group $Q_8 \times C_3$.
\end{prop} 

We, furthermore, determine the Hodge diamond of such hyperelliptic fourfolds in part (D) below.

\emph{(A) Representation Theory of $Q_8$.} \\
The quaternion group
\begin{align*}
Q_8 = \langle g,h \ | \ g^4 = 1, \ g^2 = h^2, \ gh = h^{-1}g\rangle
\end{align*}
of order $8$ has -- up to equivalence -- a unique irreducible representation $\rho_2$ of degree $2$. It is well-known that it cannot be defined over $\QQ$. It can, however, be defined over $\QQ(\zeta_3)$:
\begin{align*}
\rho_2(g) =  \begin{pmatrix}
1+2\zeta_3 & -1 \\ -2 & -1-2\zeta_3 
\end{pmatrix}, \qquad \rho_2(h) =  \begin{pmatrix}
-1 & \zeta_3^2 \\ -2\zeta_3 & 1
\end{pmatrix}.
\end{align*}
Its four linear characters $\chi_{a,b}$, $a,b \in \{0,1\}$ are given by $\chi_{a,b}(g) = (-1)^a$, $\chi_{a,b}(h) = (-1)^b$. The three non-trivial linear characters form an orbit under the action of $\Aut(Q_8)$. Indeed,  $\chi_{1,0} \circ \varphi_1 = \chi_{1,1}$ and $\chi_{1,0} \circ \varphi_2 = \chi_{0,1}$, where $\varphi_1, \varphi_2 \in \Aut(Q_8)$ are defined by
\begin{align*}
&\varphi_1(g) = g, \qquad \varphi_1(h) = gh, \\
&\varphi_2(g) = h, \qquad \varphi_2(h) = g.
\end{align*}

\emph{(B) The complex representation and the isogeny type of the torus.} \\

In this subsection, we determine the complex representation of a hyperelliptic fourfold with holonomy group $Q_8 \times C_3$. As a first step, the complex representation of a hyperelliptic fourfold with holonomy $Q_8$ is determined:

%\begin{align} \label{24-11-meta}
%Q_8 \times C_3 = \langle a,b \ | \ a^{12} = [a^4,b] = 1, \ a^6 = b^2, \ ab = ba^{-1}\rangle.
%\end{align}
%
%\hyperref[lemma-two-generators]{Proposition~\ref*{lemma-two-generators}} thus implies that $\rho$ decomposes as the direct sum of three irreducible representations whose respective degrees are $2$, $1$ and $1$. Clearly, the degree $2$ summand is the representation $\rho_2$ defined above. We denote the degree $1$ constituents by $\chi$ and $\chi'$, respectively. 

\begin{lemma} \label{lemma:cplxrep-q8}
	Up to automorphisms of $Q_8$, the complex representation $\rho$ of a hyperelliptic fourfold with holonomy group $Q_8$ is equivalent to the direct sum $\rho_2 \oplus \chi_{1,0} \oplus \chi_{0,0}$, where $\rho_2$ and the $\chi_{a,b}$ are as in (A).
\end{lemma}

\begin{proof}
	\hyperref[cor:metacyclic-rep]{Corollary~\ref*{cor:metacyclic-rep}} \ref{cor:metacyclic-rep-1} shows that $\rho$ is equivalent to a direct sum of three irreducible representations, whose respective degrees are $2$, $1$, $1$. Clearly, the irreducible degree $2$ summand is (equivalent to) $\rho_2$. Denote the degree $1$ summands by $\chi$ and $\chi'$. Now, according to \hyperref[lemma-two-generators]{Proposition~\ref*{lemma-two-generators}}, we may assume $\chi$ to be non-trivial. Now, the discussion in part (A) shows that -- after applying a suitable automorphism of $Q_8$ -- we may assume that $\chi = \chi_{1,0}$. Finally, the matrices
	\begin{align*}
	\rho(g) = \diag(\rho_2(g), \ -1, \ \chi'(g)) \quad \text{ and } \quad \rho(gh) = \diag(\rho_2(gh), \ -1, \ \chi'(gh)) 
	\end{align*}
	must have the eigenvalue $1$, and hence it follows that $\chi'$ is trivial, since $\rho_2(g)$ and $\rho_2(gh)$ do not have the eigenvalue $1$.
\end{proof}

Denote now by $k$ an element of order $3$ that commutes with $g$ and $h$, so that $\langle g,h,k\rangle = Q_8 \times C_3$. If $\rho \colon Q_8 \times C_3 \to \GL(4,\CC)$ denotes the complex representation of a hyperelliptic fourfold $X = T/(Q_8 \times C_3)$, then \hyperref[prop:metacyclic-c3]{Proposition~\ref*{prop:metacyclic-c3}} implies that -- after possibly replacing $k$ by $k^2$ -- we may assume that
\begin{align*}
\rho(k) = \diag(\zeta_3, \ \zeta_3, \ 1, \ 1).
\end{align*}
Together with \hyperref[lemma:cplxrep-q8]{Lemma~\ref*{lemma:cplxrep-q8}}, this completely determines $\rho$, hence completing the proof of \hyperref[24-11-prop]{Proposition~\ref*{24-11-prop}} \ref{24-11-prop1}. \\

The discussion in \hyperref[isogeny]{Section~\ref*{isogeny}} shows that $\rho$ induces a decomposition of the complex torus $T$ up to isogeny:
\begin{align*}
T \sim_{(Q_8 \times C_3)-\isog} S \times E \times E'.
\end{align*}
Here, $S \subset T$ is a sub-torus of dimension $2$ and $E$, $E' \subset T$ are elliptic curves. Since $\rho(k)$ acts on $S$ by multiplication by $\zeta_3$, \hyperref[prop:isogenous-ord-3-4]{Proposition~\ref*{prop:isogenous-ord-3-4}} shows that $S$ is equivariantly isogenous to the square of Fermat's elliptic curve $F = \CC/(\ZZ+ \zeta_3 \ZZ)$. This completes the proof of \hyperref[24-11-prop]{Proposition~\ref*{24-11-prop}} \ref{24-11-prop2}. \\

\emph{(C) An Example.} \\

We give an example of a hyperelliptic fourfold with holonomy group $Q_8 \times C_3$, which then proves part \ref{24-11-prop3} of \hyperref[24-11-prop]{Proposition~\ref*{24-11-prop}} and hence completes the proof of the proposition. According to part \ref{24-11-prop1} of the cited proposition, we may assume that 
\begin{align*}
&T = (F \times F \times E \times E')/H, \text{ where} \\
&F = \CC/(\ZZ+\zeta_3\ZZ), \quad E = \CC/(\ZZ+\tau\ZZ), \quad E' = \CC/(\ZZ+\tau'\ZZ).
\end{align*}
Moreover, $H$ is a finite group of translations.
Furthermore, by \ref{24-11-prop2} and after a change of origin in the respective elliptic curves, we may assume that the action of $g$, $h$, and $k$ on $T$ can be written in the following form:
\begin{align*}
&g(z) = \left((1+2\zeta_3)z_1 - z_2 + a_1, \ -2z_1 - (1+2\zeta_3)z_2 + a_2, \ -z_3, \ z_4 + a_4\right), \\
&h(z) = \left(-z_1 + \zeta_3^2 z_2 + b_1, \ -2\zeta_3 z_1 + z_2 + b_2, \ z_3 + b_3, \ z_4 + b_4\right), \\
&k(z) = \left(\zeta_3 z_1, \ \zeta_3 z_2, \ z_3 + c_3, \ z_4 + c_4\right).
\end{align*}

We spell out conditions on the translation parts of $g$, $h$ and $k$ which guarantee that $\langle g,h,k \rangle \subset \Bihol(T)$ is isomorphic to $Q_8 \times C_3$:

\begin{lemma} \label{24-11-relations} \
	\begin{enumerate}[ref=(\theenumi)]
		\item The relation $g^4 = \id_T$ holds if and only if $4a_4 = 0$ in $E'$.
		\item The relation $g^2 = h^2$ holds if and only if $(2\zeta_3^2a_1 + a_2 + \zeta_3^2 b_2, \ 2(a_1+b_2) + 2\zeta_3 (a_2-b_1), \ 2b_3, \ 2(b_4-a_4))$ is zero in $T$.
		\item The relation $gh = h^{-1}g$ holds if and only if $(\zeta_3^2(a_2-2b_1) + \zeta_3b_2, \ -2\zeta_3 a_1 + 2a_2 + (2\zeta_3^2-4)b_1 + 2\zeta_3^2 b_2, \ 0, \ -2b_4)$ is zero in $T$.
		\item The relation $k^3 = \id_T$ holds if and only if $(0, \ 0, \ 3c_3, \ 3c_4)$ is zero in $T$.
		\item The elements $g$ and $k$ commute if and only if $((\zeta_3-1)a_1, \ (\zeta_3-1)a_2, \ 2c_3, \ 0)$ is zero in $T$.
		\item The elements $h$ and $k$ commute if and only if $((\zeta_3-1)b_1, \ (\zeta_3-1)b_2, \ 0, \ 0)$ is zero in $T$.
	\end{enumerate}
\end{lemma}

\begin{rem}
	The set $\sC= \{\id_T$, $g$, $g^2$, $h$, $gh\}$ is a system of representatives for the conjugacy classes of the quaternion group $Q_8$. Thus $\sC \cup k \sC \cup k^2 \sC$ is a system of representatives for the conjugacy classes of $Q_8 \times C_3$.
\end{rem}

Finally, we present a concrete example.

\begin{example} \label{24-11-example}
	Define $T = T'/H$, where $T' = F \times F \times E \times E'$, and
	\begin{itemize}
		\item $F = \CC/(\ZZ + \zeta_3 \ZZ)$ is the Fermat elliptic curve,
		\item $E$, $E'$ are arbitrary elliptic curves in normal form,
		\item $H \subset T'$ is the subgroup of order $2$ generated by $\left(0, \ 0, \ \frac12, \ \frac12\right)$.
	\end{itemize}
	We moreover consider the following automorphisms of $T'$:
	\begin{align*}
	&g(z) = \left((1+2\zeta_3)z_1 - z_2, \ -2z_1 - (1+2\zeta_3)z_2, \ -z_3, \ z_4 + \frac14\right), \\
	&h(z) = \left(-z_1 + \zeta_3^2 z_2, \ -2\zeta_3 z_1 + z_2, \ z_3 + \frac14, \ z_4\right), \\
	&k(z) = \left(\zeta_3 z_1, \ \zeta_3 z_2, \ z_3, \ z_4 + \frac13\right).
	\end{align*}
	Clearly, the linear parts of these elements map $H$ to $H$, and thus they define automorphisms $g,h,k$ of $T$. Viewed as such,  \hyperref[24-11-relations]{Lemma~\ref*{24-11-relations}} shows that $\langle g,h,k \rangle$ is isomorphic to $Q_8 \times C_3$ and does not contain any translations. That the action is free is immediate, and thus $T/\langle g,h,k\rangle$ defines a hyperelliptic fourfold with holonomy group $Q_8 \times C_3$.
\end{example}

The proof of \hyperref[24-11-prop]{Proposition~\ref*{24-11-prop}} is, therefore, complete. \\

\emph{(D) The Hodge diamond.} \\
The Hodge diamond of a hyperelliptic fourfold with holonomy group $Q_8 \times C_3$ is 
\begin{center}
	$\begin{matrix}
	&   &  &  & 1 &  &  &  &  \\
	&   &  & 1 &  & 1 &  &  &  \\
	&   & 0 &  & 3 &  & 0  &  & \\
	&  0 &  & 3 &   & 3 &   &  0 &  \\
	0&    & 1 &  & 6  &  & 1  &     &  0 \\
	\end{matrix}$
\end{center}

\subsection{$\Heis(3)$ (ID [27,3])} \label{27-3-section}\  \\ 

Here, we describe hyperelliptic fourfolds $T/\Heis(3)$, where
\begin{align*}
\Heis(3) = \langle g,h,k \ | \ g^3 = h^3 = k^3, \ k \text{ central},\ [g,h] = k \rangle
\end{align*}
is the Heisenberg group of order $27$. More precisely, we will prove (see  \hyperref[heis3-rigid]{Remark~\ref*{heis3-rigid}}, \hyperref[heis3-rho]{Corollary~\ref*{heis3-rho}} and  \hyperref[heis3-example]{Example~\ref*{heis3-example}}):

\begin{prop} \label{27-3-prop}
	Let $X = T/\Heis(3)$ be a hyperelliptic fourfold with associated complex representation $\rho$. Then:
	\begin{enumerate}[ref=(\theenumi)]
		\item \label{27-3-prop-1}Up to equivalence and automorphisms, $\rho$ is given as follows:
		\begin{align*}
		\rho(g) = \begin{pmatrix}
		0 & 0 & 1 & \\ 1 & 0 & 0 & \\ 0 & 1 & 0 & \\ & && 1
		\end{pmatrix}, \qquad \rho(h) = \begin{pmatrix}
		1 &&& \\ & \zeta_3^2 && \\ && \zeta_3 & \\ &&& \zeta_3
		\end{pmatrix}, \qquad \rho(k) = \begin{pmatrix}
		\zeta_3 & && \\ & \zeta_3 && \\ && \zeta_3 & \\ &&& 1
		\end{pmatrix}.
		\end{align*}
		\item The representation $\rho$ induces an equivariant isogeny $A \times E' \to T$, where $A \subset T$ is a complex sub-torus of dimension $3$, which isogenous to the third power of an elliptic curve $E \subset T$, and $E' \subset T$ is another elliptic curve. Both elliptic curves $E$ and $E'$ are isomorphic to Fermat's elliptic curve $F = \CC/(\ZZ+\zeta_3\ZZ)$.
		\item Hyperelliptic fourfolds with holonomy group $\Heis(3)$ exist.
	\end{enumerate}
	In particular, hyperelliptic manifolds with holonomy group $\Heis(3)$ are rigid.
\end{prop}

The Hodge diamond of a hyperelliptic fourfold $T/\Heis(3)$ is determined in part (D). Furthermore, \hyperref[27-3-prop]{Proposition~\ref*{27-3-prop}} allows us to give a  quick proof of the non-hyperellipticity of the following two groups containing $\Heis(3)$ in dimension $4$:

\begin{prop} \label{prop:cont-heis3-excluded}
	The following groups are not hyperelliptic in dimension $4$:
	\begin{enumerate}[ref=(\roman*)]
		\item $\Heis(3) \rtimes C_2$ (ID $[54,8]$). Here, $C_2 = \langle u \rangle$ acts on $\Heis(3) = \langle g,h,k ~ | ~ g^3 = h^3 = k^3 = 1, ~ [g,h] = k, ~ k \text{ central}\rangle$ via
		\begin{align*}
		u^{-1}gu = g^2, \qquad u^{-1}hu = h^2, \qquad uk = ku.
		\end{align*}
		\item $\Heis(3) \times C_2$ (ID $[54,10]$).
	\end{enumerate}
\end{prop}

\emph{(A) Representation Theory of $\Heis(3)$.} \\
Consider the Heisenberg group $\Heis(3) = \langle g,h,k \ | \ g^3 = h^3 = k^3, \ k \text{ central},\ [g,h] = k \rangle$
of order $27$. It is the unique non-Abelian group of order $27$ and exponent $3$. Observe that $\Heis(3)$ is generated by $g$ and $h$: we, however, find it convenient to denote by $k$ the commutator of $g$ and $h$. \\
The group $\Heis(3)$ has precisely nine irreducible representations of degree $1$. Since $k$ is a commutator, these are given by
\begin{align*}
\chi_{a,b}(g) = \zeta_3^a, \qquad \chi_{a,b}(h) = \zeta_3^b, \qquad \chi_{a,b}(k) = 1
\end{align*}
for $a,b \in \{0,1,2\}$. Furthermore, $\Heis(3)$ has two irreducible representations of degree $3$:
\begin{align*}
\rho_3(g) = \begin{pmatrix}
0 & 0 & 1 \\ 1 & 0 & 0 \\ 0 & 1 & 0
\end{pmatrix}, \qquad \rho_3(h) = \begin{pmatrix}
1 & & \\ & \zeta_3^2 & \\ && \zeta_3
\end{pmatrix}, \qquad \rho_3(k) = \begin{pmatrix}
\zeta_3 && \\ & \zeta_3 & \\ && \zeta_3
\end{pmatrix}
\end{align*}
and its complex conjugate $\overline{\rho_3}$. Consider  $\psi \in \Aut(\Heis(3))$ defined via 
\begin{align*}
\psi(g) = g^2, \qquad \psi(h) = h, \qquad \psi(k) = k^2. 
\end{align*}

As usual, we consider the action of $\Aut(\Heis(3))$ on the set of equivalence classes of irreducible representations:

%By evaluating the characters $\overline{\rho_3}$ and $\rho_3 \circ \psi$ at $k$, we see that $\overline{\rho_3}$ are $\rho_3 \circ \psi$ equivalent representations.\\
%Similarly, $\Aut(\Heis(3))$ acts transitively on the set $\{\chi_{a,b} \ | \ (a,b) \neq (0,0)\}$ of non-trivial linear characters. In fact, the following stronger statement holds:

\begin{lemma} \label{heis3-stabilizer}
	The following statements hold:
	\begin{enumerate}[ref=(\theenumi)]
		\item \label{heis3-stabilizer-1} An automorphism $\psi \in \Aut(\Heis(3))$ stabilizes the equivalence class of $\rho_3$ if and only if $\psi(k) = k$. Furthermore, there exists an automorphism $\varphi \in \Aut(\Heis(3))$ such that $\varphi(k) = k^2$. In particular, $\rho_3$ and $\overline{\rho_3}$ form an orbit under the action of $\Aut(\Heis(3))$.
		\item \label{heis3-stabilizer-2} The stabilizer $\Stab(\rho_3) = \{\psi \in \Aut(\Heis(3)) \ | \ \rho_3 \circ \psi \sim \rho_3 \}$ acts transitively on the set $\{\chi_{a,b} \ | \ (a,b) \neq (0,0)\}$ of non-trivial linear characters.
	\end{enumerate}
\end{lemma}

\begin{proof}
	\ref{heis3-stabilizer-1} Observe first of all that the center $Z(\Heis(3)) = \{1, k, k^2\}$ is a characteristic subgroup and hence $\psi(k) \in \{k,k^2\}$ for any $\psi \in \Aut(\Heis(3))$.\\ Now, assuming that $\psi$ stabilizes the class of $\rho_3$, then the matrices $\rho_3(\psi(k))$ and $\rho_3(k)$ are similar and hence $\psi(k) = k$. Conversely, if $\rho_3 \circ \psi \sim \overline{\rho_3}$, then clearly $\psi(k) = k^2$. \\
	An automorphism $\varphi$ as in the statement is given by $\varphi(g) = g^2$ and $\varphi(h) = h$: an easy computation using the defining relations of $\Heis(3)$ shows that $[\varphi(g),\varphi(h)] = k^2$. (Another possibility is to evaluate the character of $\rho_3 \circ \varphi$ at $k$.) \\
	
	\ref{heis3-stabilizer-2} By the orbit-stabilizer formula, The following automorphisms $\psi_i$ of $\Heis(3)$ stabilize $k$ (which -- as in the previous assertion -- can be seen by direct calculation, using either the relations of $\Heis(3)$ or by evaluating the character of $\rho_3 \circ \psi_i$ at $k$):
	\begin{center}
		\begin{tabular}{l|r}
			\begin{tabular}{l|llc}
				$i$ & $\psi_i(g)$ & $\psi_i(h)$ & $\chi_{1,0} \circ \psi_i$ \\ \hline 
				$1$ & $g$ & $h$ & $\chi_{1,0}$ \\
				$2$ & $h^2$ & $g$ & $\chi_{0,1}$ \\
				$3$ & $h$ & $g^2$ & $\chi_{0,2}$ \\
				$4$ & $g$ & $gh$ & $\chi_{1,1}$ \\
			\end{tabular} &
			\begin{tabular}{l|llc}
				$i$ & $\psi_i(g)$ & $\psi_i(h)$ & $\chi_{1,0} \circ \psi_i$ \\ \hline 
				$5$ & $g^2$ & $h^2$ & $\chi_{2,0}$ \\
				$6$ & $g^2$ & $gh^2$ & $\chi_{2,1}$ \\
				$7$ & $g$ & $g^2h$ & $\chi_{1,2}$ \\
				$8$ & $g^2$ & $g^2h^2$ & $\chi_{2,2}$
			\end{tabular}
		\end{tabular}
	\end{center}
\end{proof}

\emph{(B) The complex representation and the isogeny type of the torus.} \\

We determine the complex representation $\rho$ of a hyperelliptic fourfold $T/\Heis(3)$. Clearly, $\rho$ must contain an irreducible constituent of degree $3$. Since the two irreducible representations $\rho_3$ and $\overline{\rho_3}$ of degree $3$ form an $\Aut(\Heis(3))$-orbit, we may assume that the irreducible degree $3$ summand of $\rho$ is $\rho_3$. \\
We denote by $\chi$ the degree $1$ summand of $\rho$. 

\begin{lemma} \label{heis3-isog-1}
	The complex torus $T$ is equivariantly isogenous to $T' \times E$, where 
	\begin{itemize}
		\item $T' \subset T$ is a $3$-dimensional complex subtorus that itself is equivariantly isogenous to $F \times F \times F$, where $F = \CC/(\ZZ+\zeta_3\ZZ)$ is the Fermat elliptic curve, and
		\item $E \subset T$ is an elliptic curve. 
	\end{itemize}
\end{lemma}

\begin{proof}
	The existence of an equivariant isogeny $T' \times E \to T$ follows from  \hyperref[isogeny]{Section~\ref*{isogeny}}. The statement about the isogeny type of $T'$ follows because the central element $k$ acts on $T'$ by multiplication by a primitive third root of unity.
\end{proof}

\begin{lemma} \label{heis3-isog-2}
	The linear character $\chi$ is non-trivial. In particular, the elliptic curve $E$ is isomorphic to the Fermat elliptic curve $F$.
\end{lemma}

\begin{proof}
	If $\chi$ is trivial, the relation $[g,h] = k$ implies that $k$ acts trivially on $E$ and thus has a fixed point on $T \sim_{\Heis(3)-\isog} T' \times E$.
\end{proof}

\begin{rem} \label{heis3-rigid}
	Together,  \hyperref[heis3-isog-1]{Lemma~\ref*{heis3-isog-1}} and  \hyperref[heis3-isog-2]{Lemma~\ref*{heis3-isog-2}} show that $T$ is equivariantly isogenous to $F \times F \times F \times F$. In particular, any hyperelliptic fourfold $T/\Heis(3)$ is necessarily a \emph{rigid} compact complex manifold: this means that it does not admit any non-trivial local deformations. This is reflected by the fact that the rigidity of a hyperelliptic manifold is equivalent to requiring that the associated complex representation does not contain a pair of complex conjugate characters (in particular, it does not contain any real characters). For details, we refer to \cite{Ekedahl}. Out of all the examples of hyperelliptic fourfolds we discuss in this book, the only rigid one is the one given in  \hyperref[heis3-example]{Example~\ref*{heis3-example}} below. However, the restriction of the action given in the cited example to the subgroup $C_3 \times C_3 = \langle h,k \rangle$ remains rigid.
\end{rem}

We can now determine the complex representation of a hyperelliptic fourfold $T/\Heis(3)$.

\begin{cor} \label{heis3-rho}
	Up to equivalence and automorphisms, the complex representation $\rho$ of a hyperelliptic fourfold $T/\Heis(3)$ is $\rho=\rho_3 \oplus \chi_{0,1}$, i.e.:
	\begin{align*}
	\rho(g) \mapsto \begin{pmatrix}
	0 & 0 & 1 & \\ 1 & 0 & 0 & \\ 0 & 1 & 0 \\ &&& 1
	\end{pmatrix}, \qquad \rho(h) = \begin{pmatrix}
	1 &&& \\ & \zeta_3^2 && \\ &&\zeta_3 & \\ &&&\zeta_3
	\end{pmatrix}, \qquad \rho(k) = \begin{pmatrix}
	\zeta_3 &&& \\ & \zeta_3 && \\ && \zeta_3 & \\ &&& 1
	\end{pmatrix}.
	\end{align*}
\end{cor}

\begin{proof}
	The discussion in part (A) shows that $\rho_3$ and $\overline{\rho}_3$ are in the same $\Aut(\Heis(3))$-orbit: we may therefore assume without loss of generality that $\rho_3$ is the degree $3$ summand of $\rho$. As shown in \hyperref[heis3-isog-2]{Lemma~\ref*{heis3-isog-2}}, the degree $1$ summand $\chi$ is non-trivial. \hyperref[heis3-stabilizer]{Lemma~\ref*{heis3-stabilizer}} shows that we may assume $\rho_3$ to be the degree $3$ constituent of $\rho$ and that we can find an automorphism $\psi$ of $\Heis(3)$ that stabilizes (the equivalence class of) $\rho_3$ and $\chi \circ \psi = \chi_{0,1}$.
\end{proof}

\emph{(C) An example.} \\
According to part (B), we may write $T = (F \times F \times F \times F)/H$, where $H$ is a finite subgroup of translations. After appropriate changes of origins in the four elliptic curves, we may write the action on $T$ as follows:
\begin{align*}
&g(z) = (z_3 + a_1, \ z_1 + a_2, \ z_2 +a_3, \ z_4 + a_4), \\
&h(z) = (z_1 + b_1, \ \zeta_3^2z_2 + b_2, \ \zeta_3z_3 + b_3, \ \zeta_3 z_4), \\
&k(z) = (\zeta_3 z_1, \ \zeta_3 z_2, \ \zeta_3 z_3, \ z_4 + c_4).
\end{align*}
The following lemma tells us when we get a well-defined action of $\Heis(3)$.

\begin{lemma} \label{heis3-rels} \
	\begin{enumerate}[ref=(\theenumi)]
		\item The relation $g^3 = \id_T$ holds if and only if $(a_1+a_2+a_3, \ a_1+a_2+a_3, \ a_1+a_2+a_3, \ 3a_4)$ is zero in $T$.
		\item The relation $h^3 = \id_T$ holds if and only if $3b_2 = 0$ in $F$.
		\item The relation $k^3 = \id_T$ holds if and only if $3c_4 = 0$ in $F$.
		\item The elements $g$ and $k$ commute if and only if $((1-\zeta_3)a_1, \ (1-\zeta_3)a_2, \ (1-\zeta_3)a_3, \ 0)$ is zero in $T$.
		\item The elements $h$ and $k$ commute if and only if $((1-\zeta_3)b_1, \ (1-\zeta_3)b_2, \ (1-\zeta_3)b_3, \ (\zeta_3-1)c_4)$ is zero in $T$.
		\item \label{heis3-rels-k} The relation $g^{-1}h^{-1}gh = k$ holds if and only if $$ (\zeta_3-1)a_2 + \zeta_3(b_1 - b_2), \ (\zeta_3^2-1)a_3 + \zeta_3^2(b_2-b_3), \ b_3-b_1, \ (\zeta_3^2-1)a_4 - c_4)$$ is zero in $T$.
	\end{enumerate}
\end{lemma}

\begin{rem}
	The following elements form a system of representatives for the non-trivial conjugacy classes of $\Heis(3)$:
	\begin{align*}
	g, \quad g^2, \quad h, \quad h^2, \quad k, \quad k^2, \quad gh, \quad (gh)^2, \quad gh^2, \quad (gh^2)^2.
	\end{align*}
	Since all of these elements have order $3$, it suffices that the elements $g$, $h$, $k$, $gh$, and $gh^2$ act freely for the whole group $\Heis(3)$ to act freely.
\end{rem}

\begin{lemma} \label{heis3-freeness} \ 
	Let $T = F^4/H$ as above. Assume that every element of $H$ is of the form $(t_1,t_2,t_3,0)$, where $t_1+t_2+t_3 = 0$ and $(\zeta_3-1)t_i = 0$ for $i = \{1,2,3\}$.
	Then:
	\begin{enumerate}[ref=(\theenumi)]
		\item \label{heis3-freeness-g} The element $g$ acts freely on $T$ if and only if $a_4 \neq 0$ or $a_1+a_2+a_3 \neq 0$.
		\item \label{heis3-freeness-h} The element $h$ acts freely on $T$ if and only if $b_1 \neq 0$.
		\item \label{heis3-freeness-k} The element $k$ acts freely on $T$ if and only if $(\zeta_3-1)a_4 \neq 0$.
		\item \label{heis3-freeness-gh} The element $gh$ acts freely on $T$ if and only if $\zeta_3^2(a_1+a_2+b_1+b_3) + a_3 + b_2 \neq 0$.
		\item \label{heis3-freeness-gh^2} The element $gh^2$ acts freely on $T$ if and only if $\zeta_3(a_1+a_2+2b_1-b_2) + a_3-b_3 \neq 0$.
	\end{enumerate}
\end{lemma}

\begin{proof}
	\ref{heis3-freeness-g} A point $z = (z_1, ..., z_4)  \in T$ is fixed by $g$ if and only if
	\begin{align} \label{heis3-fixed-g}
	g(z) - z = (z_3 - z_1 + a_1, \ z_1 - z_2 + a_2, \ z_2 - z_3 +a_3, \ a_4)
	\end{align}
	is zero in $T$. Reading \ref{heis3-fixed-g} as an equation in $F^4$ and defining $w_1 := z_3-z_1$, $w_2 := z_1-z_2$, it is satisfied if and only if there is $(t_1,t_2,t_3,0) \in H$ such that
	\begin{align} \label{heis3-fixed-g-2}
	(w_1 + a_1, \ w_2 + a_2, \ -w_1-w_2 + a_3, \ a_4) = (t_1,t_2,t_3,0) \text{ in } F^4.
	\end{align}
	It is necessary for equation \ref{heis3-fixed-g-2} to have a solution that $w_1 = t_1-a_1$, $w_2 = t_2-a_2$ and $a_4 = 0$. Comparing the third coordinates of both sides yields
	\begin{align*}
	-w_1-w_2 + a_3 = -t_1-t_2 + a_1+a_2+a_3 = t_3 + a_1+a_2+a_3,
	\end{align*}
	so that \ref{heis3-fixed-g-2} has a solution if $a_1+a_2+a_3 = 0$. \\
	
	\ref{heis3-freeness-h} is clear. \\
	
	\ref{heis3-freeness-k} Clearly, the element 
	\begin{align*}
	k(z) = (\zeta_3 z_1, \ \zeta_3 z_2, \ \zeta_3 z_3, \ z_4 + c_4)
	\end{align*}
	acts freely on $T$ if and only if $c_4 \neq 0$. By  \hyperref[heis3-rels]{Lemma~\ref*{heis3-rels}} \ref{heis3-rels-k}, the element
	$$ (\zeta_3-1)a_2 + \zeta_3(b_1 - b_2), \ (\zeta_3^2-1)a_3 + \zeta_3^2(b_2-b_3), \ b_3-b_1, \ (\zeta_3^2-1)a_4 - c_4) $$ is zero in $T = F^4/H$. Equivalently, we may view it as an element in $F^4$ and require it to be contained in the kernel $H$. Then, by assumption, its last coordinate $(\zeta_3^2-1)a_4 - c_4$ is zero in $F$, i.e.,
	\begin{align*}
	c_4 = (\zeta_3^2-1)a_4.
	\end{align*}
	This means that $(\zeta_3-1)a_4 = 0$ is equivalent to $c_4 = 0$, which shows the desired statement. 
	
	\ref{heis3-freeness-gh}, \ref{heis3-freeness-gh^2} We leave the proofs as exercises, them being very similar to the proof of \ref{heis3-freeness-g}.
\end{proof}

\begin{rem}
	In \hyperref[heis3-freeness]{Lemma~\ref*{heis3-freeness}}, we require $H$ to be of a very special form. This is for the following reasons: in \hyperref[heis3-example]{Lemma~\ref*{heis3-example}} below, the torsion group $H$ is of the said form. This is not a coincidence: it can be shown that $H$ actually satisfies this property. We do not give a proof and precise statements here, but this investigation will be carried out in a forthcoming paper.
\end{rem}

\begin{example} \label{heis3-example}
	We define $T := (F \times F \times F \times F)/H$, where 
	\begin{align*}
	H = \left\langle \left(\frac{1-\zeta_3}{3}, \ \frac{1-\zeta_3}{3}, \ \frac{1-\zeta_3}{3}, \ 0 \right) \right\rangle, \qquad \text{ and } \qquad 
	F = \CC/(\ZZ + \zeta_3\ZZ).
	\end{align*}
	Furthermore, we define the following automorphisms of $F^4$:
	\begin{align*}
	&g(z) = \left(z_3, \ z_1, \ z_2, \ z_4 + \frac13\right), \\
	&h(z) = \left(z_1 + \frac13, \ \zeta_3^2 z_2 + \frac13, \ \zeta_3 z_3 + \frac13, \ \zeta_3 z_4\right), \\
	&k(z) = \left(\zeta_3 z_1, \ \zeta_3 z_2, \ \zeta_3 z_3, \ z_4 + \frac{\zeta_3^2-1}{3}\right).
	\end{align*}
	Since $\frac{1-\zeta_3}{3} \in F$ is fixed by multiplication by $\zeta_3$, the linear parts of $g$, $h$ and $k$ map $H$ to $H$. It follows that $g$, $h$ and $k$ indeed define automorphisms of $T$.  \hyperref[heis3-rels]{Lemma~\ref*{heis3-rels}} shows that $\langle g,h,k\rangle \subset \Bihol(T)$ is isomorphic to $\Heis(3)$ and does not contain any translations. Furthermore, the action on $T$ is free, see  \hyperref[heis3-freeness]{Lemma~\ref*{heis3-freeness}}. Thus $T/\langle g,h,k \rangle$ is a hyperelliptic fourfold with holonomy group $\Heis(3)$. 
\end{example}

\emph{(D) The Hodge diamond.} \\
The Hodge diamond of a hyperelliptic fourfold with holonomy group $\Heis(3)$ is 
\begin{center}
	$\begin{matrix}
	&   &  &  & 1 &  &  &  &  \\
	&   &  & 0 &  & 0 &  &  &  \\
	&   & 0 &  & 2 &  & 0  &  & \\
	&  1 &  & 1 &   & 1 &   &  1 &  \\
	0&    & 0 &  & 2  &  & 0  &     &  0 \\
	\end{matrix}$
\end{center}
In the calculation of the Hodge diamond we had to use that $\wedge^2 \chi_3 = \overline{\chi}_3$, where $\chi_3$ is the character of $\rho_3$. Indeed, if $V$ is the representation associated with $\rho_3$, then $V$ and $\wedge^2 V^*$ are equivalent representations because $\rho_3$ maps to $\SL(3,\CC)$. \\

\emph{(E) Proof of \hyperref[prop:cont-heis3-excluded]{Proposition~\ref*{prop:cont-heis3-excluded}}.}

We first show that $\Heis(3) \rtimes C_2$ is not hyperelliptic in dimension $4$. The action of $C_2 = \langle u \rangle$ on $\Heis(3) = \langle g,h,k ~ | ~ g^3 = h^3 = k^3 = 1, ~ [g,h] = k, ~ k \text{ central}\rangle$ is given by
\begin{align*}
u^{-1}gu = g^2, \qquad u^{-1}hu = h^2, \qquad uk = ku.
\end{align*}
In particular, $g$ and $h$ are commutators and thus mapped to $1$ by any linear character of $\Heis(3) \rtimes C_2$. However, the linear character of the complex representation of a hyperelliptic fourfold with holonomy group $\Heis(3)$ is non-trivial, see \hyperref[27-3-prop]{Proposition~\ref*{27-3-prop}} \ref{27-3-prop-1}. \\

We now exclude $\Heis(3) \times C_2$. Denote again a generator of $C_2$ by $u$. If $\rho \colon \Heis(3) \times C_2 \to \GL(4,\CC)$ is a faithful representation such that $\rho(u)$ has the eigenvalue $1$ and $\rho|_{\Heis(3)}$ is as in \hyperref[27-3-prop]{Proposition~\ref*{27-3-prop}} \ref{27-3-prop-1}, we easily find a matrix without the eigenvalue $1$ in the image of $\rho$:
\begin{align*}
&\rho(u) = \diag(-1, ~ -1, ~ -1, ~ 1) \implies \rho(hu) = \diag(-1, ~ - \zeta_3^2, ~ -\zeta_3, ~ \zeta_3) \text{ does not have the eigenvalue } 1, \\
&\rho(u) = \diag(1, ~ 1, ~ 1, ~ -1) \implies \rho(hk) = \diag(\zeta_3, ~ \zeta_3, ~ \zeta_3, ~ -1) \text{ does not have the eigenvalue } 1.
\end{align*}
This shows that $\Heis(3) \times C_2$ is not hyperelliptic in dimension $4$.

\subsection{$G(8,4,5)$ (ID [32,4])} \label{32-4-section}\  \\

We investigate hyperelliptic fourfolds $T/G(8,4,5)$, where
\begin{align*}
G(8,4,5) = \langle g,h \ | \ g^8 = h^4 = 1, ~ h^{-1}gh = g^5\rangle.
\end{align*}
Our result concerning these manifolds is as follows (see  \hyperref[32-4-rho]{Lemma~\ref*{32-4-rho}}, the text below the lemma and \hyperref[32-4-example]{Example~\ref*{32-4-example}}):

\begin{prop}
	Let $X = T/G(8,4,5)$ be a hyperelliptic fourfold with associated complex representation $\rho$. Then:
	\begin{enumerate}[ref=(\theenumi)]
		\item Up to equivalence and automorphisms, $\rho$ is given as follows:
		\begin{align*}
		\rho(g) = \begin{pmatrix}
		0 & i && \\ 1 & 0 && \\ && 1& \\ &&&1
		\end{pmatrix}, \qquad \rho(h) = \begin{pmatrix}
		i & && \\  & -i && \\ &&i & \\ &&& 1
		\end{pmatrix}.
		\end{align*}
		\item The representation $\rho$ induces an equivariant isogeny $E_i \times E_i \times E_i \times E \to T$, where $E \subset T$ is an elliptic curve and $E_i = \CC/(\ZZ+i\ZZ) \subset T$ is Gauss' elliptic curve.
		\item Hyperelliptic fourfolds with holonomy group $G(8,4,5)$ exist.
	\end{enumerate}
	In particular, $X$ moves in a complete $1$-dimensional family of hyperelliptic fourfolds with holonomy group $G(8,4,5)$.
\end{prop}

Furthermore, we give the Hodge diamond of a hyperelliptic fourfold $T/G(8,4,5)$ in part (D). \\

\emph{(A) Representation Theory of $G(8,4,5)$.} \\
Consider the metacyclic group $G(8,4,5) = \langle g,h \ | \ g^8 = h^4 = 1, \ h^{-1}gh = g^5\rangle$ of order $32$. Since $g^4$ is a commutator, we find that $G(8,4,5)$ has exactly $16$ representations of degree $1$. They are given by 
\begin{align*}
\chi_{a,b}(g) = i^a, \qquad \chi_{a,b}(h) = i^b
\end{align*}
for $a,b \in \{0,...,3\}$.
Furthermore, $G(8,4,5)$ has four irreducible degree $2$ representations up to equivalence:
\begin{center}
	\begin{tabular}{llll}
		$\rho_{2,1}$: & $g \mapsto \begin{pmatrix}
		0 & i \\ 1 & 0
		\end{pmatrix}$, & $h \mapsto \begin{pmatrix}
		i & \\ & -i
		\end{pmatrix}$, & $\ker(\rho_{2,1}) = \langle g^4h^2\rangle$, \\
		$\rho_{2,2}$: & $g \mapsto \begin{pmatrix}
		0 & -i \\ 1 & 0
		\end{pmatrix}$, & $h \mapsto \begin{pmatrix}
		i & \\ & -i
		\end{pmatrix}$, & $\ker(\rho_{2,2}) = \langle g^4h^2\rangle$, \\
		$\rho_{2,3}$: & $g \mapsto \begin{pmatrix}
		0 & i \\ 1 & 0
		\end{pmatrix}$, & $h \mapsto \begin{pmatrix}
		1 & \\ & -1
		\end{pmatrix}$, & $\ker(\rho_{2,3}) = \langle h^2\rangle$, \\
		$\rho_{2,4}$: & $g \mapsto \begin{pmatrix}
		0 & -i \\ 1 & 0
		\end{pmatrix}$, & $h \mapsto \begin{pmatrix}
		1 & \\ & -1
		\end{pmatrix}$, & $\ker(\rho_{2,1}) = \langle h^2\rangle$.
	\end{tabular}
\end{center}
These representations are non-faithful, reflecting that the center of $G(8,4,5)$ is the non-cyclic group generated by $g^2$ and $h^2$.
Moreover, the four degree $2$ representations above form an orbit under the action of $\Aut(G(8,4,5))$. Indeed,
\begin{align*}
&\rho_{2,1} \circ \psi_1 = \rho_{2,2}, \quad \text{where } \psi_1(g) = g^3h^3, \quad \psi_1(h) = h, \\
&\rho_{2,1} \circ \psi_2 = \rho_{2,3}, \quad \text{where } \psi_2(g) = g, \quad \psi_2(h) = g^6h, \\
&\rho_{2,1} \circ \psi_3 = \rho_{2,4}, \quad \text{where } \psi_3 = \psi_2 \circ \psi_1.
\end{align*}

\emph{(B) The complex representation and the isogeny type of the torus.} \\
Our main goal is to prove the existence of hyperelliptic fourfolds with holonomy group $G(8,4,5)$ -- indeed, we will give a concrete example in part (C). The discussion in this part however investigates the associated complex representation $\rho \colon G(8,4,5) \to \GL(4,\CC)$ of such a hyperelliptic fourfold $T/G(8,4,5)$. Furthermore, we determine the isogeny type of $T$ in this section. \\

First, we study the decomposition of $\rho$ into irreducible representations such that properties \ref{nec-prop1} -- \ref{nec-prop3} on p. \pageref{nec-prop1} are satisfied. The group $G(8,4,5)$ is metacyclic and consequently \hyperref[cor:metacyclic-rep]{Corollary~\ref*{cor:metacyclic-rep}} \ref{cor:metacyclic-rep-1} as well as \hyperref[lemma-two-generators]{Proposition~\ref*{lemma-two-generators}} implies that the complex representation $\rho$ splits as an irreducible degree $2$ summand and two linear characters $\chi$, $\chi'$, at least one of which is non-trivial.\\ 
The discussion in part (A) shows that all irreducible degree $2$ representations form an $\Aut(G(8,4,5))$-orbit: we may therefore assume without loss of generality that the irreducible degree $2$ summand of $\rho$ is (equivalent to) $\rho_{2,1}$. Since $\rho$ is faithful (property \ref{nec-prop1}) and the kernel of $\rho_{2,1}$ is generated by $g^4h^2$, we may assume that $\chi(h)$ is a primitive fourth root of unity. Furthermore, the matrix
\begin{align*}
\rho(h) = \diag(i, \ -i, \ \chi(h), \chi'(h))
\end{align*}
must have the eigenvalue $1$ (property \ref{nec-prop2}) and thus $\chi'(h) = 1$. \\
We now take the matrix
\begin{align*}
\rho(g) = \begin{pmatrix}
0 & i && \\ 1 & 0 && \\ && \chi(g) & \\ &&& \chi'(g)
\end{pmatrix}
\end{align*}
into account. Using property \ref{nec-prop2} again, we conclude that $\chi(g) = 1$ or $\chi'(g) = 1$. However, the matrices
\begin{align*}
\rho(gh) = \begin{pmatrix}
0 & 1 && \\ i & 0 && \\ && \chi(gh) & \\ &&& \chi'(g)
\end{pmatrix} \quad \text{ and } \quad	\rho(gh^2) = \begin{pmatrix}
0 & -i && \\ -1 & 0 && \\ && -\chi(g) & \\ &&& \chi'(g)
\end{pmatrix}
\end{align*}
must have the eigenvalue $1$ as well: we obtain that $\chi'(g) = 1$. \\ 
The next lemma completes the description of the complex representation $\rho$, the first goal of this part.

\begin{lemma} \label{32-4-rho}
	The two automorphisms 
	\begin{align*}
	g \mapsto g, \qquad h \mapsto h^3,\qquad\qquad \text{ and } \qquad\qquad g \mapsto gh^3, \qquad h \mapsto h
	\end{align*}
	both stabilize the equivalence class of $\rho_{2,1}$. In particular: the representation $\rho$ is equivalent to $\rho_{2,1} \oplus \chi_{0,1} \oplus \chi_{0,0}$ up to an automorphism of $G(8,4,5)$.
\end{lemma}

\begin{proof}
	The first assertion follows from a GAP computation: we check that the listed automorphisms leave the character of $\rho_{2,1}$ invariant. The second assertion follows from the first one: after possibly applying the first listed automorphism, we may assume that $\chi(h) = i$, and by applying the second listed automorphism several times, we may assume that $\chi(g) = 1$. The statement now follows since these automorphisms leave $\rho_{2,1}$ invariant.  
\end{proof}
%g := [[0,E(4)][1,0]];
%h := DiagonalMat([E(4),-E(4)]);
%g2 := g*h^3;
%h2 := h^3;
%sum := 0;
%for i in [0..7] do
%for j in [0..3] do
%u := g^i*h^j;
%u2 := g^i*h2^j;
%sum := sum + TraceMat(u) - TraceMat(u2);
%od;
%od;

Finally, having completed the description of the complex representation $\rho$, we observe that $\rho$ is the direct sum of three irreducible representations that lie in different $\Aut(\CC)$-orbits. Applying the decomposition procedure described in \hyperref[isogeny]{Section~\ref*{isogeny}} shows that the complex torus $T$ is equivariantly isogenous to a product of a $2$-dimensional complex torus $S \subset T$ and two elliptic curves $E, E' \subset T$: here, $G(8,4,5)$ acts on $S$ via $\rho_{2,1}$ and on $E, E'$ via $\chi$, $\chi'$, respectively. It follows from $\chi(h) = i$ that the elliptic curve $E$ is isomorphic to Gauss' elliptic curve $E_i = \CC/(\ZZ + i \ZZ)$. Similarly, $\rho_2(g^2) = \diag(i, \ i)$, and hence $S$ is isogenous\footnote{In fact, $S$ is even \emph{isomorphic} to $E_i \times E_i$. However, an isomorphism does not have to be equivariant.} to $E_i \times E_i$. In total, we have established that $T$ is equivariantly isogenous to $E_i \times E_i \times E_i \times E'$. \\

\emph{(C) An example.} \\
We now construct an example of a hyperelliptic fourfold with holonomy group $G(8,4,5)$. Part (B) shows that we may write
\begin{align*}
T = (E_i \times E_i \times E_i \times E)/H,
\end{align*}
where $E_i = \CC/(\ZZ+i\ZZ)$ and $E = \CC/(\ZZ+\tau\ZZ)$ is another elliptic curve. By a change of origin in the respective elliptic curves, we may write the action $G(8,4,5)$ on $T$ in the following form:
\begin{align*}
&g(z) = \left(iz_2, \ z_1, \ z_3 + a_3, \ z_4 + a_4\right), \\
&h(z) = \left(iz_1 + b_1, \ -iz_2 + b_2, \ iz_3, \ z_4 + b_4\right).
\end{align*}
We determine the conditions on the translation parts of $g$ and $h$ such that $G(8,4,5) \cong \langle g,h\rangle \subset \Bihol(T)$.

\begin{lemma} \label{32-4-rels} \
	\begin{enumerate}[ref=(\theenumi)]
		\item The relation $g^8 = \id_T$ holds if and only if $(0,\ 0, \ 8a_3, \ 8a_4)$ is zero in $T$.
		\item The relation $h^4 = \id_T$ holds if and only if $4b_4 = 0$ in $E$.
		\item The relation $h^{-1}gh = g^5$ holds if and only if $(ib_1+b_2, \ ib_1 - ib_2, \ -(i-5)a_3, \ -4a_4)$ is zero in $T$.
	\end{enumerate}
\end{lemma}

%\begin{rem}
%The group $G(8,4,5)$ has $19$ non-trivial conjugacy classes. The following elements form a system of representatives for these classes:
%\begin{align*}
%\underline{g}, \quad \underline{g^2}, \quad g^3, \quad \underline{g^4}, \quad g^6, \quad \underline{h}, \quad \underline{h^2}, \quad h^3, \quad \underline{g^2h^2}, \quad (g^2h^2)^3, \quad \underline{g^4h^2}, \\ \underline{gh}, \quad (gh)^3, \quad \underline{g^2h}, \quad (g^2h)^3, \quad \underline{h^2g}, \quad (h^2g)^3, \quad \underline{h^3g}, \quad (h^3g)^3.
%\end{align*}
%Here, the elements that are not third powers of previously listed elements are underlined. Since $G(8,4,5)$ is a $2$-group, we only need to prove that the eleven underlined elements act freely. 
%\end{rem}

Finally, we give¸ an example of a hyperelliptic fourfold with holonomy group $G(8,4,5)$.

\begin{example} \label{32-4-example}
	Let $T := T'/H$, where $T' := E_i \times E_i \times E_i \times E$, $E_i = \CC/(\ZZ+i\ZZ)$, $E = \CC/(\ZZ+ \tau \ZZ)$ and $H \subset T'$ is the subgroup spanned by $\left(0,\ 0, \ \frac{1+i}{2}, \frac12\right)$. We define the following automorphisms of $T'$:
	\begin{align*}
	&g(z) = \left(iz_2, \ z_1, \ z_3 + \frac12, \ z_4 + \frac18\right), \\
	&h(z) = \left(iz_1, \ -iz_2, \ iz_3, \ z_4 + \frac{\tau}{4}\right).
	\end{align*}
	By construction, their linear parts map $H$ to $H$ and hence these maps define elements $g,h \in \Bihol(T)$.  \hyperref[32-4-rels]{Lemma~\ref*{32-4-rels}} and our construction now immediately show that the subgroup of $\Bihol(T)$ that is spanned by $g$ and $h$ is isomorphic to $G(8,4,5)$ and does not contain any translation. \\
	The freeness of the action of $G(8,4,5)$ on $T$ is proved as follows. The last two coordinates of $g^jb^k(z)$ ($0 \leq j < 8$, $0 \leq k < 4$, $(j,k) \neq (0,0)$) are as follows:
	\begin{align*}
	\left(i^kz_4 + \frac{j}{2}, \ z_4 + \frac{j}{8} + \frac{k\tau}{4}\right).
	\end{align*}
	Given our definition of $H$, these elements act freely on $T$, unless possibly when the last coordinate is a translation by $\frac12$: this is the case if and only if $j = 4$ and $k = 0$ and in this case, the element is given by
	\begin{align*}
	g^4(z) = \left(-z_1, \ -z_2, z_3, z_4 + \frac12 \right).
	\end{align*}
	This element acts freely by the definition of $H$ as well. We have therefore proved that $T/\langle g,h\rangle$ is a hyperelliptic fourfold with holonomy group $G(8,4,5)$.
\end{example}

\emph{(D) The Hodge diamond.} \\
The Hodge diamond of a hyperelliptic fourfold with holonomy group $G(8,4,5)$ is 
\begin{center}
	$\begin{matrix}
	&   &  &  & 1 &  &  &  &  \\
	&   &  & 1 &  & 1 &  &  &  \\
	&   & 0 &  & 3 &  & 0  &  & \\
	&  0 &  & 2 &   & 2 &   &  0 &  \\
	0&    & 0 &  & 4  &  & 0  &     &  0 \\
	\end{matrix}$
\end{center}

\subsection{$C_4^2 \rtimes C_2$ (ID [32,11])} \label{32-11-section}\  \\ 

Consider the action of $C_2 = \langle c \rangle$ on $C_4^2 = \langle a,b \rangle$ that is given by
\begin{align*}
c^{-1}ac = a^{-1}, \qquad c^{-1}bc = ab.
\end{align*}
The resulting semidirect product $C_4^2 \rtimes C_2$ is labeled [32,11] in the Database of Small Groups. In this section, we show the following structure result on hyperelliptic fourfolds with holonomy group $C_4^2 \rtimes C_2$  \hyperref[32-11-rem-rho]{Remark~\ref*{32-11-rem-rho}} and \hyperref[32-11-example]{Example~\ref*{32-11-example}}):

\begin{prop} \label{32-11-prop}
	Let $X = T/(C_4^2 \rtimes C_2)$ be a hyperelliptic fourfold with associated complex representation $\rho$. Then:
	\begin{enumerate}[ref=(\theenumi)]
		\item \label{32-11-prop1} Up to equivalence and automorphisms, $\rho$ is given as follows:
		\begin{align*}
		\rho(a) = \begin{pmatrix}
		i &  && \\  & -i && \\ && 1& \\ &&&1
		\end{pmatrix}, \qquad \rho(b) = \begin{pmatrix}
		i &  && \\  & -1 && \\ &&1 & \\ &&& 1
		\end{pmatrix}, \qquad \rho(c) = \begin{pmatrix}
		0 & 1 && \\ 1 & 0 && \\ && -1 & \\ &&&1
		\end{pmatrix}.
		\end{align*}
		\item The representation $\rho$ induces an equivariant isogeny $E_i \times E_i \times E \times E' \to T$, where $E, E' \subset T$ are elliptic curves and $E_i = \CC/(\ZZ+i\ZZ) \subset T$ are copies of Gauss' elliptic curve. 
		\item Hyperelliptic fourfolds with holonomy group $C_4^2\rtimes C_2$ exist.
	\end{enumerate}
	In particular, $X$ moves in a complete $2$-dimensional family of hyperelliptic fourfolds with holonomy group $C_4^2\rtimes C_2$.
\end{prop}

The Hodge diamond of such manifolds will be given in part (D). \\

\emph{(A) Representation Theory of $C_4^2 \rtimes C_2$.} \\

We summarize the representation theory of $C_4^2 \rtimes C_2$. Starting with the degree $1$ representations, observe that the relation $c^{-1}bc = ab$ implies that $a$ is a commutator. It follows that $C_4^2 \rtimes C_2$ has the following linear characters:
\begin{align*}
\chi_{n,m}(a) = 1, \qquad \chi_{n,m}(b) = i^n, \qquad \chi_{n,m} = (-1)^m, \qquad \text{where } n \in \{0,...,3\}, \quad m \in \{0,1\}.
\end{align*}
Up to equivalence, the irreducible degree $2$ representations of $C_4^2 \rtimes C_2$ are as follows:
\begin{center}
	\begin{tabular}{llll}
		$\rho_{2,1}$: & $a \mapsto \begin{pmatrix}
		i & \\ & -i
		\end{pmatrix}$ & $b \mapsto \begin{pmatrix}
		i & \\ & -1
		\end{pmatrix}$ & $c \mapsto \begin{pmatrix}
		0 & 1 \\ 1 & 0
		\end{pmatrix}$ \\
		$\rho_{2,2}$: & $a \mapsto \begin{pmatrix}
		i & \\ & -i
		\end{pmatrix}$ & $b \mapsto \begin{pmatrix}
		-1 & \\ & -i
		\end{pmatrix}$ & $c \mapsto \begin{pmatrix}
		0 & 1 \\ 1 & 0
		\end{pmatrix}$ \\
		$\rho_{2,3}$: & $a \mapsto \begin{pmatrix}
		i & \\ & -i
		\end{pmatrix}$ & $b \mapsto \begin{pmatrix}
		1 & \\ & i
		\end{pmatrix}$ & $c \mapsto \begin{pmatrix}
		0 & 1 \\ 1 & 0
		\end{pmatrix}$ \\
		$\rho_{2,4}$: & $a \mapsto \begin{pmatrix}
		i & \\ & -i
		\end{pmatrix}$ & $b \mapsto \begin{pmatrix}
		-i & \\ & 1
		\end{pmatrix}$ & $c \mapsto \begin{pmatrix}
		0 & 1 \\ 1 & 0
		\end{pmatrix}$ \\
		$\rho_{2,5}$: & $a \mapsto \begin{pmatrix}
		-1 & \\ & -1
		\end{pmatrix}$ & $b \mapsto \begin{pmatrix}
		0 & -1 \\ 1& 0
		\end{pmatrix}$ & $c \mapsto \begin{pmatrix}
		1 &  \\  & -1
		\end{pmatrix}$ \\
		$\rho_{2,6}$: & $a \mapsto \begin{pmatrix}
		-1 & \\ & -1
		\end{pmatrix}$ & $b \mapsto \begin{pmatrix}
		0 & 1 \\ 1& 0
		\end{pmatrix}$ & $c \mapsto \begin{pmatrix}
		1 &  \\  & -1
		\end{pmatrix}$
	\end{tabular}
\end{center}

%[ [ -1, 0 ], [ 0, -1 ] ]      [ [ 0, -1 ], [ 1, 0 ] ]             [ [ 1, 0 ], [ 0, -1 ] ]
%
%[ [ -1, 0 ], [ 0, -1 ] ]      [ [ 0, 1 ], [ 1, 0 ] ]             [ [ 1, 0 ], [ 0, -1 ] ]
%
%[ [ E(4), 0 ], [ 0, -E(4) ] ]      [ [ -1, 0 ], [ 0, -E(4) ] ]             [ [ 0, 1 ], [ 1, 0 ] ]
%
%[ [ E(4), 0 ], [ 0, -E(4) ] ]      [ [ 1, 0 ], [ 0, E(4) ] ]             [ [ 0, 1 ], [ 1, 0 ] ]
%
%[ [ E(4), 0 ], [ 0, -E(4) ] ]      [ [ -E(4), 0 ], [ 0, 1 ] ]             [ [ 0, 1 ], [ 1, 0 ] ]
%
%[ [ E(4), 0 ], [ 0, -E(4) ] ]      [ [ E(4), 0 ], [ 0, -1 ] ]             [ [ 0, 1 ], [ 1, 0 ] ]

In contrast to $\rho_{2,1}, ..., \rho_{2,4}$, the representations $\rho_{2,5}$ and $\rho_{2,6}$ are non-faithful, since they map the element $a$ of order $4$ to a matrix of order $2$. \\

The equivalence classes corresponding to the faithful representations $\rho_{2,1}, ..., \rho_{2,4}$ form an orbit under the action of $\Aut(C_4^2 \rtimes C_2)$. Indeed, 
\begin{align*}
&\rho_{2,1} \circ \psi_1 = \rho_{2,2}, \qquad \text{where} \quad \psi_1(a) = a, \quad \psi_1(b) = a^3b^3, \quad \psi_1(c) = c, \\
&\rho_{2,1} \circ \psi_2 = \rho_{2,3}, \qquad \text{where}\quad \psi_1(a) = a, \quad \psi_1(b) = a^3b, \quad \psi_1(c) = c, \\
&\rho_{2,1} \circ \psi_3 = \rho_{2,4}, \qquad \text{where}\quad \psi_3 = \psi_2 \circ \psi_1, \quad \text{i.e.,} \quad \psi_3(b)=a^2b.
\end{align*}

\emph{(B) The complex representation and the isogeny type of the torus.} \\
We determine the complex representation $\rho \colon C_4^2 \rtimes C_2$ of a hyperelliptic fourfold $T/(C_4^2 \rtimes C_2)$ in several steps. We obtain the isogeny type of the complex torus $T$ as a byproduct. \\

As a first step, we decompose $\rho$ into irreducible representations such that properties \ref{nec-prop1} -- \ref{nec-prop3} on p. \pageref{nec-prop1} are satisfied, i.e., such that $\rho$ is faithful and every matrix in the image of $\rho$ has the eigenvalue $1$: any irreducible degree $2$ representation of $C_4^2 \rtimes C_2$ maps $a$ to a matrix without the eigenvalue $1$ and thus $\rho$ is the direct sum of three irreducible representations whose respective degrees are $2$, $1$, $1$. Furthermore, as the kernel of the non-faithful irreducible degree $2$ representations $\rho_{2,5}$ and $\rho_{2,6}$ intersects the derived subgroup $\langle a \rangle$ of $C_4^2 \rtimes C_2$ non-trivially, the irreducible degree $2$ summand of $\rho$ must be faithful. According to part (A), the faithful irreducible representations of $C_4^2 \rtimes C_2$ form an orbit under the action of $\Aut(C_4^2 \rtimes C_2)$. This proves that we may assume without loss of generality that $\rho_{2,1}$ is a constituent of $\rho$. \\

\begin{rem} \label{32-11-rem-rho}
	The group $C_4^2 \rtimes C_2$ contains the following metacyclic subgroups:
	\begin{align*}
	\langle a,c \rangle \cong D_4, \qquad \langle a, b^2c \rangle \cong Q_8, \qquad \langle a,bc\rangle \cong G(8,2,5).
	\end{align*}
	Denoting by $\chi$ and $\chi'$ the constituents of $\rho$ of degree $1$, \hyperref[lemma-two-generators]{Proposition~\ref*{lemma-two-generators}} and the observation that $\chi(a) = \chi'(a) = 1$ show that the following statements hold:
	\begin{itemize}
		\item $\chi(c) = -1$ or $\chi'(c) = -1$, 
		\item $\chi(b^2c) = -1$ or $\chi'(b^2c) = -1$,
		\item $\chi(bc) \neq 1$ or $\chi'(bc) \neq 1$.
	\end{itemize}
	We may therefore assume that $\chi(c) = -1$. The three bullet points above and the obstruction that the three matrices 
	\begin{align*}
	\rho(b) = \begin{pmatrix}
	i &&& \\ & -1 && \\ && \chi(b) & \\ &&& \chi'(b)
	\end{pmatrix}, \quad \rho(b^2c) = \begin{pmatrix}
	0 & -1 && \\ 1 & 0 && \\ && -\chi(b^2) & \\ &&& \chi'(b^2c)
	\end{pmatrix} \quad	\rho(bc) = \begin{pmatrix}
	0 & i && \\ -1 & 0 && \\ && -\chi(b) & \\ &&& \chi'(bc)
	\end{pmatrix}
	\end{align*}
	must have the eigenvalue $1$ (property \ref{nec-prop2}) leads us to $\chi(b) = \chi'(b) = \chi'(c) = 1$. The linear characters $\chi$ and $\chi'$ are now completely determined.
\end{rem}

The knowledge of the complex representation $\rho$ makes determining the isogeny type of $T$: $\rho$ decomposes in three pairwise non-equivalent representations, all of which are in different orbits of $\Aut(\CC)$. The application of the decomposition algorithm of  \hyperref[isogeny]{Section~\ref*{isogeny}} thus gives an equivariant isogeny $S \times E \times E' \to T$, where $S \subset T$ is a $2$-dimensional complex torus and $E,E' \subset T$ are elliptic curve. Since $\rho(bc)^6$ acts on $S$ by multiplication by $i$, we conclude that $S$ is equivariantly isogenous to $E_i \times E_i$, where $E_i = \CC/(\ZZ + i \ZZ)$ is Gauss' elliptic curve, the unique elliptic curve admitting a linear automorphism of order $4$. \\

\emph{(C) An Example.}

We give an example of a hyperelliptic fourfold $T/(C_4^2 \rtimes C_2)$. As proved above, we may assume that $T = (E_i \times E_i \times E \times E')/H$, where $E_i = \CC/(\ZZ + i \ZZ)$, $E = \CC/(\ZZ+\tau\ZZ)$ and $E' = \CC/(\ZZ+\tau'\ZZ)$. Up to a change of origin, we may write
\begin{align*}
&a(z) = (iz_1, \ -iz_2, \ z_3 + a_3, \ z_4 + a_4), \\
&b(z) = (iz_1 + b_1, \ -z_2 + b_2, \ z_3 + b_3, \ z_4 + b_4), \\
&c(z) = (z_2 + c_1, \ z_1 + c_2, \ -z_3, \ z_4 + c_4). 	
\end{align*}

The following lemma tells us the conditions we have to impose on the translation parts of $a$, $b$ and $c$ in order for the subgroup $\langle a,b,c \rangle$ of $\Bihol(T)$ to be isomorphic to $C_4^2 \rtimes C_2$.

\begin{lemma} \label{32-11-relations} \
	\begin{enumerate}[ref=(\theenumi)]
		\item \label{32-11-rels-first} The relation $a^4 = \id_T$ holds if and only if $(0, \ 0,\ 4a_3, \ 4a_4)$ is zero in $T$.
		\item The relation $b^4 = \id_T$ holds if and only if $(0, \ 0,\ 4b_3, \ 4b_4)$ is zero in $T$.
		\item The relation $c^2 = \id_T$ holds if and only if $(c_1+c_2, \ c_1+c_2,\ 0, \ 2c_4)$ is zero in $T$.
		\item The elements $a$ and $b$ commute if and only if $((1-i)b_1,\ (1+i)b_2, \ 0, \ 0)$ is zero in $T$.
		\item The relation $c^{-1}ac = a^{-1}$ holds if and only if $(-(1+i)c_2, \ (i-1)c_1,\ 0, \ 2a_4)$ is zero in $T$.
		\item \label{32-11-rels-last} The relation $c^{-1}bc = ab$ holds if and only if $(2c_2+ib_1-b_2,\ (1-i)c_1 - b_1 - ib_2, \ a_3+2b_3, \ a_4)$ is zero in $T$.
	\end{enumerate}
	Consequently, if the ten properties \ref{32-11-rels-first} -- \ref{32-11-rels-last} are satisfied, then the group $\langle a,b,c \rangle \subset \Bihol(T)$ is isomorphic to $C_4^2 \rtimes C_2$ and does not contain any translations.
\end{lemma}

\begin{rem} \label{32-11-classes}
	The following elements form a system of representatives for the $13$ non-trivial conjugacy classes of $C_4^2 \rtimes C_2$:
	\begin{align*}
	\underline{a}, \quad \underline{a^2}, \quad \underline{b}, \quad \underline{b^2}, \quad b^3, \quad \underline{c}, \quad \underline{bc}, \quad \underline{(bc)^2}, \quad (bc)^3, \quad (bc)^6, \quad \underline{a^2b}, \quad (a^2b)^3, \quad \underline{b^2c}.
	\end{align*}
	The underlined elements are precisely the ones that are not conjugate to the third powers of elements that were already listed before. Since $C_4^2 \rtimes C_2$ is a $2$-group, we only have to prove that the underlined nine elements act freely for the whole group $C_4^2 \rtimes C_2$ to act freely.
\end{rem}

Presenting an example of a hyperelliptic fourfold $T/(C_4^2 \rtimes C_2)$ is now feasible.

\begin{example} \label{32-11-example}
	Let $T = (E_i \times E_i \times E \times E'/H$, where 
	\begin{itemize}
		\item $E_i = \CC/(\ZZ+i\ZZ)$ is Gauss' elliptic curve,
		\item $E = \CC/(\ZZ + \tau\ZZ)$ and $E' = \CC/(\ZZ+\tau'\ZZ)$ are arbitrary elliptic curves in standard form, and
		\item $H$ is the subgroup of order $2$ spanned by $\left(0, \ 0, \ \frac12, \ \frac12\right)$.
	\end{itemize}
	Consider the following automorphisms of $E_i \times E_i \times E \times E'$:
	\begin{align*}
	&a(z) = \left(iz_1, \ -iz_2, \ z_3 + \frac14, \ z_4\right), \\
	&b(z) = \left(iz_1, \ -z_2, \ z_3 - \frac18, \ z_4 + \frac{1}8\right), \\
	&c(z) = \left(z_2, \ z_1, \ -z_3, \ z_4 + \frac{\tau'}{2} \right).
	\end{align*}
	As usual, $a$, $b$ and $c$ descend to holomorphic automorphisms of $T$. Viewed as such,  \hyperref[32-11-relations]{Lemma~\ref*{32-11-relations}} implies that the group $\langle a,b,c \rangle \subset \Bihol(T)$ is isomorphic to $C_4^2 \rtimes C_2$ and a fortiori does not contain any translations. \\
	The nine underlined elements from  \hyperref[32-11-classes]{Remark~\ref*{32-11-classes}} all act freely by construction; this can be seen fairly easily by investigating the third and fourth coordinates of these elements and comparing them with our definition of $T$.

	%except possibly $g^4$ and $ghg^{-1}h^{-1}$ act freely on $T$, since they act on $E$ or $E'$ by a translation of order $\geq 4$. Since
	%\begin{align*}
	%	&g^4(z) = \left(-z_1, \ -z_2, \ z_3 + \frac12, \ z_4\right), \\
	%	&ghg^{-1}h^{-1}(z) = \left(iz_1, \ -iz_2, \ z_3, \ z_4 - \frac14\right),
	%\end{align*}
	%the remaining two elements act freely as well. 
	In total, we have proved the existence of a hyperelliptic fourfold with holonomy group $C_4^2 \rtimes C_2$.
\end{example}

\emph{(C) The Hodge diamond.} \\
The Hodge diamond of a hyperelliptic fourfold with holonomy group $C_4^2 \rtimes C_2$ (ID [32,11]) is
\begin{center}
	$\begin{matrix}
	&   &  &  & 1 &  &  &  &  \\
	&   &  & 1 &  & 1 &  &  &  \\
	&   & 0 &  & 3 &  & 0  &  & \\
	&  0 &  & 3 &   & 3 &   &  0 &  \\
	0&    & 1 &  & 6  &  & 1  &     &  0 \\
	\end{matrix}$
\end{center}

\subsection{$C_4^2 \rtimes C_2$ (ID [32,24])} \label{32-24-section}\  \\ 

Consider the following action of $C_2 = \langle c \rangle$ on $C_4^2 = \langle a,b \rangle$:
\begin{align*}
[a,c] = 1, \qquad c^{-1}bc = a^2b.
\end{align*}
The resulting group is a semidirect product $C_4^2 \rtimes C_2$ -- its ID in the Database of Small Groups is [32,24]. In this section, we show the following result (see \hyperref[32-24-example]{Example~\ref*{32-24-example}} and the discussion in part (B), especially \hyperref[32-24-rho]{Lemma~\ref*{32-24-rho}}):

\begin{prop}  \label{32-24-prop}
	Let $X = T/(C_4^2 \rtimes C_2)$ be a hyperelliptic fourfold with associated complex representation $\rho$. Then:
	\begin{enumerate}[ref=(\theenumi)]
		\item Up to equivalence and automorphisms, $\rho$ is given as follows:
		\begin{align*}
		\rho(a) = \begin{pmatrix}
		i &  && \\  & i && \\ && 1& \\ &&&1
		\end{pmatrix}, \qquad \rho(b) = \begin{pmatrix}
		0 & -1 && \\ 1 & 0 && \\ &&i & \\ &&& 1
		\end{pmatrix}, \qquad \rho(c) = \begin{pmatrix}
		1 & && \\ & -1 && \\ && 1 & \\ &&&1
		\end{pmatrix}.
		\end{align*}
		\item The representation $\rho$ induces an equivariant isogeny $E_i \times E_i \times E_i \times E \to T$, where $E \subset T$ is an elliptic curve and $E_i = \CC/(\ZZ+i\ZZ) \subset T$ are copies of Gauss' elliptic curve.
		\item Hyperelliptic fourfolds with holonomy group $C_4^2\rtimes C_2$ exist.
	\end{enumerate}
	In particular, $X$ moves in a complete $1$-dimensional family of hyperelliptic fourfolds with holonomy group $C_4^2\rtimes C_2$.
\end{prop}

Moreover, we present the Hodge diamond of such manifolds in part (D). \\

\emph{(A) Representation Theory of $C_4^2 \rtimes C_2$.}

We recall the representation theory of $C_4^2 \rtimes C_2$. 
Observe that since $a$ commutes with both $b$ and $c$, it is a central element of $C_4^2 \rtimes C_2$. Furthermore, one can see from the above relations that $b^2$ is a central element as well: this shows that the center of $C_4^2 \rtimes C_2$ is non-cyclic, and hence this group does not have a faithful irreducible representation. Indeed, its irreducible representations of degree $2$ are $\rho_{2,1}$, $\overline{\rho_{2,1}}$, $\rho_{2,2}$ and $\overline{\rho_{2,2}}$, where
\begin{tabular}{llll}
	$\rho_{2,1}:$ & $a \mapsto \begin{pmatrix}
	i & \\ & i
	\end{pmatrix}$, & $b \mapsto \begin{pmatrix}
	0 & -1 \\ 1 & 0
	\end{pmatrix}$, & $c \mapsto \begin{pmatrix}
	1 & \\ & -1
	\end{pmatrix}$, \\
	$\rho_{2,2}:$ & $a \mapsto \begin{pmatrix}
	i & \\ & i
	\end{pmatrix}$, & $b \mapsto \begin{pmatrix}
	0 & 1 \\ 1 & 0
	\end{pmatrix}$, & $c \mapsto \begin{pmatrix}
	1 & \\ & -1
	\end{pmatrix}$.
\end{tabular} \\
These four representations form an orbit under the action of $\Aut(C_4^2 \rtimes C_2)$. Indeed,
\begin{align*}
&\rho_{2,1} \circ \varphi_1 = \overline{\rho_{2,1}}, \qquad \text{where } \varphi_1(a) = a^3, \quad \varphi_1(b) = b, \quad \varphi_1(c) = c, \\
&\rho_{2,1} \circ \varphi_2 = \rho_{2,2}, \qquad \text{where } \varphi_2(a) = a, \quad \varphi_2(b) = bc, \quad \varphi_2(c) = c, \\
&\rho_{2,1} \circ \varphi_3 = \overline{\rho_{2,2}}, \qquad \text{where } \varphi_3 = \varphi_2 \circ \varphi_1.
\end{align*}
The relation $c^{-1}bc = a^2b$ implies that $a^2$ is a commutator and thus the following list exhausts all linear characters of $C_4^2 \rtimes C_2$:
\begin{align*}
a \mapsto (-1)^n, \qquad b \mapsto i^m, \qquad c \mapsto (-1)^\ell, \qquad \text{where } n,\ell \in \{0,1\}, \quad m \in \{0,...,3\}.
\end{align*}

\emph{(B) The complex representation and the isogeny type of the torus.} \\
In this part, we describe the complex representation $\rho \colon C_4^2 \rtimes C_2 \to \GL(4,\CC)$ of a hyperelliptic fourfold $T/(C_4^2 \rtimes C_2)$. Moreover, we determine the isogeny type of $T$. \\
As usual, we first decompose $\rho$ into irreducible representations such that properties \ref{nec-prop1} -- \ref{nec-prop3} on p. \pageref{nec-prop1} are satisfied. \hyperref[rho2'-reducible]{Lemma~\ref*{rho2'-reducible}} implies that the complex representation $\rho$ is a direct sum of three irreducible representations of respective degrees $2$, $1$, $1$. Since $\Aut(C_4^2 \rtimes C_2)$ acts transitively on the set of degree $2$ representations, we may assume that the degree $2$ summand of $\rho$ is $\rho_{2,1}$. \\

\begin{lemma} \label{32-24-rho}
	Denote by $\chi$ and $\chi'$ the degree $1$ summands of $\rho$. Then there is an automorphism $\psi \in \Aut(C_4^2 \rtimes C_2)$ that stabilizes the equivalence class of $\rho_{2,1}$ and $\{\chi \circ \psi, \ \chi' \circ \psi\} = \{\chi_{\triv}, \ \chi_{0,1,0}\}$, where $\chi_{0,1,0}$ is the linear character defined by $a \mapsto 1$, $b \mapsto i$, $c \mapsto 1$. 
\end{lemma}

\begin{proof}
	We first show that $\rho$ contains the trivial character $\chi_{\triv}$. As already observed in part (A), the kernel of $\rho_{2,1}$ is generated by $a^2b^2$. Since $\rho$ is faithful (property \ref{nec-prop1}), we may assume that $\chi(b)$ is a primitive fourth root of unity. Now, $\rho(b)$ must have the eigenvalue $1$ and hence $\chi'(b) = 1$. Similarly, both matrices
	\begin{align*}
	&\rho(a) = \diag(i, \ i, \ \chi(a), \chi'(a)), \text{ and} \\
	&\rho(ab^2) = \diag(-i, \ -i, \ -\chi(a), \ \chi'(a))
	\end{align*}
	must have the eigenvalue $1$ and so we obtain $\chi'(a) = 1$. The remaining statement $\chi'(c) = 1$ follows analogously by considering the matrix
	\begin{align*}
	\rho(abc) = \begin{pmatrix}
	0 & i && \\ i & 0 & & \\ && \chi(abc) & \\ &&& \chi'(c)
	\end{pmatrix}
	\end{align*}
	and observing that since $\chi(abc) \neq 1$, we have the desired value $\chi'(c) = 1$. \\
	
	We now construct $\psi \in \Aut(C_4^2 \rtimes C_2)$ that stabilizes the character of $\rho_{2,1}$ and such that $\chi \circ \psi = \chi_{0,1,0}$ as a suitable composition of the following automorphisms:
	\begin{itemize}
		\item $\psi_1$ defined by $a \mapsto a$, $b \mapsto b^3$, $c \mapsto c$,
		\item $\psi_2$ defined by $a \mapsto a\cdot (bc)^2$, $b \mapsto b$, $c \mapsto c$, 
		\item $\psi_3$ defined by $a \mapsto a$, $b \mapsto b$, $c \mapsto b^2c$.
	\end{itemize}
	Indeed, one checks that every $\psi_j$ stabilizes the character of $\rho_{2,1}$. Thus, if $\chi(b) = -i$, we may replace $\chi$ by $\chi \circ \psi_1$ to assume that $\chi(b) = i$. Furthermore, if $\chi(a) = -1$, we may apply $\psi_2$ to achieve $\chi(a) = 1$. Similarly, possibly applying $\psi_3$ guarantees that $\chi(c) = 1$, as claimed.
\end{proof}

We finish part (B) by describing the isogeny type of $T$. Indeed, it follows from \hyperref[isogeny]{Section~\ref*{isogeny}} that $T$ is equivariantly isogenous to $E_i \times E_i \times E_i \times E$, where the $E_i$ are copies of Gauss' elliptic curve embedded in $T$ and $E \subset T$ is another elliptic curve. \\

\emph{(C) An example.} \\
We construct an explicit example of a hyperelliptic fourfold $T/(C_4^2 \rtimes C_2)$. By a change of origin, we may write the action of $\langle a,b,c \rangle$ on $T \sim_{(C_4^2 \rtimes C_2)-\isog} T' := E_i \times E_i \times E_i \times E$ as follows:
\begin{align*}
&a(z) = (iz_1, \ iz_2, \ z_3 + a_3, \ z_4 + a_4), \\
&b(z) = (-z_2 + b_1, \ z_1 + b_2, \ iz_3 + b_3, \ z_4 + b_4), \\
&c(z) = (z_1 + c_1, \ -z_2 + c_2, \ z_3 + c_3, \ z_4 + c_4).
\end{align*}
The following conditions on the $a_i$, $b_j$ and $c_k$ then tell us when $\langle a,b,c \rangle$ is isomorphic to $C_4^2 \rtimes C_2$:

\begin{lemma}\label{32-24-rels} \
	\begin{enumerate}[ref=(\theenumi)]
		\item It holds $a^4 = \id_T$ if and only if $(0,\ 0,\ 4a3,\ 4a4)$ is zero in $T$.
		\item \label{32-24-rels2} It holds $b^4 = \id_T$ if and only if $4b_4 = 0$ in $E$.
		\item It holds $c^2 = \id_T$ if and only if $(2c_1,\ 0,\ 2c_3,\ 2c_4)$ is zero in $T$.
		\item It holds $[a,b] = \id_T$ if and only if $((i-1)b_1,\ (i-1)b_2,\  (1-i)a_3,\ 0)$ is zero in $T$.
		\item It holds $[a,c] = \id_T$ if and only if $((i-1)c_1,\  (i-1)c_2,\ 0,\ 0)$ is zero in $T$.
		\item It holds $c^{-1}bc = a^2b$ if and only if $(c_1+c_2-2b_1,\  c_1-c_2,\  2a_3-(i-1)c_3,\ 2a_4)$ is zero in $T$.
	\end{enumerate}
\end{lemma}

%\begin{proof}
%This is only a computation. As for \ref{32-24-rels2}, we only observe that $b^4 = \id_T$ is satisfied if and only if $(0, \ 0, \ 0, \ 4b_4)$ is zero in $T$. However, since $E \subset T$, this is equivalent to $4b_4 = 0$ in $E$.
%\end{proof}

\begin{rem} \label{32-24-free}
	The $19$ non-trivial conjugacy classes of $C_4^2 \rtimes C_2$ are represented by
	\begin{center}
		$a, \quad a^2, \quad a^3, \quad b, \quad b^2, \quad b^3, \quad c, \quad ab, \quad (ab)^3, \quad ac$, \\
		$bc, \quad b^2c, \quad (bc)^3, \quad ab^2, \quad a^2b^2, \quad abc, \quad ab^2c, \quad (ab^2)^3, \quad (abc)^3$.
	\end{center}
	Since $3$ does not divide the order of $C_4^2 \rtimes C_2$, an element acts freely if and only if its third power acts freely. We, therefore, only have to prove that the $13$ elements
	\begin{center}
		$a, \quad a^2, \quad b, \quad b^2, \quad c, \quad ab, \quad ac, \quad bc, \quad b^2c, \quad ab^2, \quad a^2b^2, \quad abc, \quad ab^2c$
	\end{center}
	act freely.
\end{rem}

We are now in the situation to give an example of a hyperelliptic fourfold with holonomy group $C_4^2 \rtimes C_2$. 

\begin{example} \label{32-24-example}
	Let $T' = E_i \times E_i \times E_i \times E$, as above. For convenience, we will assume that the elliptic curve $E$ is in normal form, i.e., we will write $E = \CC/(\ZZ+\tau \ZZ)$ for some $\tau \in \CC$, $\mathrm{Im}(\tau) > 0$. Define $T := T'/H$, where
	\begin{align*}
	H := \left\langle \left(0, \ 0, \ \frac{1+i}{2}, \ \frac12\right) \right\rangle. 
	\end{align*}
	Furthermore, we define the following holomorphic automorphisms of $T'$:
	\begin{align*}
	&a(z) = \left(iz_1, \ iz_2, \ z_3, \ z_4 + \frac14\right), \\
	&b(z) = \left(-z_2, \ z_1, \ iz_3, \ z_4 + \frac{1+\tau}{4}\right), \\
	&c(z) = \left(z_1 + \frac{1+i}{2}, \ -z_2 + \frac{1+i}{2}, \ z_3, \ z_4 + \frac{\tau}{2}\right).
	\end{align*}
	Since the linear parts of $a$, $b$ and $c$ map $H$ to $H$, respectively, they induce automorphisms of $T$, which we will again denote by $a$, $b$ and $c$. Viewed as automorphisms of $T$, \hyperref[32-24-rels]{Lemma~\ref*{32-24-rels}} immediately shows that the group $\langle a,b,c \rangle$ is isomorphic to $C_4^2 \rtimes C_2$. Moreover, the group $\langle a,b,c \rangle$ does not contain any translations since the associated complex representation is by construction faithful. \\
	It remains to show that $\langle a,b,c\rangle$ acts freely on $T$. That most of the thirteen elements listed in \hyperref[32-24-free]{Remark~\ref*{32-24-free}} act freely is clear -- we only check the delicate ones:
	\begin{align*}
	&a^2(z) = \left(-z_1, \ -z_2, \ z_3, \ z_4+ \frac12\right), \\
	&b^2c(z) = \left(-z_1 + \frac{1+i}{2}, \ z_2 + \frac{1+i}{2}, \ -z_3, \ z_4 + \frac12\right), \text{ and} \\
	&a^2b^2(z) = \left(z_1, \ z_2, \ -z_3, \ z_4 + \frac{\tau}{2}\right).
	\end{align*}
	These elements act freely by our definition of $T$. This proves that $T/\langle a,b,c\rangle$ indeed defines a hyperelliptic fourfold with holonomy group $\langle a,b,c \rangle \cong C_4^2 \rtimes C_2$. 
\end{example}

\emph{(D) The Hodge diamond.} \\ 
The Hodge diamond of a hyperelliptic fourfold with holonomy group $C_4^2 \rtimes C_2$ (ID [32,24]) is
\begin{center}
	$\begin{matrix}
	&   &  &  & 1 &  &  &  &  \\
	&   &  & 1 &  & 1 &  &  &  \\
	&   & 0 &  & 3 &  & 0  &  & \\
	&  0 &  & 2 &   & 2 &   &  0 &  \\
	0&    & 0 &  & 4  &  & 0  &     &  0 \\
	\end{matrix}$
\end{center}

\emph{(E) Excluding a group containing $C_4^2 \rtimes C_2$.}

The above discussion allows us to give a simple proof of the following

\begin{lemma} \label{32-24-supgroup-excluded}
	The group $(C_4^2 \rtimes C_2) \times C_3$ (ID $[96,164]$) is not hyperelliptic in dimension $4$.
\end{lemma}

\begin{proof}
	This follows essentially from the \hyperref[order-cyclic-groups]{Integrality Lemma~\ref*{order-cyclic-groups}} \ref{ocg-3} and  \hyperref[32-24-prop]{Proposition~\ref*{32-24-prop}}. To be more precise, let $k$ be an element of order $3$ commuting with $a$, $b$ and $c$. If $(C_4^2 \rtimes C_2) \times C_3$ is hyperelliptic with associated complex representation $\rho$ such that $\rho|_{C_4^2 \rtimes C_2}$ is as in \hyperref[32-24-prop]{Proposition~\ref*{32-24-prop}}, then property \ref{nec-prop2} implies that $\rho(bk)$ has an eigenvalue of order $12$. According to \hyperref[order-cyclic-groups]{Integrality Lemma~\ref*{order-cyclic-groups}} \ref{ocg-3}, $\rho(bk)$ has exactly two eigenvalues of order $12$, and thus
	\begin{align*}
	\rho(k) = \diag(\zeta_3, \ \zeta_3, \ 1, \ 1) \text{ or }	\rho(k) = \diag(\zeta_3^2, \ \zeta_3^2, \ 1, \ 1). 
	\end{align*}
	However, then $\rho(ak)$ has an eigenvalue of order $12$ with multiplicity $2$, a contradiction.
\end{proof}

\subsection{$G(8,2,5) \times C_2$ (ID [32,37])} \label{32-37-section}\  \\ 

Our next goal is to investigate hyperelliptic fourfolds $T/(G(8,2,5) \times C_2)$, where $G(8,2,5)$ is generated by two elements $g$ and $h$ of respective orders $8$ and $2$ such that $h^{-1}gh = g^5$. Assume, furthermore, that $C_2$ is generated by an element $k$ that commutes with $g$ and $h$, so that $\langle g,h,k\rangle = G(8,2,5) \times C_2$. We then establish the following result (see  \hyperref[32-37-rho]{Lemma~\ref*{32-37-rho}}, the discussion after the lemma and  \hyperref[32-37-example]{Example~\ref*{32-37-example}})

\begin{prop} \label{32-37-prop}
	Let $X = T/(G(8,2,5) \times C_2)$ be a hyperelliptic fourfold with associated complex representation $\rho$. Then:
	\begin{enumerate}[ref=(\theenumi)]
		\item Up to equivalence and automorphisms, $\rho$ is given as follows:
		\begin{align*}
		\rho(g) = \begin{pmatrix}
		0 & i && \\ 1 & 0 && \\ && 1& \\ &&&i
		\end{pmatrix}, \qquad \rho(h) = \begin{pmatrix}
		1 & && \\ & -1 && \\ &&1 & \\ &&& -1
		\end{pmatrix}, \qquad \rho(k) = \begin{pmatrix}
		1 & && \\ & 1 && \\ && 1 & \\ &&&-1
		\end{pmatrix}.
		\end{align*}
		\item The representation $\rho$ induces an equivariant isogeny $E_i \times E_i \times E_i \times E \to T$, where $E \subset T$ is an elliptic curve and $E_i = \CC/(\ZZ+i\ZZ) \subset T$ are copies of Gauss' elliptic curve.
		\item Hyperelliptic fourfolds with holonomy group $G(8,2,5) \times C_2$ exist.
	\end{enumerate}
	In particular, $X$ moves in a complete $1$-dimensional family of hyperelliptic fourfolds with holonomy group $G(8,2,5) \times C_2$.
\end{prop}

The Hodge diamond of a hyperelliptic fourfold $T/(G(8,2,5) \times C_2)$ is determined in part (D). Moreover, we will show the following negative result in part (E):

\begin{prop} \label{prop-64-20-64-85-excluded}
	\
	\begin{enumerate}[ref=(\theenumi)]
		\item Consider the following action of $C_4 = \langle u \rangle$ on $G(4,4,3) = \langle a,b ~ | ~ a^4 = b^4 = 1, ~ b^{-1}ab = a^3\rangle$:
		\begin{align*}
		[u,a] = 1, \qquad u^{-1}bu = ab.
		\end{align*}
		Then the resulting semidirect product $G(4,4,3) \rtimes C_4$ (ID $[64,20]$) is not hyperelliptic in dimension $4$.
		\item The group $G(8,2,5) \times C_4$ (ID $[64,85]$) is not hyperelliptic in dimension $4$.
	\end{enumerate}
\end{prop}

We remark both of the groups of \hyperref[prop-64-20-64-85-excluded]{Proposition~\ref*{prop-64-20-64-85-excluded}} contain a subgroup, which is isomorphic to $G(8,2,5) \times C_2$. It is obvious for $G(8,2,5) \times C_4$, and the subgroup $\langle g := ub, ~ h := u^2, ~ k := b^2\rangle$ of $G(4,4,3) \rtimes C_4$ (ID $[64,20]$) is isomorphic to $G(8,2,5) \times C_2$, since
\begin{align*}
\ord(ub) = 8, \qquad \ord(u^2) = 2, \qquad \ord(b^2) = 2, \qquad u^{-2}(ub)u^2 = (ub)^5, \qquad [ub,b^2] = [u^2,b^2] = 1.
\end{align*}
\hyperref[prop-64-20-64-85-excluded]{Proposition~\ref*{prop-64-20-64-85-excluded}} will therefore be a straightforward consequence of \hyperref[32-37-prop]{Proposition~\ref*{32-37-prop}}.  \\

\emph{(A) Representation Theory of $G(8,2,5)$.} \\
As already mentioned above, the group $G(8,2,5)$ is the standard metacyclic group presented by	$G(8,2,5) = \langle g,h \ | \ g^8 = h^2 = 1, \ h^{-1}gh = g^5\rangle$.

The relation $h^{-1}gh = g^5$ implies that $g^4$ is a commutator. It follows that $G(8,2,5)$ has precisely eight degree $1$ representations, namely
\begin{align*}
\chi_{a,b}(g) = i^a, \qquad \chi_{a,b}(h) = (-1)^b, \qquad \text{where } a \in \{0,1,2,3\}, \quad  b \in \{0,1\}.
\end{align*}
Furthermore, $G(8,2,5)$ has two irreducible representations of degree $2$:
\begin{align*}
\rho_2(g) = \begin{pmatrix}
0 & i \\ 1 & 0
\end{pmatrix}, \qquad \rho_2(h) = \begin{pmatrix}
1 & \\ & -1
\end{pmatrix}
\end{align*}
and its complex conjugate $\overline{\rho_2}$. These are faithful and form an orbit under the action of $\Aut(G(8,2,5))$, since
\begin{align*}
\rho_2 \circ \psi = \overline{\rho_2}, \qquad \text{where } \psi(g) = g^5, \quad \psi(h) = h.
\end{align*}

\emph{(B) The complex representation and the isogeny type of the torus.} \\
As is customary by now, we describe the complex representation $\rho \colon G(8,2,5) \times C_2 \to \GL(4,\CC)$ of a hyperelliptic fourfold $T/(G(8,2,5) \times C_2)$ up to equivalence and automorphisms. To achieve this, we first decompose $\rho$ into irreducible representations such that the three properties \ref{nec-prop1} -- \ref{nec-prop3} of p. \ref{nec-prop1} are satisfied. In particular, we want $\rho$ to be faithful (property \ref{nec-prop1}) and that every matrix in the image of $\rho$ has the eigenvalue $1$ (property \ref{nec-prop2}). \\

Let therefore $k \in \Bihol(T)$ be an element such that $\langle g,h,k\rangle$ is isomorphic to $G(8,2,5) \times C_2$. \\

The discussion of part (A) immediately shows that $\rho$ is not the direct sum of two irreducible representations of degree $2$. (This also follows from \hyperref[rho2'-reducible]{Lemma~\ref*{rho2'-reducible}}, since the center of $G(8,2,5) \times C_2$ is non-cyclic.) It follows that $\rho$ is the direct sum of an irreducible degree $2$ representation and two representations of degree $1$. 

\begin{lemma} \label{32-37-rho}
	Up to symmetry and the action of the automorphism group, the complex representation of a hyperelliptic fourfold with holonomy group $G(8,2,5) \times C_2$ is given by
	\begin{align*}
	\rho(g) = \begin{pmatrix}
	0 & i && \\ 1 & 0 && \\ &&1& \\ &&&i 
	\end{pmatrix}, \quad \rho(h) = \begin{pmatrix}
	1 &&& \\ & -1 && \\ &&1& \\ &&&-1
	\end{pmatrix}, \quad \rho(k) = \begin{pmatrix}
	1 &&& \\ & 1 && \\ &&1 & \\ &&&-1
	\end{pmatrix}.
	\end{align*}
\end{lemma}

\begin{proof}
	Denote $\tilde\rho_2$ the irreducible degree $2$ summand and by $\chi$, $\chi'$ the degree $1$ summands of $\rho$. 
	The automorphism $g \mapsto g$, $h \mapsto h$ and $k \mapsto g^4k$ of $G(8,2,5) \times C_2$ fixes $G(8,2,5) = \langle g,h\rangle$ and replaces $\tilde\rho_2(k)$ by $-\tilde\rho_2(k)$. We may therefore assume that $\tilde\rho_2(k)$ is the identity. Now, since $\rho$ is faithful, we have $\chi'(k) = -1$, up to symmetry. Since
	\begin{align*}
	\rho(g^4k) = \diag(-1, \ -1, \ \chi(k), \ -1)
	\end{align*}
	must have the eigenvalue $1$, we infer that $\chi(k) = 1$. \\
	As a next step, we observe that after possibly applying one or both of the automorphisms
	\begin{align*}
	&g \mapsto gk, \qquad h \mapsto h, \qquad k \mapsto k, \text{ or} \\
	&g \mapsto g, \qquad h \mapsto hk, \qquad k \mapsto k,
	\end{align*} 
	we may assume that $\chi'(g) \in \{1,i\}$ and $\chi'(h) = 1$. \\
	
	Our next step is to exclude $\chi'(g) = 1$. Indeed, if $\chi'(g) = 1$, we consider the matrices
	\begin{align} \label{32-37-eq}
	\rho(gk) = \begin{pmatrix}
	0 & i && \\ 1 & 0 && \\ && \chi(g) & \\ &&& -1
	\end{pmatrix} \qquad \text{ and } \qquad \rho(ghk) = \begin{pmatrix}
	0 & -i && \\ 1 & 0 && \\ && \chi(gh) & \\ &&& -1
	\end{pmatrix},
	\end{align}
	which both must have the eigenvalue $1$. It follows that $\chi(g) = 1$ and $\chi(h) = 1$ -- a contradiction to \hyperref[lemma-two-generators]{Proposition~\ref*{lemma-two-generators}}. Consequently, $\chi'(g) = i$, and because $\rho(g)$ must have the eigenvalue $1$, it follows that $\chi(g) = 1$. Equation \ref{32-37-eq} finally implies that $\chi(h) = 1$.
\end{proof}

The isogeny type of $T$ is now easily determined: the discussion in Section \hyperref[isogeny]{Section~\ref*{isogeny}} immediately implies that $\rho$ induces an equivariant isogeny $E_i \times E_i \times E \times E_i \to T$, where $E_i = \CC/(\ZZ+i\ZZ)$ and $E \subset T$ is another elliptic curve. \\

\emph{(C) An example.} \\

We now show that hyperelliptic fourfolds with holonomy group $G(8,2,5) \times C_2$ exist. By the previous part (B), we may write the action of $g$, $h$ and $k$ on $T \sim_{(G(8,2,5) \times C_2)-\isog} E_i \times E_i \times E \times E_i$ as follows:
\begin{align*}
&g(z) = \left(iz_2 + a_1, \ z_1+a_2, \ z_3 + a_3, \ iz_4 + a_4\right), \\
&h(z) = \left(z_1 + b_1, \ -z_2 + b_2, \ z_3 + b_3, \ z_4 + b_4\right),
&k(z) = \left(z_1 + c_1, \ z_2 + c_2, \ z_3 + c_3, \ -z_4 + c_4\right).
\end{align*}
By changing the origin in the respective elliptic curves, we may assume that $a_1 = a_2 = 0$, $a_4 = 0$. The following lemma contains the criteria the $a_i$, $b_j$ and $c_n$ have to satisfy in order for $\langle g,h,k\rangle$ to be isomorphic to $G(8,2,5) \times C_2$.

\begin{lemma} \label{32-37-rels} \
	\begin{enumerate}[ref=(\theenumi)]
		\item It holds $g^8 = \id_T$ if and only if $8a_3 = 0$ in $E$.
		\item It holds $h^2 = \id_T$ if and only if $(2b_1, \ 0, \ 2b_3, \ 2b_4)$ is zero in $T$.
		\item It holds $k^2 = \id_T$ if and only if $(2c_1, \ 2c_2, \ 2c_3, \ 0)$ is zero in $T$.
		\item It holds $[g,k] = \id_T$ if and only if $(ic_2 - c_1, \ c_1-c_2, \ 0, \ (i-1)c_4)$ is zero in $T$
		\item It holds $[h,k] = \id_T$ if and only if $(0,\ 2c_2,\ 0, \ -2b_4)$ is zero in $T$.
		\item It holds $h^{-1}gh = g^5$ if and only if $(b_1-ib_2, \ b_1-b_2, \ 4a_3, \ (1-i)b_4)$ is zero in $T$.
	\end{enumerate}
\end{lemma}

\begin{rem} \label{32-37-free}
	The group $G(8,2,5)$ has ten conjugacy classes. We give a set of representatives $\sC$:
	\begin{align*}
	1, \quad g, \quad g^2, \quad g^3, \quad g^4, \quad g^6, \quad h, \quad gh, \quad (gh)^3, \quad g^2h.
	\end{align*}
	A set of representatives for the conjugacy classes of $G(8,2,5) \times C_2$ then is $\sC \cup k\sC$. Observe that $G(8,2,5) \times C_2$ is a $2$-group, and thus an element acts freely if and only if their third power acts freely. We, therefore, only need to prove that the elements
	\begin{align*}
	g, \quad g^2, \quad g^4, \quad h, \quad gh, \quad g^2h,\quad k, \quad kg, \quad kg^2, \quad kg^4, \quad kh, \quad kgh, \quad kg^2h
	\end{align*}
	act freely.
\end{rem}

\begin{example} \label{32-37-example}
	Let $T' := E_i \times E_i \times E \times E_i$, where $E_i = \CC/(\ZZ+i\ZZ)$ is Gauss' elliptic curve and let $E = \CC/(\ZZ+ \tau \ZZ)$ be another elliptic curve. Moreover, define $T := T'/H$,
	where
	\begin{align*}
	H := \left\langle \left(0, \ 0, \ \frac12, \frac{1+i}{2} \right) \right\rangle.
	\end{align*}
	Consider $g,h,k \in \Bihol(T')$ defined by
	\begin{align*}
	&g(z) = \left(iz_2, \ z_1, \ z_3 + \frac18, \ iz_4\right), \\
	&h(z) = \left(z_1, \ -z_2, \ z_3, \ z_4 + \frac{1}{2} \right), \\
	&k(z) = \left(z_1, \ z_2, \ z_3 + \frac{\tau}{2}, \ -z_4 \right).
	\end{align*}
	The linear parts of $g$, $h$ and $k$ map $H$ to $H$: therefore, they yield maps $g,h,k \in \Bihol(T)$. Viewed as automorphisms of $T$,  \hyperref[32-37-rels]{Lemma~\ref*{32-37-rels}} shows that $g$, $h$ and $k$ span a subgroup isomorphic to $G(8,2,5) \times C_2$. It is clear by construction that $\langle g,h,k\rangle$ does not contain any translations. \\
	It is more or less clear from the definition of $T$ and $g$, $h$ and $k$ that the thirteen elements from \hyperref[32-37-free]{Remark~\ref*{32-37-free}} act freely on $T$. It follows that $T/\langle g,h,k\rangle$ defines a hyperelliptic fourfold with holonomy group $G(8,2,5) \times C_2$.
\end{example}

\emph{(D) The Hodge diamond.} \\
The Hodge diamond of a hyperelliptic fourfold with holonomy group $G(8,2,5) \times C_2$ is 
\begin{center}
	$\begin{matrix}
	&   &  &  & 1 &  &  &  &  \\
	&   &  & 1 &  & 1 &  &  &  \\
	&   & 0 &  & 3 &  & 0  &  & \\
	&  0 &  & 2 &   & 2 &   &  0 &  \\
	0&    & 0 &  & 4  &  & 0  &     &  0 \\
	\end{matrix}$
\end{center}

\textit{(E) Proof of \hyperref[prop-64-20-64-85-excluded]{Proposition~\ref*{prop-64-20-64-85-excluded}}.} \\

Consider the action of $C_4 = \langle u \rangle$ on $G(4,4,3) = \langle a,b \ | \ a^4 = b^4 = 1, ~ b^{-1}ab = b^3\rangle$, which is defined by
\begin{align*}
[u,a] = 1, \qquad u^{-1}bu = ab.
\end{align*}

The resulting group is the semidirect product $G(4,4,3) \rtimes C_4$ labeled as the group $[64,20]$ in the Database of Small Groups. We show that it is not hyperelliptic in dimension $4$. \\
As usual, we let $\rho \colon G(4,4,3) \rtimes C_4 \to \GL(4,\CC)$ be a faithful representation and assume that it is the complex representation of some hyperelliptic fourfold. Then, since the center of $G(4,4,3) \rtimes C_4$ is the non-cyclic group generated by $u^2a$ and $b^2$, the complex representation $\rho$ is equivalent to the direct sum of an irreducible representation $\rho_2$ of degree $2$ and two linear characters $\chi$, $\chi'$ (see \hyperref[rho2'-reducible]{Lemma~\ref*{rho2'-reducible}}).

\begin{rem} \label{64-20-rem}
	As already mentioned below the statement of \hyperref[prop-64-20-64-85-excluded]{Proposition~\ref*{prop-64-20-64-85-excluded}}, the defining relations of $G(4,4,3) \rtimes C_4$ imply that 
	\begin{align*}
	\ord(ub) = 8, \qquad \ord(u^2) = 2, \qquad \ord(b^2) = 2, \qquad u^{-2}(ub)u^2 = (ub)^5, \qquad [ub,b^2] = [u^2,b^2] = 1.
	\end{align*}
	It follows that the subgroup $\langle g := ub, h := u^2, k := b^2 \rangle$ is isomorphic to $G(8,2,5) \times C_2$. According to  \hyperref[32-37-prop]{Proposition~\ref*{32-37-prop}}, we may therefore assume that
	\begin{align*}
	\rho(ub) = \begin{pmatrix}
	0 & i && \\ 1 & 0 && \\ && 1& \\ &&&i
	\end{pmatrix}, \qquad \rho(u^2) = \begin{pmatrix}
	1 & && \\ & -1 && \\ &&1 & \\ &&& -1
	\end{pmatrix}, \qquad \rho(h^2) = \begin{pmatrix}
	1 & && \\ & 1 && \\ && 1 & \\ &&&-1
	\end{pmatrix}.
	\end{align*}
\end{rem}

Our goal is an immediate consequence.

\begin{cor}
	The group $G(4,4,3) \rtimes C_4$ is not hyperelliptic in dimension $4$.
\end{cor}

\begin{proof}
	According to the discussion in \hyperref[64-20-rem]{Remark~\ref*{64-20-rem}}, we may assume that $\chi'(ub) = i$, $\chi'(u^2) = -1$ and $\chi'(b^2) = -1$. This is impossible, since if $\chi'(u)$ and $\chi'(b)$ are primitive fourth roots of unity, then $\chi'(ub)^2 = 1$.
\end{proof}
\vspace{1.5cm}

We now turn to $G(8,2,5) \times C_4$. Suppose that $\ell$ is an element of order $4$ commuting with the group
\begin{align*}
G(8,2,5) = \langle g,h ~ | ~ g^8 = h^2 = 1, ~ h^{-1}gh = g^5 \rangle,
\end{align*}
so that $\langle g,h,\ell^2 \rangle$ is isomorphic to $G(8,2,5) \times C_2$. \\
Assume again that $\rho \colon G(8,2,5) \times C_4 \to \GL(4,\CC)$ is the complex representation of some hyperelliptic fourfold. We may assume that the restriction of $\rho$ to $\langle g,h, \ell^2\rangle = G(8,2,5) \times C_2$ is given as in  \hyperref[32-37-prop]{Proposition~\ref*{32-37-prop}} -- in particular, we assume that
\begin{align*}
\rho(\ell^2) = \diag(1, \ 1, \ 1, \ -1).
\end{align*}
By replacing $\ell$ by $g^4 \ell$ and $\ell$ by $\ell^3$ if necessary, we may assume that $\rho(\ell)$ is one of
\begin{align*}
\diag(1, \ 1, \ 1,  \ i) \qquad \text{ or } \diag (1, \ 1, \ -1, \ i).
\end{align*}
The latter of these possibilities is easily excluded since $\rho(g^4\ell) = -I_4$ does not have the eigenvalue $1$ in this case. The first possibility is excluded as well, since here $\langle g\ell^3, h \ell^2\rangle \cong G(8,2,5)$ is metacyclic and hyperelliptic in dimension $4$ but contains two copies of the trivial character, which is impossible by  \hyperref[lemma-two-generators]{Proposition~\ref*{lemma-two-generators}}. Thus $G(8,2,5) \times C_4$ is not hyperelliptic in dimension $4$ either.

\subsection{$((C_2 \times C_4) \rtimes C_2) \times C_3$ (ID [48,21])} \label{48-21-section}\  \\ 

Consider the action of $C_2 = \langle c \rangle$ on $C_2 \times C_4 = \langle a,b \ | \ a^2 = b^4 = [a,b] = 1\rangle$ defined by
\begin{align*}
[a,c] = 1, \qquad c^{-1}bc = ab.
\end{align*}
The resulting group is the semidirect product $(C_2 \times C_4) \rtimes C_2$ labeled by [16,3] in the Database of Small Groups. This section is dedicated to proving the following result (see in particular  \hyperref[48-21-rho]{Lemma~\ref*{48-21-rho}}, the text below the lemma and  \hyperref[48-21-example]{Example~\ref*{48-21-example}}):

\begin{prop} \label{48-21-prop}
	Let $X = T/G$, $G = ((C_2 \times C_4) \rtimes C_2) \times C_3$ be a hyperelliptic fourfold with associated complex representation $\rho$. Then:
	\begin{enumerate}[ref=(\theenumi)]
		\item Up to equivalence and automorphisms, $\rho$ is given as follows:
		\begin{align*}
		&\rho(a) = \begin{pmatrix}
		-1 &&& \\ & -1 && \\ && 1 & \\ &&&1
		\end{pmatrix}, \qquad \ \rho(b) = \begin{pmatrix}
		0 & -1 && \\ 1 & 0 && \\ && i & \\ &&&1
		\end{pmatrix}, \\
		&\rho(c) = \begin{pmatrix}
		1 &&& \\ & -1 && \\ && 1 & \\ &&&1
		\end{pmatrix}, \qquad \quad \rho(k) = \begin{pmatrix}
		\zeta_3 &&& \\ & \zeta_3&& \\ && 1 & \\ &&&1
		\end{pmatrix}.
		\end{align*}
		\item The representation $\rho$ induces an equivariant isogeny $F \times F \times E_i \times E \to T$, where  $E \subset T$ is an elliptic curve, $F = \CC/(\ZZ+\zeta_3\ZZ)$ and $E_i = \CC/(\ZZ + i\ZZ)$.
		\item Hyperelliptic fourfolds with holonomy group $((C_2 \times C_4) \rtimes C_2) \times C_3$ exist.
	\end{enumerate}
	In particular, $X$ moves in a complete $1$-dimensional family of hyperelliptic fourfolds with holonomy group $((C_2 \times C_4) \rtimes C_2) \times C_3$.
	
\end{prop}

The Hodge diamond of such fourfolds is given in part (D).\\ Furthermore, observe that the above proposition implies that a central element $k$ of order $3$ is not contained in the kernel of the irreducible degree $2$ summand of $\rho$. Hence we obtain the following immediate consequence:

\begin{prop} \label{144-102-excluded}
	The group $((C_2 \times C_4) \rtimes C_2) \times C_3^2$ (ID $[144,102]$) is not hyperelliptic in dimension $4$.
\end{prop}

\emph{(A) Representation Theory of $(C_2 \times C_4) \rtimes C_2$.} \\
The group in discussion has exactly two irreducible representations of degree $2$: 
\begin{center}
	\begin{tabular}{llll}
		$\rho_{2,1}$: & $a \mapsto \begin{pmatrix}
		-1 & \\ & -1 
		\end{pmatrix}$, & $b \mapsto \begin{pmatrix}
		0 & -1 \\ 1 & 0
		\end{pmatrix}$, & $c \mapsto \begin{pmatrix}
		1 & \\ & -1
		\end{pmatrix}$, \\
		$\rho_{2,2}$: & $a \mapsto \begin{pmatrix}
		-1 & \\ & -1 
		\end{pmatrix}$, & $b \mapsto \begin{pmatrix}
		0 & 1 \\ 1 & 0
		\end{pmatrix}$, & $c \mapsto \begin{pmatrix}
		1 & \\ & -1
		\end{pmatrix}$.
	\end{tabular} 
\end{center}
Both of these representations are non-faithful, which reflects the fact that the center of $(C_2 \times C_4) \rtimes C_2$ is the non-cyclic group generated by $a$ and $b^2$. \\
We now consider the action of $\Aut((C_2 \times C_4) \rtimes C_2)$ on the set of irreducible characters of $(C_2 \times C_4) \rtimes C_2$. The two representations $\rho_{2,1}$ and $\rho_{2,2}$ lie in the same orbit of this action: indeed, $\rho_{2,1} \circ \varphi = \rho_{2,2}$, where $\varphi$ is the automorphism defined by $a \mapsto a$, $b \mapsto bc$, $c \mapsto c$. \\

Turning our attention to the linear characters, the relation $c^{-1}bc = ab$ implies that $a$ is a commutator and hence the linear characters of $(C_2 \times C_4) \rtimes C_2$ are exactly the eight characters $\chi_{n,m}$  ($n \in \{0,...,3\}$, $m \in \{0,1\}$) defined by
\begin{align*}
\chi_{n,m}(a) = 1, \qquad \rho_{n,m}(b) = i^n, \qquad \rho_{n,m}(c) = (-1)^m.
\end{align*}

\emph{(B) The complex representation and the isogeny type of the torus.} \\
In this part, we investigate the complex representation $\rho$ of a hyperelliptic fourfold $T/G$ with holonomy group $G = ((C_2 \times C_4) \rtimes C_2) \times C_3$. We moreover describe the isogeny type of $T$. A simple starting point is the following:

\begin{lemma}
	The representation $\rho$ is the direct sum of an irreducible representation of degree $2$ and two linear characters.
\end{lemma}

\begin{proof}
	The center of $(C_2 \times C_4) \rtimes C_2$ is non-cyclic, and thus the assertion follows from \hyperref[rho2'-reducible]{Lemma~\ref*{rho2'-reducible}}.
\end{proof}

As we saw in part (A), both $\rho_{2,1}$ and $\rho_{2,2}$ are equivalent up to an automorphism of $(C_2 \times C_4) \rtimes C_2$: we may therefore assume that the  irreducible degree $2$ subrepresentation of $\rho|_{(C_2 \times C_4) \rtimes C_2}$ is $\rho_{2,1}$. Denote by $\chi$ and $\chi'$ the two degree $1$ summands of $\rho$. The upcoming lemma determines some values of $\chi$ and $\chi'$.

\begin{lemma}
	Up to symmetry and automorphisms, $\chi(b) = i$ and $\chi(c) = 1$ (i.e., $\chi=\chi_{1,0}$) and $\chi'(b) = 1$.
\end{lemma}

\begin{proof}
	Since the kernel of $\rho_{2,1}$ is generated by $ab^2$, but $\rho$ is faithful, we may assume that $\chi(b)$ is a primitive fourth root of unity. Now, $\rho(b)$ must have the eigenvalue $1$ and hence $\chi'(b) = 1$.\\ 
	Up to applying the automorphism $a \mapsto a$, $b \mapsto b^3$, $c \mapsto c$ (that stabilizes the equivalence class of $\rho_{2,1}$), we may assume that $\chi(b) = i$. Similarly, after having applied the automorphism $a \mapsto a$, $b \mapsto b$, $c \mapsto b^2c$ (that also stabilizes $\rho_{2,1}$!) if needed, we obtain $\chi(c) = 1$. 
\end{proof}

Finally, to determine the remaining values, we consider the action of the full group $((C_2 \times C_4) \rtimes C_2) \times C_3$: let therefore $k$ be a central element of order $3$ that commutes with $a$, $b$ and $c$. The statement is then as follows:

\begin{lemma} \label{48-21-rho}
	Up to equivalence and an automorphism of $G := ((C_2 \times C_4) \rtimes C_2) \times C_3$, the complex representation $\rho$ of a hyperelliptic fourfold $T/G$ is given as follows:
	\begin{align*}
	&\rho(a) = \begin{pmatrix}
	-1 &&& \\ & -1 && \\ && 1 & \\ &&&1
	\end{pmatrix}, \qquad \ \rho(b) = \begin{pmatrix}
	0 & -1 && \\ 1 & 0 && \\ && i & \\ &&&1
	\end{pmatrix}, \\
	&\rho(c) = \begin{pmatrix}
	1 &&& \\ & -1 && \\ && 1 & \\ &&&1
	\end{pmatrix}, \qquad \quad \rho(k) = \begin{pmatrix}
	\zeta_3 &&& \\ & \zeta_3&& \\ && 1 & \\ &&&1
	\end{pmatrix}.
	\end{align*}
\end{lemma}

\begin{proof}
	First of all, $\chi(k) = \chi'(k) = 1$, since 
	\begin{align*}
	\rho(bk) = \diag(\rho_2(bk), ~ i\chi(k), ~ \chi'(k))
	\end{align*}
	has the eigenvalue $1$ and an even number of eigenvalues of order $12$ (\hyperref[order-cyclic-groups]{Integrality Lemma~\ref*{order-cyclic-groups}} \ref{ocg-3}). Up to replacing $k$ by $k^2$, we may thus assume that $\rho(k)$ is as asserted. \\
	The remaining value ($\chi'(c) = 1$) is now easily determined because
	\begin{align*}
	\rho(bck) = \begin{pmatrix}
	0 & \zeta_3 && \\ \zeta_3 & 0 && \\ & & i & \\ &&& \chi'(c)
	\end{pmatrix}
	\end{align*} 
	must have the eigenvalue $1$ as well.
\end{proof}

The complex representation $\rho$ is, therefore, completely described, which also allows us to determine the isogeny type of $T$: it follows immediately from  \hyperref[isogeny]{Section~\ref*{isogeny}} that $T$ is equivariantly isogenous to $F \times F \times E_i \times E$, where $F = \CC/(\ZZ + \zeta_3\ZZ)$, $E_i = \CC/(\ZZ+i\ZZ)$, and $E \subset T$ is another elliptic curve. \\

\emph{(C) An example.} \\
According to the previous section, we may write $T = (F \times F \times E_i \times E)/H$
for a finite subgroup $H$ of translation. By changing the origin in the respective elliptic curves, we may write the action of $a$, $b$, $c$, and $k$ on $T$ as follows:
\begin{align*}
&a(z) = \left(-z_1 + a_1, \ -z_2 + a_2, \ z_3 + a_3, \ z_4 + a_4\right), \\
&b(z) = \left(-z_2 + b_1, \ z_1 + b_2, \ iz_3, \ z_4 + b_4\right), \\
&c(z) = \left(z_1 + c_1, \ -z_2 + c_2, \ z_3 + c_3, \ z_4 + c_4\right), \\
&k(z) = \left(\zeta_3 z_1, \ \zeta_3 z_2, \ z_3 + c_3, \ z_4 + c_3\right).
\end{align*}
We investigate when $\langle a,b,c,k\rangle \subset \Bihol(T)$ is isomorphic to $((C_2 \times C_4) \rtimes C_2) \times C_3$.

\begin{lemma} \label{48-21-rels} \
	\begin{enumerate}[ref=(\theenumi)]
		\item \label{48-21-rels-first} The element $a$ has order $2$ if and only if $(0, \ 0, \ 2a_3, \  2a_4)$ is zero in $T$.
		\item The element $b$ has order $4$ if and only if $4b_4 = 0$ in $E$.
		\item The elements $a$ and $b$ commute if and only if $(a_1+a_2-2b_1, \ a_2-a_1-2b_2, \ (1-i)a_3, \ 0)$ is zero in $T$.
		\item The element $c$ has order $2$ if and only if $(2c_1, \ 0, \ 2c_3, \  2c_4)$ is zero in $T$.
		\item The elements $a$ and $c$ commute if and only if $(2c_1, \ 2c_2 - 2a_2, \ 0, \ 0)$ is zero in $T$.
		\item The relation $c^{-1}bc = ab$ holds if and only if $(a_1 - 2b_1 + c_1+c_2, \ a_2+c_1-c_2, \ a_3 + (1-i)c_3, \ a_4)$ is zero in $T$.
		\item The element $k$ has order $3$ if and only if $(0, \ 0, \ 3k_3, \  3k_4)$ is zero in $T$.
		\item The elements $a$ and $k$ commute if and only if $((\zeta_3-1)a_1, \ (\zeta_3-1)a_2, \ 0, \ 0)$ is zero in $T$.
		\item \label{48-21-rels-last} The elements $b$ and $k$ commute if and only if $((\zeta_3-1)b_1, \ (\zeta_3-1)b_2, \ (i-1)k_3, \ 0)$ is zero in $T$.
	\end{enumerate}
	Consequently, if the nine properties \ref{48-21-rels-first} -- \ref{48-21-rels-last} are satisfied, then the group $\langle a,b,c,k \rangle \subset \Bihol(T)$ is isomorphic to $((C_2 \times C_4) \rtimes C_2) \times C_3$ and does not contain any translations.
\end{lemma}

We finally present a concrete example of a hyperelliptic fourfold with holonomy group $((C_2 \times C_4) \rtimes C_2) \times C_3$. 

\begin{example} \label{48-21-example}
	Let $T = F \times F \times E_i \times E$, where $F = \CC/(\ZZ+ \zeta_3 \ZZ)$, $E_i = \CC/(\ZZ+i\ZZ)$ and $E = \CC/(\ZZ+\tau\ZZ)$ for some $\tau \in \CC$ with positive imaginary part. Define the following automorphisms of $T$:
	\begin{align*}
	&a(z) = \left(-z_1, \ -z_2, \ z_3 + \frac{1+i}{2}, \ z_4\right), \\
	&b(z) = \left(-z_2, \ z_1, \ iz_3, \ z_4 + \frac14\right), \\
	&c(z) = \left(z_1, \ -z_2, \ z_3, \ z_4 + \frac{\tau}{2}\right), \\
	&k(z) = \left(\zeta_3 z_1, \ \zeta_3 z_2, \ z_3, \ z_4 + \frac13\right).
	\end{align*}
	\hyperref[48-21-rels]{Lemma~\ref*{48-21-rels}} immediately implies that $G := \langle a,b,c,k\rangle$ is isomorphic to the desired group $((C_2 \times C_4) \rtimes C_2) \times C_3$. Furthermore, the action of $G$ on $T$ is clearly free: by construction non-trivial element acts on $E_i$ or $E$ by a non-trivial translation. Therefore, $T/G$ is a hyperelliptic manifold.
\end{example}

\emph{(D) The Hodge diamond.} \\
The Hodge diamond of a hyperelliptic fourfold with holonomy group $((C_2 \times C_4) \rtimes C_2) \times C_3$  is 
\begin{center}
	$\begin{matrix}
	&   &  &  & 1 &  &  &  &  \\
	&   &  & 1 &  & 1 &  &  &  \\
	&   & 0 &  & 3 &  & 0  &  & \\
	&  0 &  & 2 &   & 2 &   &  0 &  \\
	0&    & 0 &  & 4  &  & 0  &     &  0 \\
	\end{matrix}$
\end{center}

%\emph{(E) Excluding a certain group containing $((C_2 \times C_4) \rtimes C_2) \times C_3$.}
%
%Our discussion above implies that the central element $k$ of $((C_2 \times C_4) \rtimes C_2) \times C_3$ of order $3$ is necessarily mapped to 
%\begin{align*}
%\diag(\zeta_3, \ \zeta_3, \ 1, \ 1)
%\end{align*}
%or its square by $\rho$. It follows that
%
%\begin{cor} \label{48-22-supgroup-excluded}
%	The group $((C_2 \times C_4) \rtimes C_2) \times C_3 \times C_3$ (ID $[144,102]$) is not hyperelliptic in dimension four.
%\end{cor}

\subsection{$G(4,4,3) \times C_3$ (ID [48,22])} \label{48-22-section}\  \\ 

We investigate hyperelliptic fourfolds with holonomy group
\begin{align*}
G(4,4,3) \times C_3 = \langle g,h,k ~ | ~ g^4 = h^4 = k^3 = 1, h^{-1}gh = h^3, ~ k \text{ central}\rangle.
\end{align*}
More precisely, we show:

\begin{prop} \label{48-22-prop} 
	Let $X = T/(G(4,4,3) \times C_3)$ be a hyperelliptic fourfold with associated complex representation $\rho$. Then:
	\begin{enumerate}[ref=(\theenumi)]
		\item \label{48-22-prop1} Up to equivalence and automorphisms, there are the following two possibilities for $\rho$:
		\begin{enumerate}[label=(\roman*), ref=(\roman*)]
			\item \label{48-22-rho1} $\rho(g) = \begin{pmatrix}
			1+2\zeta_3 & -1 && \\ -2 & -1-2\zeta_3 && \\  && 1 & \\ &&& 1
			\end{pmatrix}$, \quad $\rho(h) = \begin{pmatrix}
			-1 & \zeta_3^2 && \\ -2\zeta_3 & 1 && \\ && 1 & \\ &&& i
			\end{pmatrix}$, \quad $\rho(k) = \begin{pmatrix}
			\zeta_3&&& \\ & \zeta_3&& \\ && 1 & \\ &&& 1
			\end{pmatrix}$, or
			\item \label{48-22-rho2} $\rho(g) = \begin{pmatrix}
			0& -1 && \\ 1 & 0 && \\  && 1 & \\ &&& 1
			\end{pmatrix}$, \quad $\rho(h) = \begin{pmatrix}
			1 & && \\ & -1 && \\ && 1 & \\ &&& i
			\end{pmatrix}$, \quad $\rho(k) = \begin{pmatrix}
			\zeta_3&&& \\ & \zeta_3&& \\ && 1 & \\ &&& 1
			\end{pmatrix}$
		\end{enumerate}
		The equivalence classes of the above two representations belong to different $\Aut(G(4,4,3) \times C_3)$-orbits.
		\item \label{48-22-prop2} The representation $\rho$ induces an equivariant isogeny $F \times F \times E \times E_i \to T$, where $F = \CC/(\ZZ+\zeta_3\ZZ)$ and $E_i = \CC/(\ZZ+i\ZZ)$ denote Fermat's and Gauss' elliptic curve, respectively, and $E \subset T$ is another elliptic curve.
		\item \label{48-22-prop3} Hyperelliptic fourfolds with holonomy group $G(4,4,3) \times C_3$ exist for both choices \ref{48-22-rho1} and \ref{48-22-rho2} of $\rho$.
	\end{enumerate}
	In particular, $X$ moves in a complete $1$-dimensional family of hyperelliptic fourfolds with holonomy group $G(4,4,3) \times C_3$.
\end{prop}

Furthermore, the Hodge diamond of a hyperelliptic fourfold with holonomy group $G(4,4,3) \times C_3$ is determined in part (D). \\

\emph{(A) Representation Theory of $G(4,4,3) \times C_3$.}

Since the irreducible representations of the direct product $G(4,4,3) \times C_3$ are given by the tensor products of irreducible representations of $G(4,4,3)$ and $C_3$, we only list the irreducible representations of $G(4,4,3)$. Since $g^2$ is a commutator, the degree $1$ representations of $G(4,4,3)$ are $\chi_{a,b}$, $a \in \{0,1\}$, $b \in \{0,...,3\}$, where
\begin{center}
	$\chi_{a,b}(g) = (-1)^a, \qquad \chi_{a,b}(h) = i^b$.
\end{center}
Furthermore, the group $G(4,4,3)$ has two irreducible representations of degree $2$. As it will turn out in part (B), defining these representations over $\QQ(\zeta_3)$ is necessary. They are given by
\begin{center}
	\begin{tabular}{lll}
		$\rho_{2,1}$: & $g \mapsto \begin{pmatrix}
		1+2\zeta_3 & -1 \\ -2 & -1-2\zeta_3
		\end{pmatrix}$, & $h \mapsto \begin{pmatrix}
		-1 & \zeta_3^2 \\ -2\zeta_3 & 1
		\end{pmatrix}$, \\
		$\rho_{2,2}$: & $g \mapsto \begin{pmatrix}
		0 & -1 \\ 1 & 0
		\end{pmatrix}$, & $h \mapsto \begin{pmatrix}
		1 & \\ & -1
		\end{pmatrix}$.
	\end{tabular}
\end{center}
Both of these representations are not faithful. The kernel of $\rho_{2,1}$ is generated by $g^2h^2$, whereas the kernel of $\rho_{2,2}$ is generated by $h^2$. Since $g^2h^2$ is not the square of an element of order $4$, but $h^2$ is, the representations $\rho_{2,1}$ and $\rho_{2,2}$ are even non-equivalent up to the action of $\Aut(G(4,4,3))$. \\
%[48,22]
%
%g := [[0,-1,0,0],[1,0,0,0],[0,0,1,0],[0,0,0,1]];
%h := DiagonalMat([1,-1,1,E(4)]);
%k := DiagonalMat([E(3),E(3),1,1]);
%g2 := [[1+2*E(3),-1,0,0],[-2, -1-2*E(3),0,0],[0,0,1,0],[0,0,0,1]];
%h2 := [[-1,E(3)^2,0,0],[-2*E(3),1,0,0],[0,0,1,0],[0,0,0,E(4)]];
%k2 := DiagonalMat([E(3),E(3),1,1]);
%
%G := Group([g,h,k]);
%Aut := AutomorphismGroup(G);
%
%AssignNiceMonomorphismAutomorphismGroup(Aut,G);
%for f in Aut do
%sum := 0;
%for i in [0..3] do
%for j in [0..3] do
%for m in [0,1,2] do
%U := Image(f,g^i*h^j*k^m);
%u := [[U[1][1],U[1][2]],[U[2][1],U[2][2]]];
%U2 := g2^i*h2^j*k2^m;
%u2 := [[U2[1][1],U2[1][2]],[U2[2][1],U2[2][2]]];
%sum := sum + TraceMat(u)*ComplexConjugate(TraceMat(u));
%od;
%od;
%od;
%Print(sum);
%if sum <> 0 then Print(" yeeeeeeeeeeeeeee"); fi;
%Print("\n");
%od;

\emph{(B) The complex representation and the isogeny type of the torus.} \\

We determine the possible complex representations $\rho \colon G(4,4,3) \times C_3 \to \GL(4,\CC)$ of a hyperelliptic fourfold $T/(G(4,4,3) \times C_3)$ up to equivalence and automorphisms. Furthermore, we describe the isogeny type of $T$ induced by the representation $\rho$. As a first step, we investigate $\rho|_{G(4,4,3)}$:

\begin{lemma} \label{G(4,4,3)-rho}
	Up to symmetry and automorphisms, $\rho|_{G(4,4,3)}$ can be described as follows:
	\begin{enumerate}[ref=(\theenumi)]
		\item \label{G(4,4,3)-rho21} if $\rho_{2,1}$ is a constituent of $\rho|_{G(4,4,3)}$, then
		$\rho|_{G(4,4,3)} = \rho_{2,1} \oplus \chi_{0,0} \oplus \chi_{0,1}$.
		\item \label{G(4,4,3)-rho22} if $\rho_{2,2}$ is a constituent of $\rho_{G(4,4,3)}$, then $\rho|_{G(4,4,3)} = \rho_{2,1} \oplus \chi_{0,b} \oplus \chi_{0,1}$ for some $b \in \{0,2\}$.
	\end{enumerate}
\end{lemma}

\begin{proof}
	First of all, \hyperref[lemma-two-generators]{Proposition~\ref*{lemma-two-generators}} implies that $\rho|_{G(4,4,3)}$ is equivalent to the direct sum $\rho_2 \oplus \chi \oplus \chi'$
	of an irreducible degree $2$ representation $\rho_2$ and two linear characters $\chi$, $\chi'$, where we may assume $\chi'$ to be non-trivial. Moreover, recall from part (A) that every automorphism of $G(4,4,3)$ stabilizes the equivalence class of $\rho_2$. \\
	
	\ref{G(4,4,3)-rho21} Here, $\rho_2 = \rho_{2,1}$. Since $g^2h^2 \in \ker(\rho_{2,1})$ and $\rho|_{G(4,4,3)}$ is faithful, we may assume that $\chi'(h)$ is a primitive fourth root of unity. After applying the automorphism $g \mapsto g$, $h \mapsto h^{-1}$ of $G(4,4,3)$, we may assume that $\chi'(h) = i$. Now, 
	\begin{align*}
	\rho(h) = \diag(\rho_{2,1}(h), ~ \chi(h), ~ i)
	\end{align*}
	has the eigenvalue $1$ and thus $\chi(h) = 1$. Consider now
	\begin{align*}
	\rho(g) = \diag(\rho_{2,1}(g), ~ \chi(g), ~ \chi'(g)) \quad \text{ and } \quad \rho(gh^2) = \diag(-\rho_{2,1}(g), ~ \chi(g), ~ -\chi'(g)).
	\end{align*}
	Both of these matrices have the eigenvalue $1$, and hence $\chi(g) = 1$. Furthermore, if $\chi'(g) = -1$, we may apply the automorphism $g \mapsto gh^2$, $h \mapsto h$ to assume that $\chi(g) = 1$. \\
	
	\ref{G(4,4,3)-rho22} In this case, $\rho_2 = \rho_{2,2}$. Similarly as above, $h^2 \in \ker(\rho_{2,2})$, and thus we may assume that $\chi'(h) = i$. Again, after possibly applying $g \mapsto gh^2$, $h \mapsto h$, we may assume that $\chi(g) = 1$. Taking the matrices
	\begin{align*}
	\rho(gh^2) = \diag(\rho_{2,2}(g), ~ \chi(gh^2), ~ -1) \quad \text{ and } 	\rho(g^2h^2) = \diag(\rho_{2,2}(g^2), ~ \chi(h^2), ~ -1).
	\end{align*}
	into account, which both have the eigenvalue $1$, we conclude that $\chi(gh^2) = \chi(h^2) = 1$, i.e., $\chi= \chi_{0,0}$  or $\chi=\chi_{0,1}$.
\end{proof}

%First of all, observe that
%\begin{align*}
%	G(4,4,3) \times C_3 \cong G(12,4,7) 
%\end{align*}
%is metacyclic as well, so that  \hyperref[lemma-two-generators]{Proposition~\ref*{lemma-two-generators}} \ref{ltg-2} implies that $\rho$ splits into the direct sum of a degree $2$ representation and two representations $\chi$, $\chi'$ of degree $1$. We determine $\chi$ and $\chi'$ to $G(4,4,3) = \langle g,h \rangle$.

%\begin{proof}
%	Recall that $\ker(\rho_{2,1}) = \langle g^2h^2 \rangle$ and that $\ker(\rho_{2,2}) = \langle h^2 \rangle$. Since $g^2$ is a commutator, the faithfulness of $\rho$ implies that we may assume $\chi'(h)$ to be a primitive fourth root of unity. After possibly applying automorphism $g \mapsto g$, $h \mapsto h^{-1}$, we may assume that $\chi'(h) = i$. Similarly, we may apply the automorphism $g \mapsto gh^2$, $h \mapsto h$ to assume that $\chi'(g) = 1$. \\
%	In order to determine $\chi(g)$ and $\chi(h)$, we distinguish between whether $\rho$ contains $\rho_{2,1}$ or $\rho_{2,2}$ as a direct summand. If $\rho_{2,1}$ is a direct summand of $\rho$, then
%	\begin{align*}
%		\rho(h) = \diag(\rho_{2,1}(h), \ \chi(h), \ i)
%	\end{align*}
%	does not have the eigenvalue $1$, unless $\chi(h) = 1$. Similarly, the condition that
%	\begin{align*}
%	\rho(gh^2) = \diag(-\rho_{2,1}(g), \ \chi(g), \ -1)
%	\end{align*}
%	has the eigenvalue $1$ implies that $\chi(g) = 1$. \\
%	On the other hand, if $\rho_{2,2}$ is a direct summand of $\rho$, then we first observe that $\chi(g) = \chi(h^2)$, since
%	\begin{align*}
%	\rho(g^3h^2) = \rho(g)\rho(g^2h^2) = \diag(\rho_{2,2}(g), \ \chi(g)\chi(h^2), \ -1)
%	\end{align*}
%	must have the eigenvalue $1$.
%\end{proof}

The representation $\rho|_{G(4,4,3)}$ is, therefore, completely determined in Case \ref{G(4,4,3)-rho21}. To obtain a unique decomposition also in Case \ref{G(4,4,3)-rho22}, we consider the action of the full group $G(4,4,3) \times C_3$. Denoting by $k$ a generator of $C_3$, we easily obtain:

\begin{lemma}
	Up to replacing $k$ by $k^2$, 
	\begin{align*}
	\rho(k) = \diag(\zeta_3, \ \zeta_3, \ 1, \ 1).
	\end{align*}
	In particular, $\rho|_{G(4,4,3)}$ is equivalent to the direct sum $\rho_{2,j} \oplus \chi_{0,0} \oplus \chi_{0,1}$ for some $j \in \{1,2\}$.
\end{lemma}

\begin{proof}
	The first assertion follows directly from \hyperref[prop:metacyclic-c3]{Proposition~\ref*{prop:metacyclic-c3}}. The second assertion follows from \hyperref[G(4,4,3)-rho]{Lemma~\ref*{G(4,4,3)-rho}} if $j = 1$. If $j = 2$, we obtain the result by noting that the matrix
	\begin{align*}
	\rho(hk) = \diag(\zeta_3 \rho_{2,2}(h), \ \chi_{0,b}(h), \ i)
	\end{align*}
	must have the eigenvalue $1$, which implies $b = 0$.
\end{proof}

This completes the proof of \hyperref[48-22-prop]{Proposition~\ref*{48-22-prop}} \ref{48-22-prop1}.

Now, if $T$ is a complex torus of dimension $4$ admitting an action of $G(4,4,3) \times C_3$, whose associated complex representation $\rho$ is one of the two representations of \hyperref[48-22-prop]{Proposition~\ref*{48-22-prop}} \ref{48-22-prop1}, we easily obtain that $T$ is equivariantly isogenous to $F \times F \times E \times E'$, where $F = \CC/(\ZZ+\zeta_3\ZZ)$, $E_i = \CC/(\ZZ+i\ZZ)$ and $E$ is an elliptic curve (see \hyperref[isogeny]{Section~\ref*{isogeny}}). \\

\emph{(C) Examples} \\
We have seen above that there are two complex representations up to equivalence and automorphisms, which potentially allow a free and translation-free action of $G(4,4,3) \times C_3$. In the following, we will give two examples of hyperelliptic fourfolds $T/(G(4,4,3) \times C_3)$, one for each possible complex representation. \\

By part (B), we may assume that $T$ is (equivariantly) isogenous to $T' := F \times F \times E \times E_i$, where $F = \CC/(\ZZ+\zeta_3\ZZ)$ is the Fermat elliptic curve, $E_i = \CC/(\ZZ+i\ZZ)$ is the Gauss elliptic curve and $E = \CC/(\ZZ+\tau\ZZ)$ is another elliptic curve.

We first deal with the case where $\rho_{2,1}$ is a summand of $\rho|_{G(4,4,3)}$. After a change of coordinates in the elliptic curves $F$ and $E_i$, we may write
\begin{align*}
&g(z) = ((1+2\zeta_3)z_1 - z_2, \ -2z_1 - (1+2\zeta_3)z_2, \ z_3 + a_3, \ z_4+a_4), \\
&h(z) = (-z_1 + \zeta_3^2 z_2 + b_1, \ -2\zeta_3 z_1 + z_2 + b_2, \ z_3 + b_3, \ iz_4), \\
&k(z) = (\zeta_3 z_1 + c_1, \ \zeta_3 z_2 + c_2, \ z_3 + c_3, \ z_4 + c_4).
\end{align*}

As usual, we give sufficient and necessary criteria for $\langle g,h,k\rangle \subset \Bihol(T)$ to be isomorphic to $G(4,4,3) \times C_3$. 

\begin{lemma} \label{48-22-rels-case1} \ 
	\begin{enumerate}[ref=(\theenumi)]
		\item \label{48-22-rels-1-first} It holds $g^4 = \id_T$ if and only if $(0, \ 0, \ 4a_3, \ 4a_4)$ is zero in $T$.
		\item It holds $h^4 = \id_T$ if and only if $4b_3 = 0$ in $E$.
		\item It holds $h^{-1}gh = g^3$ if and only if
		$(2b_1 - (2\zeta_3^2 + 1) b_2, \ -2(2\zeta_3^2+1)b_1 - 2b_2, \ 2a_3, \ (i+3)a_4)$ is zero in $T$.
		\item It holds $k^3 = \id_T$ if and only if $(0, \ 0,\ 3c_3, \ 3c_4)$ is zero in $T$. 
		\item The maps $g$ and $k$ commute if and only if $(2\zeta_3 c_1 - c_2, \ - 2c_1 + 2\zeta_3^2 c_2, \ 0, \ 0)$ is zero in $T$.
		\item \label{48-22-rels-1-last} The maps $h$ and $k$ commute if and only if $((\zeta_3-1)b_1 + 2c_1 - \zeta_3^2c_2, \ (\zeta_3-1)b_2 +2\zeta_3c_1, \ 0, \ (1-i)c_4)$ is zero in $T$.
	\end{enumerate}
	%F := CyclotomicField(12);
	%
	%z := F.1^4;
	%i := F.1^3;
	%
	%P<a3,a4,b1,b2,b3,c1,c2,c3,c4> := PolynomialRing(F,9);
	%
	%g := Matrix(P,5,5,[[1+2*z, -1,0,0,0],[-2,-1-2*z,0,0,0],[0,0,1,0,a3],[0,0,0,1,a4],[0,0,0,0,1]]);
	%h := Matrix(P,5,5,[[-1, z^2,0,0,b1],[-2*z,1,0,0,b2],[0,0,1,0,b3],[0,0,0,i,0],[0,0,0,0,1]]);
	%k := Matrix(P,5,5,[[z, 0,0,0,c1],[0,z,0,0,c2],[0,0,1,0,c3],[0,0,0,1,c4],[0,0,0,0,1]]);
	%
	%
	%g^4;
	%h^4;
	%g^3 - h^-1*g*h;
	%k^3;
	%g*k - k*g;
	%k*h - h*k;
	%[   1    0    0    0    0]
	%[   0    1    0    0    0]
	%[   0    0    1    0 4*a3]
	%[   0    0    0    1 4*a4]
	%[   0    0    0    0    1]
	%
	%
	%[   1    0    0    0    0]
	%[   0    1    0    0    0]
	%[   0    0    1    0 4*b3]
	%[   0    0    0    1    0]
	%[   0    0    0    0    1]
	%
	%
	%[0   0   0   0   2*b1 + (2*zeta_12^2 - 1)*b2]
	%[0   0   0   0   (4*zeta_12^2 - 2)*b1 - 2*b2]
	%[0   0   0   0   2*a3]
	%[0   0   0   0   (zeta_12^3 + 3)*a4]
	%[0   0   0   0   0]
	%
	%
	%[   1    0    0    0    0]
	%[   0    1    0    0    0]
	%[   0    0    1    0 3*c3]
	%[   0    0    0    1 3*c4]
	%[   0    0    0    0    1]
	%
	%
	%[0   0   0   0   (2*zeta_12^2 - 2)*c1 - c2]
	%[0   0   0   0   -2*c1 - 2*zeta_12^2*c2]
	%[0   0   0   0   0]
	%[0   0   0   0   0]
	%[0   0   0   0   0]
	%
	%
	%[0   0   0   0   (zeta_12^2 - 2)*b1 + 2*c1 + zeta_12^2*c2]
	%[0   0   0   0   (zeta_12^2 - 2)*b2 + (2*zeta_12^2 - 2)*c1]
	%[0   0   0   0   0]
	%[0   0   0   0   (-zeta_12^3 + 1)*c4]
	%[0   0   0   0   0]

	Consequently, if the six properties \ref{48-22-rels-1-first} -- \ref{48-22-rels-1-last} are satisfied, then the group $\langle g,h,k \rangle \subset \Bihol(T)$ is isomorphic to $G(4,4,3) \times C_3$ and does not contain any translations.
\end{lemma}

An explicit example is presented below.

\begin{example} \label{48-22-example-case1}
	Define $T = (F \times F \times E \times E_i)/H$, where
	\begin{itemize}
		\item $F = \CC/(\ZZ+\zeta_3\ZZ)$ is Fermat's elliptic curve,
		\item $E_i = \CC/(\ZZ+i\ZZ)$ is Gauss' elliptic curve,
		\item $E = \CC/(\ZZ+\tau\ZZ)$ is an elliptic curve, and
		\item $H$ is the subgroup of order $2$ spanned by $\left(0, \ 0, \ \frac12, \frac{1+i}{2}\right)$.
	\end{itemize}
	Furthermore, we define the following holomorphic automorphisms of $F \times F \times E \times E_i$:
	\begin{align*}
	&g(z) = \left((1+2\zeta_3)z_1 - z_2, \ -2z_1 - (1+2\zeta_3)z_2, \ z_3 + \frac{1}{4}, \ z_4 + \frac12\right), \\
	&h(z) = \left(-z_1 + \zeta_3^2 z_2, \ -2\zeta_3 z_1 + z_2, \ z_3 + \frac{\tau}4, \ iz_4\right), \\
	&k(z) = \left(\zeta_3 z_1, \ \zeta_3 z_2, \ z_3 + \frac13, \ z_4\right).
	\end{align*}
	By construction, their linear parts map $H$ to $H$ and thus $g$, $h$ and $k$ can be viewed as elements in $\Bihol(T)$. Viewed as such,  \hyperref[48-22-rels-case1]{Lemma~\ref*{48-22-rels-case1}} shows that $\langle g,h,k\rangle \subset \Bihol(T)$ is isomorphic to $G(4,4,3) \times C_3$ and does not contain any translations. \\
	Let us prove that $\langle g,h,k\rangle$ acts freely on $T$. Indeed, every element of this group can be written uniquely in the form $k^ag^b h^c$, where $a \in \{0,1,2\}$ and $b,c \in \{0,...,3\}$. This element acts on $S := (E \times E_i)/\langle (\frac12,\ \frac{1+i}{2})\rangle$ by
	\begin{align*}
	k^a g^b h^c (z_3,z_4) = \left(z_3 + \frac{a}{3} + \frac{b + c\tau}{4}, \ i^c z_4 + \frac{b}{2} \right).
	\end{align*}
	Clearly, $g^ah^bk^c$ only has a fixed point on $S$ (which is equivalent to having a fixed point on $T$) if $a = b = c = 0$. This completes the proof that $T/\langle g,h,k\rangle$ is a hyperelliptic fourfold with holonomy group $G(4,4,3) \times C_3$. 
\end{example}

We repeat the process for the other, i.e., the case in which $\rho_{2,2}$ is a constituent of the complex representation $\rho|_{G(4,4,3)}$. Here, we may write the action as follows:
\begin{align*}
&g(z) = (-z_2, \ z_1, \ z_3 + a_3, \ z_4 + a_4), \\
&h(z) = (z_1+b_1, \ -z_2 + b_2, \ z_3 + b_3, \ iz_4), \\
&k(z) = (\zeta_3 z_1 + c_1, \ \zeta_3 z_2 + c_2, \ z_3 + c_3, \ z_4 + c_4).
\end{align*}
The following lemma is the analog of  \hyperref[48-22-rels-case1]{Lemma~\ref*{48-22-rels-case1}} in the second case.

\begin{lemma} \label{48-22-rels-case2} \
	\begin{enumerate}[ref=(\theenumi)]
		\item \label{48-22-rels-2-first} It holds $g^4 = \id_T$ if and only if $(0, \ 0, \ 4a_3, \ 0)$ is zero in $T$.
		\item It holds $h^4 = \id_T$ if and only if $(4b_1, \ 0, \ 4b_3, \ 4b_4)$ is zero in $T$.
		\item It holds $h^{-1}gh = g^3$ if and only if $(b_1+b_2, \ b_1-b_2, \ 2a_3, \ (i+3)a_4)$ is zero in $T$.
		\item It holds $k^3 = \id_T$ if and only if $(0, \ 0,\ 3c_3, \ 3c_4)$ is zero in $T$. 
		\item The maps $g$ and $k$ commute if and only if $(c_1+c_2, \ c_2-c_1, \ 0, \ 0)$ is zero in $T$.
		\item \label{48-22-rels-2-last} The maps $h$ and $k$ commute if and only if $((\zeta_3-1)b_1, \ (\zeta_3-1)b_2 + 2c_2 \ 0, \ (1-i)c_4)$ is zero in $T$.
	\end{enumerate}
	Consequently, if the six properties \ref{48-22-rels-2-first} -- \ref{48-22-rels-2-last} are satisfied, then the group $\langle g,h,k \rangle \subset \Bihol(T)$ is isomorphic to $G(4,4,3) \times C_3$ and does not contain any translations. 
\end{lemma}
%F := CyclotomicField(12);
%
%z := F.1^4;
%i := F.1^3;
%
%P<a3,a4,b1,b2,b3,c1,c2,c3,c4> := PolynomialRing(F,9);
%
%g := Matrix(P,5,5,[[0, -1,0,0,0],[1,0,0,0,0],[0,0,1,0,a3],[0,0,0,1,a4],[0,0,0,0,1]]);
%h := Matrix(P,5,5,[[1,0,0,0,b1],[0,-1,0,0,b2],[0,0,1,0,b3],[0,0,0,i,0],[0,0,0,0,1]]);
%k := Matrix(P,5,5,[[z, 0,0,0,c1],[0,z,0,0,c2],[0,0,1,0,c3],[0,0,0,1,c4],[0,0,0,0,1]]);
%
%
%g^4;
%h^4;
%g^3 - h^-1*g*h;
%k^3;
%g*k - k*g;
%k*h - h*k;
%
%[   1    0    0    0    0]
%[   0    1    0    0    0]
%[   0    0    1    0 4*a3]
%[   0    0    0    1 4*a4]
%[   0    0    0    0    1]
%[   1    0    0    0 4*b1]
%[   0    1    0    0    0]
%[   0    0    1    0 4*b3]
%[   0    0    0    1    0]
%[   0    0    0    0    1]
%[0   0   0   0   b1 + b2]
%[0   0   0   0   b1 - b2]
%[0   0   0   0   2*a3]
%[0   0   0   0   (zeta_12^3 + 3)*a4]
%[0   0   0   0   0]
%[   1    0    0    0    0]
%[   0    1    0    0    0]
%[   0    0    1    0 3*c3]
%[   0    0    0    1 3*c4]
%[   0    0    0    0    1]
%[       0        0        0        0 -c1 - c2]
%[       0        0        0        0  c1 - c2]
%[       0        0        0        0        0]
%[       0        0        0        0        0]
%[       0        0        0        0        0]
%[0   0   0   0   (zeta_12^2 - 2)*b1]
%[0   0   0   0   (zeta_12^2 - 2)*b2 + 2*c2]
%[0   0   0   0   0]
%[0   0   0   0   (-zeta_12^3 + 1)*c4]
%[0   0   0   0   0]

\begin{example}
	Let $T$ be as in  \hyperref[48-22-example-case1]{Example~\ref*{48-22-example-case1}}. Similarly as in the cited example, consider the following automorphisms of $T$:
	\begin{align*}
	&g(z) = \left(-z_2, \ z_1, \ z_3 + \frac{1}{4}, \ z_4 + \frac12\right), \\
	&h(z) = \left(z_1, \ -z_2, \ z_3 + \frac{\tau}4, \ iz_4\right), \\
	&k(z) = \left(\zeta_3 z_1, \ \zeta_3 z_2, \ z_3 + \frac13, \ z_4\right).
	\end{align*}
	\hyperref[48-22-rels-case2]{Lemma~\ref*{48-22-rels-case2}} proves that $\langle g,h,k\rangle \subset \Bihol(T)$ is isomorphic to $G(4,4,3) \times C_3$ and does not contain any translations. Observing that these automorphisms differ from the ones given in \hyperref[48-22-example-case1]{Example~\ref*{48-22-example-case1}} only by the linear action on $F \times F$, we may check the freeness of the action exactly as in the previous example. Therefore $T/\langle g,h,k \rangle$ is a hyperelliptic fourfold, which falls under Case 2.
\end{example}

The statement of \hyperref[48-22-prop]{Proposition~\ref*{48-22-prop}} \ref{48-22-prop3} is, therefore, completely proved.

\emph{(D) The Hodge diamond.} \\

The Hodge diamond of a hyperelliptic fourfold with holonomy group $G(4,4,3) \times C_3$ is
\begin{center}
	$\begin{matrix}
	&   &  &  & 1 &  &  &  &  \\
	&   &  & 1 &  & 1 &  &  &  \\
	&   & 0 &  & 3 &  & 0  &  & \\
	&  0 &  & 2 &   & 2 &   &  0 &  \\
	0&    & 0 &  & 4  &  & 0  &     &  0 \\
	\end{matrix}$
\end{center}
It is independent of the choice of complex representation.

\subsection{$A_4 \times C_4$ (ID [48,31])} \label{48-31-section}\ \\

This section is dedicated to describing hyperelliptic fourfolds with holonomy group
\begin{align*}
A_4 \times C_4 = \langle \sigma, \xi, \kappa \ | \ \sigma^3 = \xi^2 = \kappa^4 = (\sigma \xi)^3 = 1, ~ \kappa \text{ central}\rangle.
\end{align*}
More precisely, we show the following proposition in part (B) and  \hyperref[48-31-example]{Example~\ref*{48-31-example}} of part (C) below:

\begin{prop} \label{48-31-prop}
	Let $X = T/(A_4 \times C_4)$ be a hyperelliptic fourfold with associated complex representation $\rho$. Then:
	\begin{enumerate}[ref=(\theenumi)]
		\item Up to equivalence and automorphisms, $\rho$ is given as follows:
		\begin{align*}
		\rho(\sigma) = \begin{pmatrix}
		0 & 0 & 1& \\  1& 0 &0& \\ 0&1&0 & \\ &&&1
		\end{pmatrix}, \qquad \rho(\xi) = \begin{pmatrix}
		1& && \\  & -1 && \\ &&-1 & \\ &&& 1
		\end{pmatrix}, \qquad \rho(\kappa) = \begin{pmatrix}
		i & && \\ & i && \\ && i & \\ &&&1
		\end{pmatrix}.
		\end{align*}
		\item The representation $\rho$ induces an equivariant isogeny $E_i \times E_i \times E_i \times E \to T$, where $E_i = \CC/(\ZZ+i\ZZ) \subset T$ are copies of Gauss' elliptic curve and $E \subset T$ is another elliptic curve.
		\item Hyperelliptic fourfolds with holonomy group $A_4 \times C_4$ exist.
	\end{enumerate}
	In particular, $X$ moves in a complete $1$-dimensional family of hyperelliptic fourfolds with holonomy group $A_4 \times C_4$.
\end{prop}

In addition to the above, the Hodge diamond of hyperelliptic fourfolds $T/(A_4 \times C_4)$ is presented in part (D). \\

\emph{(A) Representation Theory of $A_4 \times C_4$.} \\

As mentioned above, the alternating group $A_4$ is generated by a $3$-cycle $\sigma$ and a double transposition $\xi$.
It is well-known that the derived subgroup of $A_4$ is the Klein four subgroup generated by the double transpositions. The linear characters of $A_4$ are hence given by $\chi_j(\sigma) = \zeta_3^j$ and $\chi_j(\xi) = 1$ for $j \in \{0,1,2\}$. Furthermore, $A_4$ has only one irreducible representation $\rho_3$ of degree larger than $1$. It is given by
\begin{align*}
\rho_3(\sigma) = \begin{pmatrix}
0 & 0 & 1 \\ 1 & 0 & 0 \\ 0 & 1 & 0
\end{pmatrix}, \qquad \rho_3(\xi) = \begin{pmatrix}
1 && \\ & -1 & \\ && -1
\end{pmatrix}.
\end{align*}

\emph{(B) The complex representation and the isogeny type of the torus.} \\

We determine the complex representation $\rho \colon A_4 \times C_4 \to \GL(4,\CC)$ of a hyperelliptic fourfold $T/(A_4 \times C_4)$ and determine the isogeny type of $T$. As a first goal, we decompose $\rho|_{A_4}$ into irreducible representations  (\hyperref[A4-rep]{Corollary~\ref*{A4-rep}}) and only then consider the action of $A_4 \times C_4$. Finally, the equivalence class of $\rho$ is then determined in  \hyperref[A4-kappa]{Lemma~\ref*{A4-kappa}}, see also equation \ref{A4xC4-kappa} below. \\

The following result in the $3$-dimensional case was already established in \cite[Proposition 7]{Uchida-Yoshihara}. Nevertheless, we give the proof for the reader's convenience.

\begin{lemma} \label{A4-lemma-fixed-point}
	Let $T'$ be a $3$-dimensional complex torus that admits an action of the group $A_4 = \langle \sigma, \xi \rangle$ given by the formulae
	\begin{align*}
	&\sigma(z_1,z_2,z_3) = (z_3 + s_1, \ z_1 + s_2, \ z_2 + s_3), \\
	&\xi(z_1,z_2,z_3) = (z_1 + t_1, \ -z_2 + t_2, \ -z_3 + t_3).
	\end{align*}
	Then $\sigma$ has a fixed point on $T'$.
\end{lemma}

\begin{proof}
	The condition that $\sigma^3 = \id_{T'}$ is equivalent to the element $(s_1+s_2+s_3, \ s_1+s_2+s_3, \ s_1+s_2+s_3)$ being zero in $T'$. Applying the linear part of $\sigma^2\xi$ to this element, we obtain that the element $(-(s_1+s_2+s_3), \ -(s_1+s_2+s_3), \ s_1+s_2+s_3)$ is zero in $T'$ as well. We now calculate
	\begin{align*}
	\sigma(0, \ -s_1-s_3, \ s_2+s_3) = (s_1+s_2+s_3, \ s_2, \ -s_1) =  (0, \ -s_1-s_3, \ s_2+s_3).
	\end{align*}
\end{proof}

As an easy consequence of \hyperref[A4-lemma-fixed-point]{Lemma~\ref*{A4-lemma-fixed-point}}, we complete the description of $\rho|_{A_4}$. Observe that $\rho|_{A_4}$ induces an equivariant isogeny $T' \times E \to T$, where $T' \subset T$ is a subtorus of dimension $3$ and $E \subset T$ is an elliptic curve (see \hyperref[isogeny]{Section~\ref*{isogeny}}).

\begin{cor} \label{A4-rep}
	Up to equivalence, the complex representation $\rho|_{A_4}$ is given by
	\begin{align*}
	\rho(\sigma) = \begin{pmatrix}
	0 & 0 & 1 & \\ 1 & 0 & 0 & \\ 0 & 1 & 0 & \\ &&& 1
	\end{pmatrix}, \qquad \rho(\xi) = \begin{pmatrix}
	1 && & \\ & -1 & & \\ && -1 & \\ &&& 1
	\end{pmatrix}.
	\end{align*}
\end{cor}

\begin{proof}
	According to \hyperref[A4-lemma-fixed-point]{Lemma~\ref*{A4-lemma-fixed-point}}, $\sigma$ is mapped to $1$ by the degree $1$ summand of $\rho$ (else, $\sigma$ would have a fixed point on $T \sim_{(A_4 \times C_4)-\isog} T' \times E$)¸.
\end{proof}

So far, we have only considered the action of $A_4$ on $T$ -- however, our goal is to construct an example of a hyperelliptic fourfold with holonomy group $A_4 \times C_4$. We thus take the additional structure given by an element $\kappa$ of order $4$ that commutes with $\sigma$ and $\xi$ into account. The upcoming lemma determines the image of $\kappa$ under the complex representation $\rho$. 

\begin{lemma} \label{A4-kappa}
	The element $\kappa$ is contained in the kernel of the degree $1$ summand of $\rho$.
\end{lemma}

\begin{proof}
	Assume the contrary. Up to replacing $\kappa$ by $\kappa^3$, we may then assume that $\rho(\kappa) = \diag(1,1,1,i)$. Write the action of $\sigma$ and $\kappa$ on $T$ as follows:
	\begin{align*}
	&\sigma(z) = (z_3 + s_1, \ z_1 + s_2, \ z_3 + s_3, \ z_4 + s_4), \\
	&\kappa(z) = (z_1 + k_1, \ z_2 + k_2, \ z_3 + k_3, \ iz_4 + k_4).
	\end{align*}
	The elements $\sigma^3$ and $[\sigma,\kappa]$ being the identity on $T$ yields that 
	\begin{align*}
	v := (s_1+s_2+s_3, \ s_1+s_2+s_3, \ s_1+s_2+s_3, \ 3s_4) \text{ and } w := (k_1-k_3, \ k_2-k_1, \ k_3-k_2, \ (i-1)s_4)
	\end{align*} 
	are zero in $T$. It follows that the element $(\rho(\kappa) + \id_T)w = (0, \ 0, \ 0, \ 2s_4)$ is zero in $T$. Subtracting this element from $v$, we obtain that 
	\begin{align*}
	(s_1+s_2+s_3, \ s_1+s_2+s_3, \ s_1+s_2+s_3, \ s_4)
	\end{align*}
	is zero in $T$ and thus, we may rewrite $\sigma$ as follows:
	\begin{align*}
	\sigma(z) = (z_3 + s_1', \ z_1 + s_2', \ z_3 + s_3', \ z_4), \quad \text{ where } s_i' := s_i - (s_1+s_2+s_3).
	\end{align*}
	The same computation as in the proof of \hyperref[A4-lemma-fixed-point]{Lemma~\ref*{A4-lemma-fixed-point}} now shows that $\sigma$ has a fixed point on $T$.
\end{proof}

\begin{rem} \label{A4-rem-kappa}
	The proof of  \hyperref[A4-kappa]{Lemma~\ref*{A4-kappa}} shows more generally that if an element of the form $(w_1, \ w_2, \ w_3, \ s_4)$ is zero in $T$, then $\sigma$ does not act freely on $T$. Indeed, $\sigma$ is equal to
	\begin{align*}
	\sigma(z) = (z_3 + s_1', \ z_1 + s_2', \ z_3 + s_3', \ z_4), \quad \text{ where } s_i' := s_i - w_i
	\end{align*}
	and we may again conclude by invoking  \hyperref[A4-lemma-fixed-point]{Lemma~\ref*{A4-lemma-fixed-point}}.
\end{rem}

It follows that up to replacing $\kappa$ by $\kappa^3$, 
\begin{align} \label{A4xC4-kappa}
\rho(\kappa) = \diag(i, \ i, \ i, \ 1),
\end{align}
which completes the description of the complex representation of a hyperelliptic fourfold with holonomy group $A_4 \times C_4$. Before we give an example of such a hyperelliptic fourfold, we observe that $\rho(\kappa)$ acts on the $3$-dimensional complex torus $T'$ by multiplication by $i$: hence $T'$ is equivariantly isogenous to the cube of Gauss' elliptic curve $E_i = \CC/(\ZZ + i \ZZ)$. \\

%\begin{rem}
%{$E_i$ canonically isomorphic}
%\end{rem}

\emph{(C) An example.} \\
The aim of this part is to give an explicit example of a hyperelliptic fourfold $T/(A_4 \times C_4)$. According to the previous section, we may assume that
\begin{align*}
T = (E_i \times E_i \times E_i \times E)/H,
\end{align*}
where $H$ is a finite subgroup of translations. By changing the origin in the respective elliptic curves, we may write the action of $\sigma$, $\xi$, and $\kappa$ on $T$ as follows:
\begin{align*}
&\sigma(z) = (z_3 + s_1, \ z_1 + s_2, \ z_3 + s_3, \ z_4 + s_4), \\
&\xi(z) = (z_1 + t_1, \ -z_2, \ -z_3, \ z_4+t_4), \\
&\kappa(z) = (iz_1 + k_1, \ iz_2 + k_2, \ iz_3 + k_3, \ z_4 + k_4).
\end{align*} 
As usual, we first develop necessary and sufficient conditions on the translation parts for the action to be well-defined.

\begin{lemma} \label{A4-relations} \ 
	\begin{enumerate}[ref=(\theenumi)]
		\item \label{A4-rels-first} The relation $\sigma^3 = \id_T$ holds if and only if $(s_1 + s_2 + s_3, \ s_1 + s_2 + s_3, \ s_1 + s_2 + s_3, \ 3s_4)$ is zero in $T$.
		\item The relation $\xi^2 = \id_T$ holds if and only if $(2t_1, \ 0, \ 0, \ 2t_4)$ is zero in $T$.
		\item The relation $(\sigma \xi)^3 = \id_T$ holds if and only if $$(s_1+s_2-s_3 + t_1, \ s_1+s_2-s_3 + t_1, \ -s_1-s_2+s_3 - t_1, \ 3s_4 +3t_4 )$$ is zero in $T$. 
		\item The relation $\kappa^4 = \id_T$ holds if and only if $4k_4$ is zero in $E$.
		\item The elements $\sigma$ and $\kappa$ commute if and only if $$((i-1)s_1 + k_1 - k_3, \ (i-1)s_2 + k_2 - k_1, \ (i-1)s_3 + k_3 - k_2, \ 0 )$$ is zero in $T$.
		\item \label{A4-rels-last} The elements $\xi$ and $\kappa$ commute if and only if $((i-1)t_1, \ 2k_2, \ 2k_3, \ 0)$ is zero in $T$.
	\end{enumerate}
	Consequently, if the six properties \ref{A4-rels-first} --\ref{A4-rels-last} are satisfied, then the group $\langle \sigma, \xi, \kappa \rangle \subset \Bihol(T)$ is isomorphic to $A_4 \times C_4$ and does not contain any translations.
\end{lemma}

It is well-known that $\id_T$, $\sigma$, $\sigma^2$ and $\xi$ forms a system of representatives for the conjugacy classes of $A_4 = \langle \sigma, \xi \rangle$. The next lemma explains how to check that $A_4 \times C_4$ acts freely on $T$.

\begin{lemma} \label{A4-freeness} \
	\begin{enumerate}[ref=(\theenumi)]
		\item \label{A4-freeness-1} The elements $\kappa^n\sigma^m$, $0 \leq n < 4$, $0 \leq m < 3$ act freely on $T$ if and only if $H$ does not contain an element whose last coordinate is $nk_4 + ms_4$.
		\item The element $\xi$ acts freely on $T$ if and only if $H$ does not contain an element of the form $(t_1, \ w_2, \ w_3, \ t_4)$ for some $w_2, w_3 \in  E_i$.  
		\item The elements $\kappa^n \xi$, $n \in \{1,2,3\}$ act freely on $T$ if and only if $H$ does not contain an element whose last coordinate is $nk_4 + t_1$. 
	\end{enumerate}
	
\end{lemma}

\begin{proof}
	We only prove the delicate part, that is, the criterion for the freeness of $\sigma$ and $\sigma^2$ given in \ref{A4-freeness-1}. First of all, observe that since $\sigma$ is an element of order $3$, the element $\sigma^2$ acts freely if and only if $\sigma$ does. Now  \hyperref[A4-rem-kappa]{Remark~\ref*{A4-rem-kappa}} shows $H$ does not contain an element whose last coordinate is $s_4$ if $\sigma$ acts freely. Conversely, if $\sigma$ has a fixed point $w = (w_1,...,w_4)$ on $T$, then
	\begin{align*}
	\sigma(w)-w = (w_3-w_1 + s_1, \ w_1 - w_2 + s_2, \ w_2 - w_3 + s_3, \ s_4) = 0 \text{ in } T.
	\end{align*} 
\end{proof}

Finally, we give explicit translation parts such that $T/(A_4 \times C_4)$ is a hyperelliptic fourfold:

\begin{example} \label{48-31-example}
	As above, let $T = E_i \times E_i \times E_i \times E$, where $E_i = \CC/(\ZZ + i\ZZ)$ and $E$ is another elliptic curve in standard form. Consider the following automorphisms of $T$:
	\begin{align*}
	&\sigma(z) = \left(z_3+ \frac{1+i}{4}, \ z_1, \ z_2 - \frac{1+i}{4}, \ z_4 + \frac13\right), \\
	&\xi(z) = \left(z_1 + \frac{i+1}{2}, \ -z_2, \ -z_3, \ z_4\right), \\
	&\kappa(z) = \left(iz_1, \ iz_2, \ iz_3 + \frac{1}{2}, \ z_4 + \frac14\right).
	\end{align*}
	Since the relations of  \hyperref[A4-relations]{Lemma~\ref*{A4-relations}} hold, the group $\langle \sigma, \xi, \kappa \rangle \subset \Bihol(T)$ is isomorphic to $A_4 \times C_4$ and does not contain any translations by construction. Moreover,  \hyperref[A4-freeness]{Lemma~\ref*{A4-freeness}} immediately implies that the action of $\langle \sigma, \xi, \kappa \rangle$ on $T$ is free. It follows that $T/\langle \sigma, \xi, \kappa \rangle$ indeed defines a hyperelliptic manifold with holonomy group $A_4 \times C_4$.
\end{example}

\emph{(D) The Hodge diamond.} \\
The Hodge diamond of a hyperelliptic fourfold with holonomy group $A_4 \times C_4$ is 
\begin{center}
	$\begin{matrix}
	&   &  &  & 1 &  &  &  &  \\
	&   &  & 1 &  & 1 &  &  &  \\
	&   & 0 &  & 2 &  & 0  &  & \\
	&  0 &  & 1 &   & 1 &   &  0 &  \\
	0&    & 0 &  & 2  &  & 0  &     &  0 \\
	\end{matrix}$
\end{center}

\emph{(D) The Hodge diamond.} \\
The Hodge diamond of a hyperelliptic fourfold with holonomy group $A_4 \times C_4$ is 
\begin{center}
	$\begin{matrix}
	&   &  &  & 1 &  &  &  &  \\
	&   &  & 1 &  & 1 &  &  &  \\
	&   & 0 &  & 2 &  & 0  &  & \\
	&  0 &  & 1 &   & 1 &   &  0 &  \\
	0&    & 0 &  & 2  &  & 0  &     &  0 \\
	\end{matrix}$
\end{center}

\subsection{$G(3,8,2) \times C_3$ (ID [72,12])} \label{72-12-section}\  \\ 

In this section, hyperelliptic fourfolds with holonomy group
\begin{align*}
G(3,8,2) \times C_3 = \langle g,h,k ~ | ~ g^3 = h^8 = k^3 = 1,~ h^{-1}gh = g^2, ~ k \text{ central}\rangle
\end{align*}
are investigated. More precisely, we show in parts (B) and (C): 

\begin{prop} \label{72-12-prop}
	Let $T/(G(3,8,2) \times C_3)$ be a hyperelliptic fourfold with associated complex representation $\rho$. Then:
	\begin{enumerate}[ref=(\theenumi)]
		\item \label{72-12-prop-rep} Up to equivalence and automorphisms, $\rho$ is given as follows:
		\begin{align*}
		\rho(g) = \begin{pmatrix}
		-1 & -1 && \\ 1 & 0 && \\ && 1& \\ &&&1
		\end{pmatrix}, \qquad \rho(h) = \begin{pmatrix}
		1 & i && \\ i-1& -1 && \\ &&1 & \\ &&& -1
		\end{pmatrix}, \qquad \rho(k) = \begin{pmatrix}
		1 & && \\ & 1 && \\ && 1 & \\ &&&\zeta_3
		\end{pmatrix}.
		\end{align*}
		\item  \label{72-12-prop-isog} The representation $\rho$ induces an equivariant isogeny $E_i \times E_i \times E \times F \to T$, where $E \subset T$ is an elliptic curve, $E_i = \CC/(\ZZ+i\ZZ)$ and $F = \CC/(\ZZ+\zeta_3 \ZZ)$.
		\item Hyperelliptic fourfolds with holonomy group $G(3,8,2) \times C_3$ exist.
	\end{enumerate}
	In particular, $X$ moves in a complete $1$-dimensional family of hyperelliptic fourfolds with holonomy group $G(3,8,2) \times C_3$.
\end{prop}

In part (D), we also give the Hodge diamond of hyperelliptic fourfolds $T/(G(3,8,2) \times C_3)$. \\

\emph{(A) Structure and Representation Theory of $G(3,8,2)$.} \\

We discuss the representation theory of the metacyclic group $G(3,8,2) = \langle g,h ~ | ~ g^3 = h^8 = 1, ~ h^{-1}gh = g^2 \rangle$. First of all, the relation $h^{-1}gh = g^2$ implies that $g$ is a commutator, and hence the linear characters of $G(3,8,2)$ are 
\begin{align*}
\chi_j(g) = 1, \qquad \chi_j(h) = \zeta_8^j
\end{align*}
for $j \in \{0,...,7\}$. Furthermore, $G(3,8,2)$ has four irreducible representations of degree $2$. Among these are two faithful ones, namely
\begin{align*}
\rho_{2,1}(g) = \begin{pmatrix}
-1 & -1 \\ 1 & 0
\end{pmatrix}, \qquad  \rho_{2,1}(h) = \begin{pmatrix}
1 & i \\ i-1 & -1
\end{pmatrix}
\end{align*}
and its complex conjugate $\overline{\rho_{2,1}}$. 

\begin{lemma}
	The representations $\rho_{2,1}$ and $\overline{\rho_{2,1}}$ are in the same orbit under the action of $\Aut(G(3,8,2))$.
\end{lemma}

\begin{proof}
	A GAP computation shows that the character of $\rho_{2,1}$ is stabilized by an index $2$ subgroup of $\Aut(G(3,8,2))$. Thus there is an automorphism $\varphi$ of $G(3,8,2)$ such that $\rho_{2,1}$ and $\rho_{2,1} \circ \varphi$ are non-equivalent. However, since $\rho_{2,1}$ and $\overline{\rho_{2,1}}$ are the only faithful irreducible representations of $G(3,8,2)$, the representations $\rho_{2,1} \circ \varphi$ and $\overline{\rho_{2,1}}$ must be equivalent. A concrete automorphism exchanging $\rho_{2,1}$ and $\overline{\rho_{2,1}}$ is given by $g \mapsto g$, $h \mapsto h^3$.
\end{proof}

%g := [[-1,-1],[1,0]];
%h := [[1,E(4)],[E(4)-1,-1]];
%G := Group([g,h]);
%AutG := AutomorphismGroup(G);
%AssignNiceMonomorphismAutomorphismGroup(AutG,G);
%
%Order(AutG);
%
%counter := 1;
%for f in AutG do
%sum := 0;
%for u in G do
%sum := sum + TraceMat(u)*ComplexConjugate(TraceMat(Image(f,u)));
%od;
%if sum/24 = 0 then
%Print(counter," Fuer uns relevant: ", Image(f,g), "            ", Image(f,h), "\n");
%fi;
%if sum/24 = 1 then
%Print(counter, " Wird fixiert: ", Image(f,g), "            ", Image(f,h), "\n");
%fi;
%counter := counter+1;
%od;

The non-faithful irreducible degree $2$ representations of $G(3,8,2)$ are:
\begin{center}
	\begin{tabular}{lll}
		$\rho_{2,2}$: & $g \mapsto \begin{pmatrix}
		\zeta_3 & \\ & \zeta_3^2
		\end{pmatrix}$ & $h\mapsto \begin{pmatrix}
		0 & 1 \\ 1 & 0
		\end{pmatrix}$, \\
		$\rho_{2,3}$: & $g \mapsto \begin{pmatrix}
		\zeta_3 & \\ & \zeta_3^2
		\end{pmatrix}$ & $h\mapsto \begin{pmatrix}
		0 & -1 \\ 1 & 0
		\end{pmatrix}$.
	\end{tabular}
\end{center}
Indeed, they map the element $g$ of order $8$ to a matrix of order $\leq 4$. \\

\emph{(B) The complex representation and the isogeny type of the torus.} \\

The upcoming two lemmas determine the complex representation $\rho \colon G(3,8,2) \times C_3 \to \GL(4,\CC)$ of a hyperelliptic fourfold with holonomy group $G(3,8,2) \times C_3$.

\begin{lemma} \label{72-12-lemma}
	The following statements hold:
	\begin{enumerate}[ref=(\theenumi)]
		\item \label{72-12-1} The representation $\rho$ splits as a direct sum of three irreducible representations whose respective degrees are $2$, $1$ and $1$.
		\item \label{72-12-2} The irreducible degree $2$ constituent of $\rho|_{G(3,8,2)}$ is faithful, i.e., equivalent to one of $\rho_{2,1}$ or $\overline{\rho_{2,1}}$.
		\item \label{72-12-3} The central element $k$ is contained in the kernel of the irreducible degree $2$ constituent of $\rho$.
	\end{enumerate}
\end{lemma}

\begin{proof}
	\ref{72-12-1} This follows directly from \hyperref[cor:metacyclic-rep]{Corollary~\ref*{cor:metacyclic-rep}} \ref{cor:metacyclic-rep-1}. (Alternatively, the description of the degree $2$ irreducible representations of $G(3,8,2)$ in (A) shows that $\rho(g)$ does not have the eigenvalue $1$, if $\rho$ is the direct sum of two irreducible representations of degree $2$.) \\
	
	\ref{72-12-2} According to the  \hyperref[order-cyclic-groups]{Integrality Lemma~\ref*{order-cyclic-groups}} \ref{ocg-3}, the matrix $\rho(h)$ must have exactly two eigenvalues of order $8$. It follows that if the degree $2$ constituent of $\rho|_{G(3,8,2)}$ is one of the non-faithful representations $\rho_{2,2}$ or $\rho_{2,3}$, then
	\begin{align*}
	\rho(gh^4) = \diag(\zeta_3, ~ \zeta_3^2, ~ -1, ~ -1)
	\end{align*}
	does not have the eigenvalue $1$. \\
	
	\ref{72-12-3} This follows from the previous assertion, \ref{72-12-2}, and \hyperref[order-cyclic-groups]{Lemma~\ref*{order-cyclic-groups}} \ref{ocg-1}, the latter of which asserts that $\rho(hk)$ cannot have eigenvalues of order $24$.
\end{proof}

As a consequence of the above lemma and the discussion part (A), we may assume that $\rho|_{G(3,8,2)}$ contains $\rho_{2,1}$ as an irreducible constituent. We denote by $\chi$, $\chi'$ the degree $1$ constituents of $\rho$. To complete the proof of \hyperref[72-12-prop]{Proposition~\ref*{72-12-prop}}, it remains to determine $\chi$ and $\chi'$ precisely  \ref{72-12-prop-rep}:

\begin{lemma}
	Up to symmetry and automorphisms, the following statements hold:
	\begin{enumerate}[ref=(\theenumi)]
		\item \label{72-12-chi1} $\chi(g) = \chi'(g) = 1$, 
		\item \label{72-12-chi2} $\chi(h) = 1$, $\chi'(h) = -1$, 
		\item \label{72-12-chi3} $\chi(k) = 1$, $\chi'(k) = \zeta_3$.
	\end{enumerate}
\end{lemma}

\begin{proof}
	\ref{72-12-chi1} This is clear, since $g$ is a commutator. \\
	
	\ref{72-12-chi2}, \ref{72-12-chi3} 
	We apply \hyperref[lemma-two-generators]{Proposition~\ref*{lemma-two-generators}} to $G(3,8,2) = \langle g,h\rangle$ to obtain that one of the values $\chi(h)$, $\chi'(h)$ is different from $1$. We may therefore assume that $\chi'(h) \neq 1$: then $\chi(h) = 1$ follows, since $\rho(h)$ must have the eigenvalue $1$. To determine $\chi(k)$ and $\chi'(k)$, we invoke  \hyperref[72-12-lemma]{Lemma~\ref*{72-12-lemma}} \ref{72-12-3}, which tells us that one of $\chi(k)$, $\chi'(k)$ is a primitive third root of unity. Since $\rho(hk)$ must have the eigenvalue $1$, we conclude that $\chi(k) = 1$ and hence, by possibly replacing $k$ by $k^2$, we may assume that $\chi'(k) = \zeta_3$. The remaining statement ($\chi'(h) = -1$) now follows from the \hyperref[order-cyclic-groups]{Integrality Lemma~\ref*{order-cyclic-groups}} \ref{ocg-3}, since $\rho(hk)$ must have an even number of eigenvalues of order $12$ and cannot have eigenvalues of order $24$.
\end{proof}

The complex representation $\rho \colon G(3,8,2) \times C_3 \to \GL(4,\CC)$ is, therefore, completely described, which proves \hyperref[72-12-prop]{Proposition~\ref*{72-12-prop}} \ref{72-12-prop-rep}. From this, we may easily deduce part \ref{72-12-prop-isog} of  the cited proposition from the discussion in  \hyperref[isogeny]{Section~\ref*{isogeny}}. Indeed, $\rho$ induces an equivariant isogeny $S \times E \times E' \to T$, where $S \subset T$ is a complex subtorus of dimension $2$ and $E, E' \subset T$ are elliptic curves. Moreover, $\rho(k)$ acts on $E'$ by multiplication by $\zeta_3$ and hence $E' \cong F = \CC/(\ZZ+\zeta_3 \ZZ)$. Similarly, $\rho(h^2)$ acts on $S$ by multiplication by $i$, which shows that $S$ is equivariantly isogenous to $E_i \times E_i$, where $E_i = \CC/(\ZZ+i\ZZ)$ denotes Gauss' elliptic curve. \\

\emph{(C) An example.} \\
We show that hyperelliptic fourfolds with holonomy group $G(3,8,2) \times C_3$ exist. By part (B), the complex torus $T$ is isogenous to $E_i \times E_i \times E \times F$, where $E_i$ is the Gauss elliptic curve, $F$ is the Fermat elliptic curve and $E = \CC/(\ZZ+\tau \ZZ)$ is another elliptic curve. Up to a change of coordinates in $F$ and the $E_i$, we may write
\begin{align*}
&g(z) = \left(-z_1 - z_2, \ z_1, \ z_3 + a_3, \ z_4 + a_4\right), \\
&h(z) = \left(z_1 + iz_2 + b_1, \ (i-1)z_1 - z_2 + b_2, \ z_3 + b_3, \ -z_4\right), \\
&k(z) = \left(z_1 + c_1, \ z_2 + c_2, \ z_3 + c_3, \ \zeta_3z_4 + c_4\right).
\end{align*}

As usual, we investigate when $\langle g,h,k\rangle \subset \Bihol(T)$ is isomorphic to $G(3,8,2) \times C_3$:

\begin{lemma}\label{72-12-rels} \ 
	\begin{enumerate}[ref=(\theenumi)]
		\item \label{72-12-rels-first} It holds $g^3 = \id_T$ if and only if $(0, \ 0, \ 3a_3, \ 3a_4)$ is zero in $T$.
		\item It holds $h^8 = \id_T$ if and only if $8b_3 = 0$ in $E$.
		\item It holds $k^3 = \id_T$ if and only if $(3c_1, \ 3c_2, \ 3c_3,\ 0)$ is zero in $T$.
		\item It holds $h^{-1}gh = g^2$ if and only if $((2i+1)b_1 + (i-1)b_2, \ -(i+2)b_1 - (2i+1)b_2, \ a_3, \ 3a_4)$ is zero in $T$.
		\item The elements $g$ and $k$ commute if and only if $(2c_1+c_2, \ c_2-c_1, \ 0, \ (\zeta_3-1)a_4)$ is zero in $T$.
		\item \label{72-12-rels-last} The elements $h$ and $k$ commute if and only if $(-ic_2, \ (1-i)c_1 + 2c_2, \ 0, \ 2c_4)$ is zero in $T$.
	\end{enumerate}
	Consequently, if the six properties \ref{72-12-rels-first} -- \ref{72-12-rels-last} are satisfied, then the group $\langle g,h,k \rangle \subset \Bihol(T)$ is isomorphic to $G(3,8,2) \times C_3$ and does not contain any translations.
\end{lemma}

We now present a concrete example. 

\begin{example} \label{72-12-example}
	We define the following holomorphic automorphisms of $T := E_i \times E_i \times E \times F$:
	\begin{align*}
	&g(z) = \left(-z_1-z_2, \ z_1, \ z_3, \ z_4 + \frac{1-\zeta_3}{3}\right), \\
	&h(z) = \left(z_1+iz_2, \ (i-1)z_1 - z_2, \ z_3 + \frac18, \ -z_4\right), \\
	&k(z) = \left(z_1, \ z_2, \ z_3 + \frac13, \ \zeta_3 z_4\right).
	\end{align*}
	By \hyperref[72-12-rels]{Lemma~\ref*{72-12-rels}}, the above elements satisfy the defining relations of $G(3,8,2) \times C_3$ and that $\langle g,h,k\rangle \subset \Bihol(T)$ does not contain any translations. \\
	We prove that $G(3,8,2) \times C_3 = \langle g,h,k\rangle$ acts freely on $T$. Every element of said group can be uniquely written in the form $g^a h^b k^c$ for $0 \leq a < 3$, $0 \leq b < 8$ and $0 \leq c < 3$.  Such an element acts on $E \times F \subset T$ by
	\begin{align*}
	(g^a h^b k^c)|_{E\times F}(z_3, z_4) = \left(z_3 + \frac{b}{8} + \frac{c}{3}, \ (-1)^b \zeta_3^c z_4 + a \cdot \frac{1-\zeta_3}{3}\right).
	\end{align*}
	This description shows that $g^a h^b k^c$ can only have a fixed point if $a = b = c = 0$, i.e., if the element is trivial. In total, our arguments show that $T/\langle g,h,k\rangle$ is a hyperelliptic fourfold with holonomy group $G(3,8,2) \times C_3$.
\end{example}

\emph{(D) The Hodge diamond.} \\
The Hodge diamond of a hyperelliptic fourfold with holonomy group $G(3,8,2) \times C_3$ is 

\begin{center}
	$\begin{matrix}
	&   &  &  & 1 &  &  &  &  \\
	&   &  & 1 &  & 1 &  &  &  \\
	&   & 0 &  & 3 &  & 0  &  & \\
	&  0 &  & 2 &   & 2 &   &  0 &  \\
	0&    & 0 &  & 4  &  & 0  &     &  0 \\
	\end{matrix}$
\end{center}

\subsection{$S_3 \times C_{12}$ (ID [72,27]) and $S_3 \times C_3 \times C_6$ (ID [108,42])} \label{72-27-and-108-42-section}\  \\ 

We investigate hyperelliptic fourfolds with holonomy groups $S_3 \times C_{12}$ or $S_3 \times C_3 \times C_6$. Here, we use the notations
\begin{align*}
&S_3 \times C_{12} = \langle \sigma, \tau, \kappa \ | \ \sigma^3 = \tau^2 = \kappa^{12} = 1, \ \tau^{-1}\sigma\tau = \sigma^2, \ \kappa \text{ central} \rangle, \\
&S_3 \times C_3 \times C_6 =\langle \sigma, \tau, \kappa_1, \kappa_2 \ | \ \sigma^3 = \tau^2 = \kappa_1^3 = \kappa_2^6  = 1, \ \tau^{-1}\sigma\tau = \sigma^2, \ \kappa_1 \text{ and } \kappa_2 \text{ central} \rangle.
\end{align*}
Other than determining the Hodge diamond of such fourfolds, our main results are the following two structure theorems:

\begin{prop}[$S_3 \times C_{12}$-Case]
	Let $X = T/(S_3 \times C_{12})$ be a hyperelliptic fourfold with associated complex representation $\rho$. Then:
	\begin{enumerate}[ref=(\theenumi)]
		\item Up to equivalence and automorphisms, $\rho$ is given as follows:
		\begin{align*}
		\rho(\sigma) = \begin{pmatrix}
		-1 & -1 && \\ 1 & 0 && \\ && 1& \\ &&&1
		\end{pmatrix}, \qquad \rho(\tau) = \begin{pmatrix}
		0 & 1 && \\ 1 & 0 && \\ &&-1 & \\ &&& 1
		\end{pmatrix}, \qquad \rho(\kappa) = \begin{pmatrix}
		i &  && \\  & i && \\ && \zeta_3 & \\ &&&1
		\end{pmatrix}.
		\end{align*}
		\item The representation $\rho$ induces an equivariant isogeny $E_i \times E_i \times F \times E \to T$, where $E \subset T$ is an elliptic curve and $E_i = \CC/(\ZZ+i\ZZ)$, $F = \CC/(\ZZ+\zeta_3\ZZ)$.
		\item Hyperelliptic fourfolds with holonomy group $S_3 \times C_{12}$ exist.
	\end{enumerate}
	In particular, $X$ moves in a complete $1$-dimensional family of hyperelliptic fourfolds with holonomy group $S_3 \times C_{12}$.
\end{prop}

\begin{prop}[$S_3 \times C_3 \times C_6$-Case]
	Let $X = T/(S_3 \times C_3 \times C_6)$ be a hyperelliptic fourfold with associated complex representation $\rho$. Then:
	\begin{enumerate}[ref=(\theenumi)]
		\item Up to equivalence and automorphisms, $\rho$ is given as follows:
		\begin{align*}
		&\rho(\sigma) = \begin{pmatrix}
		-1 & -1 && \\ 1 & 0 && \\ && 1& \\ &&&1
		\end{pmatrix}, \qquad \ \rho(\tau) = \begin{pmatrix}
		0 & 1 && \\ 1 & 0 && \\ &&-1 & \\ &&& 1
		\end{pmatrix}, \\ 
		&\rho(\kappa_1) = \begin{pmatrix}
		1 &  && \\  & 1 && \\ && \zeta_3 & \\ &&&1
		\end{pmatrix}, \qquad  \quad \rho(\kappa_2) = \begin{pmatrix}
		-\zeta_3 &&& \\ &-\zeta_3 && \\ && 1 & \\ &&& 1
		\end{pmatrix}.
		\end{align*}
		\item The representation $\rho$ induces an equivariant isogeny $F \times F \times F \times E \to T$, where $E \subset T$ is an elliptic curve and $F = \CC/(\ZZ+\zeta_3\ZZ)$ is Fermat's elliptic curve.
		\item Hyperelliptic fourfolds with holonomy group $S_3 \times  C_3 \times C_6$ exist.
	\end{enumerate}
	In particular, $X$ moves in a complete $1$-dimensional family of hyperelliptic fourfolds with holonomy group $S_3 \times C_3 \times C_6$.
\end{prop}

Furthermore, we use \hyperref[s3-complexrep]{Corollary~\ref*{s3-complexrep}} below to prove:

\begin{prop} \label{s3xc2xc2-excluded}
	The group $S_3 \times C_2 \times C_2$ is not hyperelliptic in dimension $4$.
\end{prop}

The proof will be given in part (E). \\

\emph{(A) Representation Theory of $S_3$.} \\
We briefly recall the representation theory of the symmetric group
\begin{align*}
G(3,2,2) = S_3 = \langle \sigma, \tau \ | \ \sigma^3 = \tau^2 = 1, \ \tau^{-1}\sigma\tau = \sigma^2\rangle.
\end{align*}
Being a $3$-cycle, $\sigma$ spans the derived subgroup $A_3$ of $S_3$ and hence the degree $1$ representations of $S_3$ are given by $\sigma \mapsto 1$ and $\tau \mapsto (-1)^j$, $j \in \{0,1\}$. In addition to that, the group $S_3$ has -- up to equivalence -- a unique irreducible representation of degree $> 1$. It is given by
\begin{align*}
\sigma \mapsto \begin{pmatrix}
-1 & -1 \\ 1 & 0
\end{pmatrix}, \qquad \tau \mapsto \begin{pmatrix}
0 & 1 \\ 1 & 0
\end{pmatrix}.
\end{align*}

\emph{(B) The complex representation and the isogeny type of the torus.} \\
We determine the complex representations of hyperelliptic fourfolds with holonomy group $S_3 \times C_{12}$ (\hyperref[S3xC12-rep]{Lemma~\ref*{S3xC12-rep}}) or $S_3 \times C_3 \times C_6$ (\hyperref[S3xC3xC6-rep]{Lemma~\ref*{S3xC3xC6-rep}}). As a byproduct, we will obtain the isogeny type of the corresponding complex torus. However, we first study the complex representation of hyperelliptic fourfolds with holonomy group $S_3$. The main part of this investigation is the following lemma that first appeared in \cite[Lemma 5]{Uchida-Yoshihara}.

\begin{lemma} \label{s3-fixed-point}
	Let $S$ be a complex torus on dimension $2$ on which the group $S_3 = \langle \tau, \sigma \rangle$ acts holomorphically and without translations. Then the transposition $\tau$ has a fixed point on $S$.
\end{lemma}

\begin{proof}
	Since $S_3$ contains no translations, the complex representation $\rho_2 \colon S_3 \to \GL(2,\CC)$ is the unique irreducible representation of degree $2$ of $S_3$. We write
	\begin{align*}
	&\tau(z_1,z_2) = (z_2 + t_1, \ z_1 + t_2).
	\end{align*}
	Since $\tau^2 = \id_S$, the element $(t_1+t_2, \ t_1+t_2)$ is zero in $S$. Thus the element
	\begin{align} \label{s3-non-free}
	\rho_2(\sigma)(t_1+t_2, \ t_1+t_2) = (-2(t_1+t_2), \ \zeta_3^2(t_1+t_2))
	\end{align}
	is zero, too. Using $\zeta_3^2 + \zeta_3 + 1 = 0$, the following calculation now shows that $\tau$ has a fixed point on $S$:
	\begin{align*}
	\tau(\zeta_3 t_2 - \zeta_3^2 t_1,0) = (t_1, \ - \zeta_3^2(t_1+t_2)) \stackrel{\ref{s3-non-free}}{=} (\zeta_3t_2 - \zeta_3^2 t_1,0).
	\end{align*}
\end{proof}

Consequently, we obtain the complex representation of a hyperelliptic fourfold with holonomy group $S_3$.

\begin{cor} \label{s3-complexrep}
	The complex representation of a hyperelliptic fourfold with holonomy group $S_3$ is equivalent to the following:
	\begin{align*}
	\tau \mapsto \begin{pmatrix}
	0 & 1 && \\ 1& 0 && \\ && -1 & \\ &&& 1
	\end{pmatrix}, \qquad \sigma \mapsto  \begin{pmatrix}
	-1& -1 && \\ 1 & 0 && \\ &&1& \\ &&&1
	\end{pmatrix}.
	\end{align*}
\end{cor}

\begin{proof}
	We only have to prove that the degree $1$ summands of $\rho$ are as asserted. Indeed, \hyperref[lemma-two-generators]{Proposition~\ref*{lemma-two-generators}} asserts that at least one degree $1$ summand of $\rho$ is non-trivial, whereas it is an straightforward consequence of  \hyperref[s3-fixed-point]{Lemma~\ref*{s3-fixed-point}} that $\tau$ does not act freely on $T$ if $\tau$ is mapped to $-1$ by both degree $1$ constituents of $\rho$. 
\end{proof}

So far, we have only considered the action of $S_3$. We will now consider actions of $S_3 \times C_{12}$ and $S_3 \times C_3 \times C_6$, respectively. \\

Starting with $S_3 \times C_{12}$, let $\kappa$ be a central element of order $12$ such that $\langle \sigma, \tau, \kappa \rangle = S_3 \times C_{12}$. Then the complex representation $\rho$ of a hyperelliptic fourfold $T/(S_3 \times C_{12})$ is given as follows:

\begin{lemma} \label{S3xC12-rep}
	Let $T/(S_3 \times C_{12})$ be a hyperelliptic fourfold with associated complex representation $\rho$. Assume that $\rho|_{S_3}$ is given as in \hyperref[s3-complexrep]{Corollary~\ref*{s3-complexrep}}. 
	Then, up to replacing $\kappa$ by an appropriate power,
	\begin{align*}
	\rho(\kappa) = \diag(i, \ i, \ \zeta_3, \ 1).
	\end{align*}
	Furthermore, $T$ is equivariantly isogenous to $E_i \times E_i \times F \times E$, where $E_i \subset T$ are copies of Gauss' elliptic curve, $F \subset T$ is a Fermat elliptic curve and $E \subset T$ is another elliptic curve.
\end{lemma}

\begin{proof}
	First of all, it follows from  the \hyperref[order-cyclic-groups]{Integrality Lemma~\ref*{order-cyclic-groups}} \ref{ocg-3} that $\rho(\kappa)$ cannot have eigenvalues of order $12$, because $\kappa$ is central. Furthermore, there is no $j$ such that $\rho(\kappa^j) = \diag(1, \ 1, \ -1, \ 1)$: otherwise the restriction of $\rho$ to $\langle \kappa^j \tau, \sigma \rangle \cong S_3$ contains two copies of the trivial representation, which contradicts \hyperref[s3-complexrep]{Corollary~\ref*{s3-complexrep}}. It follows that the third diagonal entry of $\rho(\kappa)$ is a third root of unity. Recalling that $\rho$ is faithful (property \ref{nec-prop1}), we observe that up to replacing $\kappa$ by an appropriate power, there are only the two cases
	\begin{align*}
	\rho(\kappa) = \diag(i, \ i, \ \zeta_3, \ 1) \qquad \text{ or }\qquad 	\rho(\kappa) = \diag(1, \ 1, \ \zeta_3, \ i) .
	\end{align*}
	In the latter case, $\rho(\sigma \kappa)$ does not have the eigenvalue $1$, which contradicts property \ref{nec-prop2}. Therefore $\rho(\kappa)$ is as asserted. \\
	It remains to prove the statement concerning the isogeny type of $T$. Applying the equivariant decomposition up to isogeny described in  \hyperref[isogeny]{Section~\ref*{isogeny}} to $T$ gives an equivariant isogeny $S \times E' \times E$ for a $2$-dimensional complex torus $S \subset T$ and elliptic curves $E, E' \subset T$. Since $\rho(\kappa)$ acts on $S$ by multiplication by $i$, the surface $S$ is equivariantly isogenous to $E_i \times E_i$. Similarly, $E'$ is isomorphic to $F$.
\end{proof}

Turning our attention to $S_3 \times C_3 \times C_6$, let $\kappa_1$ and $\kappa_2$ be central elements of respective orders $3$ and $6$ such that $\langle \sigma, \tau, \kappa_1, \kappa_2\rangle = S_3 \times C_3 \times C_6$. We determine the complex representation of a hyperelliptic fourfold $T/(S_3 \times C_3 \times C_6)$:

\begin{lemma} \label{S3xC3xC6-rep}
	Let $T/(S_3 \times C_3 \times C_6)$ be a hyperelliptic fourfold with associated complex representation $\rho$. Assume that $\rho|_{S_3}$ is given as in \hyperref[s3-complexrep]{Corollary~\ref*{s3-complexrep}}. 
	Then, up to automorphism,
	\begin{align*}
	\rho(\kappa_1) = \diag(1, \ 1, \ \zeta_3, \ 1) \qquad \text{and} \qquad \rho(\kappa_2) = \diag(-\zeta_3, \ -\zeta_3, \ 1, \ 1).
	\end{align*}
	Furthermore, $T$ is equivariantly isogenous to $F \times F \times F \times E$, where the curves $F \subset T$ are copies of Fermat's elliptic curve, and $E \subset T$ is another elliptic curve.
\end{lemma}

\begin{proof}
	Similarly as in the proof of \hyperref[S3xC12-rep]{Lemma~\ref*{S3xC12-rep}}, we conclude that no element of $\langle \kappa_1, \kappa_2 \rangle$ is mapped to $\diag(1, \ 1, \ -1, \ 1)$ and that the fourth diagonal entry $\rho(\kappa_1)$ and $\rho(\kappa_2)$ is $1$. It is then clear that $\rho(\kappa_1)$ and $\rho(\kappa_2)$ are as asserted, up to an automorphism of $S_3 \times C_3 \times C_6$ that fixes $S_3$ element-wise. The statement concerning the isogeny type of $T$ follows as in the proof of \hyperref[S3xC12-rep]{Lemma~\ref*{S3xC12-rep}}.
\end{proof}

\emph{(C) Examples.} \\
In this part, we will give one example of a hyperelliptic fourfold for each of the two groups $S_3 \times C_{12}$ and $S_3 \times C_3 \times C_6$.

\begin{rem}
	It is well-known that the conjugacy classes of $S_3$ correspond to the different cycle types. Therefore, a system of representatives for the conjugacy classes of $S_3$ is given by $\{\id, \tau, \sigma\}$.
\end{rem}
First, we construct an example of a hyperelliptic fourfold $T/(S_3 \times C_{12})$.  \hyperref[S3xC12-rep]{Lemma~\ref*{S3xC12-rep}} shows that $T = (E_i \times E_i \times F \times E)/H$ and that after suitable change of origin, we may write the action of $S_3 \times C_{12}$ on $T$ as follows:
\begin{align*}
&\sigma(z) = (-z_1-z_2 + s_1, \ z_1 + s_2, \ z_3 + s_3, \ z_4 + s_4), \\
&\tau(z) = (z_2 + t_1, \ z_1 + t_2, \ -z_3 + t_3, \ z_4 + t_4), \\
&\kappa(z) = (iz_1, \ iz_2, \ \zeta_3 z_3, \ z_4 + k_4).
\end{align*}

The next lemma tells us necessary and sufficient conditions on the translation parts such that $\langle \sigma, \tau, \kappa \rangle$ is isomorphic to $S_3 \times C_{12}$. 

\begin{lemma} \label{S3xC12-rels} \
	\begin{enumerate}[ref=(\theenumi)]
		\item \label{S3xC12-rels-first} The relation $\sigma^3 = \id_T$ holds if and only if $(0, \ 0, \ 3s_3, \ 3s_4)$ is zero in $T$.
		\item The relation $\tau^2 = \id_T$ holds if and only if $(t_1+t_2, \ t_1+t_2, \ 0, \ 2t_4)$ is zero in $T$.
		\item The relation $\tau^{-1}\sigma \tau = \sigma^2$ holds if and only if $(t_2-t_1-2s_2, \ 2t_1+t_2 + s_2, \ 3s_3, \ s_4)$ is zero in $T$.
		\item The relation $\kappa^{12} = \id_T$ holds if and only if $12k_4 = 0$ in $E$.
		\item The elements $\sigma$ and $\kappa$ commute if and only if $((i-1)s_1, \ (i-1)s_2, \ (\zeta_3-1)s_3, \ 0)$ is zero in $T$.
		\item \label{S3xC12-rels-last} The elements $\tau$ and $\kappa$ commute if and only if $((i-1)t_1, \ (i-1)t_2, \ (\zeta_3-1)t_3, \ 0)$ is zero in $T$
	\end{enumerate}
	Consequently, if the six properties \ref{S3xC12-rels-first} --\ref{S3xC12-rels-last} are satisfied, then the group $\langle \sigma, \tau, \kappa \rangle \subset \Bihol(T)$ is isomorphic to $S_3 \times C_{12}$ and does not contain any translations.
\end{lemma}

\begin{example} \label{S3xC12-example}
	Let $T = E_i \times E_i \times F \times E$, where
	\begin{itemize}
		\item $E_i = \CC/(\ZZ + i \ZZ)$ is Gauss' elliptic curve,
		\item $F = \CC/(\ZZ+\zeta_3 \ZZ)$ is Fermat's elliptic curve, and
		\item $E = \CC/(\ZZ+ t\ZZ)$ is an elliptic curve in normal form.
	\end{itemize}
	Consider the following automorphisms of $T$:
	\begin{align*}
	&\sigma(z) = \left(-z_1-z_2, \ z_1, \ z_3 + \frac{1-\zeta_3}{3}, \ z_4\right), \\
	&\tau(z) = \left(z_2, \ z_1, \ -z_3, \ z_4 + \frac12\right), \\
	&\kappa(z) = \left(iz_1, \ iz_2, \ \zeta_3z_3, \ z_4 + \frac{t}{12}\right).
	\end{align*}
	All six properties of \hyperref[S3xC12-rels]{Lemma~\ref*{S3xC12-rels}} are satisfied, hence $\langle \sigma, \tau, \kappa \rangle \cong S_3 \times C_{12}$ and does not contain any translations. The freeness of the action on $T$ is clear since each non-trivial element acts on one of the elliptic curves $F$ or $E$ by a non-trivial translation. It follows that $T/\langle  \sigma, \tau, \kappa\rangle$ is indeed a hyperelliptic manifold with holonomy group $S_3 \times C_{12}$.
\end{example}

We proceed similarly for $S_3 \times C_3 \times C_6$. According to  \hyperref[S3xC3xC6-rep]{Lemma~\ref*{S3xC3xC6-rep}}, we may change the origin to write the action of $S_3 \times C_3 \times C_6$ on $T = (F \times F \times F \times E)/H$ as follows:
\begin{align*}
&\sigma(z) = (-z_1-z_2 + s_1, \ z_1 + s_2, \ z_3 + s_3, \ z_4 + s_4), \\
&\tau(z) = (z_2 + t_1, \ z_1 + t_2, \ -z_3 + t_3, \ z_4 + t_4), \\
&\kappa_1(z) = (z_1 + k_{11}, \ z_2 + k_{12}, \ \zeta_3 z_3, \ z_4 + k_{14}), \\
&\kappa_2(z) = (-\zeta_3z_1, \ -\zeta_3z_2, \ z_3 + k_{23}, \ z_4 + k_{24}).
\end{align*}

We again investigate the defining relations of $S_3 \times C_3 \times C_6$. (Of course, the conditions for the relations between $\sigma$ and $\tau$ to hold remain unchanged from \hyperref[S3xC12-rels]{Lemma~\ref*{S3xC12-rels}}.)

\begin{lemma} \label{S3xC3xC6-rels} \
	\begin{enumerate}[ref=(\theenumi)]
		\item \label{S3xC3cC6-rels-first} The relation $\sigma^3 = \id_T$ holds if and only if $(0, \ 0, \ 3s_3, \ 3s_4)$ is zero in $T$.
		\item The relation $\tau^2 = \id_T$ holds if and only if $(t_1+t_2, \ t_1+t_2, \ 0, \ 2t_4)$ is zero in $T$.
		\item The relation $\tau^{-1}\sigma \tau = \sigma^2$ holds if and only if $(t_2-t_1-2s_2, \ 2t_1+t_2 + s_2, \ 3s_3, \ s_4)$ is zero in $T$.
		\item The relation $\kappa_1^3 = \id_T$ holds if and only if $(3k_{11}, \ 3k_{12}, \ 0, \ 3k_{14})$ is zero in $T$.
		\item The relation $\kappa_2^6 = \id_T$ holds if and only if $(0, \ 0, \ 6k_{23}, \ 6k_{24})$ is zero in $T$.
		\item The elements $\kappa_1$ and $\sigma$ commute if and only if $(2k_{11}+k_{12}, \ k_{12} - k_{11}, \ (\zeta_3-1)s_3, \ 0)$ is zero in $T$.
		\item The elements $\kappa_1$ and $\tau$ commute if and only if $(k_{11} - k_{12}, \ k_{12} - k_{11}, \ (\zeta_3-1)t_3, \ 0)$ is zero in $T$.
		\item \label{S3xC3cC6-rels-8} The elements $\kappa_2$ and $\sigma$ commute if and only if $(s_1, \ s_2, \ 0, \ 0)$ is zero in $T$.
		\item \label{S3xC3cC6-rels-9} The elements $\kappa_2$ and $\tau$ commute if and only if $(t_1, \ t_2, \ 2k_{23}, \ 0)$ is zero in $T$.
		\item \label{S3xC3cC6-rels-last} The elements $\kappa_1$ and $\kappa_2$ commute if and only if $(k_{11}, \ k_{12}, \ (\zeta_3-1)k_{23}, \ 0)$ is zero in $T$.
	\end{enumerate}
	Consequently, if the ten properties \ref{S3xC3cC6-rels-first} --\ref{S3xC3cC6-rels-last} are satisfied, then the group $\langle \sigma, \tau, \kappa_1, \kappa_2 \rangle \subset \Bihol(T)$ is isomorphic to $S_3 \times C_3 \times C_6$ and does not contain any translations.
\end{lemma}

\begin{proof}
	Only \ref{S3xC3cC6-rels-8} -- \ref{S3xC3cC6-rels-last} need to be commented on. Consider for instance the relation $\sigma \kappa_2 = \kappa_2 \sigma$ from \ref{S3xC3cC6-rels-8}: a priori, it is satisfied if and only if
	\begin{align} \label{S3-eq}
	((-\zeta_3-1)s_1, \ (-\zeta_3-1)s_2, \ 0, \ 0) = (\zeta_3^2s_1, \ \zeta_3^2s_2, \ 0, \ 0)
	\end{align}
	is zero in $T$. However, since $\rho(\kappa_2^2) = \diag(\zeta_3, \ \zeta_3, \ 0, \ 0)$ is an automorphism of $T$, the element \ref{S3-eq} is zero if and only if $(s_1, \ s_2, \ 0, \ 0)$ is zero, as asserted. The arguments for \ref{S3xC3cC6-rels-9} and \ref{S3xC3cC6-rels-last} work in the same way.
\end{proof}

\begin{example}
	Let $T = F \times F \times F \times E$, where $F = \CC/(\ZZ + \zeta_3 \ZZ)$ is Fermat's elliptic curve and $E = \CC/(\ZZ+t\ZZ)$ is an arbitrary elliptic curve in standard form.¸
	Consider the following automorphisms of $T$:
	\begin{align*}
	&\sigma(z) = \left(-z_1-z_2, \ z_1, \ z_3 + \frac{1-\zeta_3}{3}, \ z_4\right), \\
	&\tau(z) = \left(z_2, \ z_1, \ -z_3, \ z_4 + \frac12\right), \\
	&\kappa_1(z) = \left(z_1, \ z_2, \ \zeta_3z_3, \ z_4 + \frac{1}{3}\right), \\
	&\kappa_2(z) = \left(-\zeta_3z_1, \ -\zeta_3z_2, \ z_3, \ z_4 + \frac{t}{6}\right).
	\end{align*}
	All ten properties of \hyperref[S3xC3xC6-rels]{Lemma~\ref*{S3xC3xC6-rels}} are satisfied and hence $\langle \sigma, \tau, \kappa_1, \kappa_2 \rangle \subset \Bihol(T)$ is isomorphic to $S_3 \times C_3 \times C_6$ and does not contain any translations. Similarly as in  \hyperref[S3xC12-example]{Example~\ref*{S3xC12-example}}, the freeness of the action is clear and hence we obtain a hyperelliptic manifold $T/\langle \sigma, \tau, \kappa_1, \kappa_2\rangle$ with holonomy group $S_3 \times C_3 \times C_6$.
\end{example}

\emph{(D) The Hodge diamond.} \\
The Hodge diamonds of hyperelliptic fourfolds with holonomy group $S_3 \times C_{12}$ or $S_3 \times C_3 \times C_6$ coincide. They are given as follows:
\begin{center}
	$\begin{matrix}
	&   &  &  & 1 &  &  &  &  \\
	&   &  & 1 &  & 1 &  &  &  \\
	&   & 0 &  & 3 &  & 0  &  & \\
	&  0 &  & 2 &   & 2 &   &  0 &  \\
	0&    & 0 &  & 4  &  & 0  &     &  0 \\
	\end{matrix}$
\end{center}

\emph{(E) Proof of \hyperref[s3xc2xc2-excluded]{Proposition~\ref*{s3xc2xc2-excluded}}.} \\

\hyperref[s3-complexrep]{Corollary~\ref*{s3-complexrep}} above allows us to give a quick proof of the non-hyperellipticity of $S_3 \times C_2 \times C_2$ in dimension $4$. Indeed, let $\rho \colon S_3 \times C_2 \times C_2 \to \GL(4,\CC)$ be a faithful representation. Let $\langle \sigma, \tau \rangle$ be a copy of $S_3$, i.e., $\sigma$ denotes a $3$-cycle and $\tau$ a transposition. Assume that  $\rho|_{\langle \sigma, \tau \rangle}$ is as in \hyperref[s3-complexrep]{Corollary~\ref*{s3-complexrep}}. Then, by the faithfulness of $\rho$, there exists a central element $\kappa \in S_3 \times C_2 \times C_2$ such that $\kappa$ is mapped to $-1$ by one of the linear characters contained in $\rho$. Hence $\langle \sigma, \tau \kappa \rangle$ is a subgroup isomorphic to $S_3$ with the property that the representation $\rho|_{\langle \sigma, \tau \kappa \rangle}$ is not equivalent to $\rho|_{\langle \sigma, \tau \rangle}$. \hyperref[s3-complexrep]{Corollary~\ref*{s3-complexrep}} now shows that $\rho|_{\langle \sigma, \tau \kappa \rangle}$ and hence also $\rho$ cannot occur as complex representations of hyperelliptic fourfolds.

\subsection{$((C_2 \times C_6) \rtimes C_2) \times C_3$ (ID [72,30])} \label{72-30-section}\  \\

Consider the following action of $C_2 = \langle c \rangle$ on $C_2 \times C_6 = \langle a,b \ | \ a^6 = b^2 = 1, \ ab = ba\rangle$:
\begin{align} \label{72-30-action}
c^{-1}ac = a^{-1}, \qquad c^{-1}bc = a^3b.
\end{align}
The resulting semidirect product $(C_2 \times C_6) \rtimes C_2$ is labeled as $[24,8]$ in the Database of Small Groups. Furthermore, we consider a central element $k$ of order $3$, which commutes with $a$, $b$, and $c$. The group generated by $a$, $b$, $c$, and $k$ is the group $((C_2 \times C_6) \rtimes C_2) \times C_3$ (ID $[72,30]$). This section aims to investigate hyperelliptic fourfolds whose holonomy groups are the group $[72,30]$. More precisely, we prove

\begin{prop} \label{72-30-prop}
	
	Let $T/G$ be a hyperelliptic fourfold with holonomy group $G = ((C_2 \times C_6) \rtimes C_2) \times C_3$ and associated complex representation $\rho$. Then:
	\begin{enumerate}[ref=(\theenumi)]
		\item \label{72-30-prop1} Up to equivalence and automorphisms, $\rho$ is given as follows:
		\begin{align*}
		&\rho(a) = \begin{pmatrix}
		- \zeta_3 &&& \\ & -\zeta_3^2 && \\ && 1 & \\ &&& 1 
		\end{pmatrix}, \qquad \rho(b) = \begin{pmatrix}
		-1 & && \\ & 1 && \\ && 1 & \\ &&& 1
		\end{pmatrix}, \\
		&\rho(c) = \begin{pmatrix}
		0 & 1 && \\ 1 & 0 && \\ &&-1 & \\ &&&1 
		\end{pmatrix}, \qquad \rho(k) = \begin{pmatrix}
		\zeta_3 &&& \\ & \zeta_3 && \\ && 1 & \\ &&& 1
		\end{pmatrix}.
		\end{align*}
		\item \label{72-30-prop2} The representation $\rho$ induces an equivariant isogeny $F \times F \times E \times E' \to T$, where $F = \CC/(\ZZ+\zeta_3\ZZ)$ is the equianharmonic elliptic curve, and $E,E' \subset T$ are elliptic curves.
		\item \label{72-30-prop3} Hyperelliptic fourfolds with holonomy group $((C_6 \times C_2) \rtimes C_2) \times C_3$ exist.
	\end{enumerate}
	In particular, $X$ moves in a complete $1$-dimensional family of hyperelliptic fourfolds with holonomy group $((C_6 \times C_2) \rtimes C_2) \times C_3$.
\end{prop}

Moreover, we determine the Hodge diamond of a hyperelliptic fourfold whose holonomy group is the group labeled $[72,30]$. \\

\emph{(A) Representation Theory of $(C_2 \times C_6) \rtimes C_2$.} \\

We briefly discuss the representation theory of $(C_2 \times C_6) \rtimes C_2$. \\ 

The defining relations (\ref{72-30-action}) show that both $a^2$ and $a^3$ are commutators and hence $a$ is contained in the derived subgroup of $(C_2 \times C_6) \rtimes C_2$. The degree $1$ representations of $(C_2 \times C_6) \rtimes C_2$ are therefore the four representations $\chi_{i,j}$, $i,j \in \{0,1\}$ given by
\begin{align*}
a \mapsto 1, \qquad b \mapsto (-1)^i, \qquad c \mapsto (-1)^j.
\end{align*}
Furthermore, the group in discussion has the following faithful irreducible representations of degree $2$: 
\begin{center}
	\begin{tabular}{llll}
		$\rho_{2,1}$: & $a \mapsto \begin{pmatrix}
		-\zeta_3 & \\ & -\zeta_3^2 
		\end{pmatrix}$, & $b \mapsto \begin{pmatrix}
		-1 & \\  & 1
		\end{pmatrix}$, & $c \mapsto \begin{pmatrix}
		0 & 1 \\ 1 & 0
		\end{pmatrix}$, \\
		$\rho_{2,2}$: & $a \mapsto \begin{pmatrix}
		-\zeta_3 & \\ & -\zeta_3^2 
		\end{pmatrix}$, & $b \mapsto \begin{pmatrix}
		1 & \\  & -1
		\end{pmatrix}$, & $c \mapsto \begin{pmatrix}
		0 & 1 \\ 1 & 0
		\end{pmatrix}$.
	\end{tabular}
\end{center}
Observe that $\rho_{2,1} \circ \psi \sim \rho_{2,2}$, where $\psi \in \Aut((C_6 \times C_2) \rtimes C_2)$ is given by $a \mapsto a$, $b \mapsto ba^3$, $c \mapsto c$. \\
In addition to the above, $(C_2 \times C_6) \rtimes C_2$ has three more irreducible degree $2$ representations. They are non-faithful and given by:
\begin{center}
	\begin{tabular}{lll}
		$a \mapsto \begin{pmatrix}
		-1 & \\ & -1 
		\end{pmatrix}$, & $b \mapsto \begin{pmatrix}
		0 & 1 \\ 1 & 0
		\end{pmatrix}$, & $c \mapsto \begin{pmatrix}
		-1 &  \\ & 1
		\end{pmatrix}$, \\
		$a \mapsto \begin{pmatrix}
		\zeta_3 & \\ & \zeta_3^2 
		\end{pmatrix}$, & $b \mapsto \begin{pmatrix}
		1 & \\  & 1
		\end{pmatrix}$, & $c \mapsto \begin{pmatrix}
		0 & 1 \\ 1 & 0
		\end{pmatrix}$, \\
		$a \mapsto \begin{pmatrix}
		\zeta_3 & \\ & \zeta_3^2 
		\end{pmatrix}$, & $b \mapsto \begin{pmatrix}
		-1 & \\  & -1
		\end{pmatrix}$, & $c \mapsto \begin{pmatrix}
		0 & 1 \\ 1 & 0
		\end{pmatrix}$,
	\end{tabular}
\end{center}

%[ [ -1, 0 ], [ 0, -1 ] ]              [ [ 0, -1 ], [ -1, 0 ] ]              [ [ -1, 0 ], [ 0, 1 ] ]
%
%[ [ E(3), 0 ], [ 0, E(3)^2 ] ]              [ [ 1, 0 ], [ 0, 1 ] ]              [ [ 0, 1 ], [ 1, 0 ] ]
%
%[ [ E(3), 0 ], [ 0, E(3)^2 ] ]              [ [ -1, 0 ], [ 0, -1 ] ]              [ [ 0, 1 ], [ 1, 0 ] ]
%
%[ [ -E(3), 0 ], [ 0, -E(3)^2 ] ]              [ [ -1, 0 ], [ 0, 1 ] ]              [ [ 0, -1 ], [ -1, 0 ] ]
%
%[ [ -E(3), 0 ], [ 0, -E(3)^2 ] ]              [ [ 1, 0 ], [ 0, -1 ] ]              [ [ 0, -1 ], [ -1, 0 ] ]

%The representations $\rho_{2,1}$ and $\rho_{2,2}$ however \textit{are} faithful and form an orbit under the action of $\Aut((C_6 \times C_2) \rtimes C_2)$. 

\emph{(B) The complex representation and the isogeny type of the torus.} \\

We determine the complex representation $\rho$ of a hyperelliptic fourfold with holonomy group $((C_2 \times C_6) \rtimes C_2) \times C_3$. A simple starting point is:

\begin{lemma} \
	\begin{enumerate}[ref=(\theenumi)]
		\item \label{72-30-1} The complex representation $\rho$ is the direct sum of three irreducible representations of respective degrees $2$, $1$, $1$.
		\item \label{72-30-2} The irreducible degree $2$ constituent of $\rho|_{(C_6 \times C_2) \rtimes C_2}$ is faithful.
	\end{enumerate}
\end{lemma}

\begin{proof}
	\ref{72-30-1} Since $\rho$ is faithful, it must contain an irreducible constituent of degree $2$. If $\rho$ is the direct sum of two irreducible degree $2$ summands, then the description of the irreducible degree $2$ representations of $(C_6 \times C_2) \rtimes C_2$ given in part (A) implies that $\rho(a)$ does not have the eigenvalue $1$, contradicting property \ref{nec-prop2} on p. \pageref{nec-prop1}. \\
	
	\ref{72-30-2} By \ref{72-30-1}, the faithful representation $\rho$ contains two degree $1$ summands. Since $a$ is contained in the derived subgroup and the kernel of every non-faithful irreducible degree $2$ representation contains a non-trivial power of $a$, we may conclude that the degree $2$ summand of $\rho$ must be faithful.
\end{proof}

Now, since $\rho_{2,1}$ and $\rho_{2,2}$ are equivalent up to an isomorphism, we may assume without loss of generality that $\rho_{2,1}$ is a constituent of $\rho$. \\

\begin{rem} The group $(C_6 \times C_2)\rtimes C_2$ contains the metacyclic groups $S_3$, $D_4$ and $G(3,4,2)$ as subgroups, and hence we may use their structure theory developed in \hyperref[section-metacyclic]{Section~\ref*{section-metacyclic}}:
	\begin{enumerate}[ref=(\theenumi)]
		\item The subgroup $\langle a^2, c \rangle$ of $(C_6 \times C_2) \rtimes C_2$ is isomorphic to $S_3$. The complex representation of a hyperelliptic fourfold with holonomy group $S_3$ was determined in \hyperref[s3-complexrep]{Corollary~\ref*{s3-complexrep}}, which shows that $c$ is mapped to $1$ and $-1$ by the two degree $1$ constituents of $\rho$, respectively. We may therefore assume that
		\begin{align*}
		\rho(c) =  \begin{pmatrix}
		0 & 1  && \\ 1 & 0 && \\ && -1 & \\ &&& 1
		\end{pmatrix}.
		\end{align*}
		In a similar spirit, the subgroup $\langle a, bc \rangle$ is isomorphic to $G(3,4,2)$. Since $a$ is a commutator and $\rho(bc)$ must have the eigenvalue $1$, \hyperref[lemma-two-generators]{Proposition~\ref*{lemma-two-generators}}  implies that exactly one of the values $\chi(bc) = -\chi(b)$ and $\chi'(bc) = \chi'(b)$ is $1$. Observing that the matrix $\rho(ab)$ must have the eigenvalue $1$ as well, respectively, we find that
		\begin{align*}
		\chi(b) = \chi'(b) = 1
		\end{align*}
		is the only possibility. Therefore,
		\begin{align*}
		\rho(b) =  \begin{pmatrix}
		-1 &  && \\ & 1 && \\ && 1 & \\ &&& 1
		\end{pmatrix}.
		\end{align*}
		The representation $\rho|_{(C_6 \times C_2) \rtimes C_2}$ is, therefore, completely determined.
		\item Finally, observe that the subgroup $\langle b, c\rangle$ is isomorphic to the dihedral group $D_4$ (an isomorphism $D_4 \to \langle b,c \rangle$ maps the rotation to $bc$ and the symmetry to $c$). We now take the central element $k$ of order $3$ into account and exploit that
		\hyperref[prop:metacyclic-c3]{Proposition~\ref*{prop:metacyclic-c3}} shows that it is necessary for $\rho|_{\langle b,c,k\rangle}$ to be the complex representation of a hyperelliptic manifold with holonomy group $D_4 \times C_3$ that $\chi(k) = \chi'(k) = 1$. After possibly replacing $k$ by $k^2$, we may therefore assume that
		\begin{align*}
		\rho(k) = \diag(\zeta_3, \ \zeta_3, \ 1, \ 1).
		\end{align*} 
	\end{enumerate}
	The complex representation of a hyperelliptic fourfold with holonomy group $((C_6 \times C_2) \rtimes C_2) \times C_3$ is, therefore, completely determined, which completes the proof of \hyperref[72-30-prop]{Proposition~\ref*{72-30-prop}} \ref{72-30-prop1}. 
\end{rem}

Denoting now by $T$ a complex torus of dimension $4$ admitting an action of $G$ with associated complex representation $\rho$, the procedure described in \hyperref[isogeny]{Section~\ref*{isogeny}} induces the decomposition of $T$ up to equivariant isogeny given in \hyperref[72-30-prop]{Proposition~\ref*{72-30-prop}} \ref{72-30-prop2}. \\

%
%\begin{lemma} \label{72-30-rep}
%	Up to replacing $k$ by $k^2$, we have $\rho(k) = \diag(\zeta_3, \ \zeta_3, \ 1, \ 1)$.
%\end{lemma}
%
%\begin{proof}
%	The statement follows by the faithfulness of $\rho$ if we prove that $\chi(k) = \chi'(k) = 1$. Since the matrices $\rho(a^3k)$ and $\rho(bck)$ must have the eigenvalue $1$, we infer that $\chi'(k) = 1$. {\ ähnliches problem wie bei $Q_8 \times C_3$, wenn ich hier weiß, dass der Kern $2^\infty$-torsion ist, dann ist es ok.	}
%\end{proof}

\emph{(C) An example.} \\
We give a concrete example. Our previous discussion shows that we may assume that $T = (F \times F \times E \times E')/H$, where $F = \CC/(\ZZ+\zeta_3 \ZZ)$ is the equianharmonic elliptic curve, $E$, $E'$ are elliptic curves and $H$ is a finite subgroup of translations. Furthermore, we may write the action of $G = ((C_2 \times C_6) \rtimes C_2) \times C_3$ on $T$ as follows:
\begin{align*}
&a(z) = (-\zeta_3 z_1 + a_1, \ - \zeta_3^2 z_2 + a_2, \ z_3 + a_3, \ z_4+a_4), \\
&b(z) = (-z_1 + b_1, \ z_2 + b_2, \ z_3 + b_3, \ z_4 + b_4), \\
&c(z) = (z_2 + c_1, \ z_1 + c_2, \ - z_3, \ z_4 + c_4), \\
&k(z) = (\zeta_3 z_1, \ \zeta_3 z_2, \ z_3 + k_3, \ z_4 + k_4).
\end{align*}

The upcoming lemma investigates when the subgroup $\langle a,b,c,k \rangle \subset \Bihol(T)$ is isomorphic to $((C_6 \times C_2) \rtimes C_2) \times C_3$.

%F := CyclotomicField(3);
%z := F.1;
%
%P<a1,a2,a3,a4,b1,b2,b3,b4,c1,c2,c4,k3,k4> := PolynomialRing(F,13);
%
%a := Matrix(P,5,5,[[-z,0,0,0,a1],[0,-z^2,0,0,a2],[0,0,1,0,a3],[0,0,0,1,a4],[0,0,0,0,1]]);
%b := Matrix(P,5,5,[[-1,0,0,0,b1],[0,1,0,0,b2],[0,0,1,0,b3],[0,0,0,1,b4],[0,0,0,0,1]]);
%c := Matrix(P,5,5,[[0,1,0,0,c1],[1,0,0,0,c2],[0,0,-1,0,0],[0,0,0,1,c4],[0,0,0,0,1]]);
%k := Matrix(P,5,5,[[z,0,0,0,0],[0,z,0,0,0],[0,0,1,0,k3],[0,0,0,1,k4],[0,0,0,0,1]]);
%
%c*k - k*c;
\begin{lemma} \ \label{72-30-rels}
	\begin{enumerate}[ref=(\theenumi)]
		\item \label{72-30-rels-first} The element $a$ has order $6$ if and only if $(0, \ 0, \ 6a_3, \  6a_4)$ is zero in $T$.
		\item The element $b$ has order $2$ if and only if $(0, \ 2b_2, \ 2b_3, \ 2b_4)$ is zero in $T$.
		\item The elements $a$ and $b$ commute if and only if $(2a_1 + \zeta_3^2 b_1, \ \zeta_3 b_2, \ 0, \ 0)$ is zero in $T$.
		%	[0   0   0   0   2*a1 + (-zeta_3 - 1)*b1]
		%	[0   0   0   0   zeta_3*b2]
		%	[0   0   0   0   0]
		%	[0   0   0   0   0]
		%	[0   0   0   0   0]
		\item The element $c$ has order $2$ if and only if $(c_1+c_2, \ c_1+c_2, \ 0, \  2c_4)$ is zero in $T$.
		\item The relation $c^{-1}ac = a^{-1}$ holds if and only if $(a_2-\zeta_3^2 a_1 + \zeta_3 c_2, \ a_1 - \zeta_3 a_2 + \zeta_3^2 c_1, \ 0, \ 2a_4)$ is zero in $T$. 
		%	[0   0   0   0   (zeta_3 + 1)*a1 + a2 + zeta_3*c2]
		%	[0   0   0   0   a1 - zeta_3*a2 + (-zeta_3 - 1)*c1]
		%	[0   0   0   0   0]
		%	[0   0   0   0   2*a4]
		%	[0   0   0   0   0]
		\item The relation $c^{-1}bc = a^3b$ holds if and only if $(2\zeta_3 a_1 + b_1+b_2, \ 2\zeta_3^2 a_2 + b_1 + b_2 - 2c_1, \ -3a_3 - 2b_3, \ -3a_4)$ is zero in $T$.
		%	[0   0   0   0   2*zeta_3*a1 + b1 + b2]
		%	[0   0   0   0   (-2*zeta_3 - 2)*a2 + b1 + b2 - 2*c1]
		%	[0   0   0   0   -3*a3 - 2*b3]
		%	[0   0   0   0   -3*a4]
		%	[0   0   0   0   0]
		\item The element $k$ has order $3$ if and only if $(0, \ 0, \ 3k_3, \  3k_4)$ is zero in $T$.
		\item The elements $a$ and $k$ commute if and only if $((\zeta_3-1)a_1, \ (\zeta_3-1)a_2, \ 0, \ 0)$ is zero in $T$.
		%	[0   0   0   0   (-zeta_3 + 1)*a1]
		%	[0   0   0   0   (-zeta_3 + 1)*a2]
		%	[0   0   0   0   0]
		%	[0   0   0   0   0]
		%	[0   0   0   0   0]
		\item The elements $b$ and $k$ commute if and only if $((\zeta_3-1)b_1, \ (\zeta_3-1)b_2, \ 0, \ 0)$ is zero in $T$.
		%	[0   0   0   0   (-zeta_3 + 1)*b1]
		%	[0   0   0   0   (-zeta_3 + 1)*b2]
		%	[0   0   0   0   0]
		%	[0   0   0   0   0]
		%	[0   0   0   0   0]
		\item \label{72-30-rels-last} The elements $c$ and $k$ commute if and only if $((\zeta_3-1)c_1, \ (\zeta_3-1)c_2, \ 2k_3, \ 0)$ is zero in $T$.
		%	[0   0   0   0   (-zeta_3 + 1)*c1]
		%	[0   0   0   0   (-zeta_3 + 1)*c2]
		%	[0   0   0   0   -2*k3]
		%	[0   0   0   0   0]
		%	[0   0   0   0   0]
	\end{enumerate}
	Consequently, if the ten properties \ref{72-30-rels-first} -- \ref{72-30-rels-last} are satisfied, then the group $\langle a,b,c,k \rangle \subset \Bihol(T)$ is isomorphic to $((C_2 \times C_6) \rtimes C_2) \times C_3$ and does not contain any translations.
\end{lemma}

Finally, we give a concrete example.

\begin{example}
	Define $T = (F \times F \times E \times E')/H$, where
	\begin{itemize}
		\item $F = \CC/(\ZZ+\zeta_3\ZZ)$ is Fermat's elliptic curve,
		\item $E = \CC/(\ZZ+\tau\ZZ)$ and $E' = \CC/(\ZZ+\tau'\ZZ)$ are elliptic curves in standard form, and
		\item $H \subset F \times F \times E \times E'$ is the group of order $2$ spanned by $\left(0, \ 0, \ \frac12, \frac12\right)$.
	\end{itemize}
	Consider the following automorphisms of $F \times F \times E \times E'$:
	\begin{align*}
	&a(z) = \left(-\zeta_3 z_1, \ - \zeta_3^2 z_2, \ z_3 + \frac16, \ z_4\right), \\
	&b(z) = \left(-z_1, \ z_2, \ z_3 + \frac14, \ z_4 + \frac14\right), \\
	&c(z) = \left(z_2, \ z_1, \ -z_3, \ z_4 + \frac{\tau'}2\right), \\
	&k(z) = \left(\zeta_3 z_1, \ \zeta_3 z_2, \ z_3, \ z_4 + \frac13 \right).
	\end{align*}
	Their linear parts map $H$ to $H$, and hence they define automorphisms of $T$.  \hyperref[72-30-rels]{Lemma~\ref*{72-30-rels}} now shows that $\langle a,b,c,k \rangle \subset \Bihol(T)$ is isomorphic to $((C_2 \times C_6) \rtimes C_2) \times C_3$ and does not contain any translations. \\
	It remains to prove that $\langle a,b,c,k \rangle$ acts freely on $T$. This follows in the following way: e¸very element of $\langle a,b,c \rangle$ can be uniquely written in the form $a^ib^jc^m$ for $i \in \{0,...,5\}$ and $j,m \in \{0,1\}$ and therefore acts on $S := (E \times E')/\langle (\frac12, \ \frac12)\rangle$ via
	\begin{align*}
	a^ib^jc^m(z_3,z_4) = \left((-1)^m z_3 + \frac{i}{6} + \frac{j}{4}, \ z_4 + \frac{m\tau'}{2} + \frac{j}{4}\right).
	\end{align*}
	This description implies that $a^ib^jc^m$ can only have a fixed point on $S$ (which is necessary and sufficient for it to have a fixed point on $T$) if it is trivial, i.e., $i = j = m = 0$. This proves that $\langle a,b,c,k \rangle$ acts freely on $T$ and hence $T/\langle a,b,c,k\rangle$ is a hyperelliptic manifold.
\end{example}

\emph{(D) The Hodge diamond.} \\
The Hodge diamond of a hyperelliptic fourfold with holonomy group $((C_2 \times C_6) \rtimes C_2) \times C_3$ is 
\begin{center}
	$\begin{matrix}
	&   &  &  & 1 &  &  &  &  \\
	&   &  & 1 &  & 1 &  &  &  \\
	&   & 0 &  & 3 &  & 0  &  & \\
	&  0 &  & 3 &   & 3 &   &  0 &  \\
	0&    & 1 &  & 6  &  & 1  &     &  0 \\
	\end{matrix}$
\end{center}

\subsection{$G(3,4,2) \times C_3 \times C_3$ (ID $[108,32]$)} \label{108-32-section}

Consider the group
\begin{align*}
G(3,4,2) \times C_3 \times C_3 = \langle g,h,k_1,k_2 ~ | ~ g^3 = h^4 = k_1^3 = k_2^3 = 1, ~ h^{-1}gh = g^2, ~ k_1 \text{ and } k_2 \text{ central}\rangle.
\end{align*}
In parts (B) and (C) of this section, we prove:

\begin{prop} \label{108-32-prop}
	Let $X = T/(G(3,4,2) \times C_3^2)$ be a hyperelliptic fourfold with associated complex representation $\rho$. Then:
	\begin{enumerate}[ref=(\theenumi)]
		\item \label{108-32-prop1} Up to equivalence and automorphisms, $\rho$ is given as follows:
		\begin{align*}
		&\rho(g) = \begin{pmatrix}
		\zeta_3 &  && \\  & \zeta_3^2 && \\ && 1& \\ &&&1
		\end{pmatrix}, \qquad \rho(h) = \begin{pmatrix}
		0 & -1 && \\ 1 & 0 && \\ &&1 & \\ &&& -1
		\end{pmatrix}, \\
		&\rho(k_1) = \begin{pmatrix}
		\zeta_3 &  && \\  & \zeta_3 && \\ && 1& \\ &&&1
		\end{pmatrix}, \qquad \rho(k_2) = \begin{pmatrix}
		1 & && \\ & 1 && \\ &&1 & \\ &&& \zeta_3
		\end{pmatrix}.
		\end{align*}
		\item \label{108-32-prop2} The representation $\rho$ induces an equivariant isogeny $F \times F \times E \times F \to T$, where $E \subset T$ is an elliptic curve and $F = \CC/(\ZZ+\zeta_3\ZZ) \subset T$ are copies of Fermat's elliptic curve.
		\item Hyperelliptic fourfolds with holonomy group $G(3,4,2) \times C_3^2$ exist.
	\end{enumerate}
	In particular, $X$ moves in a complete $1$-dimensional family of hyperelliptic fourfolds with holonomy group $G(3,4,2) \times C_3^2$.
\end{prop}

Moreover, the Hodge diamond of a hyperelliptic fourfold with holonomy group $G(3,4,2) \times C_3^2$ will be given in part (D). \\

\emph{(A) Representation Theory of $G(3,4,2)$.} \\

The representation theory of $G(3,4,2)$ is easily described. First of all, the relation $h^{-1}gh =g^2$ implies that $g$ is a commutator, and hence the linear characters of $G(3,4,2)$ are simply the ones given by
\begin{align*}
\chi_a(g) = 1, \qquad \chi_j(h) = i^a
\end{align*}
for $a \in \{0,...,3\}$. Furthermore, $G(3,4,2)$ has two irreducible degree $2$ representations:
\begin{center}
	\begin{tabular}{lll}
		$\rho_{2,1}$: & $g \mapsto \begin{pmatrix}
		\zeta_3 & \\ & \zeta_3^2
		\end{pmatrix}$, & $h \mapsto \begin{pmatrix}
		0 & -1 \\ 1 & 0
		\end{pmatrix}$, \\
		$\rho_{2,2}$: & $g \mapsto \begin{pmatrix}
		\zeta_3 & \\ & \zeta_3^2
		\end{pmatrix}$, & $h \mapsto \begin{pmatrix}
		0 & 1 \\ 1 & 0
		\end{pmatrix}$.
	\end{tabular}
\end{center}
Note that $\rho_{2,2}$ is not faithful. \\

\emph{(B) The Complex Representation and the Isogeny Type of the Torus.} \\

Let $\rho \colon G(3,4,2) \times C_3^2 \to \GL(4,\CC)$ be a representation. We investigate how $\rho$ decomposes into irreducible representations, if $\rho$ satisfies the three properties \ref{rho-prop1} -- \ref{rho-prop3} from p. \pageref{rho-prop1}. First of all, we immediately see that $\rho$ decomposes as the direct sum of an irreducible degree $2$ representation and two linear characters, which we will denote by $\rho_2$, $\chi$ and $\chi'$, respectively. We are then able to show

\begin{lemma}\label{lemma:108-32}
	The following statements hold:
	\begin{enumerate}[ref=(\theenumi)]
		\item \label{lemma:108-32-1} for any $k \in \langle k_1, k_2 \rangle$, it holds $\chi(k) = 1$ or $\chi'(k) = 1$,
		\item \label{lemma:108-32-2} $\langle k_1, k_2 \rangle \nsubseteq \ker(\rho_2)$,
		\item \label{lemma:108-32-3} $\rho_2|_{G(3,4,2)} = \rho_{2,1}$.
	\end{enumerate}
\end{lemma}

\begin{proof}
	\ref{lemma:108-32-1} Assume for a contradiction that there is a non-trivial element $k \in \langle k_1,k_2 \rangle$ such that $\chi(k) \neq 1$ and $\chi'(k) \neq 1$. Since
	\begin{align*}
	\rho(k) = \diag(\rho_2(k), ~ \chi(k), ~ \chi'(k))
	\end{align*}
	has the eigenvalue $1$, we conclude that $k \in \ker(\rho_2)$. However then
	\begin{align*}
	\rho(gk) = \diag(\zeta_3, ~ \zeta_3^2, ~ \chi(k), ~ \chi'(k))
	\end{align*}
	does not have the eigenvalue $1$. \\
	
	\ref{lemma:108-32-2} If  $\langle k_1, k_2 \rangle \subset \ker(\rho_2)$, then the faithfulness of $\rho = \rho_2 \oplus \chi \oplus \chi'$ implies that $(\chi \oplus \chi')|_{\langle k_1,k_2\rangle}$ is faithful -- in particular, we find $k \in \langle k_1,k_2 \rangle$ such that $\chi(k), \chi'(k) \neq 1$. This contradicts \ref{lemma:108-32-1}. \\
	
	\ref{lemma:108-32-3} The previous assertions imply that we are able to choose $k_1$ and $k_2$ appropriately such that
	\begin{align*}
	\rho(k_1) = \diag(\zeta_3, ~ \zeta_3, ~ 1, ~ 1) \qquad \text{ and } \qquad \rho(k_2) = \diag(1, ~ 1, ~ 1, ~ \zeta_3),
	\end{align*}
	Since the matrix
	\begin{align*}
	\rho(ghk_1k_2) = \diag(\zeta_3 \rho_2(gh), ~ \chi(h), ~ \zeta_3 \chi'(h))
	\end{align*}
	has the eigenvalue $1$, we deduce that $\chi(h) = 1$. Now, if $\rho_2|_{G(3,4,2)}$ is the non-faithful representation $\rho_{2,2}$ (which maps the element $h$ of order $4$ to a matrix of order $2$), the faithfulness of $\rho$ implies that $\chi'(h)$ is a primitive fourth unity. However then $\rho(hk_2)$ has exactly one eigenvalue of order $12$, contradicting property \ref{rho-prop3}.
\end{proof}

The proof of \hyperref[lemma:108-32]{Lemma~\ref*{lemma:108-32}} shows in particular that:
\begin{itemize}
	\item $\rho(k_1) = \diag(\zeta_3, ~ \zeta_3, ~ 1, ~ 1)$ and $\rho(k_2) = \diag(1, ~ 1, ~ 1, ~ \zeta_3)$, 
	\item $\chi$ is trivial,
	\item $\chi'(h)$ is not a primitive fourth root of unity.
\end{itemize}

Finally, we recall from \hyperref[lemma-two-generators]{Proposition~\ref*{lemma-two-generators}} that $\chi'|_{G(3,4,2)}$ is non-trivial, since $G(3,4,2)$ is metacyclic. We thus obtain that $\chi'(h) = -1$ is the only possibility that may potentially allow a free action. This completes the proof of \hyperref[108-32-prop]{Proposition~\ref*{108-32-prop}} \ref{108-32-prop1}. Part \ref{108-32-prop2} of the cited proposition follows immediately from the description of the complex representation $\rho$. \\

\emph{(C) An Example.} \\

We now construct an explicit example. According to the previous section, we may assume that $T$ is a quotient of $F \times F \times E \times F$ by a finite group of translations. Furthermore, after a change of origin in the three copies of $F$, we may write the action on $T$ in the following form:

\begin{align*}
&g(z) = \left(\zeta_3 z_1 + a_1, ~ \zeta_3^2 z_2 + a_2, ~ z_3 + a_3, ~ z_4 + a_4\right), \\
&h(z) = \left(-z_2 + b_1, ~ z_1 + b_2, ~ z_3 + b_3, ~ -z_4 + b_4\right), \\
&k_1(z) = \left(\zeta_3 z_1, ~ \zeta_3 z_2, ~ z_3 + k_{13}, ~ z_4 + k_{14}\right), \\
&k_2(z) = \left(z_1 + k_{21}, ~ z_2 + k_{22}, ~ z_3 + k_{23}, ~ \zeta_3 z_4\right).
\end{align*}

\begin{lemma} \label{108-32-relations}
	\begin{enumerate}[ref=(\theenumi)]
		\item \label{108-32-rels-first} The relation $g^3 = \id_T$ holds if and only if $(0, \ 0, \ 3a_3, \ 3a_4)$  is zero in $T$.
		\item The relation $h^4 = \id_T$ holds if and only if $4b_3 = 0$ in $E$.
		\item The relation $h^{-1}gh = g^2$ holds if and only if $(-\zeta_3^2 a_1 - a_2 + (1-\zeta_3^2)b_2, ~ a_1-\zeta_3a_2 + (\zeta_3-1)b_1, ~ a_3, ~ 3a_4)$ is zero in $T$.
		\item The relation $k_1^3 = \id_T$ holds if and only if $(0, \ 0, \ 3k_{13}, \ 3k_{14})$ is zero in $T$.
		\item The relation $k_2^3 = \id_T$ holds if and only if $(3k_{21}, \ 3k_{22}, \ 3k_{23}, \ 0)$ is zero in $T$.
		\item The elements $g$ and $k_1$ commute if and only if $((\zeta_3-1)a_1, \ (\zeta_3-1)a_2, \ 0, \ 0)$ is zero in $T$.
		\item The elements $h$ and $k_1$ commute if and only if $((\zeta_3-1)b_1, \ (\zeta_3-1)b_2, \ 0, \ 2k_{14})$ is zero in $T$.
		\item  The elements $g$ and $k_2$ commute if and only if $((\zeta_3-1)k_{21},  \ (\zeta_3^2-1)k_{22}, \ 0, \ (1-\zeta_3)a_4)$ is zero in $T$.
		\item The elements $h$ and $k_2$ commute if and only if $(k_{21} + k_{22}, \ k_{22} - k_{21}, ~ 0, ~ (\zeta_3-1)b_4)$ is zero in $T$.
		\item \label{108-32-rels-last} The elements $k_1$ and $k_2$ commute if and only if $((\zeta_3-1)k_{21}, \ (\zeta_3-1)k_{22}, \ 0, \ (\zeta_3-1)k_{14})$ is zero in $T$.
	\end{enumerate}
	Consequently, if the ten properties \ref{108-32-rels-first} --\ref{108-32-rels-last} are satisfied, then the group $\langle g,h,k_1,k_2 \rangle \subset \Bihol(T)$ is isomorphic to $G(3,4,2) \times C_3 \times C_3$ and does not contain any translations.
\end{lemma}

\begin{example}
	Define $T := F \times F \times E \times F$, where $F := \CC/(\ZZ+\zeta_3\ZZ)$ is Fermat's elliptic curve, and $E = \CC/(\ZZ + \tau \ZZ)$ is an arbitrary elliptic curve. Furthermore, we define the following biholomorphisms of $T$:
	\begin{align*}
	&g(z) = \left(\zeta_3 z_1, ~ \zeta_3^2 z_2, ~ z_3, ~ z_4 + \frac{1-\zeta_3}{3}\right), \\
	&h(z) = \left(-z_2, ~ z_1, ~ z_3 + \frac14, ~ -z_4\right), \\
	&k_1(z) = \left(\zeta_3 z_1, ~ \zeta_3 z_2, ~ z_3 + \frac{ 1}{3}, ~ z_4\right), \\
	&k_2(z) = \left(z_1, ~ z_2, ~ z_3 + \frac\tau3, ~ \zeta_3 z_4\right).
	\end{align*}
	Clearly, the relations of \hyperref[108-32-relations]{Lemma~\ref*{108-32-relations}} are satisfied, and hence $\langle g,h,k_1,k_2\rangle \Bihol(T)$ is isomorphic to $G(3,4,2) \times C_3 \times C_3$ and does not contain any translations. \\
	The freeness of the action is easily checked: the element $g^a h^b k_1^c k_2^d$ acts on the last two factors, $E \times F$, in the following way:
	\begin{align*}
	(z_3,z_4) \mapsto \left(z_3 + \frac{b}{4} + \frac{c+d\tau}{3}, ~ (-1)^b \zeta_3^d z_4 + a \cdot \frac{1-\zeta_3}{3}\right).
	\end{align*}
	It is immediate from this description that only the identity has a fixed point in $T$, and thus we have indeed defined a $4$-dimensional hyperelliptic manifold $T/\langle g,h,k_1,k_2\rangle$ with holonomy group $\langle g,h,k_1,k_2 \rangle \cong G(3,4,2) \times C_3 \times C_3$.
\end{example}

The proof of \hyperref[108-32-prop]{Proposition~\ref*{108-32-prop}} is, therefore, complete. \\

\emph{(D) The Hodge diamond.} \\
The Hodge diamond of a hyperelliptic fourfold with holonomy group $G(3,4,2) \times C_3 \times C_2$ is given as follows:
\begin{center}
	$\begin{matrix}
	&   &  &  & 1 &  &  &  &  \\
	&   &  & 1 &  & 1 &  &  &  \\
	&   & 0 &  & 3 &  & 0  &  & \\
	&  0 &  & 2 &   & 2 &   &  0 &  \\
	0&    & 0 &  & 4  &  & 0  &     &  0 \\
	\end{matrix}$
\end{center}

\section{Non-Examples} \label{non-examples}

%[ [24,7], [24,14], [36,5], [48,9], [48,24], [48,26], [48,45], [48,47], [54,8], [54,10], [72,14], [72,37], [72,38], [96,46], [96,47], [96,164], [108,32], [144,101], [144,102], [144,103], [216,150], [216,177] ];

%% \red give table???

\subsection{$S_4$ (ID [24,12])} \label{S4-excluded-section}

We prove that the symmetric group $S_4$ on four letters is not hyperelliptic in dimension $4$. The group $S_4$ is generated by $A_4$ (which in turn is generated by a $3$-cycle $\sigma$ and a double transposition $\xi$ subject to the relation $ (\xi \sigma)^3 = 1$) and a transposition $\tau$. Its three irreducible representations of degree $>1$ are given by
\begin{center}
	%\bgroup \tablinesep=2ex\tabcolsep=10pt
	\begin{tabular}{llll}
		Name & Image of $\sigma$ & Image of $\xi$ & Image of  $\tau$ \\ \hline \hline 
		$\rho_2$ & $\begin{pmatrix}
		\zeta_3 & \\ & \zeta_3^2
		\end{pmatrix}$ & $\begin{pmatrix}
		1 & \\ & 1
		\end{pmatrix}$ & $\begin{pmatrix}
		0 & 1 \\ 1 & 0
		\end{pmatrix}$ \\
		$\rho_{3,1}$ & $\begin{pmatrix}
		0 & 0 & 1 \\ 1 & 0 & 0 \\ 0 & 1 & 0
		\end{pmatrix}$ & $\begin{pmatrix}
		-1 && \\ & 1 & \\ && -1
		\end{pmatrix}$ & $\begin{pmatrix}
		0 & 1 & 0 \\ 1 & 0 & 0 \\ 0 & 0 & 1
		\end{pmatrix}$ \\
		$\rho_{3,2}$ & $\begin{pmatrix}
		0 & 0 & 1 \\ 1 & 0 & 0 \\ 0 & 1 & 0
		\end{pmatrix}$ & $\begin{pmatrix}
		-1 && \\ & 1 & \\ && -1
		\end{pmatrix}$ & $\begin{pmatrix}
		0 & -1 & 0 \\ -1 & 0 & 0 \\ 0 & 0 & -1
		\end{pmatrix}$
	\end{tabular}
	%\egroup
\end{center}

In the following, we prove that if $\rho \colon S_4 \to \GL(4,\CC)$ is a faithful representation, then it is not the associated complex representation of a hyperelliptic manifold. We exclude the representations above individually and start by excluding $\rho_2$.

\begin{lemma}
	Suppose that $\rho \colon S_4 \to \GL(4,\CC)$ is a faithful representation, which contains the degree $2$ irreducible representation $\rho_2$ of $S_4$. Then $\rho$ does not occur as the complex representation of a hyperelliptic manifold with holonomy group $S_4$.
\end{lemma}

\begin{proof}
	The kernel of $\rho_2$ is the Klein four subgroup generated by the double transpositions of $S_4$. Since the derived subgroup of $S_4$ is $A_4 = \langle \sigma, \xi \rangle$, the representation $\rho$ cannot be faithful, unless it is equivalent to the direct sum of two copies of $\rho_2$. However then $\rho(\sigma)$ does not have the eigenvalue $1$.
\end{proof}

Next, the representation $\rho_{3,1}$ is excluded.

\begin{lemma}
	Suppose that $\rho \colon S_4 \to \GL(4,\CC)$ is a faithful representation, which contains the representation $\rho_{3,1}$. Then $\rho$ does not occur as the complex representation of a hyperelliptic manifold with holonomy group $S_4$.
\end{lemma}

\begin{proof}
	Write $\rho = \rho_{3,1}\oplus \chi$, where $\chi$ is a linear character of $S_4$. Observe that the subgroup of $S_4$ spanned by $\sigma$ and $\tau$ is isomorphic to $S_3$. It is easy to see that $\rho_{3,1}|_{\langle \sigma, \tau \rangle}$ splits as the direct sum of the irreducible degree $2$ representation of $S_3$ and the trivial character. According to  \hyperref[lemma-two-generators]{Proposition~\ref*{lemma-two-generators}}, only $\chi(\tau) = -1$ may possibly allow for a free action on a $4$-dimensional complex torus. In this case, however, the matrix
	\begin{align*}
	\rho(\tau \xi) = \begin{pmatrix}
	0 & 1 & 0 & \\ -1 & 0 & 0 & \\ 0 & 0 & -1 & \\ &&& -1
	\end{pmatrix}
	\end{align*} 
	does not have the eigenvalue $1$.
\end{proof}

It remains to exclude the last representation, $\rho_{3,2}$.

\begin{lemma} \label{S4-case-c}
	Suppose that $\rho \colon S_4 \to \GL(4,\CC)$ is a faithful representation, which contains the representation $\rho_{3,2}$. Then $\rho$ does not occur as the complex representation of a hyperelliptic manifold with holonomy group $S_4$.
\end{lemma}

\begin{proof}
	Write $\rho = \rho_{3,2} \oplus \chi$, where $\chi$ is a linear character. By \hyperref[s3-complexrep]{Corollary~\ref*{s3-complexrep}} on p. \pageref{s3-complexrep}, we infer that $\chi$ is the trivial character. Let $T$ be a complex torus of dimension $4$ and write the action of $S_3 = \langle \sigma, \tau \rangle$ on $T$ as follows:
	\begin{align*}
	&\sigma(z) = (z_3 + s_1, \ z_1 + s_2, \ z_2 + s_3, \ z_4 + s_4), \\
	&\tau(z) = (-z_2 + t_1, \ -z_1 + t_2, \ -z_3 + t_3, \ z_4 + t_4).
	\end{align*}
	The relations $\sigma^3 = \tau^2 = (\sigma \tau)^2 = \id_T$ yield that the following elements are zero in $T$, respectively:
	\begin{align*}
	&v_1 := (s_1+s_2+s_3, \ s_1+s_2+s_3, \ s_1+s_2+s_3, \ 3s_4), \\
	&v_2 := (t_1 - t_2, \ t_2 - t_1, \ 0, \ 2t_4), \\
	&v_3 := (0, \ s_3 - s_1 + t_2 - t_3, \ s_1 - s_3 + t_3 - t_2, \ 2t_4 + 2s_4).
	\end{align*}
	In particular,
	\begin{align*}
	v_1 + v_2 - v_3 = (w_1, \ w_2, \ w_3, \ s_4) \qquad \quad \text{(for some unspecified } w_1,w_2,w_3\text{)}
	\end{align*}
	is zero in $T$. The action of $\sigma$ is thus congruent to an action of the form
	\begin{align*}
	\sigma(z) = (z_3 + s_1', \ z_1 + s_2', \ z_2 + s_3', \ z_4).
	\end{align*}
	\hyperref[A4-lemma-fixed-point]{Lemma~\ref*{A4-lemma-fixed-point}} now proves that $\sigma$ does not act freely on $A$. This completes the proof.
\end{proof}

\begin{rem}
	At first glance, the proof of  \hyperref[S4-case-c]{Lemma~\ref*{S4-case-c}} seems suspicious: it seems like only the subgroup $S_3 = \langle \sigma, \tau \rangle$ was used to construct a fixed point. This seems like a mistake, since $S_3$ was shown to be hyperelliptic in dimension $4$ in  \hyperref[72-27-and-108-42-section]{Section~\ref*{72-27-and-108-42-section}}. However, the double transposition $\xi$ was used in the proof of  \hyperref[A4-lemma-fixed-point]{Lemma~\ref*{A4-lemma-fixed-point}} to construct a fixed point of $\sigma$.
\end{rem}

In total, the three lemmas above prove that $S_4$ is not hyperelliptic in dimension $4$.

\subsection{$D_4 \times C_4$ (ID [32,25])} \label{D4xCd-excluded-section}

We aim to prove that $D_4 \times C_4$ is not hyperelliptic in dimension $4$. \\
We examine the complex representation of a hypothetical hyperelliptic fourfold with holonomy group $D_4 \times C_4$. We use the standard presentation of the dihedral group, i.e.,
\begin{align*}
D_4 \times C_d = \langle r,s,\ell \ | \ r^4 = s^2 = (rs)^2 = \ell^4 = [r,\ell] = [s,\ell] = 1\rangle.
\end{align*}
In the following, we assume that $\rho \colon D_4 \times C_4 \to \GL(4,\CC)$ is the complex representation of some hyperelliptic fourfold. 

\begin{lemma} \label{d4xcd-lemma}
	Up to automorphisms of $D_4 \times C_4$ and equivalence, the representation $\rho$ is given as follows:
	\begin{align*}
	\rho(r) = \begin{pmatrix}
	0 & -1 && \\ 1 & 0 && \\ && 1 & \\ &&& 1
	\end{pmatrix},\quad \rho(s) = \begin{pmatrix}
	1 &  && \\  & -1 && \\ && -1 & \\ &&& -1
	\end{pmatrix},\quad \rho(\ell) = \begin{pmatrix}
	1 & && \\ & 1 && \\ && 1 & \\ &&& i
	\end{pmatrix}.
	\end{align*}
\end{lemma}

\begin{proof}
	We use that $\rho|_{ \langle r,s,\ell^2\rangle}$ is the complex representation of a hyperelliptic fourfold with holonomy group $D_4 \times C_2$: \hyperref[16-11-prop]{Proposition~\ref*{16-11-prop}} then shows that we may assume that $\rho(r)$, $\rho(s)$ and $\rho(\ell^2)$ are as in the statement of the lemma. In particular,
	\begin{enumerate}[ref=(\theenumi)]
		\item \label{d4-cd-1} the first three diagonal entries of $\rho(\ell)$ are contained in $\{-1,1\}$,
		\item the fourth diagonal entry of $\rho(\ell)$ is a primitive fourth root of unity. After replacing $\ell$ by $\ell^{-1}$ if necessary, we may assume that it is equal to $i$.
	\end{enumerate}
	After applying the automorphism $r \mapsto r$, $s \mapsto s$ and $\ell \mapsto \ell r^2$ of $D_4 \times C_4$ if necessary, we may assume that the first two diagonal entries of $\rho(\ell)$ are $1$. Since $\rho(r^2\ell)$ must have the eigenvalue $1$, we conclude that the third diagonal entry of $\rho(\ell)$ is $1$, as asserted.
\end{proof}

Consider now a complex torus $T$ of dimension $4$ endowed with an action of $D_4 \times C_4$, whose associated complex representation is the representation $\rho$ of \hyperref[d4xcd-lemma]{Lemma~\ref*{d4xcd-lemma}}. According to \hyperref[16-11-prop]{Proposition~\ref*{16-11-prop}}, $\rho$ induces an equivariant isogeny
\begin{align*}
E \times E \times E' \times E'' \to T,
\end{align*}
where $E, E', E'' \subset T$ are elliptic curves, and $E''$ is isomorphic to the harmonic elliptic curve $E_i = \CC/(\ZZ+i\ZZ)$. After changing the origin in the respective elliptic curves, we may hence write the action of $D_4 \times C_4$ on $T$ as follows:
\begin{align*}
&r(z) = (-z_2, \ z_1, \ z_3 + r_3, \ z_4 + r_4), \\
&s(z) = (z_1 + s_1, \ -z_2 + s_2, \ -z_3, \ -z_4 + s_4), \\
&\ell(z) = (z_1 + l_1, \ z_2 + l_2, \ z_3 + l_3, \ i z_4).
\end{align*}
We will make use of the following elementary properties:

\begin{lemma} \label{D4xCd-properties}
	The following statements hold:
	\begin{enumerate}[ref=(\theenumi)]
		\item \label{D4xCd-prop1} the elements
		\begin{align*}
		&v_1 := (l_1 + l_2, \ l_2 - l_1, \ 0, (i-1)r_4), \\
		&v_2 := (0, \ 2l_2, \ 2l_3, \ (i-1)s_4), \\
		&v_3 := (2l_1, \ 2l_2, \ 0, \ 2(i-1)r_4)
		\end{align*} 
		are zero in $T$, respectively.
		\item  \label{D4xCd-prop2} it holds $2s_1 = 0$ in $E$,
		\item \label{D4xCd-prop3} if $2 l_1 = 0$, then $\ell^2$ does not act freely on $T$,
		\item \label{D4xCd-prop4} it holds $2(l_1+l_2) = 2(l_1-l_2) = 0$ and $4l_1 = 4l_2 = 0$. If, furthermore, $\ell^2$ acts freely on $T$, then $2l_1 \neq 0 \neq 2l_2$, 
		\item \label{D4xCd-prop5} if $\ell^2$ acts freely on $T$, then there is no $w_4 \in E_i$ such that $(2l_1, \ 2l_2, \ 2l_3, \ w_4)$ is zero in $T$,
		\item \label{D4xCd-prop6} if $s$ acts freely on $T$, then there are no $w_2, w_3, w_4$ such that $(s_1, ~ w_2, ~ w_3, ~ w_4)$ is zero in $T$.
	\end{enumerate}
\end{lemma}

\begin{proof} 
	\ref{D4xCd-prop1} The difference $\ell r - r\ell$ is the translation
	\begin{align*}
	z \mapsto z + (l_1 + l_2, \ l_2 - l_1, \ 0, (i-1)r_4) = z+v_1.
	\end{align*}
	Since it induces the zero morphism on $T$, we obtain $v_1 = 0$. Similarly, the assertions that $v_2 = v_3 =$ follow from computing $\ell s - s \ell$ and $\ell r^2 - r^2 \ell$, which define the zero map $T \to T$ as well. \\
	
	\ref{D4xCd-prop2} The relation $s^2 = \id_T$ implies that $(2s_1, \ 0, \ 0, \ 0) = 0$. Since the composed map 
	\begin{align*}
	E \stackrel{\text{first coord.}}\hookrightarrow E \times E \times E' \times E_i \to  T
	\end{align*}
	is injective, we obtain that $2s_1 = 0$ in $E$. \\
	
	\ref{D4xCd-prop3} If $2l_1 = 0$, the action of $\ell^2$ on $T$ is equal to
	\begin{align*}
	\ell^2(z) = (z_1, \ z_2 + 2l_2, \ z_3 + 2l_3, \ \zeta_d^2 z_4) = (z_1, \ z_2, \ z_3, \ -z_4 - (i-1)s_4).
	\end{align*}
	(It was used that $v_2 = 0$ in $T$ to simplify the expression of $\ell^2$.) Hence $\ell^2$ has a fixed point on $T$. \\
	
	\ref{D4xCd-prop4} We calculate
	\begin{align*}
	&(\id_T - \rho(s\ell^2))v_1 = (0, \ 2(l_2-l_1), \ 0, \ 0) = 0, \\
	&(\id_T - \rho(r^2s\ell^2))v_1 = (2(l_1+l_2), \ 0, \ 0, \ 0) = 0.
	\end{align*}
	Using again that the map $E \hookrightarrow T$ is injective to conclude that $2(l_1+l_2) = 2(l_1-l_2) = 0$ in $E$. It follows that $4l_1 = 2(l_1+l_2) + 2(l_1-l_2) = 0$ and $4l_2 = 2(l_1+l_2) - 2(l_1-l_2) = 0$. \\
	As for the last assertion, we have seen in part \ref{D4xCd-prop2} that the freeness of $\ell^2$ implies $2l_1 \neq 0$. Since $2(l_1-l_2) = 0$, we obtain $2l_1 \neq 0$, if $\ell^2$ acts freely. \\
	
	\ref{D4xCd-prop5} Suppose that $(2l_1, \ 2l_2, \ 2l_3, \ w_4) \in H$ for some $w_4 \in E_i$. Then
	\begin{align*}
	\ell ^2(z) &= (z_1 + 2l_1, \ z_2 + 2l_2, \ z_3 + 2l_3, \ - z_4) \\
	&= (z_1, \ z_2, \ z_3, \ -z_4 - w_4)
	\end{align*}
	has a fixed point on $T$. \\ 
	
	\ref{D4xCd-prop6} This follows exactly as in \ref{D4xCd-prop5}.
\end{proof}

We are finally able to prove the goal of this section.

\begin{prop} \label{prop-d4xcd-excluded}
	The group $D_4 \times C_4$ is not hyperelliptic in dimension $4$.
\end{prop}

\begin{proof} 
	The conditions $2(l_1+l_2) = 2(l_1-l_2) = 0$ allow us to distinguish several cases, in each of which we will find an element of $D_4 \times C_4$ that does not act freely on $T$:
	\begin{enumerate}[ref=(\theenumi)]
		\item If $l_1 = l_2$, then 
		\begin{align*}
		v_1+v_2 = (2l_1, \ 2l_2, \ 2l_3, \ (i-1)(r_4+s_4)) 
		\end{align*}
		is contained in $H$. It follows that $\ell^2$ has a fixed point, see \hyperref[D4xCd-properties]{Lemma~\ref*{D4xCd-properties}} \ref{D4xCd-prop5}. 
		\item Similarly, if $l_1 = -l_2$, then we consider the element
		\begin{align*}
		v_1+v_2+v_3 = (2l_1, \ 2l_2, \ 2l_3, \ (i-1)(3r_4+s_4)) = 0
		\end{align*}
		instead.
		\item We prove that if $l_1 \neq \pm l_2$ and $\ell^2$ acts freely, then $s$ has a fixed point on $T$. Indeed, the hypotheses imply that the elements $l_1 + l_2$ and $2l_1$ span $E[2]$ (see also \hyperref[D4xCd-properties]{Lemma~\ref*{D4xCd-properties}} \ref{D4xCd-prop3}, \ref{D4xCd-prop4}). Observe that the first coordinates of $v_1$ and $v_3$ are $l_1+l_2$ and $2l_1$, respectively. Hence the first coordinate of a suitable integral linear combination of $v_1$ and $v_3$ is equal to $s_1$, the latter element being $2$-torsion according to \hyperref[D4xCd-properties]{Lemma~\ref*{D4xCd-properties}} \ref{D4xCd-prop2}. Finally, \hyperref[D4xCd-properties]{Lemma~\ref*{D4xCd-properties}} \ref{D4xCd-prop6} shows that $s$ does not act freely on $T$.
	\end{enumerate}
\end{proof}

\subsection{$(C_8 \times C_2) \rtimes C_2$ (ID [32,5]) and $G(4,8,3)$ (ID [32,12])} \label{32-5+32-12-excluded}

We prove simultaneously that the two groups
\begin{align*}
&(C_8 \times C_2) \rtimes C_2 = \langle a,b,c \ | \ a^8 = b^2 = c^2 = [a,b] = [b,c] = 1, \ c^{-1}ac = ab\rangle, \\
&G(4,8,3) = \langle g,h \ | \ g^4 = h^8 = 1, \ h^{-1}gh = g^3\rangle
\end{align*}
are not hyperelliptic in dimension $4$. \\

Let $G$ be one of the above two groups and suppose that $\rho \colon G \to \GL(4,\CC)$ is a faithful representation. Since the centers of both groups above are non-cyclic (the centers are generated by $a^2$, $b$ and $g^2$, $h^2$, respectively), \hyperref[rho2'-reducible]{Lemma~\ref*{rho2'-reducible}} asserts that if $\rho$ is the complex representation of some hyperelliptic manifold with holonomy group $G$, then $\rho$ decomposes as
\begin{align*}
\rho = \rho_2 \oplus \chi \oplus \chi'
\end{align*}
for some irreducible degree $2$ representation $\rho_2$ and linear characters $\chi$, $\chi'$ of $G$. 

\begin{lemma} \label{32-lemma}
	The following statements hold:
	\begin{enumerate}[ref=(\theenumi)]
		\item \label{32-lemma-1} In the case of $(C_8 \times C_2) \rtimes C_2$,
		\begin{align*}
		\rho_2(b) = \diag(-1, \ -1).
		\end{align*}
		In particular, $a^2 b \in \ker(\rho_2)$ or $a^4 b \in \ker(\rho_2)$.
		\item \label{32-lemma-2} In the case of $G(4,8,3)$,
		\begin{align*}
		\rho_2(g^2) = \diag(-1, \ -1).
		\end{align*}
		In particular, $g^2 h^2 \in \ker(\rho_2)$ or $g^2 h^4 \in \ker(\rho_2)$.
	\end{enumerate}
\end{lemma}

\begin{proof}
	\ref{32-lemma-1} Since $b = c^{-1}aca^{-1}$ is a commutator, $\rho_1(b) = \rho_1'(b) = 1$ and thus $\rho_2(b) = \diag(-1, \ -1)$. The second assertion follows since $a^2$ is a central element. Assertion \ref{32-lemma-2} is proved in the same way. 
\end{proof}

\begin{prop}
	The groups
	\begin{align*}
	(C_8 \times C_2) \rtimes C_2\qquad \text{ and }\qquad G(4,8,3)
	\end{align*}
	are not hyperelliptic in dimension $4$.
\end{prop}

\begin{proof}
	We only prove that $(C_8 \times C_2) \rtimes C_2$ is not hyperelliptic in dimension $4$, since the same arguments can be used to exclude $G(4,8,3)$. Assume that $\rho \colon (C_8 \times C_2) \rtimes C_2 \to \GL(4,\CC)$ is the complex representation of some hyperelliptic fourfold. By the discussion above, $\rho= \rho_2 \oplus \chi \oplus \chi'$ for some irreducible degree $2$ representation $\rho_2$ and linear characters $\chi, \chi'$.  \\
	According to \hyperref[32-lemma]{Lemma~\ref*{32-lemma}}, there are two cases:
	\begin{itemize}
		\item If $a^2 b \in \ker(\rho_2)$, then $\rho_2(a)$ has two eigenvalues of order $4$. Since $\rho$ is faithful, at least one of $\chi(a)$ and $\chi'(a)$ must be a primitive $8$th root of unity. Now, the \hyperref[order-cyclic-groups]{Integrality Lemma~\ref*{order-cyclic-groups}} \ref{ocg-3} and the faithfulness of $\rho$ imply that both $\chi(a)$ and $\chi'(a)$ are primitive $8$th roots. In particular, $\rho(a)$ does not have the eigenvalue $1$.
		\item If $a^4 b \in \ker(\rho_2)$, then we reach a contradiction in a similar way: indeed, since $\rho$ is faithful, the element $a^4b$ of order $2$ must be mapped to $-1$ by one of the linear characters $\chi$, $\chi'$, and hence one of $\chi(a)$, $\chi'(a)$ is a primitive $8$th root of unity. It follows that $\rho(a)$ has at least three eigenvalues of order $8$, which contradicts the conclusion of the \hyperref[order-cyclic-groups]{Integrality Lemma~\ref*{order-cyclic-groups}} \ref{ocg-3}.
	\end{itemize}
\end{proof}

\subsection{$G(8,4,3)$ (ID [32,13])} \label{32-13-section}

This section is dedicated to proving that the group
\begin{align*}
G(8,4,3) = \langle g,h ~ | ~ g^8 = h^4 = 1, ~ h^{-1}gh = g^3 \rangle,
\end{align*}
which is labeled $[32,13]$ in the Database of Small Groups, is not hyperelliptic in dimension $4$. \\
As always, we assume that $\rho \colon G(8,4,3) \to \GL(4,\CC)$ is the complex representation of some hyperelliptic fourfold $X = T/G(8,4,3)$ and derive a contradiction. \\

Observe first of all that $G(8,4,3)$ is metacyclic and hence \hyperref[cor:metacyclic-rep]{Corollary~\ref*{cor:metacyclic-rep}} \ref{cor:metacyclic-rep-1} shows that $\rho$ decomposes as the direct sum of a irreducible degree $2$ representation $\rho_2$ and two linear characters $\chi$, $\chi'$, at least one of which is non-trivial. 

\begin{rem}
	Clearly, $\langle g^2,h\rangle$ is a copy of $G(4,4,3)$ inside $G(8,4,3)$. As observed in  \hyperref[G(4,4,3)-rho]{Lemma~\ref*{G(4,4,3)-rho}}, we may hence assume that 
	\begin{align*}
	\chi(g^2) = \chi(h^2) = 1 \quad  \text{ and } \quad \chi'(g^2) = 1, \ \chi'(h) = i.
	\end{align*}
	After applying the automorphism $g \mapsto gh^2$, $h \mapsto h$ of $G(8,4,3)$, we may assume that $\chi'(g) = -1$. Since $\rho_2(g)$ is a matrix of order $8$ and
	\begin{align*}
	\rho(g) = \diag(\rho_2(g), ~ \chi(g), ~ -1)
	\end{align*}
	must have the eigenvalue $1$, we may assume that $\chi(g) = 1$. \\
	
	We now use that the center of $G(8,4,3)$ is the non-cyclic group spanned by $g^4$ and $h^2$. Using again that $\ord(\rho_2(g)) = 8$, we infer that exactly one of the elements $g^4h^2$ and $h^2$ is contained in the kernel of $\rho_2$. We investigate these possibilities separately:
	\begin{enumerate}[label=(\roman*)]
		\item if $g^4h^2 \in \ker(\rho_2)$, then $\rho_2(h^2) = -I_2$. It follows that $\rho_2(h)$ does not have the eigenvalue $1$, and hence the matrix
		\begin{align*}
		\rho(h) = \diag(\rho_2(h), ~ \chi(h), ~ i)
		\end{align*} 
		has the eigenvalue $1$ only if $\chi(h) = 1$.
		\item if $h^2 \in \ker(\rho_2)$,  we observe that $(gh)^2 = g^4h^2$, and hence $\rho_2(gh) = \rho_2(g^4h^2) = -I_2$. It follows that 
		\begin{align*}
		\rho(gh) = \diag(\rho_2(gh), \ \chi(h), \ -i)
		\end{align*}
		does not have the eigenvalue $1$, unless $\chi(h) = 1$.
	\end{enumerate}
	
\end{rem}

The remark shows that the only complex representation left to exclude is
\begin{align} \label{32-13-rho}
\rho(g) = \diag(\rho_2(g), ~ 1, ~ -1) \quad \text{ and } \rho(h) = \diag(\rho_2(h), ~ 1, ~ i).
\end{align}
Denoting by $u \in \{g^4h^2, h^2\}$ the non-trivial element of $\ker(\rho_2)$, we have shown that
\begin{align} \label{32-13:u}
\rho(u) = \diag(1, \ 1, \ 1, \ -1).
\end{align}

The desired contradiction in the last case (\ref{32-13-rho}) is derived as follows. Write the action of $G(8,4,3)$ on a $4$-dimensional complex torus $T$ as follows:
\begin{align*}
&g(z) = \rho(g)z + (a_1, \ a_2, \ a_3, \ a_4), \\
&h(z) = \rho(h)z + (b_1, \ b_2, \ b_3, \ b_4).
\end{align*}
The relation $h^{-1}gh = g^3$ shows that there are $w_1,w_2,w_4$ (which we will not need to specify precisely) such that the element
\begin{align} \label{32-13:w}
w := h^{-1}gh(z) - g^3(z) = (w_1, \ w_2, \ 2a_3, \ w_4)
\end{align}
is zero in $T$. By definition of $u$ given in (\ref{32-13:u}) above, we obtain that
\begin{align} \label{32-13:w2}
(\rho(u) - \id_T)w = (0, \ 0, \ 0, \ 2w_4) = 0 \text{ in } T.
\end{align}
Combining (\ref{32-13:w}) and (\ref{32-13:w2}), we conclude that
\begin{align} \label{32-13:w3}
2w - (\rho(u)-\id_T)w = (2w_1, \ 2w_2, \ 4a_3, \ 0)
\end{align}
is zero in $T$ as well. Finally, using (\ref{32-13:w3}), we see that $g^4$ has a fixed point on $T$:
\begin{align*}
g^4(z) = &(-z_1 + w_1', \ -z_2 + w_2', \ z_3 + 4a_3, \ z_4) & \qquad \text{ (for some } w_1', w_2'\text{)} \\
\stackrel{\text{(\ref{32-13:w3})}}{=} &(-z_1 + w_1' - 2w_1, -z_2 + w_2' - 2w_2, \ z_3, \ z_4).
\end{align*}
This shows that we do not get a free action of $G(8,4,3)$ in the last case (\ref{32-13-rho}) either, and thus $G(8,4,3)$ is not hyperelliptic in dimension $4$.

\subsection{$(C_2 \times C_2) \rtimes C_9$ (ID [36,3])} \label{section-36-3-excluded}

Consider the group
\begin{align*}
&(C_2 \times C_2) \rtimes C_9 = \langle a_1,a_2,b \ | \ a_1^2 = a_2^2 = b^9 = [a_1,a_2] = 1, \ b^{-1}a_1b = a_2, \ b^{-1}a_2b = a_1a_2\rangle. %\text{ and} \\
%&A_4 = \langle \sigma, \xi \ | \ \sigma^3 = \xi^2 = (\xi \sigma)^3 = 1\rangle.
\end{align*}
This section shows that it is not hyperelliptic in dimension $4$.

Let $\rho \colon (C_2 \times C_2) \rtimes C_9 \to \GL(4,\CC)$ be a representation satisfying the usual properties
\begin{enumerate}[ref=(\theenumi)]
	\item \label{rho-prop1} $\rho$ is faithful,
	\item \label{rho-prop2} every matrix in the image of $\rho$ has the eigenvalue $1$,
	\item \label{rho-prop3} every element of order $9$ is mapped to a matrix with exactly three distinct, pairwise non-complex conjugate eigenvalues of order $9$ (and thus, the fourth eigenvalue is $1$).
\end{enumerate}

We discuss how $\rho$ decomposes into irreducible characters.

\begin{lemma}
	The representation $\rho$ is the direct sum of a faithful representation $\rho_3$ of degree $3$ and the trivial character $\chi = \chi_{\triv}$.
\end{lemma}

\begin{proof}
	The element $b$ acts on $\langle a_1,a_2\rangle$ by an automorphism of order $3$. Thus $b^3$ is central and $\langle a_1,a_2,b^3 \rangle \cong C_2 \times C_6$ is a normal Abelian subgroup of index $3$. \hyperref[thm:huppert-degree]{N. Ito's Degree Theorem~\ref*{thm:huppert-degree}} now implies that the degrees of irreducible representations of $(C_2 \times C_2) \rtimes C_9$ are $1$, $3$. The faithfulness of $\rho$ guarantees that $\rho$ is the direct sum of an irreducible degree $3$ representation $\rho_3$ and a linear character $\chi$. Observe that it suffices to prove that $\chi$ is trivial, implying that $\rho_3$ is faithful. Indeed, the derived subgroup of $(C_2 \times C_2) \rtimes C_9$ is $\langle a_1,a_2\rangle = C_2 \times C_2$, and we only have to show that $\chi(b) = 1$. As noted above, $b^3$ is central, and thus $\rho_3(b^3)$ is a scalar multiple of the identity. By properties \ref{rho-prop1} and \ref{rho-prop3} above, it follows that all eigenvalues of $\rho_3(b)$ have order $9$. Since $\rho(b)$ must have the eigenvalue $1$ by property \ref{rho-prop2}, we finally conclude that $\chi(b) = 1$, as claimed.
\end{proof}

The two faithful irreducible representations of $(C_2 \times C_2) \rtimes C_9$ are -- up to equivalence -- given as follows:
\begin{align*}
&a_1 \mapsto \begin{pmatrix}
-1 && \\ & 1 & \\ && -1
\end{pmatrix}, \quad a_2 \mapsto \begin{pmatrix}
1 && \\ & -1 & \\ && -1
\end{pmatrix}, \quad b \mapsto \begin{pmatrix}
0 & 0 & -\zeta_3 \\
-1 & 0 & 0 \\
0 & 1 & 0
\end{pmatrix}, \\
&a_1 \mapsto \begin{pmatrix}
-1 && \\ & 1 & \\ && -1
\end{pmatrix}, \quad a_2 \mapsto \begin{pmatrix}
1 && \\ & -1 & \\ && -1
\end{pmatrix}, \quad b \mapsto \begin{pmatrix}
0 & 0 & -\zeta_3^2 \\
-1 & 0 & 0 \\
0 & 1 & 0
\end{pmatrix}.
\end{align*}

A GAP computation shows that these two representations are equivalent up to automorphisms (an explicit automorphism exchanging the two equivalence classes of representations is given by $a_1 \mapsto a_1a_2$, $a_2 \mapsto a_2$, $b \mapsto b^5$). It hence suffices to only consider the first of the two listed representations. We denote it by $\rho_3$ in the sequel. \\
%f := FreeGroup("a1","a2","b");
%G := f/[f.1^2, f.2^2, Comm(f.1,f.2), f.3^9, f.3^-1*f.1*f.3*f.2^-1, f.3^-1*f.2*f.3*(f.1*f.2)^-1];
%gens := GeneratorsOfGroup(G);
%a1 := gens[1]; a2 := gens[2]; b := gens[3];
%for rep in IrreducibleRepresentations(G) do
%Print(a1, "          ",Image(rep,a1), "\n", a2, "          ",Image(rep,a2), "\n", b, "          ",Image(rep,b), "\n\n\n");
%od; 
%
%
%counteryee := 0;
%counterF := 0;
%for f in A do
%sum := 0;
%for g in G do
%chig := TraceMat(Image(rep,g));
%chifg := ComplexConjugate(TraceMat(Image(rep,Image(f,g))));
%sum := sum + chig*chifg;
%od;
%if sum = 0 then Print(Image(f,a1), "             ", Image(f,a2), "             ", Image(f,b), "\n"); counteryee := counteryee + 1; fi; if sum/36 = 1 then counterF := counterF + 1; fi; od;
%
%counteryee;
%counterF;

Let $T$ be a complex torus of dimension $4$ endowed with an action of $(C_2 \times C_2) \rtimes C_9$ whose associated complex representation is $\rho = \rho_3 \oplus \chi_{\triv}$. The results of \hyperref[isogeny]{Section~\ref*{isogeny}} show that $T$ is equivariantly isogenous to the product of a $3$-dimensional subtorus $T' \subset T$ and an elliptic curve $E' \subset T$. Here, the linear parts of the action on $T'$ and $E_4$ are via $\rho_3$ and $\chi_{\triv}$, respectively.

\begin{rem}
	The subtorus $T'$ is equivariantly isogenous to a product of elliptic curves, too. Indeed, if we define
	\begin{align*}
	E_1 := \ker(\rho(a_2) - \id_T)^0 \cap T', \qquad E_2 := \rho(b)(E_1), \qquad E_3 := \rho(b)(E_2) = \rho(b^2)(E_1),
	\end{align*}
	then $E := E_1 \cong E_2 \cong E_3$ are canonically isomorphic via $\rho(b)$ and $T'$ is isogenous to $E_1 \times E_2 \times E_3$. Since $b^3$ acts on $E \times E \times E$ by multiplication by $\zeta_3$, each of the three elliptic curves $E_1, E_2, E_3 \subset T$ is isomorphic to the equianharmonic elliptic curve. 
\end{rem}

We denote by $H$ the kernel of the addition map $E \times E \times E \times E' \to T$. Furthermore, after a change of origin in the three copies of $E$, we may write the action of $(C_2 \times C_2) \rtimes C_9$ on $(E \times E\times E \times E')/H \cong T$ as follows\footnote{It suffices to give the action of $a_1$ and $b$, since $b^{-1}a_1b = a_2$.}:
\begin{align*}
&a_1(z) = (z_1 + t_1, \ -z_2 + t_2, \ -z_3 + t_3, \ z_4 + t_4), \\
&b(z) = (z_2, \ z_3, \ \zeta_3 z_1, \ z_4 + b_4).
\end{align*}

We are finally in the situation to prove the main result of this section.

\begin{prop} \label{prop-(C2xC2):C9-excluded}
	The group $(C_2 \times C_2) \rtimes C_9$ is not hyperelliptic in dimension $4$.
\end{prop}

\begin{proof}
	It is convenient to view $a_1$, $a_2$ and $b$ as automorphisms of $T' := E \times E \times E \times E'$ whose linear parts map $H$ to $H$, respectively. The condition that $a_1^2 = \id_T$ can then be reformulated by requiring that the element
	\begin{align} \label{36-3-eq1}
	(2t_1, \ 0, \ 0, \ 2t_4) 
	\end{align}
	is contained in $H$. Therefore, the element
	\begin{align} \label{36-3-eq2}
	(\rho(a_2) - \id_{T'})(2t_1, \ 0, \ 0, \ 2t_1) = (4t_1, \ 0, \ 0, \ 0)
	\end{align}
	is contained in $H$, too. Similarly, since $b$ defines an automorphism of order $9$ of $T$, it follows that
	\begin{align} \label{36-3-eq3}
	(0, \ 0, \ 0, \ 9b_4) \in H.
	\end{align}
	A simple computation shows that the order of $a_1b$ in the abstract group $(C_2 \times C_2) \rtimes C_9$ is $9$ as well, and hence
	\begin{align} \label{36-3-eq4}
	(0, \ 0, \ 0, \ 9(b_4+t_4)) \in H.
	\end{align}
	Combining equations (\ref{36-3-eq1}) -- (\ref{36-3-eq4}), we obtain that
	\begin{align} \label{36-3-eq5}
	(2t_1, \ 0, \ 0, \ 0) \in H \text{ and } (0, \ 0, \ 0, t_4)\in H.
	\end{align}
	Therefore the action of $a_1$ on $T'$ is congruent to
	\begin{align*}
	a_1(z) \equiv (z_1 + t_1, \ -z_2 + t_2, \ -z_3 + t_3, \ z_4).
	\end{align*}
	Our goal is to show that there are $w_2, w_3 \in E$ such that $	(t_1, \ w_2, \ w_3, \ 0) \in H$: then the action of $a_1$ is congruent to
	\begin{align*}
	a_1(z) \equiv (z_1, \ -z_2 + t_3 - w_2, \ -z_3 + t_3 - w_3, \ z_4)
	\end{align*}
	and thus $a_1$ has a fixed point on $T'$ and hence also on $T$.\\ 
	We may take the identity element of $H$ if $t_1 = 0$, so let us assume that $t_1 \neq 0$. By construction, the elliptic curve $E = E_1$ embeds into $T$, and so equation (\ref{36-3-eq5}) implies that $\ord_E(t_1) = 2$.
	Now, the element $b^3$ commutes with $a_1$, and thus $b^3a_1 - a_1 b_3$ defines an element of $H$. It is given by
	\begin{align*}
	v := ((\zeta_3-1)t_1, \ (\zeta_3-1)t_2, (\zeta_3-1)t_3, \ 0).
	\end{align*}
	Thus $\rho(b^3)v \in H$ as well. Its first coordinate is equal to $\zeta_3(\zeta_3-1)t_1$. We use \hyperref[fixed-point-prime-power]{Lemma~\ref*{fixed-point-prime-power}} and the observation that $\ord_E(t_1) = 2$ to deduce that neither $t_1$ nor $(\zeta_3-1)t_1$ are fixed by $\zeta_3$. This implies that $E[2]$ is spanned by $(\zeta_3-1)t_1$ and $\zeta_3(\zeta_3-1)t_1$. Hence an appropriate linear combination of $v$ and $\rho(b^3)v$ is of the desired form $(t_1, \ w_2, \ w_3, \ 0)$.
\end{proof}

\subsection{$A_4 \times C_3$ (ID [36,11])} \label{A4xC3-excluded} 

We prove in this section that the group
\begin{align*}
A_4 \times C_3 = \langle \sigma, \xi, \kappa ~ | ~ \sigma^3 = \xi^2 = (\sigma \xi)^3 = \kappa^3 = 1, ~ \kappa \text{ central}\rangle
\end{align*}
is not hyperelliptic in dimension $4$. As usual, we let $\rho \colon A_4 \times C_3 \to \GL(4,\CC)$ be a faithful representation such that every matrix in the image of $\rho$ has the eigenvalue $1$. The properties of $\rho$ easily imply that

\begin{lemma}
	The representation $\rho$ splits as the direct sum of an irreducible degree $3$ representation $\rho_3$ and a character $\chi$. with the property that $A_4 \subset \ker(\chi)$.
\end{lemma}

We now consider a complex torus $T$ of dimension $4$ endowed with an action of $A_4 \times C_3$ whose associated complex representation is $\rho$.

\begin{lemma}
	If $\chi$ is non-trivial, then there is an element of $A_4 \times C_3$ that does not act freely on $T$.
\end{lemma}

\begin{proof}
	\hyperref[A4-rep]{Corollary~\ref*{A4-rep}} of \hyperref[48-31-section]{Section~\ref*{48-31-section}} shows that the degree $1$ summand of the complex representation of a hyperelliptic fourfold with holonomy group $A_4$ is necessarily trivial. Hence the statement of the lemma follows by applying the corollary to the two different copies $\langle \sigma,\xi\rangle$ and $\langle \sigma \kappa, \xi \rangle$ of $A_4$ in $A_4 \times C_3$. 
\end{proof}

Consequently, it remains to exclude the case $\rho = \rho_3 \oplus \chi_{\triv}$, where we may assume $\rho_3$ to be given as follows:
\begin{align*}
\rho_3(\sigma) = \begin{pmatrix}
0 & 0 & 1  \\ 1 & 0 & 0  \\ 0 & 1 & 0 
\end{pmatrix}, \qquad \rho_3(\xi) = \begin{pmatrix}
1 &&  \\ & -1 &  \\ && -1  \\
\end{pmatrix}, \qquad \rho_3(\kappa) = \begin{pmatrix}
\zeta_3 && \\ & \zeta_3 & \\ && \zeta_3
\end{pmatrix}.
\end{align*}

Excluding this possibility is very similar to excluding $(C_2 \times C_2) \rtimes C_9$ in the previous  \hyperref[section-36-3-excluded]{Section~\ref*{section-36-3-excluded}}. Indeed, it carries over verbatim that $T$ is equivariantly isogenous to $E_1 \times E_2 \times E_3 \times E'$, where $E_1, E_2, E_3, E' \subset T$ are elliptic curves and $E := E_1 \cong E_2 \cong E_3$ are canonically isomorphic via $\rho(\sigma)$. We denote again by $H$ the kernel of the addition map $E \times E \times E \times E' \to T$. By changing the origin in the three copies of $E$, we may write the action of $A_4 \times C_3$ on $(E \times E \times E \times E')/H \cong T$ as follows:
\begin{align*}
&\sigma(z) = (z_3 + s_1, \ z_1 + s_2, \ z_2 + s_3, \ z_4 + s_4), \\
&\xi(z) = (z_1 + t_1, \ -z_2 + t_2, \ -z_3 + t_3, \ z_4 + t_4), \\
&\kappa(z) = (\zeta_3 z_1, \ \zeta_3 z_2, \ \zeta_3 z_3, \ z_4 + k_4).
\end{align*}
This discussion again allows us to prove

\begin{prop} \label{A4xC3-excluded-prop}
	Let the action of $A_4 \times C_3$ on $T$ be as above. Then $\xi$ has a fixed point on $T$. In particular, the group $A_4 \times C_3$ is not hyperelliptic in dimension $4$.
\end{prop}

\begin{proof}
	Write $T' := E \times E \times E \times E'$. The relation $\xi^2 = \id_T$ shows that
	\begin{align}
	(2t_1, \ 0, \ 0, \ 2t_4) \in H.
	\end{align}
	It follows that the element
	\begin{align}
	\rho(\sigma^{-1}\xi\sigma)-\id_{T'})(2t_1, \ 0, \ 0, \ 2t_4) = (4t_1, \ 0, \ 0, \ 0)
	\end{align}
	is contained in $H$, too. Since $\sigma^3 = (\sigma \xi)^3 = \id_T$, we also obtain that $H$ contains $(s_1+s_2+s_3, \ s_1+s_2+s_3, \ s_1+s_2+s_3, \ 3s_4)$ and an element of the form $(u_1, \ u_2, \ u_3, \ 3(s_4+t_4)$. In total, we obtain that the action of $\xi$ on $T'$ is congruent to an action of the form
	\begin{align*}
	\xi(z) \equiv (z_1 + t_1', \ -z_2 + t_2', \ -z_3 + t_3', \ z_4).
	\end{align*}
	Using $\kappa$ instead of $b^3$, we may then conclude as in the proof of \hyperref[prop-(C2xC2):C9-excluded]{Proposition~\ref*{prop-(C2xC2):C9-excluded}}. We leave the details to the reader.
\end{proof}

\subsection{$(C_4 \times C_4) \rtimes C_3$ (ID [48,3])} \label{section-48-3-excluded}

Consider the following action of $C_3 = \langle h \rangle$ on $C_4 \times C_4 = \langle g_1, g_2 \rangle$:
\begin{align*}
h^{-1}g_1h = g_2, \qquad h^{-1}g_2h = (g_1g_2)^{-1}.
\end{align*}
The resulting semidirect product $(C_4 \times C_4) \rtimes C_3$ is the group with ID [48,3]. We will prove in this section that this group is not hyperelliptic in dimension $4$.\\

\hyperref[thm:huppert-degree]{N. Ito's Degree Theorem~\ref*{thm:huppert-degree}} implies that the degrees of irreducible representations of $(C_4 \times C_4) \rtimes C_3$ are $1$ or $3$. Thus a faithful representation $\rho \colon (C_4 \times C_4) \rtimes C_3 \to \GL(4,\CC)$ is equivalent to the direct sum of an irreducible representation $\rho_3$ of degree $3$ and a linear character $\chi$. 

\begin{lemma}\label{48-3-derived}
	The elements $g_1$ and $g_2$ span the derived subgroup of $(C_4 \times C_4) \rtimes C_3$. In particular, $\chi(g_1) = \chi(g_2) = 1$.
\end{lemma}

\begin{proof}
	The defining relations imply that $g_1^{-1}g_2$ and $g_1^{-1}g_2^2$ are commutators. The assertion follows immediately from this observation, since the quotient of $(C_4 \times C_4) \rtimes C_3$ by $C_4 \times C_4$ is Abelian.
\end{proof}

\begin{rem}
	The subgroup $\langle g_1^2, g_2^2, h\rangle$ is isomorphic to $A_4$, and the results of  \hyperref[48-31-section]{Section~\ref*{48-31-section}} show that the degree $1$ summand of the complex representation of a hyperelliptic fourfold with holonomy group $A_4$ maps the $3$-cycle $h$ to $1$.
\end{rem}

It thus suffices to prove that if $\chi$ is the trivial character, then $\rho$ cannot be the complex representation of a hyperelliptic manifold. This is accounted for by the upcoming proposition, which completes the proof of the non-hyperellipticity of $(C_4 \times C_4) \rtimes C_3$ in dimension $4$.

\begin{prop} \label{prop-48-3-excluded}
	Suppose that $\rho = \rho_3 \oplus \chi_{\triv}$, where $\rho_3 \colon (C_4 \times C_4) \rtimes C_3 \to \GL(3,\CC)$ is a faithful representation. Then the following statements hold:
	\begin{enumerate}[ref=(\theenumi)]
		\item \label{48-3-1} The matrix $\rho_3(g_1g_2^2)$ does not have the eigenvalue $1$.
		\item \label{48-3-2} The representation $\rho$ does not occur as the complex representation of a hyperelliptic fourfold.
	\end{enumerate}
\end{prop}

\begin{proof}
	\ref{48-3-1} The representation $\rho_3$ being faithful, we observe that the restriction of $\rho_3$ to $A_4 \cong \langle g_1^2, g_2^2, h \rangle$ is equivalent to the unique faithful representation
	\begin{align*}
	g_1^2 \mapsto \begin{pmatrix}
	1 & & \\ & -1 & \\ && -1
	\end{pmatrix}, \qquad 	g_2^2 \mapsto \begin{pmatrix}
	-1 & & \\ & -1 & \\ && 1
	\end{pmatrix}, \qquad h \mapsto \begin{pmatrix}
	0 & 0 & 1 \\ 1 & 0 & 0 \\ 0 & 1 & 0
	\end{pmatrix}
	\end{align*}
	of $A_4$. The assertion follows directly from this. \\
	
	\ref{48-3-2} Assume that $T$ is a $4$-dimensional complex torus endowed with an action of $(C_4 \times C_4) \rtimes C_3$ whose associated complex representation is $\rho$. In order words, the action of $g_1g_2^2$ and $h$ on $T$ is of the following form:
	\begin{align*}
	&g_1g_2^2(z) = \rho(g)z + (a_1, \ a_2, \ a_3, \ a_4), \\
	&h(z) = \rho(h)z + (b_1, \ b_2, \ b_3, \ b_4).
	\end{align*}
	It now follows from the relations $(g_1g_2^2)^4 = h^3 = \id_T$ and part \ref{48-3-1} that the following two elements are equal to zero in $T$, respectively:
	\begin{align*}
	(0, \ 0, \ 0, \ 4a_4), \qquad (b_1+b_2+b_3, \ b_1+b_2+b_3, \ b_1+b_2+b_3, \ 3b_4).
	\end{align*}
	Similarly, one checks that $g_1g_2^2h$ has order $3$, and hence an element of the following form is zero in $T$:
	\begin{align*}
	(w_1, \ w_2, \ w_3, \ 3a_3 + 3b_3) \qquad \qquad \text{(for some } w_1,w_2,w_3\text{)}.
	\end{align*}
	Taking differences, we obtain an element that is zero in $T$ and whose last coordinate is $a_4$. The action of $g_1g_2^2$ on $T$ is consequently congruent to an action of the form
	\begin{align*}
	g_1g_2^2(z) = \rho(g)z + (a_1', \ a_2', \ a_3', \ 0).
	\end{align*}
	It now follows from part \ref{48-3-1} that $g_1g_2^2$ cannot act freely on $T$.
\end{proof}

%\subsection{$G(24,2,5)$ (ID [48,5])} \label{section-48-5-excluded}

%%%%%%%%%%%%%%%%%%%%%%%%%%%%%%%%%%%%%%%%%%%%%%%%%%%%%%%%%%%%%%%%%%%%%%%%%%% relation implies that g^3 and g^-3 are conjugate
%{\ done}
%We prove that the metacyclic group
%\begin{align*}
%G(24,2,5) = \langle g,h \ | \ g^{24} = h^2 = 1, \ h^{-1}gh = g^5 \rangle
%\end{align*}
%is not hyperelliptic in dimension $4$. \\
%
%As usual, we argue by contradiction. Assume that $\rho \colon G(24,2,5) \to \GL(4,\CC)$ is the complex representation of some hyperelliptic fourfold. According to \hyperref[lemma-two-generators]{Proposition~\ref*{lemma-two-generators}} \ref{ltg-2}, $\rho$ decomposes as the direct sum of two degree $1$ summands and an irreducible degree $2$ summand, the latter of which we denote by $\rho_2$. \\
%The relation $h^{-1}gh = g^5$ implies that the element $g^4$ of order $6$ is a commutator and hence $\rho_2(g)$ must have eigenvalues of order $24$ in order for $\rho$ to be faithful. This however contradicts \hyperref[order-cyclic-groups]{Lemma~\ref*{order-cyclic-groups}}.

\subsection{$A_4 \rtimes C_4$ (ID [48,30]} \label{section-48-30-excluded}

Let $\sigma, \xi \in A_4$ be a $3$-cycle and a double transposition, respectively. Let $C_4 = \langle \gamma \rangle$ act on $A_4$ by conjugation, preserving the subgroups $\langle \sigma \rangle$ and $\langle \xi \rangle$. Concretely,
\begin{align*}
\gamma^{-1}\xi\gamma = \xi, \qquad \gamma^{-1}\sigma\gamma = \sigma^2.
\end{align*}
We prove that the resulting semidirect product $A_4 \rtimes C_4$ is not hyperelliptic in dimension $4$.

\begin{lemma} \label{48-30-properties}
	The following statements hold:
	\begin{enumerate}[ref=(\theenumi)]
		\item $\gamma^2$ commutes with $\sigma$. In particular, $\gamma^2$ is a central element of $A_4 \rtimes C_4$.
		\item $(\sigma \gamma)^2 = \gamma^2$. In particular, $\sigma \gamma$ has order $4$.
	\end{enumerate}
\end{lemma}

\begin{proof}
	Both statements follow immediately from the defining relations of $A_4 \rtimes C_4$. Indeed,
	\begin{align*}
	\gamma^{-2} \sigma \gamma = \gamma^{-1}\sigma^2 \gamma = (\gamma^{-1}\sigma\gamma)^2 = \sigma 
	\end{align*}
	and hence $\gamma^2$ and $\sigma$ commute. Moreover,
	\begin{align*}
	(\sigma \gamma)^2 = (\sigma \gamma)(\sigma \gamma) = \sigma \gamma^2 (\gamma^{-1}\sigma \gamma) =  \sigma \gamma^2 \sigma^2 = \gamma^2.
	\end{align*}
\end{proof}

\begin{cor} \label{48-30-cor}
	Let $\rho_3$ be a $3$-dimensional irreducible representation of $A_4 \rtimes C_4$. If the matrix $\rho_3(\gamma)$ has order $4$, all of its eigenvalues are primitive fourth roots of unity.
\end{cor}

\begin{proof}
	By  \hyperref[48-30-properties]{Lemma~\ref*{48-30-properties}}, $\gamma^2$ is a central element of order $2$ and hence is either mapped to $I_3$ or to $-I_3$ by $\rho_3$.
\end{proof}

The following corollary proves that $A_4 \rtimes C_4$ does not have an irreducible representation of dimension $4$:

\begin{cor}
	The group $A_4 \rtimes C_4$ has a normal subgroup that is isomorphic to $C_2 \times C_2 \times C_2$. In particular, the potentially possible dimensions of irreducible representations of this group are $1$, $2$, and $3$.
\end{cor}

\begin{proof}
	By  \hyperref[48-30-properties]{Lemma~\ref*{48-30-properties}}, the subgroup $\langle g^2\rangle$ is normal in $A_4 \rtimes C_4$. Moreover, the $2$-Sylow subgroup of $A_4$ is normal and isomorphic to $C_2 \times C_2$. This proves the first statement. The second statement follows directly from  \hyperref[thm:huppert-degree]{N. Ito's Degree Theorem~\ref*{thm:huppert-degree}}.
\end{proof}

Consider now a faithful representation $\rho \colon A_4 \rtimes C_4 \to \GL(4,\CC)$. Since the restriction of $\rho$ to $A_4 = \langle \sigma, \xi \rangle$ is faithful, the results of  \hyperref[48-31-section]{Section~\ref*{48-31-section}} show that $\rho$ splits into the direct sum of two irreducible representations $\rho_3$ and $\chi$, whose degrees are $3$ and $1$, respectively. Consider now the following action of $A_4 \rtimes C_4$ on a complex torus $T$ of dimension $4$:
\begin{align*}
&\sigma(z) = \rho(\sigma)z + (s_1,\ s_2,\ s_3,\ s_4), \\
&\xi(z) = \rho(\xi) z + (t_1, \ t_2, \ t_3, \ t_4), \\
&\gamma(z) = \rho(\gamma) z + (c_1, \ c_2, \ c_3, \ c_4).
\end{align*}
Using this notation, we show that

\begin{prop} \label{prop:48-30-excluded}
	The representation $\rho$ does not occur as the complex representation of any hyperelliptic fourfold.
\end{prop}

\begin{proof}	
	Choosing an appropriate basis, we may write $\rho = \rho_3 \oplus \chi$ and that $\rho|_{A_4}$ is as in  \hyperref[48-31-prop]{Proposition~\ref*{48-31-prop}} of \hyperref[48-31-section]{Section~\ref*{48-31-section}}, which in particular means that we write
	\begin{align*}
	\sigma(z) = (z_3 + s_1,\ z_1 + s_2,\ z_2 + s_3,\ z_4+s_4).
	\end{align*}
	We prove that $\sigma$ has a fixed point on $T$. \hyperref[A4-rem-kappa]{Remark~\ref*{A4-rem-kappa}} asserts that it suffices to prove that some element of the form $(w_1, \ w_2, \ w_3, \ s_4)$ is zero in $T$. Indeed, the action of $\sigma$ on $T$ is then congruent to 
	\begin{align*}
	\sigma(z) = (z_3 + s_1',\ z_1 + s_2',\ z_2 + s_3',\ z_4)
	\end{align*}
	and thus has a fixed point on $T$ by  \hyperref[A4-lemma-fixed-point]{Lemma~\ref*{A4-lemma-fixed-point}}. \\
	The existence of such an element is shown as follows. The relation $\sigma^3 = \id_T$ proves that 
	\begin{align*}
	(s_1+s_2+s_3, \ s_1+s_2+s_3, \ s_1+s_2+s_3, \ 3s_4) = 0 \text{ in } T.
	\end{align*}
	Similarly, the relations $\gamma^4 = (\sigma \gamma)^4 = \id_T$ show that elements of the form
	\begin{align*}
	(u_1, \ u_2, \ u_3, \ 4c_4) \text{ and } (v_1, \ v_2, \ v_3, \ 4c_4 + 4s_4)
	\end{align*}
	are zero in $T$. Taking differences, we obtain an element of the desired form.
\end{proof}

The above discussion completes the proof of the result that $A_4 \rtimes C_4$ is not hyperelliptic in dimension $4$.

\subsection{$(C_3 \times C_3) \rtimes C_{2^n}$, $n \in \{1,2,3\}$} \label{section-metabelian}

Finally, we show the non-hyperellipticity of the groups 
\begin{align} \label{presentation-metabelian}
(C_3 \times C_3) \rtimes C_{2^n} = \Biggl\langle 
\begin{array}{l|cl}
a_1,a_2,& \sigma^3 = a_1^3 = a_2^3 = [a_1,a_2] = b^{2^n} = 1, \\
\;\;\;b&\ b^{-1}a_1b=a_1^{-1},\ b^{-1}a_2b=a_2^{-1}  
\end{array}  \Biggr\rangle, \qquad n \in \ZZ_{>0}
\end{align}
in dimension $4$. 

\begin{rem}
	Albeit not necessary for our proof, observe that \hyperref[order-cyclic-groups]{Lemma~\ref*{order-cyclic-groups}} would allow us to restrict to the case where $n \in \{1,2,3\}$. The true content of this section hence is that the three groups labeled $[18,4]$ (if $n = 1$), $[36,7]$ (if $n = 2$) and $[72,13]$ (if $n = 3$) in the Database of Small Groups are not hyperelliptic in dimension $4$.
\end{rem}

We first collect simple properties of the groups above and then prove that they are non-hyperelliptic in dimension $4$. 

\begin{lemma}\label{lemma-meta}
	The groups $(C_3 \times C_3) \rtimes C_{2^n}$ whose presentation is given in (\ref{presentation-metabelian}) have the following properties:
	\begin{enumerate}[ref=(\theenumi)]
		\item \label{lemma-meta-1} they contain an Abelian normal subgroup of index $2$,
		\item \label{lemma-meta-2} its character degrees are $1$ and $2$,
		\item \label{lemma-meta-3} its derived subgroup is $C_3 \times C_3$,
		\item \label{lemma-meta-4} if $\rho_2$ is an irreducible degree $2$ representation of $(C_3 \times C_3) \rtimes C_{2^n}$, then $\ker(\rho_2)$ contains an element of order $3$.
	\end{enumerate}
\end{lemma}

\begin{proof}
	\ref{lemma-meta-1} Since $b$ acts on $C_3 \times C_3$ by an automorphism of order $2$, we infer that $b^2$ is central in $(C_3 \times C_3) \rtimes C_{2^n}$, hence $\langle a_1,a_2,b^2\rangle$ is an Abelian subgroup of index $2$. \\
	\ref{lemma-meta-2} This follows from the previous part and \hyperref[thm:huppert-degree]{N. Ito's Degree Theorem~\ref*{thm:huppert-degree}}. \\
	\ref{lemma-meta-3} The presentation of the group immediately implies that $a_1,a_2$ are commutators. \\
	\ref{lemma-meta-4} Since $C_3 \times C_3 = \langle a_1, a_2\rangle$ is Abelian, we may assume that both $\rho_2(a_1)$ and $\rho_2(a_2)$ are diagonal. It follows from the relations $b^{-1}a_jb = a_j^{-1}$ that $\rho_2(a_j)$ and $\rho_2(a_j)^{-1}$ have the same eigenvalues, respectively. Thus
	\begin{align*}
	\rho_2(a_j) \in \{\diag(1, 1), \ \diag(\zeta_3, \zeta_3^2), \ \diag(\zeta_3^2, \zeta_3)\}.
	\end{align*}
	The assertion thus follows from the pigeonhole principle.
\end{proof}

\begin{prop} \label{18-4+36-7+72-13-excluded}
	The groups $(C_3 \times C_3) \rtimes C_{2^n}$ whose presentations are given in (\ref{presentation-metabelian}) are not hyperelliptic in dimension $4$.
\end{prop}

\begin{proof}
	Let $\rho \colon (C_3 \times C_3) \rtimes C_{2^n} \to \GL(4,\CC)$ be a faithful representation. From \hyperref[lemma-meta]{Lemma~\ref*{lemma-meta}} \ref{lemma-meta-3} and \ref{lemma-meta-4} we infer that $\rho$ is not faithful unless it is the direct sum of two irreducible representations $\rho_2$, $\rho_2'$ of degree $2$. In this case, \hyperref[lemma-meta]{Lemma~\ref*{lemma-meta}} \ref{lemma-meta-4} allows us to find non-trivial elements $a,a'$ of order $3$ such that $a \in \ker(\rho_2)$, $a' \in \ker(\rho_2)$. But then $\rho(aa')$ does not have the eigenvalue $1$.
\end{proof}

\section{The Classification Algorithm} \label{section:running-algo}

%[ [24,7], [24,14], [36,5], [48,9], [48,24], [48,26], [48,45], [48,47], [54,8], [54,10], [72,14], [72,37], [72,38], [96,46], [96,47], [96,164], [108,32], [144,101], [144,102], [144,103], [216,150], [216,177] ];

In the previous sections, \hyperref[section:examples]{Section~\ref*{section:examples}} and \hyperref[non-examples]{Section~\ref*{non-examples}}, plenty of examples and non-examples of hyperelliptic groups of order $2^a \cdot 3^b$ in dimension $4$ were given. Here, we perform the actual classification, which is carried out using the computer algebra system GAP \cite{GAP}. \\

The main input for our classification algorithm consists of three lists:
\begin{enumerate}[ref=(\theenumi)]
	\item The list \textsf{2Sylows} given in \hyperref[table:2sylows]{Table~\ref*{table:2sylows}} containing all the possible $2$-Sylow subgroups of a  hyperelliptic group in dimension $4$.  We explain why this list is complete:
	\begin{itemize}
		\item If $G$ is an Abelian hyperelliptic $2$-group in dimension $4$, then $G$ is contained in the table, see  \hyperref[abelian-2-grp]{Corollary~\ref*{abelian-2-grp}}.
		\item If $G$ is non-Abelian, then the discussion at the end of \hyperref[section:2groups-summary]{Section~\ref*{section:2groups-summary}} implies that $G$ is contained in the table or is one of the groups $[16,12]$, $[32,5]$, $[32,9]$, $[32,12]$, $[32,13]$, $[32,25]$, $[64,20]$, $[64,85]$ as a subgroup. We give references to the places where the latter groups were excluded:
		\begin{itemize}
			\item $[16,12]$ is the ID of $Q_8 \times C_2$, which was excluded in \hyperref[metacyclic-c2-excluded]{Proposition~\ref*{metacyclic-c2-excluded}},
			\item the groups with IDs $[32,5]$ and $[32,12]$ were shown to not occur in  \hyperref[32-5+32-12-excluded]{Section~\ref*{32-5+32-12-excluded}},
			\item the group $[32,9]$ was shown to not occur at the end of  \hyperref[16-11-section]{Section~\ref*{16-11-section}},
			\item $[32,13]$ was shown to be non-hyperelliptic in dimension $4$ in  \hyperref[32-13-section]{Section~\ref*{32-13-section}},
			\item the group $D_4 \times C_4$ (ID $[32,25]$) was excluded in  \hyperref[D4xCd-excluded-section]{Section~\ref*{D4xCd-excluded-section}},
			\item finally, $[64,20]$ and $[64,85]$ were both excluded in \hyperref[prop-64-20-64-85-excluded]{Proposition~\ref*{prop-64-20-64-85-excluded}} of \hyperref[32-37-section]{Section~\ref*{32-37-section}}.
		\end{itemize}
	\end{itemize}
	Hence \hyperref[table:2sylows]{Table~\ref*{table:2sylows}} contains exactly those groups, which occur as $2$-Sylow subgroups of hyperelliptic groups in dimension $4$.
	\begin{center}
		\begin{table}
			\begin{tabular}{lll}
				ID & Group & Occurs, because... \\ \hline \hline 
				$[1,1]$ & $\{1\}$ & Clear \\
				$[2,1]$ & $C_2$ & Subgroup of all other non-trivial groups in this list\\
				$[4,1]$ & $C_4$ & Subgroup of $C_2 \times C_4 \times C_4$ \\
				$[4,2]$ & $C_2 \times C_2$ & Subgroup of $C_2 \times C_4 \times C_4$\\
				$[8,1]$ & $C_8$ & Subgroup of $C_4 \times C_8$ \\
				$[8,2]$ & $C_2 \times C_4$ & Subgroup of $C_2 \times C_4 \times C_4$ \\
				$[8,3]$ & $D_4$ & Subgroup of $D_4 \times C_2$\\
				$[8,4]$ & $Q_8$ & Subgroup of $Q_8 \times C_3$, cf.  \hyperref[24-11-section]{Section~\ref*{24-11-section}} \\
				$[8,5]$ & $C_2 \times C_2 \times C_2$ & Subgroup of $C_2 \times C_4 \times C_4$ \\
				$[16,2]$ & $C_4 \times C_4$ & Subgroup of $C_2 \times C_4 \times C_4$ \\
				$[16,3]$ & $(C_2 \times C_4) \rtimes C_2$ & Subgroup of $(C_4 \times C_4) \rtimes C_2$ (ID $[32,24]$)  \\
				$[16,4]$ & $G(4,4,3)$ & Subgroup of $G(4,4,3) \times C_3$, cf.  \hyperref[48-22-section]{Section~\ref*{48-22-section}}\\
				$[16,5]$ & $C_2 \times C_8$ & Subgroup of $C_4 \times C_8$ \\
				$[16,6]$ & $G(8,2,5)$ & Subgroup of $G(8,2,5) \times C_2$ \\
				$[16,8]$ & $G(8,2,3)$ &  \hyperref[16-8-section]{Section~\ref*{16-8-section}} \\
				$[16,10]$ & $C_2 \times C_2 \times C_4$ & Subgroup of $C_2 \times C_4 \times C_4$ \\
				$[16,11]$ & $D_4 \times C_2$ &  \hyperref[16-11-section]{Section~\ref*{16-11-section}} \\
				$[16,13]$ & $D_4 \curlyvee C_4$ & Subgroup of $(C_4 \times C_4) \rtimes C_2$ (ID [32,11]) \\
				$[32,3]$ & $C_4 \times C_8$ &   \hyperref[abelian:cor]{Corollary~\ref*{abelian:cor}} \\
				$[32,4]$ & $G(8,4,5)$ &  \hyperref[32-4-section]{Section~\ref*{32-4-section}} \\
				$[32,11]$ & $(C_4 \times C_4) \rtimes C_2$ &  \hyperref[32-11-section]{Section~\ref*{32-11-section}} \\
				$[32,21]$ & $C_2 \times C_4 \times C_4$ & Subgroup of $C_2 \times C_4 \times C_{12}$, cf.  \hyperref[abelian:ex]{Example~\ref*{abelian:ex}}\\
				$[32,24]$ & $(C_4 \times C_4) \rtimes C_2$ &  \hyperref[32-24-section]{Section~\ref*{32-24-section}} \\
				$[32,37]$ & $G(8,2,5) \times C_2$ &  \hyperref[32-37-section]{Section~\ref*{32-37-section}}
			\end{tabular}
			\caption{The possible $2$-Sylow subgroups of a hyperelliptic group in dimension $4$} \label{table:2sylows}
		\end{table}
	\end{center}

	\item The list \textsf{3Sylows} of possible $3$-Sylows given in \hyperref[table:3sylows]{Table~\ref*{table:3sylows}}. We briefly explain why the list is in fact complete: according to  \hyperref[abelian-case-3-sylow]{Lemma~\ref*{abelian-case-3-sylow}}, the group $C_3 \times C_9$ is not hyperelliptic in dimension $4$. Furthermore, the group $G(9,3,4)$ and no groups of order $3^b$, $b \geq 4$ are hyperelliptic in dimension $4$, see  \hyperref[no-m27]{Proposition~\ref*{no-m27}} and \hyperref[3b-b-at-least-4]{Section~\ref*{3b-b-at-least-4}}. 
	%Finally, $\Heis(3)$ was shown to occur in \hyperref[27-3-section]{Section~\ref*{27-3-section}}, whereas $C_3 \times C_3 \times C_3$ occurs according to \hyperref[abelian:ex]{Example~\ref*{abelian:ex}}.
	\begin{center}
		\begin{table}
			\begin{tabular}{lll}
				ID & Group & Occurs, because... \\ \hline \hline 
				$[1,1]$ & $\{1\}$ & Clear \\
				$[3,1]$ & $C_3$ & Subgroup of all other non-trivial groups in this list \\
				$[9,1]$ & $C_9$ & Subgroup of $C_{18}$, see \hyperref[abelian:cor]{Corollary~\ref*{abelian:cor}} \\
				$[9,2]$ & $C_3 \times C_3$ & Subgroup of $C_3 \times C_3 \times C_3$ \\
				$[27,3]$ & $\Heis(3)$ &  \hyperref[27-3-section]{Section~\ref*{27-3-section}} \\
				$[27,5]$ & $C_3 \times C_3 \times C_3$ & Subgroup of $C_3 \times C_6 \times C_6$, cf.  \hyperref[abelian:ex]{Example~\ref*{abelian:ex}}
			\end{tabular}
			\caption{The possible $3$-Sylow subgroups of a hyperelliptic group in dimension $4$} \label{table:3sylows}
		\end{table}
	\end{center}

	\item The list \textsf{ForbiddenIDs} containing the groups listed in \hyperref[table:forbidden]{Table~\ref*{table:forbidden}}, which were seen to be non-hyperelliptic in dimension $4$. References are given in the table.
	\begin{center}
		\begin{table}
			\begin{tabular}{lll}
				ID & Group & Where is it excluded? \\ \hline \hline 
				$[18,4]$ & $(C_3 \times C_3) \rtimes C_2$ &  \hyperref[section-metabelian]{Section~\ref*{section-metabelian}}\\
				$[24,7]$ & $G(3,4,2) \times C_2$ &  \hyperref[metacyclic-c2-excluded]{Proposition~\ref*{metacyclic-c2-excluded}}\\ 
				$[24,12]$ & $S_4$ &  \hyperref[S4-excluded-section]{Section~\ref*{S4-excluded-section}}\\
				$[24,14]$ & $S_3 \times C_2 \times C_2$ & \hyperref[s3xc2xc2-excluded]{Proposition~\ref*{s3xc2xc2-excluded}} in \hyperref[72-27-and-108-42-section]{Section~\ref*{72-27-and-108-42-section}} \\
				$[36,3]$ & $(C_2 \times C_2) \rtimes C_9$ &  \hyperref[section-36-3-excluded]{Section~\ref*{section-36-3-excluded}} \\
				$[36,5]$ & $C_2 \times C_{18}$ &  \hyperref[cor:cd1xcd2-excluded]{Corollary~\ref*{cor:cd1xcd2-excluded}} \\
				$[36,7]$ & $(C_3 \times C_3) \rtimes C_4$ & \hyperref[section-metabelian]{Section~\ref*{section-metabelian}} \\
				$[36,11]$ & $A_4 \times C_3$ & \hyperref[A4xC3-excluded]{Section~\ref*{A4xC3-excluded}}\\
				$[48,3]$ & $(C_4 \times C_4) \rtimes C_3$ &  \hyperref[section-48-3-excluded]{Section~\ref*{section-48-3-excluded}}\\
				%$[48,5]$ & $G(24,2,5)$ &  \hyperref[section-48-5-excluded]{Section~\ref*{section-48-5-excluded}} \\
				$[48,9]$ & $G(3,8,2) \times C_2$ & \hyperref[metacyclic-c2-excluded]{Proposition~\ref*{metacyclic-c2-excluded}} \\
				$[48,24]$ & $G(8,2,5) \times C_3$ & \hyperref[metacyclic-2-group-times-c3-excluded]{Corollary~\ref*{metacyclic-2-group-times-c3-excluded}} \\
				$[48,26]$ & $G(8,2,3) \times C_3$ & \hyperref[metacyclic-2-group-times-c3-excluded]{Corollary~\ref*{metacyclic-2-group-times-c3-excluded}} \\
				$[48,30]$ & $A_4 \rtimes C_4$ &  \hyperref[section-48-30-excluded]{Section~\ref*{section-48-30-excluded}}\\
				$[48,45]$ & $D_4 \times C_6$ &  \hyperref[prop:d4xc6-excluded]{Proposition~\ref*{prop:d4xc6-excluded}} in \hyperref[16-11-section]{Section~\ref*{16-11-section}}\\
				$[48,47]$ & $(D_4 \curlyvee C_4) \times C_3$ & \hyperref[metacyclic-2-group-times-c3-excluded]{Corollary~\ref*{metacyclic-2-group-times-c3-excluded}}\\
				$[54,8]$ & $\Heis(3) \rtimes C_2$ & \hyperref[prop:cont-heis3-excluded]{Proposition~\ref*{prop:cont-heis3-excluded}} in \hyperref[27-3-section]{Section~\ref*{27-3-section}}\\
				$[54,10]$ & $\Heis(3) \times C_2$ &  \hyperref[prop:cont-heis3-excluded]{Proposition~\ref*{prop:cont-heis3-excluded}} in \hyperref[27-3-section]{Section~\ref*{27-3-section}}\\
				$[72,13]$ & $(C_3 \times C_3) \rtimes C_8$ & \hyperref[section-metabelian]{Section~\ref*{section-metabelian}} \\
				$[72,14]$ & $C_3 \times C_{24}$ & \hyperref[cor:cd1xcd2-excluded]{Corollary~\ref*{cor:cd1xcd2-excluded}}\\
				$[72,37]$ & $D_4 \times C_3 \times C_3$ & \hyperref[metacyclic-2-group-times-c3-excluded]{Corollary~\ref*{metacyclic-2-group-times-c3-excluded}} \\
				$[72,38]$ & $Q_8 \times C_3 \times C_3$ & \hyperref[metacyclic-2-group-times-c3-excluded]{Corollary~\ref*{metacyclic-2-group-times-c3-excluded}} \\
				$[96,46]$ & $C_4 \times C_{24}$ &  \hyperref[cor:cd1xcd2-excluded]{Corollary~\ref*{cor:cd1xcd2-excluded}}\\
				$[96,47]$ & $G(8,4,5) \times C_3$ & \hyperref[metacyclic-2-group-times-c3-excluded]{Corollary~\ref*{metacyclic-2-group-times-c3-excluded}} \\
				$[96,164]$ & $(C_4^2 \rtimes C_2) \times C_3$ &  \hyperref[32-24-supgroup-excluded]{Lemma~\ref*{32-24-supgroup-excluded}}  \\
				$[108,35]$ & $C_3 \times C_3 \times C_{12}$ &  \hyperref[c3xc3xc12-excluded]{Lemma~\ref*{c3xc3xc12-excluded}}\\
				$[144,101]$ & $C_{12} \times C_{12}$ &  \hyperref[cor:2-factors-excluded]{Corollary~\ref*{cor:2-factors-excluded}}\\
				$[144,102]$ & $((C_4 \times C_2) \rtimes C_2) \times C_3 \times C_3$ &  \hyperref[144-102-excluded]{Proposition~\ref*{144-102-excluded}} \\
				$[144,103]$ & $G(4,4,3) \times C_3 \times C_3$ & \hyperref[metacyclic-2-group-times-c3-excluded]{Corollary~\ref*{metacyclic-2-group-times-c3-excluded}} \\
				$[216,177]$ & $C_6 \times C_6 \times C_6$ &  \hyperref[cor:c6^3-c4^3-excluded]{Corollary~\ref*{cor:c6^3-c4^3-excluded}}
			\end{tabular}
			\caption{The list \textsf{ForbiddenIDs} contains the following groups that were shown \emph{not} to be hyperelliptic in dimension $4$} \label{table:forbidden}
		\end{table}
	\end{center}
\end{enumerate}

The above discussion shows that if a group $G$ of order $2^a \cdot 3^b$ is hyperelliptic in dimension $4$, then $a \leq 5$, $b \leq 3$ and $(a,b) \neq (0,0)$. This allows our computer program to run through all such groups $G$ and perform the following checks:

\begin{enumerate}[ref=(\theenumi)]
	\item There are two necessary conditions for which we must run through all (conjugacy classes of) subgroups of $G$. These are checked by the function \textsf{SubgroupConds()}: the function performs the following checks for a system of representatives $\{U_i\}$ of subgroups of $G$ up to conjugacy:
	\begin{itemize}
		\item \textsf{ForbiddenSubgroup()} checks if the ID of some $U_i$ is contained in the list \textsf{ForbiddenIDs}. If it is, then $U_i$ is not hyperelliptic in dimension $4$, so $G$ is not either.
		\item The function \textsf{CorrectCenter()} returns true if $U_i$ is Abelian and if $U_i$ is non-Abelian, then it checks whether its center is -- up to isomorphism -- a subgroup of one of the groups
		\begin{align*}
		C_2 \times C_{12}, \qquad C_4 \times C_4, \quad \text{ or } \quad C_6 \times C_6.
		\end{align*}
		If $Z(U_i)$ is not a subgroup of one of these groups, then $U_i$ and hence $G$ are not hyperelliptic in dimension $4$, see \hyperref[thm:center]{Theorem~\ref*{thm:center}}.
	\end{itemize}
	\item In addition to the above, there are two necessary conditions concerning the elements of $G$. These are summarized in the function \textsf{OrderConds()}:
	\begin{itemize}
		\item \textsf{CorrectOrder()} checks whether a group element $g \in G$ has order
		$$1,\quad 2,\quad 3,\quad 4,\quad 6,\quad 8,\quad 9,\quad 12,\quad 18,\quad 24.$$
		If this is not the case, then $G$ is not hyperelliptic in dimension $4$, see  \hyperref[order-cyclic-groups]{Lemma~\ref*{order-cyclic-groups}}.
		\item The function \textsf{Conjugate()} takes as input the group $G$ and an element $g \in G$. It returns \textsf{true} if $\ord(g) \notin \{8,9,12,24\}$, and if  $\ord(g) \in \{8,9,12,24\}$, it then checks if $g$ and $g^{-1}$ are conjugate. If this is the case, we return \textsf{false}, since then $G$ is not hyperelliptic in dimension $4$, see  \hyperref[prop:conjugate]{Proposition~\ref*{prop:conjugate}}.
	\end{itemize}
	\item Furthermore, are two remaining necessary properties we check:
	\begin{itemize}
		\item The function \textsf{CorrectSylows()} checks whether the $2$- and $3$-Sylow subgroups of $G$ are contained in the lists \textsf{2Sylows} and \textsf{3Sylows}, respectively.
		
		\item The function \textsf{MetacyclicContainedInDerived()} checks whether the commutator subgroup $[G,G]$ contains a non-Abelian metacyclic subgroup. If this is the case, then $G$ is not hyperelliptic in dimension $4$, see  \hyperref[cor-metacyclic-derived]{Corollary~\ref*{cor-metacyclic-derived}}.
	\end{itemize}
\end{enumerate}

The $69$ groups listed in \hyperref[table:2a3b]{Table~\ref*{table:2a3b}} are exactly those groups, which survived all of the above checks. We have seen throughout the text that all of these groups are hyperelliptic in dimension $4$ -- references are given in the table.

\begin{rem}
	There are usually many different ways of showing that a group contained in \hyperref[table:2a3b]{Table~\ref*{table:2a3b}} is hyperelliptic in dimension $4$, especially if its order is small. Therefore, we tried to keep the number of cross-references as low as possible by only giving references if the group is maximal among the groups in \hyperref[table:2a3b]{Table~\ref*{table:2a3b}}.
\end{rem}

\begin{center}
	\begin{longtable}{l|l|l}
		ID & Group & Reference/Reason for hyperellipticity in dimension $4$ \\ \hline \hline 
		$[2,1]$ & $C_2$ & Subgroup of $C_4 \times C_8$ (ID $[32,3]$) \\
		$[3,1]$ & $C_3$ & Subgroup of $C_3 \times C_3 \times C_3$ (ID $[27,5]$) \\
		$[4,1]$ & $C_4$ & Subgroup of $C_4 \times C_8$ (ID $[32,3]$)  \\
		$[4,2]$ & $C_2 \times C_2$ & Subgroup of $C_4 \times C_8$ (ID $[32,3]$) \\
		$[6,1]$ & $S_3$ & Subgroup of $S_3 \times C_{12}$ (ID $[72,27]$)\\
		$[6,2]$ & $C_6$ & Subgroup of $C_{18}$ (ID $[18,2]$) \\
		$[8,1]$ & $C_8$ & Subgroup of $C_4 \times C_8$ (ID $[32,3]$)  \\
		$[8,2]$ & $C_2 \times C_4$ & Subgroup of $C_4 \times C_8$ (ID $[32,3]$)  \\
		$[8,3]$ & $D_4$ & Subgroup of $D_4 \times C_2$ (ID $[16,11]$) \\
		$[8,4]$ & $Q_8$ &Subgroup of $Q_8 \times C_3$ (ID $[24,11]$)\\
		$[8,5]$ & $C_2 \times C_2 \times C_2$ & Subgroup of $C_2 \times C_6 \times C_{12}$ (ID $[144,178]$) \\
		$[9,1]$ & $C_9$ & Subgroup of $C_{18}$ (ID $[18,2]$)  \\
		$[9,2]$ & $C_3 \times C_3$ & Subgroup of $C_3 \times C_3 \times C_3$ (ID $[27,5]$)\\
		$[12,1]$ & $G(3,4,2)$ &  Subgroup of $S_3 \times C_{12}$ (ID $[72,27]$) \\
		$[12,2]$ & $C_{12}$ & Subgroup of $C_2 \times C_{24}$ (ID $[48,23]$) \\
		$[12,3]$ & $A_4$ & Subgroup of $A_4 \times C_4$ (ID $[48,30]$) \\
		$[12,4]$ & $S_3 \times C_2$ & Subgroup of $S_3 \times C_{12}$ (ID $[72,27]$) \\
		$[12,5]$ & $C_2 \times C_6$ & Subgroup of $C_2 \times C_6 \times C_{12}$ (ID $[144,178]$) \\
		$[16,2]$ & $C_4 \times C_4$ & Subgroup of $C_4 \times C_8$ (ID $[32,3]$)  \\
		$[16,3]$ & $(C_4 \times C_2)\rtimes C_2$ & Subgroup of $C_4^2 \rtimes C_2$ (ID $[32,24]$) \\
		$[16,4]$ & $G(4,4,3)$  & Subgroup of $G(4,4,3) \times C_3$ (ID $[48,22]$) \\
		$[16,5]$ &  $C_2 \times C_8$ & Subgroup of $C_4 \times C_8$ (ID $[32,3]$)  \\
		$[16,6]$ & $G(8,2,5)$ &  Subgroup of $G(8,2,5) \times C_2$ (ID $[32,37]$) \\
		$[16,8]$ & $G(8,2,3)$ & \hyperref[16-8-section]{Section~\ref*{16-8-section}} \\
		$[16,10]$ & $C_2 \times C_2 \times C_4$ & Subgroup of $C_2 \times C_6 \times C_{12}$ (ID $[144,178]$) \\
		$[16,11]$ & $D_4 \times C_2$ & \hyperref[16-11-section]{Section~\ref*{16-11-section}}\\
		$[16,13]$ & $D_4 \curlyvee C_4$ & Subgroup of $C_4^2 \rtimes C_2$ (ID $[32,11]$) \\
		$[18,2]$ & $C_{18}$ &  \hyperref[abelian:cor]{Corollary~\ref*{abelian:cor}}\\
		$[18,3]$ & $S_3 \times C_3$ & Subgroup of $S_3 \times C_{12}$ (ID $[72,27]$)\\
		$[18,5]$ & $C_3 \times C_6$ & Subgroup of $C_2 \times C_6 \times C_{12}$ (ID $[144,178]$) \\
		$[24,1]$ & $G(3,8,2)$ & Subgroup of $G(3,8,2) \times C_3$ (ID $[72,12]$)  \\
		$[24,2]$ & $C_{24}$ & Subgroup of $C_2 \times C_{24}$ (ID $[48,23]$) \\
		$[24,5]$ & $S_3 \times C_4$ & Subgroup of $S_3 \times C_{12}$ (ID $[72,27]$)  \\
		$[24,8]$ & $(C_2 \times C_6) \rtimes C_2$ & Subgroup of $((C_2 \times C_6) \rtimes C_2) \times C_3$ (ID $[72,30]$) \\
		$[24,9]$ & $C_2 \times C_{12}$ &  Subgroup of $C_2 \times C_6 \times C_{12}$ (ID $[144,178]$)\\
		$[24,10]$ & $D_4 \times C_3$ &  Subgroup of $((C_2 \times C_6) \rtimes C_2) \times C_3$ (ID $[72,30]$)\\
		$[24,11]$ & $Q_8 \times C_3$ &  \hyperref[24-11-section]{Section~\ref*{24-11-section}} \\
		$[24,13]$ & $A_4 \times C_2$ &  Subgroup of $A_4 \times C_4$ (ID $[48,30]$)\\
		$[24,15]$ & $C_2 \times C_2 \times C_6$ & Subgroup of $C_2 \times C_6 \times C_{12}$ (ID $[144,178]$) \\
		$[27,3]$ & $\Heis(3)$ & \hyperref[27-3-section]{Section~\ref*{27-3-section}}  \\
		$[27,5]$ & $C_3 \times C_3 \times C_3$ & Subgroup of $C_3 \times C_6 \times C_{6}$ (ID $[108,45]$)  \\
		$[32,3]$ & $C_4 \times C_8$ & \hyperref[abelian:cor]{Corollary~\ref*{abelian:cor}} \\
		$[32,4]$ & $G(8,4,5)$ & \hyperref[32-4-section]{Section~\ref*{32-4-section}}  \\
		$[32,11]$ & $(C_4 \times C_4) \rtimes C_2$ & \hyperref[32-11-section]{Section~\ref*{32-11-section}} \\
		$[32,21]$ & $C_2 \times C_4 \times C_4$ & Subgroup of $C_2 \times C_4 \times C_{12}$ (ID $[96,161]$) \\
		$[32,24]$ & $(C_4 \times C_4) \rtimes C_2$ & \hyperref[32-24-section]{Section~\ref*{32-24-section}} \\
		$[32,37]$ & $G(8,4,5) \times C_2$ & \hyperref[32-37-section]{Section~\ref*{32-37-section}}\\
		$[36,6]$ & $G(3,4,2) \times C_3$ & Subgroup of $S_3 \times C_{12}$ (ID $[72,27]$)\\
		$[36,8]$ & $C_3 \times C_{12}$ & Subgroup of $C_2 \times C_6 \times C_{12}$ (ID $[144,178]$) \\
		$[36,12]$ & $S_3 \times C_6$ & Subgroup of $S_3 \times C_{12}$ (ID $[72,27]$) \\
		$[36,14]$ & $C_6 \times C_6$ & Subgroup of $C_2 \times C_6 \times C_{12}$ (ID $[144,178]$) \\
		$[48,20]$ & $C_4 \times C_{12}$ & Subgroup of $C_2 \times C_4 \times C_{12}$ (ID $[96,161]$) \\
		$[48,21]$ & $((C_4 \times C_2) \rtimes C_2) \times C_3$ &  \hyperref[48-21-section]{Section~\ref*{48-21-section}} \\
		$[48,22]$ & $G(4,4,3) \times C_3$ & \hyperref[48-22-section]{Section~\ref*{48-22-section}} \\
		$[48,23]$ & $C_2 \times C_{24}$ & \hyperref[abelian:cor]{Corollary~\ref*{abelian:cor}}  \\
		$[48,31]$ & $A_4 \times C_4$ &  \hyperref[48-31-section]{Section~\ref*{48-31-section}} \\
		$[48,44]$ & $C_2 \times C_2 \times C_{12}$ & Subgroup of $C_2 \times C_6 \times C_{12}$ (ID $[144,178]$) \\
		$[54,12]$ & $S_3 \times C_3 \times C_3$ & Subgroup of $S_3 \times C_3 \times C_6$ (ID $[108,42]$)  \\
		$[54,15]$ & $C_3 \times C_3 \times C_6$ & Subgroup of $C_3 \times C_6 \times C_{6}$ (ID $[108,45]$)  \\
		$[72,12]$ & $G(3,8,2) \times C_3$ & \hyperref[72-12-section]{Section~\ref*{72-12-section}} \\
		$[72,27]$ & $S_3 \times C_{12}$ & \hyperref[72-27-and-108-42-section]{Section~\ref*{72-27-and-108-42-section}}   \\
		$[72,30]$ & $((C_2 \times C_6) \rtimes C_2) \times C_3$ &  \hyperref[72-30-section]{Section~\ref*{72-30-section}} \\
		$[72,36]$ & $C_6 \times C_{12}$ & Subgroup of $C_2 \times C_6 \times C_{12}$ (ID $[144,178]$) \\
		$[72,50]$ & $C_2 \times C_6 \times C_6$ & Subgroup of $C_2 \times C_6 \times C_{12}$ (ID $[144,178]$) \\
		$[96,161]$ & $C_2 \times C_4 \times C_{12}$ &  \hyperref[abelian:ex]{Example~\ref*{abelian:ex}}\\
		$[108,32]$ & $G(3,4,2) \times C_3 \times C_3$ & \hyperref[108-32-section]{Section~\ref*{108-32-section}} \\
		$[108,42]$ & $S_3 \times C_6 \times C_3$  & \hyperref[72-27-and-108-42-section]{Section~\ref*{72-27-and-108-42-section}}   \\
		$[108,45] $ & $C_3 \times C_6 \times C_6$ &  \hyperref[abelian:ex]{Example~\ref*{abelian:ex}} \\
		$[144,178]$ & $C_2 \times C_6 \times C_{12}$ & \hyperref[abelian:ex]{Example~\ref*{abelian:ex}} \\
	\end{longtable}
	\captionof{table}{The groups of order $2^a \cdot 3^b$, which are hyperelliptic in dimension $4$} \label{table:2a3b}
	%\caption{the list of hyperelliptic groups in dimension $4$ whose order is $2^a \cdot 3^b$}
\end{center}

\begin{rem}
	We remark that the list \textsf{ForbiddenIDs} is minimal in the sense that if we run our code with a strictly smaller list, the output will contain at least one group that is not hyperelliptic in dimension $4$.
\end{rem}

\chapter{Hyperelliptic Groups in Dimension $4$ whose Order is Divisible by $5$ or $7$} \label{chapter:5-or-7-divides-G}

The missing piece of the puzzle of proving our classification result (\hyperref[mainthm]{Main Theorem~\ref*{mainthm}}) is the classification of hyperelliptic groups $G$ in dimension $4$, whose order is divisible by $5$ or $7$. We show in this chapter that these hyperelliptic groups are exactly the following ones:
\begin{center}
	\begin{tabular}{l|r}
		\begin{tabular}{c|c}
			ID & Group \\ \hline \hline
			$[5,1]$ & $C_5$ \\
			$[10,2]$&$C_{10}$  \\
			$[15,1]$&$C_{15}$ \\
			$[20,2]$&$C_{20}$ \\
			$[20,5]$&$C_2 \times C_{10}$ \\
			$[30,4]$&$C_{30}$ \\
			$[40,9]$&$C_2 \times C_{20}$  \\
			$[60,13]$& $C_2 \times C_{30}$ 
		\end{tabular}
		&	
		\begin{tabular}{c|c}
			ID & Group \\ \hline \hline
			$[7,1]$ & $C_7$\\
			$[14,2]$ & $C_{14}$ \\
			\, \\\, \\\, \\\, \\\, \\\, \\
			
		\end{tabular}
	\end{tabular}
\end{center}

Observe that the above list consists exactly of the Abelian hyperelliptic groups containing an element of order $5$ or $7$, see \hyperref[thm:abelian-classification]{Theorem~\ref*{thm:abelian-classification}}. \\
The chapter is divided into several sections:
\begin{itemize}
	\item We prove in \hyperref[section:2a3b5]{Section~\ref*{section:2a3b5}} that if $|G| = 2^a \cdot 3^b \cdot 5$, then $G$ is Abelian and hence contained in the list above.
	\item In \hyperref[section:2a7]{Section~\ref*{section:2a7}}, we show that the only hyperelliptic groups in dimension $4$ of order $2^a \cdot 7$ are $C_7$ and $C_{14}$.
	\item Similarly, \hyperref[section:3b7]{Section~\ref*{section:3b7}} is dedicated to proving that if $|G| = 3^b \cdot 7$ is hyperelliptic in dimension $4$, then $b = 0$.
	\item Finally, we show in \hyperref[section:2a3b5c7]{Section~\ref*{section:2a3b5c7}} that there are no hyperelliptic groups $G$ in dimension $4$ whose order is $2^a \cdot 3^b \cdot 5^c \cdot 7$, where $a,b \geq 1$.
\end{itemize}

\section{Groups of order $2^a \cdot 3^b \cdot 5$} \label{section:2a3b5}

Let us recall from \hyperref[cor:group-order-with-bounds]{Corollary~\ref*{cor:group-order-with-bounds}} that the order of a hyperelliptic group in dimension $4$ is not divisible by $25$.  This section investigates hyperelliptic groups of order $2^a \cdot 3^b \cdot 5$ in dimension $4$. More precisely, we show that

\begin{prop} \label{prop:2a3b5}
	If $|G| = 2^a \cdot 3^b  \cdot 5$ and $G$ is hyperelliptic in dimension $4$, then $G$ is Abelian and appears in the following list of groups:
	\begin{align*}
	C_5, \quad C_{10}, \quad C_{15}, \quad C_{20}, \quad C_{2} \times C_{10}, \quad C_{30}, \quad C_2 \times C_{20}, \quad C_2 \times C_{30}.
	\end{align*}
\end{prop}

For the proof, we first recall that the list of groups given in  \hyperref[prop:2a3b5]{Theorem~\ref*{prop:2a3b5}} consists exactly of the Abelian hyperelliptic groups in dimension $4$, which contain an element of order $5$, cf. \hyperref[thm:abelian-classification]{Theorem~\ref*{thm:abelian-classification}}. It thus suffices to prove that any hyperelliptic group in dimension $4$, whose order is $2^a \cdot 3^b \cdot 5$ is Abelian. The following lemma almost immediately implies our goal.

\begin{lemma} \label{lemma:2a3b5} Suppose that $G$ is hyperelliptic in dimension $4$.
	\begin{enumerate}[ref=(\theenumi)]
		\item \label{lemma:2a3b5-1} If $|G| = 2^a \cdot 5$, then $G$ is Abelian and its $2$-Sylow subgroup is one of $\{1\}$, $C_2$, $C_2 \times C_2$ or $C_4$. 
		\item \label{lemma:2a3b5-2} If $|G| = 3^b \cdot 5$, then $G$ is Abelian and its $3$-Sylow subgroup is either trivial or $C_3$.
	\end{enumerate}
\end{lemma}

\begin{proof}
	We recall from  \hyperref[cor:normal_ord_5_implies_abelian]{Corollary~\ref*{cor:normal_ord_5_implies_abelian}} that $G$ does not have a normal $5$-Sylow subgroup, unless $G$ is Abelian. Given the remark above the lemma, it suffices to prove that $G$ is Abelian. \\
	
	\ref{lemma:2a3b5-1} According to  \hyperref[table:2sylows]{\Cref*{table:2sylows}} of \hyperref[section:running-algo]{Section~\ref*{section:running-algo}}, we have $a \leq 5$. Sylow's Theorems show that the $5$-Sylow subgroup of $G$ is normal, unless possibly when $a \in \{4,5\}$. We use  \hyperref[gap-2a5]{GAP Script~\ref*{gap-2a5}} to search all groups of order $2^a \cdot 5$, $a \in \{4,5\}$, whose $5$-Sylow subgroups are not normal and which do not contain $Q_8 \times C_2$ or $C_2^4$ as a subgroup (recall that the latter groups are not hyperelliptic in dimension $4$). The output is empty, and hence the statement is proven. \\
	
	\ref{lemma:2a3b5-2} Here, we have $b \leq 3$ (see \hyperref[table:3sylows]{\Cref*{table:3sylows}}), and hence Sylow's Theorems imply that the $5$-Sylow subgroup of $G$ is normal.
\end{proof}

\begin{proof}[Proof of Proposition \ref{prop:2a3b5}.]
	Let $\rho$ be the complex representation of a hyperelliptic manifold with holonomy group $G$.  \hyperref[lemma-table]{Lemma~\ref*{lemma-table}} shows that $G$ sits in an exact sequence
	\begin{align*}
	0 \to N \to G \stackrel{\det(\rho)}{\to} C_m \to 0,
	\end{align*}
	where $m$ is divisible by $5$. In particular, $N$ is a group of order $2^{\tilde a} \cdot 3^{\tilde b}$ for some $\tilde a, \tilde b$. By \hyperref[thm:burnside]{Burnside's $p^a q^b$-Theorem~\ref*{thm:burnside}}, $N$ is solvable, and so $G$ is solvable. Hence \hyperref[thm:hall]{Hall's theorem~\ref*{thm:hall}} implies that $G$ contains subgroups of order $2^a \cdot 5$ and $3^b \cdot 5$, respectively, which by  \hyperref[lemma:2a3b5]{Lemma~\ref*{lemma:2a3b5}} are Abelian. Furthermore, the cited lemma shows that we are done, unless $a \in \{1,2\}$ and $b = 1$. We use  \hyperref[gap-2a35]{GAP Script~\ref*{gap-2a35}} to search all solvable, non-Abelian groups of order $2^a \cdot 3 \cdot 5$, $a \in \{1,2\}$, whose $5$-Sylow is not normal. There is no output, and thus we are done.
\end{proof}

\section{Groups of order $2^a \cdot 7$} \label{section:2a7}

Similarly, as in the previous section, we show

\begin{prop} \label{prop:2a7}
	If $|G| = 2^a \cdot 7$ and $G$ is hyperelliptic in dimension $4$, then $G$ is cyclic and isomorphic to one of $C_7$ or $C_{14}$.
\end{prop}

Recall that $C_7$ and $C_{14}$ are exactly the Abelian hyperelliptic groups in dimension $4$ containing an element of order $7$, see \hyperref[thm:abelian-classification]{Theorem~\ref*{thm:abelian-classification}}. \\

We divide the proof of the proposition into several steps, the main ingredient being

\begin{lemma} \label{lemma:2a7-2syl-normal}
	If $|G| = 2^a \cdot 7$ and $G$ is hyperelliptic in dimension $4$, then the $2$-Sylow subgroup of $G$ is normal.
\end{lemma}

\begin{proof}
	Let $g \in G$ be an element of order $7$ and denote by $N := N_G(\langle g\rangle)$ the normalizer of $\langle g \rangle$ in $G$. According to \hyperref[prop:conjugate]{Proposition~\ref*{prop:conjugate}}, $g$ and $g^{-1}$ are not conjugate in $G$. Since $g \mapsto g^{-1}$ is the only automorphism of $\langle g \rangle \cong C_7$ whose order is a power of $2$, the normalizer $N$ is the product of its Sylow subgroups,
	\begin{align*}
	N = \langle g \rangle \times N',
	\end{align*}
	where $N'$ is, of course, the $2$-Sylow subgroup of $N$. By  \hyperref[order-cyclic-groups]{Lemma~\ref*{order-cyclic-groups}}, the exponent of $N'$ is $\leq 2$. Hence $N'$ is Abelian and isomorphic to $C_2^r$ for some $r$. From the classification of Abelian hyperelliptic groups in dimension $4$ containing an element of order $7$ (\hyperref[thm:abelian-classification]{Theorem~\ref*{thm:abelian-classification}}) we conclude that $r \in \{0,1\}$. We will now treat the cases $r = 0$ and $r = 1$ separately. \\
	
	\textit{The Case $r = 0$:} We show that the cardinality of the set
	\begin{align} \label{eq:2a7-conjugacyclasses}
	\bigcup_{i=1}^6 \{h^{-1}g^ih ~ | ~ h \in G\}
	\end{align}
	is $2^a \cdot 6$. Once established, this implies that $G$ contains exactly $2^a$ elements whose orders are coprime to $7$, and hence the $2$-Sylow subgroup of $G$ is normal. \\
	In order to calculate the cardinality, we, first of all, observe that $g^i$ and $g^j$ can only be conjugate if $i = j$, since for $h_1,h_2 \in G \setminus N$:
	\begin{align*}
	h_1^{-1}g^i h_1 = h_2^{-1}g^j h_2 \implies (h_1h_2^{-1})^{-1} g^i (h_1h_2^{-1}) = g^j \stackrel{N \text{ Abelian}}{\implies} h_1h_2^{-1} \in N \implies i = j.
	\end{align*}
	Conversely, if $h_1h_2^{-1} \in N$, then $h_1^{-1}g^i h_1 = h_2^{-1}g^i h_2$. Since $|N| = 7$, the conjugacy class of $g^i$ has length $|G|/|N| = 2^a$. It follows that the set (\ref{eq:2a7-conjugacyclasses}) indeed has $6 \cdot 2^a$ elements. \\
	
	\textit{The Case $r = 1$:} Denote by $k$ a generator of $N$. Similar to the above case, it suffices to show that the set
	\begin{align*}
	\bigcup_{i=1, ~ i \neq 7}^{13} \{h^{-1}k^ih ~ | ~ h \in G\}
	\end{align*}
	contains $6 \cdot 2^{a}$ elements, so that $G$ again contains exactly $2^a$ elements whose orders are coprime to $7$. Since $N$ is again Abelian, this is done exactly as in the previous case and therefore omitted.
\end{proof}

We are now able to use GAP to prove \hyperref[prop:2a7]{Proposition~\ref*{prop:2a7}}:

\begin{proof}[Proof of Proposition \ref{prop:2a7}.]
	First of all, we observe that $a \leq 5$ (see for instance \hyperref[table:2sylows]{\Cref*{table:2sylows}} of \hyperref[section:running-algo]{Section~\ref*{section:running-algo}}). We run  \hyperref[gap-2a7]{GAP Script~\ref*{gap-2a7}} to search all groups $G$ of order $2^a \cdot 7$, $a \in \{1,...,5\}$ such that
	\begin{enumerate}[ref=(\theenumi)]
		\item $G$ is non-Abelian,
		\item the $2$-Sylow subgroup of $G$ is normal (see \hyperref[lemma:2a7-2syl-normal]{Lemma~\ref*{lemma:2a7-2syl-normal}}),
		\item $G$ does not contain a central element of order $7$ (see  \hyperref[cor:central-7-9]{Lemma~\ref*{cor:central-7-9}}),
		\item $G$ does not contain $C_2^4$ (ID $[16,14]$) as a subgroup (see \hyperref[lemma:many-factors]{Lemma~\ref*{lemma:many-factors}} \ref{lemma:many-factors1}).
	\end{enumerate}
	The only output is the group $C_2^3 \rtimes C_7$ (ID $[56,11]$) -- it is easily excluded by observing that  \hyperref[thm:huppert-degree]{N. Ito's Degree Theorem~\ref*{thm:huppert-degree}} implies that its character degrees are $1$ and $7$, so that $C_2^3 \rtimes C_7$ cannot have any faithful degree $4$ representation.
\end{proof}

\section{Groups of order $3^b \cdot 7$} \label{section:3b7}

The goal of this section is to prove

\begin{prop} \label{prop:3b7}
	If $|G| = 3^b \cdot 7$ and $G$ is hyperelliptic in dimension $4$, then $G$ is cyclic of order $7$.
\end{prop}

According to the discussion in \hyperref[table:3sylows]{\Cref*{table:3sylows}} of \hyperref[section:running-algo]{Section~\ref*{section:running-algo}}, $b \leq 3$ and hence the $7$-Sylow is normal in $G$. Furthermore,  \hyperref[order-cyclic-groups]{Lemma~\ref*{order-cyclic-groups}} shows that $G$ does not contain any element of order $21$. \\
In order to prove  \hyperref[prop:3b7]{Proposition~\ref*{prop:3b7}}, we exclude the cases $b = 1$, $b = 2$ and $b = 3$ separately. \\

Starting with the case $b = 1$, i.e., groups of order $21$, we observe that 
\begin{align*}
G(7,3,2) = \langle g,h ~ | ~ g^7 = h^3 = 1, ~ h^{-1}gh = g^2 \rangle
\end{align*}
is the unique non-Abelian group of order $21$.

\begin{lemma} \label{lemma:order-21}
	The group $G(7,3,2)$ is not hyperelliptic in dimension $4$.
\end{lemma}

\begin{proof}
	Suppose that $X = T/G(7,3,2)$ is a hyperelliptic fourfold with associated complex representation $\rho \colon G(7,3,2) \hookrightarrow \GL(4,\CC)$. \hyperref[thm:huppert-degree]{N. Ito's Degree Theorem~\ref*{thm:huppert-degree}} shows that the character degrees of $G(7,3,2)$ are $1$ and $3$. Thus, for $\rho$ to be faithful, $\rho$ splits as the direct sum of a faithful irreducible representation $\rho_3$ and a linear character $\chi$. The two irreducible degree $3$ representations of $G(7,3,2)$ are -- up to equivalence -- given by:
	\begin{align*}
	&g \mapsto \begin{pmatrix}
	\zeta_7 && \\ & \zeta_7^2 & \\ && \zeta_7^4
	\end{pmatrix}, \qquad h \mapsto \begin{pmatrix}
	0 & 0 & 1 \\ 1 &0 &0 \\ 0 & 1 & 0
	\end{pmatrix}, \quad \text{ and } \\
	&g \mapsto \begin{pmatrix}
	\zeta_7^3 && \\ & \zeta_7^6 & \\ && \zeta_7^5
	\end{pmatrix}, \qquad h \mapsto \begin{pmatrix}
	0 & 0 & 1 \\ 1 &0 &0 \\ 0 & 1 & 0
	\end{pmatrix}.
	\end{align*}
	The automorphism $g \mapsto g^3$, $h \mapsto h$ of $G(7,3,2)$ exchanges these two representations; hence we are allowed to focus only on the first of the two possibilities. \\
	Furthermore, by  \hyperref[lemma-two-generators]{Proposition~\ref*{lemma-two-generators}}, the linear character $\chi$ is non-trivial. Since $g = h^{-1}ghg^{-1}$ is a commutator, we may assume that $\chi(h) = \zeta_3$.  We hence write the action of $h$ on $T$ as follows:
	\begin{align*}
	h(z) = (z_3 + a_1, \ z_1 + a_2, \ z_2 + a_3, \zeta_3 z_4 + a_4).
	\end{align*}
	The condition that $h^3 = \id_T$ then amounts to
	\begin{align*}
	(a, \ a, \ a, \ 0) \text{ is zero in } T, \text{ where } a := a_1+a_2+a_3.
	\end{align*}
	Of course, this implies that
	\begin{align*}
	(\id_T + \rho(g) + \rho(g^3))(a, \ a, \ a, \ 0) = \left((1+ \zeta_7 + \zeta_7^3) \cdot a, \ (1 + \zeta_7^2 + \zeta_7^6) \cdot a, \ (1+\zeta_7^4 + \zeta_7^5) \cdot a\right) = 0 \qquad (*)
	\end{align*}
	in $T$. We claim that 
	\begin{align*}
	h^2(z) = (z_2 + a_1+a_3, \ z_3+a_1+a_2, \ z_1+a_2+a_3, \ \zeta_3^2 z_4 - \zeta_3^2 a_4)
	\end{align*}
	does not act freely on $T$. The verification of this claim starts by observing that $w = (w_1, ..., w_4)$ is a fixed point of $h^2$ if and only if
	\begin{align*}
	\left(w_2 - w_1 + a_1+a_3, \ w_3-w_2+a_1+a_2, \ w_1-w_3 + a_2+a_3, \ (\zeta_3^2-1)w_4 - \zeta_3^2 a_4\right) = 0 \text{ in } T.
	\end{align*}
	We choose
	\begin{align*}
	w_4 = \frac{\zeta_3 a_1}{2\zeta_3 + 1},
	\end{align*}
	so that $(\zeta_3^2-1)w_4 = 0$. Furthermore, we choose $w_1, w_2, w_3$ such that 
	\begin{align*}
	&w_1 - w_3 = a_2 + (\zeta_7 + \zeta_7^3) \cdot a, \\
	&w_2 - w_1 = a_3 + (\zeta_7^2 + \zeta_7^6) \cdot a.
	\end{align*}
	Then
	\begin{align*}
	w_3 - w_2 = -(w_1-w_3) - (w_2-w_1) &= -a_2-a_3 - (\zeta_7 + \zeta_7^2 + \zeta_7^3 + \zeta_7^6) \cdot a \\
	&= -a_2-a_3 + (1+\zeta_7^4+\zeta_7^5) \cdot a \\
	& = a_1 + (\zeta_7^4 + \zeta_7^5) \cdot a.
	\end{align*}
	It follows that 
	\begin{align*}
	h^2(w) - w = ((1+\zeta_7 + \zeta_7^3) \cdot a, \ (1+\zeta_7^2 + \zeta_7^6) \cdot a, \ (1+\zeta_7^4 + \zeta_7^5) \cdot a, \ 0) \stackrel{(*)}{=} 0
	\end{align*}
	and $w$ is indeed a fixed point of $h^2$.
\end{proof}

\begin{rem}
	It can be shown more generally \cite[Proposition 4.2]{Szczepanski} that the first Betti number compact flat Riemannian manifold whose holonomy group is $G(m,n,r)$, $\gcd(m,n) = 1$ is positive. This provides a quick alternative proof of  \hyperref[lemma:order-21]{Lemma~\ref*{lemma:order-21}}: indeed, that a hyperelliptic manifold has a positive first Betti number means that the associated complex representation contains at least one copy of the trivial character -- however, as we have already seen in the proof of \hyperref[lemma:order-21]{Lemma~\ref*{lemma:order-21}}, \hyperref[lemma-two-generators]{Proposition~\ref*{lemma-two-generators}} shows that this is impossible in dimension $4$. 
\end{rem}

We are now ready to complete the proof of  \hyperref[prop:2a7]{Proposition~\ref*{prop:2a7}}:

\begin{proof}[Proof of Proposition \ref{prop:3b7}.]
	It remains to show that no groups of order $3^2 \cdot 7 = 63$ and $3^3 \cdot 7 = 189$ are hyperelliptic in dimension $4$. These are easily excluded as follows. Since the $7$-Sylow subgroup is normal and $\Aut(C_7)$ is cyclic of order $6$, an element of order $7$ and an element of order $3$ span a subgroup of order $21$. We are done since no group of order $21$ is hyperelliptic in dimension $4$.
\end{proof}

\section{Groups of order $2^a \cdot 3^b \cdot 5^c \cdot 7$} \label{section:2a3b5c7}

The classification of hyperelliptic groups in dimension $4$, whose order is divisible by $5$ or $7$ is complete if we prove the following

\begin{prop} \label{prop:2a3b5c7}
	If $|G| = 2^a \cdot 3^b \cdot 5^c \cdot 7$ and $a,b \geq 1$, then $G$ is not hyperelliptic in dimension $4$.
\end{prop}

The proof is by contradiction and will be divided into several steps. In the following, we will always assume that $\rho \colon G \to \GL(4,\CC)$ is the complex representation of some hyperelliptic fourfold. The upcoming lemma reduces to the case $c = 0$ and shows that $G$ is non-solvable.

\begin{lemma} \label{lemma:non-solv}
	If $|G| = 2^a \cdot 3^b \cdot 5^c \cdot 7$, $a,b \geq 1$ is hyperelliptic in dimension $4$, then $G$ contains a non-solvable subgroup $N$ of order $2^{\tilde a} \cdot 3^{\tilde b} \cdot 7$, where $\tilde a, \tilde b \geq 1$.
\end{lemma}

\begin{proof}
	First of all, $G$ is non-solvable, since else  \hyperref[thm:hall]{Hall's Theorem~\ref*{thm:hall}} would imply that $G$ contains a subgroup of order $3^b \cdot 7$, which forces $b = 0$ (see  \hyperref[prop:3b7]{Proposition~\ref*{prop:3b7}}). Now, consider the determinant exact sequence
	\begin{align*}
	0 \to N \to G \stackrel{\det(\rho)}{\to} C_m \to 0.
	\end{align*}
	By \hyperref[lemma-table]{Lemma~\ref*{lemma-table}}, if $c \geq 1$, then $m$ is divisible by $5$. Since moreover $c \leq 1$ by \hyperref[cor:group-order-with-bounds]{Corollary~\ref*{cor:group-order-with-bounds}}, we infer that $|N|$ is not divisible by $5$, i.e.,
	\begin{align*}
	|N| = 2^{\tilde a} \cdot 3^{\tilde b} \cdot 7^{\tilde d} \text{ for some } \tilde a, \tilde b, \tilde d \geq 0.
	\end{align*}
	We observe that $N$ is non-solvable, since $G$ is non-solvable, but $G/N \cong C_m$ is solvable. By  \hyperref[thm:burnside]{Burnside's $p^a q^b$-Theorem~\ref*{thm:burnside}}, $|N|$ has to be divisible by at least three primes and hence $\tilde a, \tilde b, \tilde d \geq 1$. Since $7^2$ does not divide $|G|$, we have $\tilde d = 1$.
\end{proof}

We have therefore reduced \hyperref[prop:2a3b5c7]{Proposition~\ref*{prop:2a3b5c7}} to showing

\begin{prop} \label{prop:2a3b7}
	The non-solvable group $N$ of order $2^{\tilde a} \cdot 3^{\tilde b} \cdot 7$ is not hyperelliptic in dimension $4$.
\end{prop}

We will exploit that $N$ is by definition the kernel of $\det(\rho)$, and hence $\rho \colon G \to \GL(4,\CC)$ maps $N$ into $\SL(4,\CC)$.

\begin{lemma} \label{lemma:non-solv-N}
	We have $\tilde a \leq 4$, $\tilde b \leq 2$ and $N$ does not contain $C_2^3$ as a subgroup.
\end{lemma}

\begin{proof}
	\hyperref[lemma:many-factors]{Lemma~\ref*{lemma:many-factors}} \ref{lemma:many-factors3} shows that no complex representation of a hyperelliptic fourfold with holonomy $C_e^3$, $e \geq 2$ has image contained in $\SL(4,\CC)$, which shows the latter part of the statement. \\
	We now show that $\tilde b \leq 2$. According \hyperref[table:3sylows]{\Cref*{table:3sylows}} of  \hyperref[section:running-algo]{Section~\ref*{section:running-algo}}, $\tilde b \leq 3$ and $\tilde b = 3$ if and only if the $3$-Sylow subgroups of $N$ are isomorphic to $C_3^3$ or to the Heisenberg group $\Heis(3)$. As already noted above, $C_3^3$ does not map into $\SL(4,\CC)$ via $\rho$, and  \hyperref[27-3-prop]{Proposition~\ref*{27-3-prop}} \ref{27-3-prop-1} shows the analogous statement for $\Heis(3)$. Thus $\tilde b \leq 2$. \\
	It remains to prove that $\tilde a  \leq 4$. Let $S$ be a $2$-Sylow subgroup of $N$. Then  \hyperref[non-faithful-not-in-sl]{Lemma~\ref*{non-faithful-not-in-sl}} shows that either
	\begin{enumerate}[ref=(\theenumi)]
		\item \label{non-solv-1} $S$ is Abelian, or
		\item \label{non-solv-2} $S$ is not-Abelian and $Z(S)$ is cyclic.
	\end{enumerate}
	In Case \ref{non-solv-2}, the discussion at the end of  \hyperref[section:2-sylow-case1]{Lemma~\ref*{section:2-sylow-case1}} implies that either $|S| \leq 16$ or $S$ is the group $[32,11]$ in the Database of Small Groups. However, we have already seen in \hyperref[32-11-prop]{Proposition~\ref*{32-11-prop}} \ref{32-11-prop1} that the complex representation of a hyperelliptic fourfold whose holonomy group is the group with ID $[32,11]$ is not contained in $\SL(4,\CC)$, which shows that $\tilde b \leq 4$ if $S$ falls under Case \ref{non-solv-2}. \\
	In \ref{non-solv-1},  \hyperref[abelian-case-3-sylow]{Lemma~\ref*{abelian-case-3-sylow}} shows that the only Abelian hyperelliptic group in dimension $4$ of order $32$ is $S = C_4 \times C_8$ (see also \hyperref[table:2sylows]{\Cref*{table:2sylows}} in  \hyperref[section:running-algo]{Section~\ref*{section:running-algo}}). It is excluded as follows. Let $g,h \in S$ be generators, $\ord(g) = 8$, $\ord(h) = 4$. Assume that $S$ is diagonally embedded in $\SL(4,\CC)$ via $\rho$. According to  \hyperref[lemma-table]{Lemma~\ref*{lemma-table}}, we may assume that 
	\begin{align*}
	\rho(g) = \diag(\zeta_8, ~ \zeta_8^r, \ \varepsilon, \ 1), 
	\end{align*}
	where $r \in \{3,5\}$ and $\varepsilon$ is a fourth root of unity, chosen such that $\det(\rho(g)) = 1$. By replacing $h$ by some $hg^j$ if necessary, we may assume that at most one of the first two diagonal entries of $\rho(h)$ is different from $1$, say
	\begin{align*}
	\rho(h) = \diag(1, \ i^{a_2}, \ i^{a_3}, \ i^{a_4}),
	\end{align*}
	where $a_2,a_3,a_4 \in \{0,...,3\}$ and $a_2+a_3+a_4 \equiv 0 \pmod 4$. Moreover, the faithfulness of $\rho$ implies that at least one of the $a_j$ is odd. Hence, if $a_2$ is even, then both of $a_3, a_4$ are odd, and one of the matrices
	\begin{align*}
	&\rho(gh) = \diag(\zeta_8, \ \zeta_8^{2a_2+r}, \ \varepsilon i^{a_3}, \ i^{a_4}), \text{ or} \\
	&\rho(gh^3) = \diag(\zeta_8, \ \zeta_8^{6a_2+r}, \ \varepsilon i^{-a_3}, \ i^{-a_4})
	\end{align*}
	does not have the eigenvalue $1$, depending on the value of $\varepsilon$. On the other hand, if $a_2$ is odd, then one of the two matrices above has multiple or complex conjugate eigenvalues of order $8$, depending on $r \in \{3,5\}$ and $a_2 \in \{1,3\}$ -- this contradicts the \hyperref[order-cyclic-groups]{Integrality Lemma~\ref*{order-cyclic-groups}} \ref{ocg-3} and hence finishes the proof.

	%Considering that $h$ is not contained in the subgroup spanned by $g$, we infer that
	%\begin{align*}
	%	\rho(h^2) \neq \rho(g^4) = \diag(-1, \ -1, \ 1, \ 1).
	%\end{align*}
	%It follows that at least one of the last two diagonal entries of $\rho(h)$ is a fourth root of unity. If the last diagonal entry of $\rho(h)$ is a fourth root of unity, then one of the matrices 
	%\begin{align*}
	%\rho(gh), \qquad \rho(g^2h)
	%\end{align*}
\end{proof}

As a consequence, we obtain

\begin{cor} \label{cor:168dividesN}
	$168 = 2^3 \cdot 3 \cdot 7$ divides the order of $N$.
\end{cor}

\begin{proof}
	According to  \hyperref[lemma:non-solv]{Lemma~\ref*{lemma:non-solv}} and \hyperref[lemma:non-solv-N]{Lemma~\ref*{lemma:non-solv-N}}, 
	\begin{align*}
	|N| = 2^{\tilde a} \cdot 3^{\tilde b} \cdot 7, \qquad \text{where } \tilde a \in \{1,2,3,4\}, \ \tilde b \in \{1,2\}.
	\end{align*}
	The corollary hence follows if we prove that groups of order $2^{\tilde a} \cdot 3^{\tilde b} \cdot 7$, $\tilde a, \tilde b \in \{1,2\}$ are solvable. This can, of course, be checked by any computer algebra system; we prefer to give a computer-free proof. \\
	First of all, we observe that
	\begin{align*}
	2^{\tilde a} \cdot 3^{\tilde b} \in \{6,12,18,36\}
	\end{align*}
	for $\tilde a, \tilde b \in \{1,2\}$. Since the only divisors of $36$ that are congruent to $1$ modulo $7$ are $1$ and $36$ itself, Sylow's Theorems imply that groups of order $6 \cdot 7$, $12 \cdot 7$ and $18 \cdot 7$ contain a normal $7$-Sylow subgroup $P$. Since the order of the quotient is a product of two prime powers, \hyperref[thm:burnside]{Burnside's $p^a q^b$-Theorem~\ref*{thm:burnside}} shows the solvability of these groups. \\
	The only thing left to prove is that any group $U$ of order $252 = 2^2 \cdot 3^2 \cdot 7$ containing $36$ subgroups of order $7$ is solvable. By Sylow's Theorems, $U$ acts transitively on the set of $7$-Sylow subgroups by conjugation. Therefore,
	\begin{align*}
	N_U(P) = Z_U(P) = P,
	\end{align*}
	where $P \subset U$ is one of the $36$ subgroups of order $7$. By Burnside's Transfer Theorem \cite[Theorem 14.3.1]{Hall}, $U$ contains a normal $7$-complement, i.e., a normal subgroup of order $36$. This shows that $U$ is solvable.
\end{proof}

It is well-known that the only non-solvable group of order $168$ is the group $\GL(3,2)$, which is even simple and can also be described as the semi-direct product $C_2^3 \rtimes C_3$, where $C_3$ acts on $C_2^2$ by permutation. However, according to  \hyperref[lemma:non-solv-N]{Lemma~\ref*{lemma:non-solv-N}}, the group $N$ does not contain $C_2^3$. The remaining orders
\begin{align*}
&336 = 2^4 \cdot 3 \cdot 7, \\
&504 = 2^3 \cdot 3^2 \cdot 7, \text{ and} \\
&1008 = 2^4 \cdot 3^2 \cdot 7,
\end{align*}
are excluded by the following lemma, which we prove by running \hyperref[gap-non-solv]{GAP Script~\ref*{gap-non-solv}}.

\begin{lemma}
	All non-solvable groups of order $336$, $504$ or $1008$ contain $C_2^3$ as a subgroup.
\end{lemma}

\begin{rem}
	We remark that every non-solvable group whose order is one of $336$, $504$ or $1008$ contains one of the non-solvable groups $\GL(3,2)$ or $\SL(2,8)$. We have already seen that $C_2^3$ is a subgroup of $\GL(3,2)$. The group $\SL(2,8)$ also contains $C_2^3$ as a subgroup. Indeed, if $a \in \FF_8^*$ denotes a generator of the multiplicative group, then the matrices
	\begin{align*}
	\begin{pmatrix}
	1 & 1 \\ 0 & 1
	\end{pmatrix}, \quad \begin{pmatrix}
	1 & 1 \\ 1 & 0
	\end{pmatrix}, \quad \text{ and } \quad \begin{pmatrix}
	a & a^3 \\ a^3 & a
	\end{pmatrix}
	\end{align*}
	span a subgroup isomorphic to $C_2^3$.
\end{rem}

The proof of  \hyperref[prop:2a3b7]{Proposition~\ref*{prop:2a3b7}} (and thus also the proof of  \hyperref[prop:2a3b5c7]{Proposition~\ref*{prop:2a3b5c7}}) now finished. 

\chapter{The Canonical Divisor of a Hyperelliptic Manifold} \label{chapter:canonical-divisor}

Recall that the canonical divisor of a hyperelliptic surface $S$ is torsion. More precisely, $12K_S$ is linearly trivial. The number '$12$' is minimal with this property, since the classification of hyperelliptic surfaces implies that $4K_S$ or $6K_S$ is trivial, and $12$ is the least common multiple. This short chapter establishes similar results in dimensions $3$ and $4$, see \hyperref[mainthm-can-div]{Main Theorem~\ref*{mainthm-can-div}}. \\

Let $X = T/G$ be a hyperelliptic manifold of dimension $n$ with associated complex representation $\rho \colon G \to \GL(n,\CC)$, and denote by $\pi \colon T \to X$ the quotient map. It is well-known that line bundles on the quotient $X$ correspond to $G$-linearized line bundles on $T$, see for instance \cite[Chapter II.7]{Mumford}.
Since the canonical divisor of a complex torus is trivial, we obtain that
\begin{align*}
	\pi^* \left(\Oh_X(K_X)\right) = \Oh_T(K_T) \cong \Oh_T = \pi^*\Oh_X.
\end{align*}
However, the line bundles $\Oh_T(K_T)$ and $\Oh_T$ have different $G$-linearizations: while $G$ acts trivially on $H^0(T,\Oh_T)$, it acts on $H^0(T,\Oh_T(K_T))$ by multiplication by $\det(\rho(-))$. It follows that the minimal number $m \geq 1$ such that $mK_X$ is trivial is the order of the cyclic group $\det(\rho(G)) \subset \mathbb C^*$. Observe that $|\det(\rho(G))|$ is the least common multiple of the orders of $\det(\rho(g))$, taken over $g \in G$. The numbers
\begin{align*}
		\tau(n) := \min\{m \geq 1 ~ | ~ m K_X \text{ is trivial for any hyperelliptic manifold } X \text{ of dimension } n\}.
\end{align*}
are now easily computed for $n \in \{3,4\}$: \\
\begin{itemize}
\item  The work of Uchida-Yoshihara \cite{Uchida-Yoshihara} implies that if $G$ is hyperelliptic in dimension $3$ and $g \in G$, then
\begin{align*}
\ord(g) \in \{1,~ 2,~ 3,~ 4,~ 5,~ 6,~ 8,~ 10,~ 12\}.
\end{align*}
If $\ord(g) \in \{1,2,3,4,6\}$, then
\begin{align} \label{dim-3-1}
\ord(\det(\rho(g))) \in \{1,~2,~3,~4,~6\}
\end{align}
On the other hand, if $\ord(g) \in \{5,8,10,12\}$, then $\rho(g)$ has exactly two eigenvalues that are different from $1$. According to \hyperref[lemma-table]{Lemma~\ref*{lemma-table}}, we obtain that 
\begin{align} \label{dim-3-2}
	\ord(\det(\rho(g))) \in \{4,~5,~12\}.
\end{align}
Clearly, all possible determinant orders listed in (\ref{dim-3-1}) and (\ref{dim-3-2}) occur, and $\tau(3)$ equals their least common multiple, which is $60$.
\item The number $\tau(4)$ is computed similarly. According to \hyperref[order-cyclic-groups]{Lemma~\ref*{order-cyclic-groups}} \ref{ocg-2}, the eigenvalues of the matrices $\rho(g)$ ($g \in G$) are $d$th roots of unity, where
\begin{align*}
d \in \{1,~ 2,~ 3,~ 4,~ 5,~ 6,~ 7,~ 8,~ 9,~ 10,~ 12,~ 14,~ 18\}.
\end{align*}
Using \hyperref[lemma-table]{Lemma~\ref*{lemma-table}} again, we obtain that
\begin{align} \label{dim-4}
	\ord(\det(\rho(g))) \in \{1, ~ 2, ~ 3, ~ 4, ~ 5, ~ 6, ~ 7, ~ 9 ~ 10, ~ 12, ~ 14, ~ 15, ~ 18, ~ 20, ~ 30 \}.
\end{align}
Again, all the orders (\ref{dim-4}) occur, and hence $\tau(4)$ is their least common multiple, which is $1260$.
\end{itemize}

The above discussion concludes the proof of \hyperref[mainthm-can-div]{Main Theorem~\ref*{mainthm-can-div}}\footnote{We would like to point out the amusing coincidence that $\tau(4) = 1260$ is the concatenation of the numbers $\tau(2) = 12$ and $\tau(3) = 60$.}.

\chapter{Final Remarks and Further Questions} \label{chapter:finalremarks}

Together, the results of \hyperref[section:ab-class]{Section~\ref*{section:ab-class}} as well as \hyperref[chapter:2a3b]{Chapters~\ref*{chapter:2a3b}} and \ref{chapter:5-or-7-divides-G} constitute the proof of \hyperref[mainthm]{Main Theorem~\ref*{mainthm}}, the classification of hyperelliptic groups in dimension $4$. We want to point out the following aspects of our classification:

\begin{enumerate}[ref=(\theenumi)]
	\item We have seen in \hyperref[thm:center]{Theorem~\ref*{thm:center}} that the center of a non-Abelian hyperelliptic group in dimension $4$ is (as an abstract group) a subgroup of one of the groups $C_4 \times C_4$, $C_6 \times C_6$, or $C_2 \times C_{12}$. As already mentioned in \hyperref[rem:center]{Remark~\ref*{rem:center}}, \hyperref[mainthm]{Main Theorem~\ref*{mainthm}} implies that there is no non-Abelian hyperelliptic group in dimension $4$ such that $Z(G) \cong C_4 \times C_4, C_6 \times C_6, C_2 \times C_{12}$. 
	\item Using the groups of \hyperref[table:main]{Table~\ref*{table:main}} as input, a GAP computation shows that none of the hyperelliptic groups in dimension $4$ have an irreducible representation of degree $4$. This gives a different proof of a recent result of R. Lutowski \cite{Lutowski}, who showed that the complex representation $\rho$ of a hyperelliptic manifold can never be irreducible. His proof, which also shows that $\rho$ cannot split as the direct sum of two complex conjugate irreducible representations, is based on the classification of finite simple groups (CFSG). \\
	The (ir)reducibility of $\rho$ is related to the following question by Amerik-Rovinsky-van de Ven \cite{ARVdV}, which is in turn related to a conjecture of Amerik \cite{Amerik}:
	
	\begin{question} \label{amerik-question}
		Is there a hyperelliptic manifold whose second Betti number is $1$?
	\end{question}
	
	We claim that the complex representation of a (hypothetical) hyperelliptic manifold $X = T/G$ (where $T = V/\Lambda$ is a complex torus) satisfying $b_2(X) = 1$ is irreducible. To verify this, we first observe that $H^2(X,\CC) = H^{1,1}(X)$, since $H^{1,1}(X)$ contains the Kähler class.	It follows that
	\begin{align*}
	H^2(X,\CC) = H^{1,1}(X) = H^{1,1}(T)^G = (V \otimes \overline{V})^G.
	\end{align*}
	Now, if $V = U_1^{n_1} \oplus ... \oplus U_k^{n_k}$ is the isotypical decomposition of the complex representation $V$, we obtain (by Schur's Lemma):
	\begin{align*}
	1 = \dim (V \otimes \overline{V})^G = \dim\left(\bigoplus_{i,j=1}^k \Hom(U_i^{n_i}, U_j^{n_j})\right) = \sum_{i=1}^k n_i^2.
	\end{align*}
	Thus $k = 1$ and $n_1 = 1$, which shows that $V$ is indeed an irreducible representation of $G$. \\
	Lutowski's result hence shows that \hyperref[amerik-question]{Question~\ref*{amerik-question}} has a negative answer, while our classification gives a new proof in dimension $4$ without using CFSG. Let us remark that plenty of examples of hyperelliptic fourfolds $X$ satisfying $H^2(X,\CC) = H^{1,1}(X)$ were given in \hyperref[section:examples]{Section~\ref*{section:examples}}.
	\item It turned out that if $G$ is a non-Abelian $2$-group in dimension $4$ with associated complex representation $\rho$, then $\rho$ is the direct sum of three irreducible representations of $G$. This was shown in \hyperref[section:2-sylow-case2]{Section~\ref*{section:2-sylow-case2}} if $Z(G)$ is non-cyclic. We were, however, unable to find a proof if $Z(G)$ is cyclic.
\end{enumerate}

We finish by posing some related questions and problems. \\

The first question concerns the optimality of \hyperref[abelian-two-cases]{Theorem~\ref*{abelian-two-cases}} when applied to hyperelliptic manifolds:

\begin{question}
	Does there exist a hyperelliptic manifold of dimension $n$ with holonomy group $G = C_2 \times C_d^{n-2}$ for $d \in \{4,6\}$ and associated complex representation $\rho$ such that $\bigcap_{g \in G} \ker(\rho(g) - \id) = \{0\}$? 
\end{question}

It was shown in \cite[Chapter III.5]{Charlap} that every finite group $G$ occurs as the holonomy group of some compact flat Riemannian manifold $M$. Denote by $\rho \colon G \to \GL(m,\RR)$ be the holonomy representation of such a manifold $M$. Then $\rho \oplus \rho$ is the holonomy representation of $M \times M$, endowed with the diagonal action of $G$. According to \cite[Propositions 7.1 and 7.2]{Szczepanski}, the representation $\rho \oplus \rho$ is the complex representation of some hyperelliptic manifold. It follows that every finite group is hyperelliptic in some (complex) dimension.

\begin{problem} \label{problem:hyperell}
	Given a finite group $G$, determine the minimal $n$ such that $G$ is hyperelliptic in dimension $n$.
\end{problem}

While \hyperref[problem:hyperell]{Problem~\ref*{problem:hyperell}} is certainly interesting in full generality, there is most likely no hope for a general answer. However, there might be hope to solve the problem for Abelian groups, the representation theory of Abelian groups being well-understood. Solving  \hyperref[problem:hyperell]{Problem~\ref*{problem:hyperell}} for $C_3^{r-4} \times C_6^3$ or $C_6^{r-1}$ ($r \geq 4$) might be an interesting starting point, see \hyperref[cor:c6^3-c4^3-excluded]{Corollary~\ref*{cor:c6^3-c4^3-excluded}}.

\bibliography{citations}{}
\bibliographystyle{plain}

\begin{appendices}
	\chapter{GAP Codes} \label{appendix:gap}
	
	\begin{script} \label{script:2-groups-case1}
	The following code was used on p. \pageref{2-sylows-gap} f. to find all $2$-groups $G$ satisfying certain properties (see the cited page for a precise description of the algorithm).
	\begin{lstlisting}[language=GAP]
conj_ord_8 := function(G)
	for g in G do
		if Order(g) = 8 and IsConjugate(G,g,g^-1) then
			return false;
		fi;
	od;
			
	return true;
end;
		
correct_normal_subgr := function(G)
		
	if Order(G) = 8 and IsAbelian(G) = false then
		return true;
	fi;
		
	if Order(G) = 16 then
		for N in NormalSubgroups(G) do
			if IsCyclic(FactorGroup(G,N)) 
			and IdGroup(N) in [[4,1], [8,4]] then
				return true;
			fi;
		od;
	fi;
			
	if Order(G) = 32 then
		for N in NormalSubgroups(G) do
			if IsCyclic(FactorGroup(G,N)) and IdGroup(N) = [8,4] then
				return true;
			fi;
		od;
	fi;	
			
	return false;
end;
		
for a in [3,4,5] do
	for i in [1..NrSmallGroups(2^a)] do
		G := SmallGroup(2^a,i);
		ZG := Center(G);	
			
			
		if Order(ZG) <= 4 and Exponent(G) <= 8 and IsCyclic(ZG)
		and conj_ord_8(G) and correct_normal_subgr(G) then
			Print(IdGroup(G), "       ", 
				StructureDescription(G), "\n");
		fi;
	od;
od;
		\end{lstlisting}
	\end{script}

\begin{script} \label{script:2-sylow-non-cyc-center}
This piece of code is used to determine the non-Abelian $2$-groups with non-cyclic center that are potentially hyperelliptic in dimension $4$. The mechanism of the code is described at the end of  \hyperref[section:2-sylow-case2]{Section~\ref*{section:2-sylow-case2}}.
\begin{lstlisting}[language=GAP]
conj_ord_8 := function(G)
	for g in G do
		if Order(g) = 8 and IsConjugate(G,g,g^-1) then
			return false;
		fi;
	od;

	return true;
end;

good_center := function(G)
	ZG := Center(G);
	if Order(ZG) <= 16 
	and Length(AbelianInvariants(ZG)) = 2
	and Exponent(ZG) in [2,4]
	and IsCyclic(Intersection(G,ZG,DerivedSubgroup(G))) then
		return true;
	fi;
	
	return false;
end;

correct_subgroups := function(G)
	for C in ConjugacyClassesSubgroups(G) do
		U := Representative(C);
	if IdGroup(U) in [[16,12], [16,14], [32,5], [32,9],
	[32,12], [32,13], [32,25], [64,20], [64,85]] 
		then
			return false;
		fi;
	od;
	
	return true;
end;




correct_quotients := function(G)
	
	for N in NormalSubgroups(G) do
		C_m := FactorGroup(G,N);
		
		if IsCyclic(C_m) and Exponent(C_m) <= 4 
		and Exponent(N) <= 4 then
			for K in NormalSubgroups(N) do
						
				NmodKer := FactorGroup(N,K);

				if IdGroup(K) in [[2,1], [4,1]]
				and IsNormal(G,K)
				and (IdGroup(NmodKer) = [8,4] 
				or (Exponent(NmodKer) <= 4 and IsCyclic(NmodKer)))
				then
					return true;
				fi;
			od;
		fi;
	
	od;
	
	return false;
end;

for a in [3..7] do
	for i in [1..NrSmallGroups(2^a)] do
		G := SmallGroup(2^a,i);
		
		if IsAbelian(G) = false and Exponent(G) <= 8 
		and good_center(G) and conj_ord_8(G)
		and correct_subgroups(G) and correct_quotients(G) then
			Print(IdGroup(G), "       ", 
			StructureDescription(G), "\n");
		fi;
	od;
od;	
\end{lstlisting}
\end{script}

\begin{script} \label{gap-3groups}
We used the following code in  \hyperref[section:3^b-b>=4]{Section~\ref*{section:3^b-b>=4}} to verify that there is no group of order $81$ with exponent $3$ whose center is cyclic:
\begin{lstlisting}[language=GAP]
candidates := [];
	
for i in [1..NrSmallGroups(81)] do
	G := SmallGroup(81,i);
	if Exponent(G) = 3 and IsCyclic(Center(G)) then
		Add(candidates, IdGroup(G));
	fi;
od;
	
Print(Length(candidates));
\end{lstlisting} 
\end{script}

\begin{script}
This code classifies the hyperelliptic groups in dimension $4$ of order $2^a \cdot 3^b$. The functionality of the code is explained in \hyperref[section:running-algo]{Section~\ref*{section:running-algo}}. \hyperref[table:examples]{Table~\ref*{table:examples}} in the cited section contains the output. Running the code took around 36 minutes for us.
\begin{lstlisting}[language=GAP]
# The function "GoodCenter" checks if the center of 
# every non-Abelian subgroup of a given group "G" is a
# subgroup of one of C2 x C12 (ID [24,9]), C4 x C4 (ID [16,2])
# or C6 x C6 (ID [36,14]).

GoodCenter := function(U)
	MaxCenters := [ [16,2], [24,9], [36,14] ];
	ZU := Center(U);

	bool := false;
	if IsAbelian(U) then
		return true;
	else
		for id in MaxCenters do
			C := SmallGroup(id[1], id[2]);
			for N in NormalSubgroups(C) do
				if IdGroup(ZU) = IdGroup(N) then
					bool := true;
				fi;
			od;
		od;
	fi;

	return bool;
end;




# The function "ForbiddenSubgroup" checks if a given group "U"
# contains one of the groups in the list "ForbiddenIDs",
# which we have seen to be non-hyperelliptic in dimension 4

ForbiddenIDs:= [ [18,4], [24,7], [24,12], [24,14], [36,3], [36,5],
 [36,7], [36,11], [48,3], [48,5], [48,9], [48,24], [48,26], [48,30],
  [48,45], [48,47], [54,8], [54,10], [72,13], [72,14], [72,37],
   [72,38], [96,46], [96,47], [96,164], [108,35], [144,101],
    [144,102], [144,103], [216,150], [216,177] ];

ForbiddenSubgroup := function(U)
	if IdGroup(U) in ForbiddenIDs then
		return false;
	fi;
	return true;
end;

# The function "IsMetacyclic" checks whether "G" is metacyclic
# in the sense that "G" contains a normal cyclic subgroup "N"
# such that "G/N" is cyclic.

IsMetacyclic := function(G)
	for N in NormalSubgroups(G) do
		if IsCyclic(N) and IsCyclic(FactorGroup(G,N)) then
			return true;
		fi;
	od;
	return false;
end;

# The function "MetacyclicContainedInDerived" checks if 
# the derived subgroup of "G" contains a non-Abelian 
# metacyclic group. We have seen that this cannot be the case
# for a hyperelliptic group in dimension 4.

MetacyclicContainedInDerived := function(G)
	for class in ConjugacyClassesSubgroups(DerivedSubgroup(G)) do
		U := Representative(class);
		if IsAbelian(U) = false and IsMetacyclic(U) then
			return false;
		fi;
	od;
	return true;
end;



# "CorrectOrder" checks whether the order of "g" is one of
# 1, 2, 3, 4, 6, 8, 9, 12, 18 or 24. 

CorrectOrder := function(g)
	if Order(g) in [1,2,3,4,6,8,9,12,18,24] then
		return true;
	else
		return false;
	fi; 
end;

# We check whether the element "g" of "G" is of order
# 8, 9, 12 or 24 and conjugate to its inverse. 
# If yes, then "G" is not hyperelliptic in dimension 4.

Conjugate := function(G,g)
	if Order(g) in [8,9,12,24] and IsConjugate(G,g,g^-1) then
		return false;
	else
		return true;
	fi; 
end;


# The function "GoodSylows" checks if the 2- and 3-Sylow
# subgroups of a given group "G" are hyperelliptic in dimension 4.

2Sylows := [ [1,1], [2,1], [4,1], [4,2],
[8,1], [8,2], [8,3], [8,4], [8,5], [16,2], [16,3], [16,4],
[16,5], [16,6], [16,8], [16,10], [16,11], [16,13], [32,3],
[32,4], [32,11], [32,21], [32,24], [32,37] ];

3Sylows := [ [1,1], [3,1], [9,1], [9,2], [27,3], [27,5] ];

GoodSylows := function(G)
	if IdGroup(SylowSubgroup(G,2)) in 2Sylows 
	 and IdGroup(SylowSubgroup(G,3)) in 3Sylows then
		return true;
	fi;
	return false;
end;



# We summarize the functions above in new functions "SubgroupConds"
# and "ElementConds" 

SubgroupConds := function(G)
	for class in ConjugacyClassesSubgroups(G) do
		U := Representative(class);
		if ForbiddenSubgroup(U) = false or GoodCenter(U) = false then
			return false;
		fi;
	od;
	return true;
end;



ElementConds := function(G)
	for g in G do
		if CorrectOrder(g) = false or Conjugate(G,g) = false then
			return false;
		fi;
	od;
	return true;
end;


# The actual classification algorithm starts here

possible_orders := [];

counter := 1;
for a in [0..5] do
	for b in [0..3] do
		if a > 0 or b > 0 then
			Add(possible_orders, 2^a*3^b);
		fi;
	od;
od;

Sort(possible_orders);

for m in possible_orders do	
	for i in [1..NrSmallGroups(m)] do
		G := SmallGroup(m,i);
			if GoodSylows(G) and SubgroupConds(G) and ElementConds(G)
			and MetacyclicContainedInDerived(G) then
				Print(counter, "      ", IdGroup(G), "      ",
				StructureDescription(G), "\n"); 
				counter := counter + 1;
			fi;
	od;
od;
\end{lstlisting} 
\end{script}

\begin{script} \label{gap-2a5}
We use this piece of code to prove  \hyperref[lemma:2a3b5]{Lemma~\ref*{lemma:2a3b5}} \ref{lemma:2a3b5-1}. Here, the group with ID $[16,12]$ is $Q_8 \times C_2$, whereas the group with ID $[16,14]$ is $C_2^4$.
\begin{lstlisting}[language=GAP]
ForbiddenSubgroups := [[16,12], [16,14]];

for a in [4,5] do
	for i in [1..NrSmallGroups(2^a*5)] do
		G := SmallGroup(2^a*5,i);
		correct_subgroups := true;
		
		for c in ConjugacyClassesSubgroups(G) do
			C := Representative(c);
			if IdGroup(C) in ForbiddenSubgroups then
				correct_subgroups := false;
			fi;
		od;
			
		if IsNormal(G,SylowSubgroup(G,5)) = false 
		and correct_subgroups then
			Print(IdGroup(G), " ", StructureDescription(G), "\n");
		fi;
	od;
od;
\end{lstlisting}
\end{script}

\begin{script} \label{gap-2a35}
This script is used to conclude in the proof of  \hyperref[prop:2a3b5]{Proposition~\ref*{prop:2a3b5}}.
\begin{lstlisting}[language=GAP]
for a in [1,2] do
	for i in [1..NrSmallGroups(2^a*3*5)] do
		G := SmallGroup(2^a*3*5,i);

		if IsAbelian(G) = false
		and IsSolvable(G)
		and IsNormal(G,SylowSubgroup(G,5)) = false
		then
			Print(IdGroup(G),"  ", StructureDescription(G), "\n");
		fi;
	od;
od;
\end{lstlisting}
\end{script}

\begin{script}\label{gap-2a7}
This code is used and explained in the proof of  \hyperref[prop:2a7]{Proposition~\ref*{prop:2a7}}. Here, the group with ID $[16,14]$ is $C_2^4$.
\begin{lstlisting}[language=GAP]
for a in [1..5] do
	for i in [1..NrSmallGroups(2^a*7)] do
		G := SmallGroup(2^a*7,i);
		correct_subgroups := true;
		
		for c in ConjugacyClassesSubgroups(G) do
			C := Representative(c);
			if IdGroup(C) = [16,14] then
				correct_subgroups := false;
			fi;
		od;
	
		if IsAbelian(G) = false
		and IsNormal(G,SylowSubgroup(G,2)) 
		and IsInt(Order(Center(G))/7) = false 
		and correct_subgroups then
			Print(IdGroup(G)," ", StructureDescription(G), "\n"); 
		fi;
	od;
od;
\end{lstlisting}
\end{script}

\begin{script} \label{gap-non-solv}
We use this code to show that every non-solvable group of order $336$, $504$ or $1008$ contains $C_2^3$ as a subgroup. Here, the group with ID $[8,5]$ is $C_2^3$. Running the code gives 'true' as an output, and hence every non-solvable group of given order indeed contains $C_2^3$.
\begin{lstlisting}[language=GAP]
for m in [336, 504, 1008] do
	for i in [1..NrSmallGroups(m)] do
		G := SmallGroup(m,i);
		S := SylowSubgroup(G,2);
		c2c2c2_contained := false;
		non_solv_list := [];
		c2c2c2_list := [];
		
		if IsSolvable(G) = false then
			Add(non_solv_list, IdGroup(G));
			
			for c in ConjugacyClassesSubgroups(S) do
				C := Representative(c);
				if IdGroup(C) = [8,5] then
					c2c2c2_contained := true;
				fi;
			od;
		
			if c2c2c2_contained = true then
				Add(c2c2c2_list, IdGroup(G));
			fi;
		fi;
	od;
od;

non_solv_list = c2c2c2_list;
\end{lstlisting}
\end{script}

\end{appendices}
\end{document}